\documentclass[a4paper,12pt]{article}

\usepackage{amsmath,amssymb,amsthm,amsfonts}
\usepackage{mathtools}

\usepackage[hdivide={1in,*,1in},divide={1in,*,1in}]{geometry}
\usepackage{tikz}
\usetikzlibrary{decorations.markings}
\usetikzlibrary{arrows}
\usetikzlibrary{arrows.meta}
\usetikzlibrary{intersections}
\usetikzlibrary{calc}

\usepackage{fntsymbols}
\usepackage{pictures}

\usepackage{array}
\usepackage{libertine}
\usepackage[libertine]{newtxmath}
 
\usepackage{enumitem}

\setlist{noitemsep,topsep=0pt,parsep=0pt,partopsep=0pt}
\setenumerate{label={\arabic*)}}

\usepackage{amsmath,amssymb,amsthm,amsfonts}

\allowdisplaybreaks

\numberwithin{equation}{section}

\newtheorem{thm}{Theorem}[section]
\newtheorem{lem}[thm]{Lemma}
\newtheorem{prop}[thm]{Proposition}
\newtheorem{cor}[thm]{Corollary}
\newtheorem{exer}[thm]{Exercise}

\theoremstyle{definition}

\newtheorem{defn}[thm]{Definition}
\newtheorem{exam}[thm]{Example}

\theoremstyle{remark}

\newtheorem{rmk}[thm]{Remark}

\theoremstyle{plain}
\newtheorem{itheorem}{Theorem}
\renewcommand{\theitheorem}{\Alph{itheorem}}
\newtheorem{icor}[itheorem]{Corollary}

\long\def\pretheorem#1#2#3{%
	\expandafter\long\expandafter\edef\csname pretheorem:#2\endcsname{%
	\noexpand\begingroup
		\noexpand\setcounter{itheorem}{\the\value{itheorem}}%
		\unexpanded{\begin{itheorem}\label{#2}#3\end{itheorem}}%
	\noexpand\endgroup}%
	\stepcounter{itheorem}%
	\theoremstyle{plain}%
	\newtheorem*{PreTheorem\theitheorem}{#1~\ref{#2}}%
	\begin{PreTheorem\theitheorem}#3\end{PreTheorem\theitheorem}
}
\def\insertpretheorem#1{\csname pretheorem:#1\endcsname}

\newenvironment{hint}%
	{\small\textnormal\bgroup[Hint: \ignorespaces}%
	{]\egroup}%

\usepackage[
	colorlinks,
	pdfauthor={Beliakova, Hogancamp, Putyra, Wehrli},
	pdftitle={On functoriality of sl(2) tangle homology}
]{hyperref}

\usepackage[all]{hypcap}

\definecolor{internalLink}{rgb}{0,0,0.5}
\definecolor{citeLink}{rgb}{0,0.5,0}
\definecolor{urlLink}{rgb}{0,0.5,0.5}

\hypersetup{%
	linkcolor=internalLink,
	citecolor=citeLink,
	urlcolor=urlLink
}

\tikzset{%
	mid-arrow/.style 2 args={%
		decoration={%
			markings,
			mark=at position #1 with {\arrow{#2}}%
		},
		postaction={decorate}
	},
	->-/.style={mid-arrow={#1}{>}},           ->-/.default=0.5,
	-<-/.style={mid-arrow={#1}{<}},           -<-/.default=0.5,
	-^-/.style={mid-arrow={#1}{Bar[left]}},   -^-/.default=0.5,
	-v-/.style={mid-arrow={#1}{Bar[right]}},  -v-/.default=0.5,
}%

\colorlet{V1lineColor}{blue}
\colorlet{V2lineColor}{red!80!blue}
\colorlet{V0lineColor}{red!80!blue}
\colorlet{V2innerColor}{red!30}

\colorlet{BCshade}{black!15}

\tikzset{%
	V0/.style={%
		V0lineColor,
		line width=0.75pt,
		>={To[width=7pt, length=5pt]},
		dashed
	},
	V1/.style={%
		V1lineColor,
		line width=1.25pt,
		>={To[width=7pt, length=5pt]}
	},
	V2/.style={%
		V2innerColor,
		draw=V2lineColor,
		double=V2innerColor,
		double distance=1pt,
		line width=0.7pt,
		>={To[width=9pt]}
	},
	V1dot/.style={%
		V1lineColor,
		line width=1.25pt
	},
	V2dot/.style={%
		fill=V2innerColor,
		draw=V2lineColor,
		line width=0.75pt
	},
	BCedge/.style={%
		V1lineColor,
		line width=1.25pt,
		>={To[width=7pt, length=5pt]}
	},
	BCedge+/.style={%
		V2lineColor,
		line width=1.25pt,
		>={To[width=7pt, length=5pt]}
	}
}%

\colorlet{seamColor}{red!50!blue}
\colorlet{seamInnerColor}{seamColor!60!white}

\tikzset{%
	1facetFront/.style={blue!60!black, fill opacity=0.5, fill=blue!50},
	2facetFront/.style={red!80!blue,   fill opacity=0.5, fill=red!60},
	1facetBack/.style= {blue!60!black, fill opacity=0.5, fill=blue!30},
	2facetBack/.style= {red!80!blue,   fill opacity=0.5, fill=red!40},
	2facetInner/.style={red!80!blue,   fill opacity=0.5, fill=red!20},
	1facetLine/.style= {thin, draw=blue!60!black},
	2facetLine/.style= {thin, draw=red!80!blue},
	seam/.style={
		draw=seamColor,
		thin
	}
}

\tikzset{%
	web/bdry/.style={draw=black!80,thin},
	web/bg/.style={fill=black!10}
}

\tikzset{%
	arc hline/.style={
		color=black!60,
		thin
	},
	arc platform/.style={
		color=red,
		dashed,
		very thick
	},
	help line/.style={
		dashed,
		thin,
		color=black!60
	}
}%

\def\centertikz{\vcenter\bgroup\hbox\bgroup\tikzpicture}
\def\endcentertikz{\endtikzpicture\egroup\egroup}

\newcommand*\ric{\sb{}\kern-\scriptspace }
\newcommand*\rics{\!}

\makeatletter
\def\arXiv#1{\arXiv@#1.}
\def\arXiv@#1.#2#3#4#5.{%
	\ifx\relax#5\relax
		\href{http://front.math.ucdavis.edu/#1.5#2#3#4}{arXiv:math/#1#2#3#4}%
	\else
		\href{http://front.math.ucdavis.edu/#1.#2#3#4#5}{arXiv:#1.#2#3#4#5}%
	\fi
}
\makeatother

\makeatletter

\def\blank{-}

\begingroup
	\def\definering#1#2{\protected\gdef#1{\@ifstar{#2^*}{#2}}}%
	\definering\scalars\Bbbk
	\definering\Fld{\mathbb F}
	\definering\Z{\mathbb Z}
	\definering\Q{\mathbb Q}
	\definering\R{\mathbb R}
	\definering\C{\mathbb C}
	\definering\Zq{\mathbb Z[q^{\pm1}]}
	\definering\Zhtq{\mathbb Z[h,t,q^{\pm1}]}
\endgroup

\newcommand*\LieSL[1][2]{\mathfrak{sl}_{#1}}
\newcommand*\LieGL[1][2]{\mathfrak{gl}_{#1}}

\newcommand*{\Uqsl}[1][2]{\mathcal U_q(\mathfrak{sl}_{#1})}

\def\Disk{\@ifnextchar^{\mathbb D}{\mathbb D^2}}
\def\Ball{\@ifnextchar^{B}{B^3}}
\def\S{\@ifnextchar^{\mathbb S}{\mathbb S^1}}
\def\Ann{\mathbb A}

\def\mfld{\@ifnextchar^{M}{M^n}}%

\def\RxI{\R\times[0,1]}
\def\DxI{\Disk\times[0,1]}

\def\setcomponent@#1{%
	\@ifnextchar[{\setcomponent@weight#1}{\setcomponent@weight#1[]}
}
\def\setcomponent@weight#1[#2]#3{%
	\setcomponent@draw#1[#2]#3||\relax
}
\def\setcomponent@draw#1[#2]#3|#4|#5\relax{%
	\mathcal{#1}^{#4}_{#3}%
	\ifx\relax#2\relax\else(#2)\fi
}

\def\tangles@#1#2|#3|#4\relax{%
	{#1}^{#3}_{#2}%
}

\def\tangles#1{\tangles@{\mathcal T\!\mathit{an}}#1||\relax}
\def\ftangles#1{\tangles@{\mathcal{FT}\!\mathit{an}}#1||\relax}
\def\dtangles#1{\tangles@{\widetilde{\mathcal T}\!\mathit{an}}#1||\relax}

\def\gmatchings{\setcomponent@{GM}}

\def\revmatching#1{#1^!}

\DeclareMathOperator{\rk}{rk}		%	rank
\DeclareMathOperator{\tr}{tr}		% trace of a function

\def\tensor{\@ifstar\tensor@{\tensor@\otimes}}
\def\tensor@#1#2{\let\tensor@symb#1\tensor@@#2,\tensor@end,}
\def\tensor@@#1,{\tensor@expand#1\relax\tensor@@@}
\def\tensor@@@#1,{%
	\ifx\tensor@end#1\relax\else
		\tensor@symb\tensor@expand#1\relax
	\expandafter\tensor@@@\fi
}
\def\tensor@expand#1#2\relax{\ifx*#1#2\else#1_{#2}\fi}

\DeclareMathOperator\dcbtensor\heartsuit
\DeclareMathOperator\bdcbtensor\spadesuit

\def\utimes{\mytimes@\otimes}
\def\udtimes{\mytimes@{\hat\otimes}}

\def\mytimes@#1#2{%
	\mathchoice
		{\mytimes@@{#1}{#2}}%
		{\mytimes@@@{#1}{#2}\scriptstyle}%
		{\mytimes@@@{#1}{#2}\scriptscriptstyle}%
		{\mytimes@@@{#1}{#2}\scriptscriptstyle}%
}
\def\mytimes@clap#1#2{%
	\mathchoice
		{\mytimes@@clap{#1}{#2}}%
		{\mytimes@@@{#1}{#2}\scriptstyle}%
		{\mytimes@@@{#1}{#2}\scriptscriptstyle}%
		{\mytimes@@@{#1}{#2}\scriptscriptstyle}%
}

\def\mytimes@@#1#2{\mkern\thinmuskip\underset{\mkern-1mu#2}{#1}\mkern\thinmuskip}
\def\mytimes@@clap#1#2{\mkern\thinmuskip\underset{\mathclap{#2}}{#1}\mkern\thinmuskip}
\def\mytimes@@@#1#2#3{\mkern\thinmuskip{#1}_{\mkern-2mu\raisebox{-1pt}{$\m@th#3#2$}}\mkern\thinmuskip}

\def\shdw{\@ifnextchar[{\shdw@@}{\shdw@}}
\def\shdw@#1{\langle\!\langle#1\rangle\!\rangle}
\def\shdw@@[#1]#2{%
	\langle\!\langle#2\rangle\!\rangle
	\raisebox{-0.75ex}{$\scriptstyle\!#1$}%
}

\def\DeclareStructure#1#2{\@ifnextchar[%
	{\DeclareStructure@{#1}{#2}}%
	{\DeclareStructure@{#1}{#2}[]}%
}
\def\DeclareStructure@#1#2[#3]#4{\expandafter\def\csname #1:#2\endcsname#3{#4}}

\def\DeclareCategory{\DeclareStructure{category}}

\newcommand*\set[1]{%
	\ifcsname set:#1\endcsname
		\csname set:#1\expandafter\endcsname
	\else
		\mathit{#1}%
	\fi
}

\newcommand*\cat[1]{%
	\ifcsname category:#1\endcsname
		\csname category:#1\expandafter\endcsname
	\else
		{\normalfont\textbf{#1}}%
	\fi
}

\newcommand*\ccat[1]{%
	\ifcsname bicategory:#1\endcsname
		\csname bicategory:#1\expandafter\endcsname
	\else
		\mathbf{#1}%
	\fi
}

\begingroup
	\def\defname#1{\expandafter\protected\expandafter\gdef\csname#1\endcsname}%
	\def\defmodulecats#1#2{%
		\defname{#1}{#2}%
		\defname{l#1}##1{%{##1\text{--}#2}%
			\raisebox{-0.35ex}{$\scriptstyle##1$}#2%
		}
		\defname{r#1}##1{%{#2\text{--}##1}%
			#2\raisebox{-0.35ex}{$\scriptstyle##1$}%
		}
		\defname{b#1}##1##2{%{##1\text{--}#2\text{--}##2}%
			\raisebox{-0.35ex}{$\scriptstyle##1$}%
			#2%
			\raisebox{-0.35ex}{$\scriptstyle##2$}%
		}
	}%
	\defmodulecats{Mod}{\cat{Mod}}
	\defmodulecats{PMod}{\cat{Proj}}
	\defmodulecats{ModG}{\mathrm{gr}\mkern-2mu\cat{Mod}}
	\defmodulecats{PModG}{\mathrm{gr}\cat{Proj}}
	\defname{Bimod}{\cat{Bimod}}
	\defname{gBimod}{\mathrm g\cat{Bimod}}
	\def\defbicats#1#2{%
		\defname{#1}{\ccat{#2}}%
		\defname{D#1}{\mathcal D(\ccat{#2})}%
		\defname{DD#1}{\mathcal D^-(\ccat{#2})}%
		\defname{g#1}{\@ifstar{\ccat{g#2}_0}{\ccat{g#2}}}%
		\defname{gD#1}{\mathcal D(\@ifstar{\ccat{g#2}_0)}{\ccat{g#2})}}%
		\defname{gDD#1}{\mathcal D^-(\@ifstar{\ccat{g#2}_0)}{\ccat{g#2})}}%

		\defname{e#1}{%
			\mathrlap{\hskip-0.15em\widetilde{\phantom{\rule{1.25em}{1.5ex}}}}%
			\ccat{#2}}%
		\defname{eD#1}{\mathcal D(%
			\mathrlap{\hskip-0.15em\widetilde{\phantom{\rule{1.25em}{1.5ex}}}}%
			\ccat{#2})}%
		\defname{eDD#1}{\mathcal D^-(%
			\mathrlap{\hskip-0.15em\widetilde{\phantom{\rule{1.25em}{1.5ex}}}}%
			\ccat{#2})}%
		\defname{eg#1}{%
			\mathrlap{\phantom{\ccat g}\hskip-0.15em\widetilde{\phantom{\rule{1.25em}{1.5ex}}}}%
			\@ifstar{\ccat{g#2}_0}{\ccat{g#2}}}%
		\defname{egD#1}{\mathcal D(%
			\mathrlap{\phantom{\ccat g}\hskip-0.15em\widetilde{\phantom{\rule{1.25em}{1.5ex}}}}%
			\@ifstar{\ccat{g#2}_0)}{\ccat{g#2})}}%
		\defname{egDD#1}{\mathcal D^-(%
			\mathrlap{\phantom{\ccat g}\hskip-0.15em\widetilde{\phantom{\rule{1.25em}{1.5ex}}}}%
			\@ifstar{\ccat{g#2}_0)}{\ccat{g#2})}}%
	}%
	\defbicats{BBimod}{Bimod}
	\defbicats{CCom}{Com}
	\defbicats{RRep}{Rep}
	\defbicats{DiagBimod}{DB}
\endgroup

\protected\def\Com{\mathrm{Com}}

\def\Com{\@ifnextchar^{\Com@}{\Com@@}}
\def\Com@^#1(#2){\mathit{Com}^{#1}(#2)}
\def\Com@@(#1){\mathit{Com}(#1)}

\def\HCom{\@ifnextchar^{\HCom@}{\HCom@@}}
\def\HCom@^#1(#2){\mathit{Com}^{#1}_{\!\raisebox{0.3ex}{$\scriptstyle/$}\mkern-2mu h}(#2)}
\def\HCom@@(#1){\mathit{Com}_{\!\raisebox{0.3ex}{$\scriptstyle/$}\mkern-2mu h}(#1)}

\def\SimplMod(#1){\cat{SMod}_{#1}}
\def\HoSimplMod(#1){\cat{HoSMod}_{#1}}

\newcommand*{\Hom}{\mathrm{Hom}}

\DeclareMathOperator\Tr{Tr}

\def\qHHSymb{\mathit{qHH}}
\newcommand*{\qHoHom}{\qHHSymb}

\newcommand*{\FKh}{\mathcal{F}_{\!\textit{Kh}}}

\newcommand*{\Fweb}{\mathcal{F}_{\mkern-3mu\mathit w}}

\def\KhBracket{\@ifstar\KhBracketScaled\KhBracketSimple}
\def\KhCube{\@ifstar\KhCubeScaled\KhCubeSimple}

\newcommand*{\KhCubeScaled}[1]{\mathcal{I}\left(#1\right)}
\newcommand*{\KhCubeSimple}[1]{\mathcal{I}(#1)}
\newcommand*{\KhBracketScaled}[1]{\left\llbracket#1\right\rrbracket}
\newcommand*{\KhBracketSimple}[1]{\llbracket#1\rrbracket}

\def\wKhBracket{\@ifstar\wKhBracketScaled\wKhBracketSimple}
\def\wKhCube{\@ifstar\wKhCubeScaled\wKhCubeSimple}

\newcommand*{\wKhCubeScaled}[1]{\mathcal{I}_{\mathrm F}\left(#1\right)}
\newcommand*{\wKhCubeSimple}[1]{\mathcal{I}_{\mathrm F}(#1)}
\newcommand*{\wKhBracketScaled}[1]{\left\llbracket#1\right\rrbracket_{\mathrm F}}
\newcommand*{\wKhBracketSimple}[1]{\llbracket#1\rrbracket_{\mathrm F}}

\newcommand*{\KhCom}{\mathit{CKh}}

\newcommand*{\Kh}{\mathit{Kh}}

\newcommand*{\qAKh}{\Kh_{\mkern-2mu\Ann_q}}

\newcommand*{\WebCom}{\mathit{C}_{\mathfrak W}}

\def\twosubs#1#2#3{%
	\ifx.#1%
		#3_{#2}%
	\else
		\raisebox{-0.85ex}{$\scriptstyle#1$}%
		#3%
		\raisebox{-0.85ex}{$\scriptstyle#2$}%
	\fi
}

\def\arcalg{\@ifnextchar|{\arcalg@comp}{\arcalg@whole}}
\def\arcalg@comp|#1#2|#3{\twosubs{#1}{#2}{\!\left(\arcalg@whole{#3}\right)\!}}
\def\arcalg@whole#1{H^{#1}}

\def\arcmod{\@ifnextchar|{\arcmod@comp}{\arcmod@whole}}
\def\arcmod@comp|#1#2|#3{\arcmod@whole{\ifx.#1\else\revmatching{#1}\fi#3#2}}
\def\arcmod@whole#1{\FKh(#1)}

\def\arcmap{\@ifnextchar|{\arcmap@comp}{\arcmap@whole}}
\def\arcmap@comp|#1#2|#3{\arcmap@whole{\ifx.#1\else\revmatching{#1}\fi#3#2}}
\def\arcmap@whole#1{\FKh(#1)}

\def\CKalg{\@ifnextchar|{\CKalg@comp}{\CKalg@whole}}
\def\CKalg@comp|#1#2|#3{\twosubs{#1}{#2}{\!\left(\CKalg@whole{#3}\right)\!}}
\def\CKalg@whole#1{A^{#1}}

\def\CKmod{\@ifnextchar|{\CKmod@comp}{\CKmod@whole}}
\def\CKmod@comp|#1#2|#3{\CKmod@whole{\ifx.#1\else\revmatching{#1}\fi#3#2}}
\def\CKmod@whole#1{C_{CK}(#1)}

\def\CKmap{\@ifnextchar|{\CKmap@comp}{\CKmap@whole}}
\def\CKmap@comp|#1#2|#3{\twosubs{#1}{#2}{(\CKmap@whole{#3})}}
\def\CKmap@whole#1{\@ifnextchar'{\varphi_{#1}}{\varphi_{#1}^{\mathstrut}}}

\def\arctimes#1{\mytimes@clap\otimes{\arcalg{#1}}}
\def\CKtimes#1{\mytimes@clap\otimes{\CKalg{#1}}}

\def\arcdtimes#1{\mytimes@clap{\hat\otimes}{\arcalg{#1}}}
\def\CKdtimes#1{\mytimes@clap{\hat\otimes}{\CKalg{#1}}}

\def\webtimes#1{\mytimes@clap\otimes{\webalg{#1}}}
\def\qtwebtimes#1{\mytimes@clap\otimes{\qtwebalg{#1}}}

\def\webalg{\@ifnextchar|{\webalg@comp}{\webalg@whole}}
\def\webalg@comp|#1#2|#3{\twosubs{#1}{#2}{\!\left(\webalg@whole{#3}\right)\!}}
\def\webalg@whole#1{\mathfrak{W}^{#1}}

\def\webmod{\@ifnextchar|{\webmod@comp}{\webmod@whole}}
\def\webmod@comp|#1#2|#3{\twosubs{#1}{#2}{\webmod@whole{#3}}}
\def\webmod@whole#1{\Fweb^\circ(#1)}

\def\webmap{\@ifnextchar|{\webmap@comp}{\webmap@whole}}
\def\webmap@comp|#1#2|#3{\twosubs{#1}{#2}{\webmap@whole{#3}}}
\def\webmap@whole#1{\Fweb^\circ(#1)}

\def\qtwebalg{\@ifnextchar|{\qtwebalg@comp}{\qtwebalg@whole}}
\def\qtwebalg@comp|#1#2|#3{\twosubs{#1}{#2}{\!\left(\qtwebalg@whole{#3}\right)\!}}
\def\qtwebalg@whole#1{\mathfrak{A}^{#1}}

\def\qtwebmod{\@ifnextchar|{\qtwebmod@comp}{\qtwebmod@whole}}
\def\qtwebmod@comp|#1#2|#3{\twosubs{#1}{#2}{\qtwebmod@whole{#3}}}
\def\qtwebmod@whole#1{C_{\mathfrak A}(#1)}

\def\qtwebmap{\@ifnextchar|{\qtwebmap@comp}{\qtwebmap@whole}}
\def\qtwebmap@comp|#1#2|#3{\twosubs{#1}{#2}{\qtwebmap@whole{#3}}}
\def\qtwebmap@whole#1{\@ifnextchar'{\varphi_{#1}}{\varphi_{#1}^{\mathstrut}}}

\def\floor#1{\lfloor #1\rfloor}

\def\bdry{\Sigma}
\def\web{\omega}

\def\dbltan{\alpha}

\def\dblcob{W}
\def\sgn{\mathrm{sgn}\mskip\thinmuskip}

\def\Web{\@ifnextchar[\WebS\Web@}
\def\Web@{\cat{Web}}
\def\WebS[#1]{\Web@^{(#1)}}
\def\clWeb{\Web@^{\mathit{cl}}}

\def\Foam{\@ifnextchar[\FoamS\Foam@}
\def\Foam@{\cat{Foam}}
\def\FoamS[#1]{\Foam@^{(#1)}}
\def\clFoam{\Foam@^{\mathit{cl}}}

\def\web{\omega}
\def\foam{S}

\def\Blanchet{Z}

\def\cupfoam{\mathrm{cup}}
\def\capfoam{\mathrm{cap}}

\let\foamequiv\doteq

\def\cupfoam@#1{\@ifnextchar*{\cupfoam@get{#1}}{\@ifnextchar+{\cupfoam@getx{#1}}{\cupfoam@get{#1}}}}

\def\cupfoam@getx#1+#2{\cupfoam@get{#1#2}}
\def\cupfoam@get#1{\@ifnextchar[{\cupfoam@dots{#1}}{\cupfoam@nodots{#1}}}

\def\cupfoam@dots#1[#2]#3{#1(#3;#2)}
\def\cupfoam@nodots#1#2{#1(#2)}

\def\extpower{\@ifnextchar^\extpowersup\extpowernosup}%
\def\extpowersup^#1{\mbox{\Large $\wedge$}^{\mkern-6mu#1}}%
\def\extpowernosup{\mbox{\Large $\wedge$}}%

\def\Diff{\mathit{Diff}\@ifnextchar_{\!}{}}

\DeclareCategory{Web}{{\normalfont\textbf{Web}}}
\DeclareCategory{Foam}{{\normalfont\textbf{Foam}}}
\DeclareCategory{pFoam}{%
	\mskip-2mu{\mathrlap{\widetilde{\phantom{\normalfont\textbf{Fo}}}}}%
	\mskip 2mu{\normalfont\textbf{Foam}}}

\DeclareCategory{kMod}{{\normalfont\scalars{-}\textbf{Mod}}}

\def\EqSymb{E}
\def\EqFunc{\mathcal E}

\def\flip#1{\@ifnextchar_{\mathrlap{\phantom{#1}^!}#1}{#1^!}}
\def\dotted#1{\@ifnextchar_{\mathrlap{\phantom{#1}^\bullet}#1}{#1^\bullet}}

\makeatother
\NewFntSymbol{PosCr}[r]{%
		\draw[->](0.5,-0.5) -- (-0.55,0.55);
		\draw[fntborder](-0.5,-0.5) -- (0.45,0.45);
		\draw[->](0.45,0.45) -- (0.55,0.55);
}%

\NewFntSymbol{NegCr}[r]{%
		\draw[->](-0.5,-0.5) -- (0.55,0.55);
		\draw[fntborder](0.5,-0.5) -- (-0.45,0.45);
		\draw[->](-0.45,0.45) -- (-0.55,0.55);
}%

\protected\def\bdrypoints#1{\begin{tikzpicture}[x=1em, baseline=-0.5ex]
	\def\bdrypointx{0}%
	\drawbdrypoint#1,,
\end{tikzpicture}}
\def\drawbdrypoint#1,{
	\ifx\relax#1\relax\else
		\webendpoint#1(\bdrypointx,0);
		\edef\bdrypointx{\numexpr\bdrypointx+1}%
		\expandafter\drawbdrypoint
	\fi}

\makeatletter
\def\webdrawdots#1#2(#3)#4{%
	\ifcase#1\or
		\fill[V1] (#3) circle[radius=3pt];
	\or
		\draw[V1,fill=white,line width=0.5pt] (#3) circle[radius=3pt];
	\fi
	\ifx#4;\expandafter\@gobble\else\expandafter\@firstofone\fi
	{\webdrawdots#2#4}%
}
\def\webendpoint{\@ifnextchar[\webendpoint@{\webendpoint@[above]}}
\def\webendpoint@[#1]#2#3(#4){%
	\filldraw[V#3dot] (#4) circle[radius=\the\dimexpr#3pt/2+1.5pt\relax]
		node[#1,text=V#3lineColor] {$\scriptstyle#2\vphantom+$}}
\makeatother

\DeclarePictureFamily{web}

\defwebpict{merge}(-1,-0.7)--(1,1) M=0.15 {%
	\draw[V2] (0,0) -- (0,1);
	\draw[V1] (0,0) arc[start angle=90, end angle=150, x radius=0.8, y radius=1.4];
	\draw[V1] (0,0) arc[start angle=90, end angle= 30, x radius=0.8, y radius=1.4]; 
}
\defwebpict{merge up}(-1,-0.7)--(1,1) M=0.15 {%
	\draw[V2,->-=0.7] (0,0) -- (0,1);
	\draw[V1,-<-=0.7] (0,0) arc[start angle=90, end angle=150, x radius=0.8, y radius=1.4];
	\draw[V1,-<-=0.7] (0,0) arc[start angle=90, end angle= 30, x radius=0.8, y radius=1.4];
}
\defwebpict{merge down}(-1,-1)--(1,0.7) M=-0.15 {%
	\draw[V2,->-=0.7] (0,0) -- (0,-1);
	\draw[V1,-<-=0.7] (0,0) arc[start angle=270, end angle=210, x radius=0.8, y radius=1.4];
	\draw[V1,-<-=0.7] (0,0) arc[start angle=270, end angle=330, x radius=0.8, y radius=1.4];
}

\defwebpict{split}(-1,-1)--(1,0.7) M=-0.15 {%
	\draw[V2] (0,0) -- (0,-1);
	\draw[V1] (0,0) arc[start angle=270, end angle=210, x radius=0.8, y radius=1.4];
	\draw[V1] (0,0) arc[start angle=270, end angle=330, x radius=0.8, y radius=1.4];
}
\defwebpict{split up}(-1,-1)--(1,0.7) M=-0.15 {%
	\draw[V2,-<-=0.7] (0,0) -- (0,-1);
	\draw[V1,->-=0.7] (0,0) arc[start angle=270, end angle=210, x radius=0.8, y radius=1.4];
	\draw[V1,->-=0.7] (0,0) arc[start angle=270, end angle=330, x radius=0.8, y radius=1.4];
}
\defwebpict{split down}(-1,-0.7)--(1,1) M=0.15 {%
	\draw[V2,-<-=0.7] (0,0) -- (0,1);
	\draw[V1,->-=0.7] (0,0) arc[start angle=90, end angle=150, x radius=0.8, y radius=1.4];
	\draw[V1,->-=0.7] (0,0) arc[start angle=90, end angle= 30, x radius=0.8, y radius=1.4];
}
 
\defwebpict{circle}(-0.5,-0.5)--(0.5,0.5) M=0 [#1]=[V1]{%
	\draw[#1] (0,0) circle[radius=0.4];
}

\defwebpict{vline}(-0.3,-1)--(0.3,1) M=0 [#1]=[V1] {%
	\draw[#1] (0,-1) -- (0,1);
}
\defwebpict{line}(-0.3,-0.5)--(0.3,0.5) M=0 [V#1#2]=[V1]{%
	\draw[V#1] (-0.2,-0.4) -- (0.2,0.4);
	\webdrawdots#20(0,0);
}

\defwebpict{T-shape}(-0.1,-0.8)--(0.8,0.8) M=0 {%
	\draw[V2] (0,-0.3) -- (0.7,-0.3)
	          (0, 0.3) -- (0.7, 0.3);
	\draw[V1] (0,-0.8) -- (0,0.8);
}

\defwebpict{line and cup}(-0.1,-0.8)--(0.8,0.8) M=0 {%
	\draw[V2] (0.7,0.3) arc[x radius=0.5, y radius=0.3,
	                        start angle=90, end angle=270];
	\draw[V1] (0,-0.8) -- (0,0.8);
}

\defwebpict{<bigon}(-0.5,-1)--(0.5,1) M=0 {%
	\draw[V2] (0.25pt,-0.5) .. controls (-0.5,-0.2) and (-0.5, 0.2) ..
	          (0.25pt, 0.5);
	\draw[V1] (0,-1) -- (0,-0.5)
	          (0, 0.5) -- (0, 1)
	          (0, 0.5) .. controls (0.5, 0.2) and (0.5,-0.2) .. (0,-0.5);
}

\defwebpict{bigon>}(-0.5,-1)--(0.5,1) M=0 {%
	\draw[V2] (-0.25pt,-0.5) .. controls ( 0.5,-0.2) and ( 0.5, 0.2) ..
	          (-0.25pt, 0.5);
	\draw[V1] (0,-1) -- (0,-0.5)
	          (0, 0.5) -- (0, 1)
	          (0, 0.5) .. controls (-0.5, 0.2) and (-0.5,-0.2) .. (0,-0.5);
}

\defwebpict{bubble}(-0.5,-1)--(0.5,1) M=0 {%
	\draw[V2] (0,-1) -- (0,-0.4) (0,0.4) -- (0,1);
	\draw[V1] (0,0) circle[radius=0.4];
}

\defwebpict{vert arcs}(-0.6,-0.6)--(0.6,0.6) M=0 {%
	\draw[V2,->-=0.85] (-0.5,-0.5) .. controls (-0.1,0) .. (-0.5, 0.5);
	\draw[V2,->-=0.85] ( 0.5, 0.5) .. controls ( 0.1,0) .. ( 0.5,-0.5);
}

\defwebpict{hor arcs}(-0.6,-0.6)--(0.6,0.6) M=0 {%
	\draw[V2,->-=0.85] (-0.5,-0.5) .. controls (0,-0.1) .. ( 0.5,-0.5);
	\draw[V2,->-=0.85] ( 0.5, 0.5) .. controls (0, 0.1) .. (-0.5, 0.5);
}

\defwebpict{circles}(-1.1,-0.5)--(1.1,0.5) M=0 [V#1#2]=[V1]{%
	\draw[V#1]
		(-0.6,-0.4) arc[radius=0.4, start angle=-90, end angle=90]
		( 0.6,-0.4) arc[radius=0.4, start angle=270, end angle=90];
	\draw[V#1lineColor,dotted]
		(-0.6,-0.4) arc[radius=0.4, start angle=270, end angle=90]
		( 0.6,-0.4) arc[radius=0.4, start angle=-90, end angle=90];
	\webdrawdots #200(-0.2,0)(0.2,0);
}

\defwebpict{merged}(-1.1,-0.5)--(1.1,0.5) M=0 [#1]=[V1] {%
	\draw[#1]
		(-0.6,-0.4) .. controls (-0.3,-0.4) and (-0.3,-0.2) ..
		( 0.0,-0.2) .. controls ( 0.3,-0.2) and ( 0.3,-0.4) .. (0.6,-0.4)
		(-0.6, 0.4) .. controls (-0.3, 0.4) and (-0.3, 0.2) ..
		( 0.0, 0.2) .. controls ( 0.3, 0.2) and ( 0.3, 0.4) .. (0.6, 0.4);
	\draw[#1lineColor, dotted]
		(-0.6,-0.4) arc[radius=0.4, start angle=270, end angle=90]
		( 0.6,-0.4) arc[radius=0.4, start angle=-90, end angle=90];
}

\defwebpict{inn bigon}(-0.4,-0.9)--(0.8,0.9) M=0 [#1,#2,#3]=[V0,V2,+]{%
	\fill[black!10]
		(0.5, 0.8) arc[radius=0.5, start angle=0, end angle=-90]
		           arc[radius=0.3, start angle=90, end angle=270]
					  arc[radius=0.5, start angle=90, end angle=0]
					  -| (0.7,0.8);
	\draw[V1] (0,-0.8) -- (0,0.8);
	\draw[#2] (0.5,-0.8) arc[radius=0.5, start angle= 0, end angle=90]
	          (0,  +0.3) arc[radius=0.5, start angle=-90, end angle=0];
	\ifx#3-\relax
		\draw[#2,<-] (0,-0.8) ++(30:0.5) -- ++(132:1pt);
	\else\ifx#3+\relax
		\draw[#2,<-] (0,0.8) ++(-30:0.5) -- ++(-132:1pt);
	\fi\fi
	\draw[#1] (0, 0.3) arc[radius=0.3, start angle=90, end angle=270];
}

\defwebpict{ext bigon}(-0.7,-0.9)--(0.8,0.9) M=0 [#1,#2,#3]=[V1,V0,+]{%
	\fill[black!10]
		(0.5, 0.8) arc[radius=0.5, start angle= 0, end angle=-90]
		           arc[radius=0.3, start angle=90, end angle=270]
					  arc[radius=0.5, start angle=90, end angle=0]
					  -| (-0.6,0.8);
	\draw[V1] (0,-0.8) -- (0,0.8);
	\draw[#2] (0,-0.3) arc[radius=0.5, start angle= 90, end angle=0]
	          (0, 0.3) arc[radius=0.5, start angle=-90, end angle=0];
	\ifx#3-\relax
		\draw[#2,<-] (0,-0.8) ++(30:0.5) -- ++(132:1pt);
	\else\ifx#3+\relax
		\draw[#2,<-] (0,0.8) ++(-30:0.5) -- ++(-132:1pt);
	\fi\fi
	\draw[#1] (0, 0.3) arc[radius=0.3, start angle= 90, end angle=270];
}

\defwebpict{ext cup}(-0.5,-0.9)--(0.8,0.9) M=0 [#1,#2]=[V2,+]{%
	\fill[black!10]
		(0.5, 0.8) .. controls (0.5, 0.3) and (0.3, 0.3) .. (0.3, 0)
	              .. controls (0.3,-0.3) and (0.5,-0.3) .. (0.5,-0.8) -| (-0.4, 0.8);
	\draw[V1] (0,-0.8) -- (0,0.8);
	\draw[#1] (0.5,0.8) .. controls (0.5, 0.3) and (0.3, 0.3) .. (0.3, 0)
	                    .. controls (0.3,-0.3) and (0.5,-0.3) .. (0.5,-0.8);
	\ifx#2+\relax
		\draw[#1,->] (0.3785, 0.2803) -- (0.4, 0.325);
	\else\ifx#2-\relax
		\draw[#1,->] (0.3785,-0.2803) -- (0.4,-0.325);
	\fi\fi
}

\defwebpict{inn cup}(-0.5,-0.9)--(0.8,0.9) M=0 [#1,#2]=[V2,+]{%
	\fill[black!10]
		(0.5, 0.8) .. controls (0.5, 0.3) and (0.3, 0.3) .. (0.3, 0)
	              .. controls (0.3,-0.3) and (0.5,-0.3) .. (0.5,-0.8) -| (0.7, 0.8);
	\draw[V1] (0,-0.8) -- (0,0.8);
	\draw[#1] (0.5,0.8) .. controls (0.5, 0.3) and (0.3, 0.3) .. (0.3, 0)
		                .. controls (0.3,-0.3) and (0.5,-0.3) .. (0.5,-0.8);
	\ifx#2+\relax
		\draw[#1,->] (0.3785, 0.2803) -- (0.4, 0.325);
	\else\ifx#2-\relax
		\draw[#1,->] (0.3785,-0.2803) -- (0.4,-0.325);
	\fi\fi
}

\defwebpict{nested}(-0.8,-0.8)--(0.8,0.8) M=0 [#1]=[V2]{%
	\draw[#1]
		(60:0.3) arc[radius=0.3, start angle=60, end angle=-60]
		(30:0.7) arc[radius=0.7, start angle=30, end angle=-30];
	\draw[#1lineColor,dotted]
		(60:0.3) arc[radius=0.3, start angle=60, end angle=300]
		(30:0.7) arc[radius=0.7, start angle=30, end angle=330];
}

\defwebpict{croissant}(-0.8,-0.8)--(0.8,0.8) M=0 [#1]=[V2]{%
	\draw[#1]
		( 60:0.3) .. controls ++(-60:0.3) and ++(-30:0.2) .. ( 30:0.7)
		(-60:0.3) .. controls ++( 60:0.3) and ++( 30:0.2) .. (-30:0.7);
	\draw[#1lineColor,dotted]
		(60:0.3) arc[radius=0.3, start angle=60, end angle=300]
		(30:0.7) arc[radius=0.7, start angle=30, end angle=330];
}

\defwebpict{arc res}(-0.5,-0.5)--(0.5,0.5) M=0 {%
	\draw[V1,->] (-0.5,-0.5) .. controls (-0.08,0) and (-0.08,0) .. (-0.5,0.5);
	\draw[V1,->] ( 0.5,-0.5) .. controls ( 0.08,0) and ( 0.08,0) .. ( 0.5,0.5);
}

\defwebpict{zip res}(-0.5,-0.5)--(0.5,0.5) M=0 {%
	\draw[V2] (0,-0.2) -- (0,0.2);
	\draw[V1] (-0.5,-0.5) .. controls (-0.3,-0.3) and (-0.2,-0.2) .. (0,-0.2)
	                      .. controls (0.2,-0.2) and (0.3,-0.3) .. (0.5,-0.5);
	\draw[V1,<->] (-0.5,0.5) .. controls (-0.3,0.3) and (-0.2,0.2) .. (0,0.2)
	                      .. controls (0.2,0.2) and (0.3,0.3) .. (0.5,0.5);
}

\defwebpict{blue-red-blue}(-0.6,-0.7)--(0.6,0.7) M=0 {%
	\fill[black!10] (0,-0.7) rectangle (0.6,0.7);
	\draw[V0] (0,-0.7) -- (0,0.7);
	\draw[V1] (-0.4,-0.7) .. controls (-0.2,0) .. (-0.4,0.7)
	          ( 0.4,-0.7) .. controls ( 0.2,0) .. ( 0.4,0.7);
}

\defwebpict{bigon-blue}(-0.6,-0.7)--(0.6,0.7) M=0 {%
	\fill[black!10] (0,-0.7) .. controls (0.5,0) .. (0,0.7) -| (0.6,-0.7);
	\draw[V0] (0,-0.7) .. controls (0.5,0) .. (0,0.7);
	\begin{scope}
		\clip (0.4,-0.7) .. controls (-0.1,0) .. (0.4,0.7) -- (0.6,0) -- cycle;
		\draw[V2] (0,-0.7) .. controls (0.5,0) .. (0,0.7);
	\end{scope}
	\draw[V1] (-0.4,-0.7) .. controls (-0.3,0) .. (-0.4,0.7)
	          ( 0.4,-0.7) .. controls (-0.1,0) .. ( 0.4,0.7);
}

\defwebpict{blue-bigon}(-0.6,-0.7)--(0.6,0.7) M=0 {%
	\fill[black!10] (0,-0.7) .. controls (-0.5,0) .. (0,0.7) -| (0.6,-0.7);
	\draw[V0] (0,-0.7) .. controls (-0.5,0) .. (0,0.7);
	\begin{scope}
		\clip (-0.4,-0.7) .. controls (0.1,0) .. (-0.4,0.7) -- (-0.6,0) -- cycle;
		\draw[V2] (0,-0.7) .. controls (-0.5,0) .. (0,0.7);
	\end{scope}
	\draw[V1] (-0.4,-0.7) .. controls (0.1,0) .. (-0.4,0.7)
	          ( 0.4,-0.7) .. controls (0.3,0) .. (0.4,0.7);
}

\defwebpict{zipped}(-0.6,-0.7)--(0.6,0.7) M=0 {%
	\fill[black!10] (0,-0.7) rectangle (0.6,0.7);
	\draw[V0] (0,-0.7) -- (0,0.7);
	\draw[V2] (0,-0.25) -- (0,0.25);
	\draw[V1] (-0.4, 0.7) .. controls (-0.2, 0.1) and (0.2, 0.1) .. (0.4, 0.7)
	          (-0.4,-0.7) .. controls (-0.2,-0.1) and (0.2,-0.1) .. (0.4,-0.7);
}

\defwebpict{joined circles}(-1.1,-0.5)--(1.1,0.5) M=0 [#1]=[]{%
	\draw[V0] (-0.6,0) -- (0.6,0);
	\draw[V2] (-0.2,0) -- (0.2,0);
	\draw[V1]
		(-0.6,-0.4) arc[radius=0.4, start angle=-90, end angle=90]
		( 0.6,-0.4) arc[radius=0.4, start angle=270, end angle=90];
	\draw[V1lineColor,dotted]
		(-0.6,-0.4) arc[radius=0.4, start angle=270, end angle=90]
		( 0.6,-0.4) arc[radius=0.4, start angle=-90, end angle=90];
	\webdrawdots #100(-0.3182,0.2828)(0.3182,0.2828);
}

\defwebpict{unjoined circles}(-1.1,-0.5)--(1.1,0.5) M=0 [#1]=[]{%
	\draw[V0] (0,-0.5) -- (0,0.5);
	\draw[V1]
		(-0.6,-0.4) arc[radius=0.4, start angle=-90, end angle=90]
		( 0.6,-0.4) arc[radius=0.4, start angle=270, end angle=90];
	\draw[V1lineColor,dotted]
		(-0.6,-0.4) arc[radius=0.4, start angle=270, end angle=90]
		( 0.6,-0.4) arc[radius=0.4, start angle=-90, end angle=90];
	\webdrawdots #100(-0.3182,0.2828)(0.3182,0.2828);
}

\defwebpict{joined merged}(-1.1,-0.5)--(1.1,0.5) M=0 {%
	\draw[V0] (0,-0.5) -- (0,0.5);
	\draw[V2] (0,-0.2) -- (0,0.2);
	\draw[V1]
		(-0.6,-0.4) .. controls (-0.3,-0.4) and (-0.3,-0.2) ..
		( 0.0,-0.2) .. controls ( 0.3,-0.2) and ( 0.3,-0.4) .. (0.6,-0.4)
		(-0.6, 0.4) .. controls (-0.3, 0.4) and (-0.3, 0.2) ..
		( 0.0, 0.2) .. controls ( 0.3, 0.2) and ( 0.3, 0.4) .. (0.6, 0.4);
	\draw[V1lineColor, dotted]
		(-0.6,-0.4) arc[radius=0.4, start angle=270, end angle=90]
		( 0.6,-0.4) arc[radius=0.4, start angle=-90, end angle=90];
}

\defwebpict{unjoined merged}(-1.1,-0.5)--(1.1,0.5) M=0 {%
	\draw[V0] (-0.6,0) -- (0.6,0);
	\draw[V1]
		(-0.6,-0.4) .. controls (-0.3,-0.4) and (-0.3,-0.2) ..
		( 0.0,-0.2) .. controls ( 0.3,-0.2) and ( 0.3,-0.4) .. (0.6,-0.4)
		(-0.6, 0.4) .. controls (-0.3, 0.4) and (-0.3, 0.2) ..
		( 0.0, 0.2) .. controls ( 0.3, 0.2) and ( 0.3, 0.4) .. (0.6, 0.4);
	\draw[V1lineColor, dotted]
		(-0.6,-0.4) arc[radius=0.4, start angle=270, end angle=90]
		( 0.6,-0.4) arc[radius=0.4, start angle=-90, end angle=90];
}
\NewFntSymbol{foam-1cup}(-0.5,-0.7)(0.5,0.2)<2.5>[r]{%
	\path[1facetBack] (-0.5,0)
			arc[start angle=180, end angle=0, x radius=0.5, y radius=0.15]
			.. controls (0.5,-1) and (-0.5,-1) .. (-0.5,0);
	\path[1facetFront] (-0.5,0)
			arc[start angle=180, end angle=360, x radius=0.5, y radius=0.15]
			.. controls (0.5,-1) and (-0.5,-1) .. (-0.5,0);
	\draw[1facetLine,line width=\fntlinewidth]
	    (0.0,0) ellipse[x radius=0.5, y radius=0.15]
			(0.5,0) .. controls (0.5,-1) and (-0.5,-1) .. (-0.5,0);
}
\NewFntSymbol{foam-1cup*}(-0.5,-0.7)(0.5,0.2)<2.5>[r]{%
	\path[1facetBack] (-0.5,0)
			arc[start angle=180, end angle=0, x radius=0.5, y radius=0.15]
			.. controls (0.5,-1) and (-0.5,-1) .. (-0.5,0);
	\path[1facetFront] (-0.5,0)
			arc[start angle=180, end angle=360, x radius=0.5, y radius=0.15]
			.. controls (0.5,-1) and (-0.5,-1) .. (-0.5,0);
	\draw[1facetLine,line width=\fntlinewidth]
	    (0.0,0) ellipse[x radius=0.5, y radius=0.15]
			(0.5,0) .. controls (0.5,-1) and (-0.5,-1) .. (-0.5,0);
	\fill[black] (0.2,-0.35) circle[radius=0.3ex];
}
\NewFntSymbol{foam-2cup}(-0.5,-0.7)(0.5,0.2)<2.5>[r]{%
	\path[2facetBack] (-0.5,0)
			arc[start angle=180, end angle=0, x radius=0.5, y radius=0.15]
			.. controls (0.5,-1) and (-0.5,-1) .. (-0.5,0);
	\path[2facetFront] (-0.5,0)
			arc[start angle=180, end angle=360, x radius=0.5, y radius=0.15]
			.. controls (0.5,-1) and (-0.5,-1) .. (-0.5,0);
	\draw[2facetLine,line width=\fntlinewidth]
	    (0.0,0) ellipse[x radius=0.5, y radius=0.15]
			(0.5,0) .. controls (0.5,-1) and (-0.5,-1) .. (-0.5,0);
}

\NewFntSymbol{foam-1cap}(-0.5,1)(0.5,-0.3)<2.5>[r]{%
	\path[1facetBack] (-0.5,0)
			arc[start angle=180, end angle=360, x radius=0.5, y radius=0.15]
			.. controls (0.5, 1) and (-0.5, 1) .. (-0.5,0);
	\draw[1facetLine,line width=0.5\fntlinewidth]
	    (0.5,0) arc[start angle=0, end angle=180, x radius=0.5, y radius=0.15];
	\path[1facetFront] (-0.5,0)
			arc[start angle=180, end angle=360, x radius=0.5, y radius=0.15]
			.. controls (0.5, 1) and (-0.5, 1) .. (-0.5,0);
	\draw[1facetLine,line width=\fntlinewidth]
			(0.5,0) .. controls (0.5,1) and (-0.5,1) .. (-0.5,0)
			arc[start angle=180, end angle=360, x radius=0.5, y radius=0.15] -- cycle;
}
\NewFntSymbol{foam-1cap*}(-0.5,1)(0.5,-0.3)<2.5>[r]{%
	\path[1facetBack] (-0.5,0)
			arc[start angle=180, end angle=360, x radius=0.5, y radius=0.15]
			.. controls (0.5, 1) and (-0.5, 1) .. (-0.5,0);
	\draw[1facetLine,line width=0.5\fntlinewidth]
	    (0.5,0) arc[start angle=0, end angle=180, x radius=0.5, y radius=0.15];
	\path[1facetFront] (-0.5,0)
			arc[start angle=180, end angle=360, x radius=0.5, y radius=0.15]
			.. controls (0.5, 1) and (-0.5, 1) .. (-0.5,0);
	\draw[1facetLine,line width=\fntlinewidth]
			(0.5,0) .. controls (0.5,1) and (-0.5,1) .. (-0.5,0)
			arc[start angle=180, end angle=360, x radius=0.5, y radius=0.15] -- cycle;
	\fill[black] (0.2,0.35) circle[radius=0.3ex];
}
\NewFntSymbol{foam-2cap}(-0.5,0.7)(0.5,-0.2)<2.5>[r]{%
	\path[2facetBack] (-0.5,0)
			arc[start angle=180, end angle=360, x radius=0.5, y radius=0.15]
			.. controls (0.5, 1) and (-0.5, 1) .. (-0.5,0);
	\draw[2facetLine,line width=0.5\fntlinewidth]
	    (0.5,0) arc[start angle=0, end angle=180, x radius=0.5, y radius=0.15];
	\path[2facetFront] (-0.5,0)
			arc[start angle=180, end angle=360, x radius=0.5, y radius=0.15]
			.. controls (0.5, 1) and (-0.5, 1) .. (-0.5,0);
	\draw[2facetLine,line width=\fntlinewidth]
			(0.5,0) .. controls (0.5,1) and (-0.5,1) .. (-0.5,0)
			arc[start angle=180, end angle=360, x radius=0.5, y radius=0.15] -- cycle;
}

\NewFntSymbol{foam-saddle-front}(-0.8,-0.2)(0.8,1.2)<2.5>[r]{%
	\path[2facetBack] (-0.5,0.2)
	    arc[start angle= 90, end angle=0, radius=0.2]
			arc[start angle=180, end angle=0,  x radius=0.3, y radius=0.5]
			arc[start angle=180, end angle=90, x radius=0.5, y radius=0.2]
			-- (0.8,1.2) -- (-0.5,1.2) -- cycle;
	\draw[2facetLine,line width=\fntlinewidth]
		(-0.5,0.8) -- (-0.5,1.2) -- (0.8,1.2) -- (0.8,0.2)
	    arc[start angle=90, end angle=127, x radius=0.5, y radius=0.2];
	\draw[2facetLine,line width=0.5\fntlinewidth]
		(-0.5,0.8) -- (-0.5,0.2) arc[start angle= 90, end angle=0, radius=0.2]
		( 0.3,0.0) arc[start angle=180, end angle=127, x radius=0.5, y radius=0.2];
	\draw[2facetFront,line width=\fntlinewidth] (-0.8,-0.2)
	    arc[start angle=270, end angle=360, x radius=0.5, y radius=0.2]
			arc[start angle=180, end angle=0,   x radius=0.3, y radius=0.5]
			arc[start angle=180, end angle=270, radius=0.2]
			-- (0.5,0.8) -- (-0.8,0.8) -- cycle;
}
\NewFntSymbol{foam-saddle-side}(-0.8,-1.2)(0.8,0.2)<2.5>[r]{%
	\path[2facetBack] (-0.5,0.2)
	    arc[start angle= 90, end angle=0, radius=0.2]
			arc[start angle=180, end angle=360, x radius=0.3, y radius=0.5]
			arc[start angle=180, end angle=90,  x radius=0.5, y radius=0.2]
			-- (0.8,-0.8) -- (-0.5,-0.8) -- cycle;
	\draw[2facetLine,line width=0.5\fntlinewidth]
		(0.5,-0.8) -- (-0.5,-0.8) -- (-0.5,-0.16);
	\path[2facetFront] (-0.8,-0.2)
	    arc[start angle=270, end angle=360, x radius=0.5, y radius=0.2]
			arc[start angle=180, end angle=360, x radius=0.3, y radius=0.5]
			arc[start angle=180, end angle=270, radius=0.2]
			-- (0.5,-1.2) -- (-0.8,-1.2) -- cycle;
	\draw[2facetLine, line width=\fntlinewidth] (-0.5,-0.16) -- (-0.5,0.2)
		arc[start angle=90,  end angle=0, radius=0.2]
		arc[start angle=360, end angle=270, x radius=0.5, y radius=0.2]
		-- (-0.8,-1.2) -- (0.5,-1.2) -- (0.5,-0.2)
		arc[start angle=270, end angle=180, radius=0.2]
		arc[start angle=180, end angle=90,  x radius=0.5, y radius=0.2]
		-- (0.8,-0.8) -- (0.5,-0.8)
		(-0.3,0) arc[start angle=180, end angle=360, x radius=0.3, y radius=0.5];
}

\NewFntSymbol{foam-detach-cup}(-0.4,-0.9)(1,0.9)<3>[r]{%
	\fill[1facetFront] (-0.3, 0.1) -- (0.3, 0.9) -- (0.3, -0.1) -- (-0.3, -0.9) -- cycle;
	\draw[1facetLine,line width=0.5\fntlinewidth]
		(-0.15,-0.7) -- (0.3,-0.1) -- (0.3,0.15);
	\draw[1facetLine,line width=\fntlinewidth]
		(-0.15,-0.7) -- (-0.3,-0.9) -- (-0.3,0.1) -- (0.3,0.9) -- (0.3,0.15);
	\fill[2facetBack] (0,0)
		arc[start angle=90, end angle=0, x radius=0.15, y radius=0.3] -| (0.9, 0.7)
		arc[start angle=90, end angle=180, x radius=0.5, y radius=0.2]
		arc[start angle=0, end angle=-90, x radius=0.4, y radius=0.5];
	\draw[2facetLine,line width=0.5\fntlinewidth] (0.15,-0.3) -- (0.6,-0.3);
	\draw[seam,line width=0.5\fntlinewidth]
		(0,0) arc[start angle=90, end angle=0, x radius=0.15, y radius=0.3];
	\fill[2facetFront] (0.4,0.5)
		arc[start angle=0, end angle=-90, x radius=0.4, y radius=0.5]
		arc[start angle=90, end angle=180, x radius=0.15, y radius=0.7]
		-| (0.6,0.3) arc[start angle=270, end angle=180, radius=0.2];
	\draw[2facetLine,line width=\fntlinewidth] (0.6,-0.3) -| (0.9,0.7)
		arc[start angle= 90, end angle=180, x radius=0.5, y radius=0.2]
		arc[start angle=180, end angle=270, radius=0.2] |- (-0.15,-0.7)
		(0.4,0.5) arc[start angle=0, end angle=-90, x radius=0.4, y radius=0.5];
	\draw[seam,line width=\fntlinewidth]
		(0,0) arc[start angle=90, end angle=180, x radius=0.15, y radius=0.7];
}
\NewFntSymbol{foam-attach-cup}(-0.4,-0.9)(1,0.9)<3>[r]{%
	\fill[1facetFront] (-0.3, 0.1) -- (0.3, 0.9) -- (0.3, -0.1) -- (-0.3, -0.9) -- cycle;
	\draw[1facetLine,line width=\fntlinewidth]
		(0.3,0.7) -- (0.3,0.9) -- (-0.3,0.1) -- (-0.3,-0.9) -- (0.27,-0.14);
	\draw[1facetLine,line width=0.5\fntlinewidth]
		(0.27,-0.14) -- (0.3,-0.1) -- (0.3,0.7);
	\fill[2facetBack] (0.15, 0.7) -| (0.9, -0.3)
		arc[start angle=90, end angle=180, x radius=0.5, y radius=0.2]
		arc[start angle=0, end angle=90, x radius=0.4, y radius=0.5]
		arc[start angle=270, end angle=360, x radius=0.15, y radius=0.7];
	\draw[2facetLine,line width=0.5\fntlinewidth]
		(0.9,-0.3) arc[start angle=90, end angle=180, x radius=0.5, y radius=0.2];
	\draw[seam,line width=0.5\fntlinewidth]
		(0.15,0.7) arc[start angle=360, end angle=270, x radius=0.15, y radius=0.7];		
	\draw[seam,line width=\fntlinewidth]
		(0.15,0.7) arc[start angle=360, end angle=325, x radius=0.15, y radius=0.7];
	\fill[2facetFront] (-0.15, 0.3) -| (0.6, -0.7)
		arc[start angle=270, end angle=180, radius=0.2]
		arc[start angle=0, end angle=90, x radius=0.4, y radius=0.5]
		arc[start angle=270, end angle=180, x radius=0.15, y radius=0.3];
	\draw[seam,line width=\fntlinewidth]
		(-0.15,0.3) arc[start angle=180, end angle=270, x radius=0.15, y radius=0.3];
	\draw[2facetLine,line width=\fntlinewidth]
		(0,0) arc[start angle=90, end angle=0, x radius=0.4, y radius=0.5]
		      arc[start angle=180, end angle=270, radius=0.2] |- (-0.15,0.3)
		(0.15,0.7) -| (0.9,-0.3)
			arc[start angle=90, end angle=125, x radius=0.5, y radius=0.2];
}

\DeclarePictureFamily{foam}

\def\foamdot(#1){\fill (#1) circle[radius=0.5ex]}
\def\foamdualdot(#1){\draw[line width=0.15ex] (#1) circle[radius=0.45ex]}

\def\foamorient(#1) Rx=#2 Ry=#3 S=#4 E=#5;{\begin{scope}[shift={(#1)}]
	\coordinate (arrowst) at (#5:#2 and #3);
	\ifnum #4 < #5
		\path (arrowst) ++(#5+90:#2 and #3) coordinate(arrowdir);
	\else
		\path (arrowst) ++(#5-90:#2 and #3) coordinate(arrowdir);
	\fi
	\draw[stealth-] ($(arrowst)!3pt!(arrowdir)$) -- (arrowst)
		arc[x radius=#2, y radius=#3, start angle=#5, end angle=#4];
\end{scope}}

\def\getlength#1(#2){%
  \path let \p{x}=(#2), \n{xlen}={veclen(\x{x},\y{x})}
  in \pgfextra{\xdef#1{\n{xlen}}}%
}

\def\foamarrow#1[#2]#3(#4){%
	\ifx\relax#3\relax
		\foamarrowx#1[#2]0({#4})%
	\else
		\foamarrowx#1[#2]#3({#4})%
	\fi\foamarrownext
}

\def\foamarrowx#1[#2]#3(#4){%
	\ifx<#1\relax
		\edef\foamarrownext{\unexpanded{\draw[#2, line width=0.75pt]}
			(#4) ++(#3:-2.5pt) arc[radius=5pt,
				start angle=\the\numexpr -90 + #3,
				end angle=\the\numexpr -25 + #3]
			(#4) ++(#3:-2.5pt) arc[radius=5pt,
				start angle=\the\numexpr 90 + #3,
				end angle=\the\numexpr 25 + #3]}%
	\else\ifx>#1\relax
		\edef\foamarrownext{\unexpanded{\draw[#2, line width=0.75pt]}
			(#4) ++(#3:2.5pt) arc[radius=5pt,
				start angle=\the\numexpr -90 + #3,
				end angle=\the\numexpr -155 + #3]
			(#4) ++(#3:2.5pt) arc[radius=5pt,
				start angle=\the\numexpr 90 + #3,
				end angle=\the\numexpr 155 + #3]}%
	\else\let\foamarrownext\relax\fi\fi
}

\def\foamdotsonsphere#1{%
	\ifcase 0#1\or
		\foamdot(0.2, 0.25);\or
		\foamdot(0.22, 0.25);\foamdot(0, 0.25);\or
		\foamdot(0.22, 0.25);\foamdot(0, 0.25);\foamdot(-0.22, 0.25);\or
		\foamdualdot(0.12,0.25);\or
		\foamdot(0.2,0.25);\foamdualdot(-0.05,0.25);
	\fi
}

\deffoampict{1sphere}(-0.6,-0.6)--(0.6, 0.6) M=0 [#1]=[] {%
	\fill[1facetBack] (0,0) circle[radius=0.5];
	\draw[1facetFront] (0,0) circle[radius=0.5];
	\draw[1facetLine] (-0.5,0)
		arc[start angle=180, end angle=360, x radius=0.5, y radius=0.15];
	\foamdotsonsphere{#1}
}

\deffoampict{2sphere}(-0.6,-0.6)--(0.6, 0.6) M=0 [#1]=[] {%
	\getlength\xradius({$(0.5,0) + (-1.7pt,0)$});
	\getlength\yradius({$(0,0.15) + (0,-1pt)$});
	\fill[2facetInner] (0,0) circle[radius=0.5];
	\draw[2facetFront](0,0) circle[radius=\xradius];
	\draw[2facetLine]
		(-\xradius,0)
			arc[start angle=180, end angle=360, x radius=\xradius, y radius=\yradius];
	\draw[2facetFront] (0,0) circle[radius=0.5];
	\draw[2facetLine]
		(-0.5 ,0) arc[start angle=180, end angle=360, x radius=0.5 , y radius=0.15];
	\foamdotsonsphere{#1}
}

\deffoampict{sphere and membrane}(-0.6,-0.6)--(0.6,0.6) M=0 [#1#2]=[>]{%
	\fill[1facetBack]  (0,0) circle[radius=0.5];
	\draw[2facetInner] (0,-0.85pt) ellipse[x radius=0.5, y radius=0.15];
	\draw[2facetFront] (0, 0.85pt) ellipse[x radius=0.5, y radius=0.15];
	\draw[1facetFront] (0,0) circle[radius=0.5];
	\draw[seam]
		(0.5,-0.85pt) arc[x radius=0.5, y radius=0.15, start angle=360, end angle=180]
		(0.5, 0.85pt) arc[x radius=0.5, y radius=0.15, start angle=360, end angle=180];
	\foamarrow#1[seam](0,-0.15);
	\foamdotsonsphere{#2}
}%

\deffoampict{sphere in plane}(-1.5,-0.6)--(1.5,0.6) M=0 [#1#2]=[>]{%
	\fill[1facetBack] (0.5,0)
			arc[radius=0.5, start angle=360, end angle=180]
			arc[x radius=0.5, y radius=0.15, start angle=180, end angle=360];
	\fill[1facetFront] (0.5,0)
			arc[radius=0.5, start angle=360, end angle=180]
			arc[x radius=0.5, y radius=0.15, start angle=180, end angle=360];
	\draw[1facetLine]
			(0.5,0) arc[radius=0.5, start angle=360, end angle=180];
	\fill[2facetInner] (0.5,-0.85pt)
			arc[x radius=0.5, y radius=0.15, start angle=0, end angle=360]
			-- ++(0.9,0.3) -- ++(-0.8,-0.6) -- ++(-2,0) -- ++(0.8, 0.6) -- ++(2,0)
			-- cycle;
	\draw[2facetLine] (0,-0.85pt)
		++(1.4,0.3) -- ++(-0.8,-0.6) -- ++(-2,0) -- ++(0.8, 0.6) -- ++(2,0) -- cycle;
	\fill[2facetFront] (0.5, 0.85pt)
			arc[x radius=0.5, y radius=0.15, start angle=0, end angle=360]
			-- ++(0.9,0.3) -- ++(-0.8,-0.6) -- ++(-2,0) -- ++(0.8, 0.6) -- ++(2,0)
			-- cycle;
	\draw[2facetLine] (0, 0.85pt)
		++(1.4,0.3) -- ++(-0.8,-0.6) -- ++(-2,0) -- ++(0.8, 0.6) -- ++(2,0) -- cycle;
	\draw[seam]
		(0,  -0.85pt) ellipse[x radius=0.5, y radius=0.15]
		(0.5, 0.85pt) arc[x radius=0.5, y radius=0.15, start angle=0, end angle=180];
	\fill[1facetBack] (0.5,0)
			arc[radius=0.5, start angle=0, end angle=180]
			arc[x radius=0.5, y radius=0.15, start angle=180, end angle=360];
	\fill[1facetFront] (0.5,0)
			arc[radius=0.5, start angle=0, end angle=180]
			arc[x radius=0.5, y radius=0.15, start angle=180, end angle=360];
	\draw[1facetLine] (0.5,0)
			arc[radius=0.5, start angle=0, end angle=180];
	\draw[seam]
		(0.5, 0.85pt) arc[x radius=0.5, y radius=0.15, start angle=360, end angle=180];
	\foamarrow#1[seam](0,-0.15);
	\foamarrow#1[2facetLine](-1,-0.3);
	\foamarrow#1[2facetLine]36(1.1,0.075);
	\foamdotsonsphere{#2}
}

\def\foamdotsonplane#1{%
	\ifcase 0#1\or
		\foamdot(0,0);\or
		\foamdot(-0.15,0);\foamdot(0.15,0);\or
		\foamdot(0,0);\foamdot(-0.25,0);\foamdot(0.25,0);\or
		\foamdualdot(0,0);\or
		\foamdualdot(-0.15,0);\foamdot(0.15,0);\or
		\foamdualdot(-0.15,0);\foamdualdot(0.15,0);
	\fi
}

\deffoampict{1plane}(-0.9,-0.5)--(0.9,0.5) M=0 [#1]=[]{%
	\draw[1facetFront]
		(0.85,0.4) -- ++(-0.45,-0.8) -- ++(-1.25,0) -- ++(0.45,0.8) -- cycle;
	\foamdotsonplane{#1}
}

\deffoampict{2plane}(-0.9,-0.5)--(0.9,0.5) M=0 [#1#2]=[x]{%
	\draw[2facetInner] ( 0.2pt,-0.85pt)
		++(0.85,0.4) -- ++(-0.45,-0.8) -- ++(-1.25,0) -- ++(0.45,0.8) -- cycle;
	\draw[2facetFront] (-0.2pt, 0.85pt)
		++(0.85,0.4) -- ++(-0.45,-0.8) -- ++(-1.25,0) -- ++(0.45,0.8) -- cycle;
	\foamarrow#1[2facetLine](-0.25,-0.4);
	\foamarrow#1[2facetLine]180(0.25,0.4);
	\foamdotsonplane{#2}
}

\deffoampict{1cup}(-0.7,-0.4)--(0.7,0.75) M=0.15 [#1#2]=[0]{%
	\fill[1facetBack] (-0.5,0.5)
		.. controls (-0.5,-0.5) and (0.5,-0.5) .. (0.5,0.5)
		arc[x radius=0.5, y radius=0.15, start angle=0, end angle=180];
	\fill[1facetFront] (-0.5,0.5)
		.. controls (-0.5,-0.5) and (0.5,-0.5) .. (0.5,0.5)
		arc[x radius=0.5, y radius=0.15, start angle=360, end angle=180];
	\draw[1facetLine] (-0.5,0.5) .. controls (-0.5,-0.5) and (0.5,-0.5) .. (0.5,0.5)
		(0,0.5) ellipse[x radius=0.5, y radius=0.15];
	\ifcase #1\or
		\foamdot(0.15,0);\or
		\foamdualdot(0.15,0);%
	\fi
	\foamarrow#2[1facetLine](0,0.35);
}

\deffoampict{2cup}(-0.7,-0.4)--(0.7,0.75) M=0.15 [#1#2]=[x]{%
	\getlength\innradx({$(0.5,0) + (-1.7pt,0)$});
	\getlength\innrady({$(0,0.15) + (0,-1pt)$});
	\fill[2facetInner] (-0.5,0.5)
		.. controls (-0.5,-0.5) and (0.5,-0.5) .. (0.5,0.5)
		arc[x radius=0.5, y radius=0.15, start angle=0, end angle=180];
	\fill[2facetBack] (-\innradx,0.5)
		.. controls (-\innradx,-\innradx) and (\innradx,-\innradx)
		.. (\innradx,0.5)
		arc[x radius=\innradx, y radius=\innrady, start angle=0, end angle=180];
	\draw[2facetLine] (-\innradx,0.5)
		.. controls (-\innradx,-\innradx) and (\innradx,-\innradx)
		.. (\innradx,0.5)
		(0,0.5) ellipse[x radius=\innradx, y radius=\innrady];
	\fill[2facetInner] (-\innradx,0.5)
		.. controls (-\innradx,-\innradx) and (\innradx,-\innradx)
		.. (\innradx,0.5)
		arc[x radius=\innradx, y radius=\innrady,
		    start angle=360, end angle=180];
	\fill[2facetFront] (-0.5,0.5)
		.. controls (-0.5,-0.5) and (0.5,-0.5) .. (0.5,0.5)
		arc[x radius=0.5, y radius=0.15, start angle=360, end angle=180];
	\draw[2facetLine] (-0.5,0.5) .. controls (-0.5,-0.5) and (0.5,-0.5) .. (0.5,0.5)
		(0, 0.5) ellipse[x radius=0.5, y radius=0.15];
	\foamarrow#1[2facetLine]({$(0,0.35) + (0,0.6pt)$});
	\ifcase 0#2\or
		\foamdot(0.15,0);\or
		\foamdualdot(0.15,0);%
	\fi
}

\deffoampict{1cup and disk}(-0.7,-0.8)--(0.7,0.65) B=-0.2 [#1#2]=[>]{%
	\fill[1facetBack] (-0.5,0.5) -- (-0.5,0)
		.. controls (-0.5,-1) and (0.5,-1) .. (0.5,0) -- (0.5,0.5)
		arc[x radius=0.5, y radius=0.15, start angle=0, end angle=180];
	\draw[2facetInner] (0,-0.85pt) ellipse[x radius=0.5, y radius=0.15];
	\draw[2facetFront] (0, 0.85pt) ellipse[x radius=0.5, y radius=0.15];
	\fill[1facetFront] (-0.5,0.5) -- (-0.5,0)
		.. controls (-0.5,-1) and (0.5,-1) .. (0.5,0) -- (0.5,0.5)
		arc[x radius=0.5, y radius=0.15, start angle=360, end angle=180];
	\draw[1facetLine] (-0.5,0.5) -- (-0.5,0)
		.. controls (-0.5,-1) and (0.5,-1) .. (0.5,0) -- (0.5,0.5)
		(0,0.5) ellipse[x radius=0.5, y radius=0.15];
	\draw[seam]
		(-0.5, 0.85pt) arc[x radius=0.5, y radius=0.15, start angle=180, end angle=360]
		(-0.5,-0.85pt) arc[x radius=0.5, y radius=0.15, start angle=180, end angle=360];
	\foamarrow#1[seam](0,-0.15);
	\ifcase 0#2\or
		\foamdot(0.15,-0.5);\or
		\foamdualdot(0.15,-0.5);%
	\fi
}

\deffoampict{1cap}(-0.7, 0.4)--(0.7,-0.75) M=0 [#1]=[]{%
	\fill[1facetBack] (-0.5,-0.5)
		.. controls (-0.5, 0.5) and (0.5, 0.5) .. (0.5,-0.5)
		arc[x radius=0.5, y radius=0.15, start angle=0, end angle=180];
	\draw[1facetLine] (0.5,-0.5)
		arc[x radius=0.5, y radius=0.15, start angle=0, end angle=180];
	\draw[1facetFront] (-0.5,-0.5)
		.. controls (-0.5, 0.5) and (0.5, 0.5) .. (0.5,-0.5)
		arc[x radius=0.5, y radius=0.15, start angle=360, end angle=180];
	\ifcase 0#1\or
		\foamdot(0.15,0);\or
		\foamdualdot(0.15,0);%
	\fi
}

\deffoampict{2cap}(-0.7, 0.4)--(0.7,-0.75) M=0 [#1#2]=[x]{%
	\getlength\innradx({$(0.5,0) + (-1.7pt,0)$});
	\getlength\innrady({$(0,0.15) + (0,-1pt)$});
	\fill[2facetInner] (-0.5,-0.5)
		.. controls (-0.5, 0.5) and (0.5, 0.5) .. (0.5,-0.5)
		arc[x radius=0.5, y radius=0.15, start angle=0, end angle=180];
	\draw[2facetLine] (0.5,-0.5)
		arc[x radius=0.5, y radius=0.15, start angle=0, end angle=180];
	\fill[2facetBack] (-\innradx,-0.5)
		.. controls (-\innradx, \innradx) and (\innradx, \innradx)
		.. (\innradx,-0.5)
		arc[x radius=\innradx, y radius=\innrady, start angle=0, end angle=180];
	\draw[2facetLine] (-\innradx,-0.5)
		.. controls (-\innradx, \innradx) and (\innradx, \innradx)
		.. (\innradx,-0.5)
		(0,-0.5) ellipse[x radius=\innradx, y radius=\innrady];
	\fill[2facetInner] (-\innradx,-0.5)
		.. controls (-\innradx, \innradx) and (\innradx, \innradx)
		.. (\innradx,-0.5)
		arc[x radius=\innradx, y radius=\innrady,
		    start angle=360, end angle=180];
	\draw[2facetFront] (-0.5,-0.5)
		.. controls (-0.5, 0.5) and (0.5, 0.5) .. (0.5,-0.5)
		arc[x radius=0.5, y radius=0.15, start angle=360, end angle=180];
	\foamarrow#1[2facetLine]({$(0,-0.65) + (0,0.6pt)$});
	\ifcase 0#2\or
		\foamdot(0.15,0);\or
		\foamdualdot(0.15,0);%
	\fi
}

\deffoampict{1cap and disk}(-0.7, 0.8)--(0.7,-0.65) M=0.1 [#1#2]=[>]{%
	\fill[1facetBack] (-0.5,-0.5) -- (-0.5,0)
		.. controls (-0.5, 1) and (0.5, 1) .. (0.5,0) -- (0.5,-0.5)
		arc[x radius=0.5, y radius=0.15, start angle=0, end angle=180];
	\draw[2facetInner] (0,-0.85pt) ellipse[x radius=0.5, y radius=0.15];
	\draw[2facetFront] (0, 0.85pt) ellipse[x radius=0.5, y radius=0.15];
	\draw (0.5,-0.5) arc[x radius=0.5, y radius=0.15, start angle=0, end angle=180];
	\draw[1facetFront] (-0.5,-0.5) -- (-0.5,0)
		.. controls (-0.5, 1) and (0.5, 1) .. (0.5,0) -- (0.5,-0.5)
		arc[x radius=0.5, y radius=0.15, start angle=360, end angle=180];
	\draw[seam]
		(-0.5, 0.85pt) arc[x radius=0.5, y radius=0.15, start angle=180, end angle=360]
		(-0.5,-0.85pt) arc[x radius=0.5, y radius=0.15, start angle=180, end angle=360];
	\foamarrow#1[seam](0,-0.15);
	\ifcase 0#2\or
		\foamdot(0.15, 0.5);\or
		\foamdualdot(0.15, 0.5);%
	\fi
}

\deffoampict{3facets}(-1,-0.5)--(0.65,0.5) M=0 [#1#2]=[.]{%
	\draw[1facetFront]
		(-0.5, 0) arc[radius=1, start angle=180, end angle=150] --
		++( 1, 0) arc[radius=1, start angle=150, end angle=210] --
		++(-1, 0) arc[radius=1, start angle=210, end angle=180];
	\draw[2facetInner] (0.5,-0.85pt) -- ++(-0.5,-0.3) -- ++(-1,0) -- ++(0.5,0.3);
	\draw[seam] (0.5,-0.85pt) -- ++(-1,0);
	\draw[2facetFront] (0.5, 0.85pt) -- ++(-0.5,-0.3) -- ++(-1,0) -- ++(0.5,0.3);
	\draw[seam] (0.5, 0.85pt) -- ++(-1,0);
	\ifx+#1%
		\ifcase 0#2
			\foamdot(0.25,0.25);\or
			\foamdot(0.25,0.25);\or
			\foamdualdot(0.25,0.25);
		\fi
	\else\ifx-#1%
		\ifcase 0#2
			\foamdot(0.35,-0.3);\or
			\foamdot(0.35,-0.3);\or
			\foamdualdot(0.35,-0.3);
		\fi
	\fi\fi
}

\deffoampict{1neck}(-0.6,-1.2)--(0.6,1.2) M=0 {%
	\fill[1facetBack] (-0.5,-1) .. controls (-0.3,0) .. (-0.5,1)
		arc[x radius=0.5, y radius=0.15, start angle=180, end angle=0]
		.. controls (0.3,0) .. (0.5,-1)
		arc[x radius=0.5, y radius=0.15, start angle=0, end angle=180];
	\draw[1facetLine] (-0.5,-1)
		arc[x radius=0.5, y radius=0.15, start angle=180, end angle=0];
	\fill[1facetFront](-0.5,-1) .. controls (-0.3,0) .. (-0.5,1)
		arc[x radius=0.5, y radius=0.15, start angle=180, end angle=360]
		.. controls (0.3,0) .. (0.5,-1)
		arc[x radius=0.5, y radius=0.15, start angle=360, end angle=180];
	\draw[1facetLine] (-0.5,1) .. controls (-0.3,0) .. (-0.5,-1)
		arc[x radius=0.5, y radius=0.15, start angle=180, end angle=360]
		.. controls (0.3,0) .. (0.5,1);
	\draw[1facetLine] (0, 1) ellipse[x radius=0.5, y radius=0.15];
}
\deffoampict{2neck}(-0.6,-1.2)--(0.6,1.2) M=0 {%
	\getlength\innradx({$(0.5,0) + (-1.7pt,0)$});
	\getlength\innrady({$(0,0.15) + (0,-1pt)$});
	\fill[2facetInner](-0.5,-1) .. controls ++(0.2,1) .. (-0.5,1)
		arc[x radius=0.5, y radius=0.15, start angle=180, end angle=0]
		.. controls ++(-0.2,-1) .. (0.5,-1)
		arc[x radius=0.5, y radius=0.15, start angle=0, end angle=180];
	\draw[2facetLine] (0.5,-1)
		arc[x radius=0.5, y radius=0.15, start angle=0, end angle=180];
	\fill[2facetBack](-\innradx,-1) .. controls ++(0.2,1) .. (-\innradx,1)
		arc[x radius=\innradx, y radius=\innrady, start angle=180,end angle=0]
		.. controls ++(-0.2,-1) .. (\innradx,-1)
		arc[x radius=\innradx, y radius=\innrady, start angle=0,end angle=180];
	\draw[2facetLine] (\innradx,-1)
		arc[x radius=\innradx, y radius=\innrady, start angle=0,end angle=180];
	\fill[2facetInner](-\innradx,-1) .. controls ++(0.2,1) .. (-\innradx,1)
		arc[x radius=\innradx, y radius=\innrady, start angle=180,end angle=360]
		.. controls ++(-0.2,-1) .. (\innradx,-1)
		arc[x radius=\innradx, y radius=\innrady, start angle=360,end angle=180];
	\draw[2facetLine] (-\innradx, 1) .. controls ++(0.2,-1) .. (-\innradx,-1)
		arc[x radius=\innradx, y radius=\innrady, start angle=180,end angle=360]
		.. controls ++(-0.2,1) .. (\innradx, 1);
	\fill[2facetFront](-0.5,-1) .. controls (-0.3,0) .. (-0.5,1)
		arc[x radius=0.5, y radius=0.15, start angle=180, end angle=360]
		.. controls (0.3,0) .. (0.5,-1)
		arc[x radius=0.5, y radius=0.15, start angle=360, end angle=180];
	\draw[2facetLine] (-0.5, 1) .. controls (-0.3,0) .. (-0.5,-1)
		arc[x radius=0.5, y radius=0.15, start angle=180, end angle=360]
		.. controls (0.3,0) .. (0.5,1);
	\draw[2facetLine]
		(0, 1) ellipse[x radius=0.5, y radius=0.15]
		(0, 1) ellipse[x radius=\innradx, y radius=\innrady];
}

\deffoampict{1cap 1cup}(-0.6,-1.2)--(0.6,1.2) M=0 [#1#2]=[.]{%
	\fill[1facetBack](0.5,1) .. controls (0.5,0) and (-0.5,0) .. (-0.5,1)
		arc[x radius=0.5, y radius=0.15, start angle=180, end angle=0];
	\fill[1facetFront](0.5,1) .. controls (0.5,0) and (-0.5,0) .. (-0.5,1)
		arc[x radius=0.5, y radius=0.15, start angle=180, end angle=360];
	\draw[1facetLine]
		(0.5,1) .. controls (0.5,0) and (-0.5,0) .. (-0.5,1)
		(0,1) ellipse[x radius=0.5, y radius=0.15];
	\fill[1facetBack] (0.5,-1) .. controls (0.5,0) and (-0.5,0) .. (-0.5,-1)
		arc[x radius=0.5, y radius=0.15, start angle=180, end angle=0];
	\draw[1facetLine] (-0.5,-1)
		arc[x radius=0.5, y radius=0.15, start angle=180, end angle=0];
	\draw[1facetFront](0.5,-1) .. controls (0.5,0) and (-0.5,0) .. (-0.5,-1)
		arc[x radius=0.5, y radius=0.15, start angle=180, end angle=360];
	\ifx#1.\else
		\ifcase 0#2
			\foamdot(0.2, #10.6);\or
			\foamdot(0.2, #10.6);\or
			\foamdualdot(0.2, #10.6);
		\fi
	\fi
}
\deffoampict{2cap 2cup}(-0.6,-1.2)--(0.6,1.2) M=0 [#1#2]=[.]{%
	\getlength\innradx({$(0.5,0) + (-1.7pt,0)$});
	\getlength\innrady({$(0,0.15) + (0,-1pt)$});
	\fill[2facetInner](0.5,1) .. controls (0.5,0) and (-0.5,0) .. (-0.5,1)
		arc[x radius=0.5, y radius=0.15, start angle=180, end angle=0];
	\fill[2facetBack] (\innradx,1)
		.. controls (\innradx,1.7pt) and (-\innradx,1.7pt) .. (-\innradx,1)
		arc[x radius=\innradx,y radius=\innrady,start angle=180,end angle=-180];
	\fill[2facetInner] (\innradx,1)
		.. controls (\innradx,1.7pt) and (-\innradx,1.7pt) .. (-\innradx,1)
		arc[x radius=\innradx,y radius=\innrady,start angle=180,end angle=360];
	\draw[2facetLine] (\innradx,1)
		.. controls (\innradx,1.7pt) and (-\innradx,1.7pt) .. (-\innradx,1)
		(0,1) ellipse[x radius=\innradx,y radius=\innrady];
	\fill[2facetFront] (0.5,1) .. controls (0.5,0) and (-0.5,0) .. (-0.5,1)
		arc[x radius=0.5, y radius=0.15, start angle=180, end angle=360];
	\draw[2facetLine]
		(0.5,1) .. controls (0.5,0) and (-0.5,0) .. (-0.5,1)
		(0,1) ellipse[x radius=0.5, y radius=0.15];
	\fill[2facetInner] (0.5,-1) .. controls (0.5,0) and (-0.5,0) .. (-0.5,-1)
		arc[x radius=0.5, y radius=0.15, start angle=180, end angle=0];
	\draw[2facetLine] (-0.5,-1)
		arc[x radius=0.5, y radius=0.15, start angle=180, end angle=0];
	\fill[2facetBack] (\innradx,-1)
		.. controls (\innradx,-1.7pt) and (-\innradx,-1.7pt) .. (-\innradx,-1)
		arc[x radius=\innradx,y radius=\innrady,start angle=180,end angle=0];
	\draw[2facetLine] (-\innradx,-1)
		arc[x radius=\innradx,y radius=\innrady,start angle=180,end angle=0];
	\draw[2facetInner] (\innradx,-1)
		.. controls (\innradx,-1.7pt) and (-\innradx,-1.7pt) .. (-\innradx,-1)
		arc[x radius=\innradx,y radius=\innrady,start angle=180,end angle=360];
	\draw[2facetFront] (0.5,-1) .. controls (0.5,0) and (-0.5,0) .. (-0.5,-1)
		arc[x radius=0.5, y radius=0.15, start angle=180, end angle=360];
	\ifx#1.\else
		\ifcase 0#2
			\foamdot(0.2, #10.6);\or
			\foamdot(0.2, #10.6);\or
			\foamdualdot(0.2, #10.6);
		\fi
	\fi
}

\deffoampict{2cyl at 1plane}(-1.8,-0.4)--(1.8,1.25) M=0.4 [#1#2]=[.]{%
	\getlength\innradx({$(0.5,0) + (-1.7pt,0)$});
	\getlength\innrady({$(0,0.15) + (0,-1pt)$});
	\draw[1facetFront] (-1.8,-0.4) -- (-0.6,0.4) -- (1.7,0.4) -- (0.5,-0.4) -- cycle;
	\fill[2facetInner] (0.5,0)
		-- ( 0.5,1) arc[x radius=0.5, y radius=0.15, start angle=0, end angle=180]
		-- (-0.5,0) arc[x radius=0.5, y radius=0.15, start angle=180, end angle=0];
	\draw[seam] (-0.5,0)
		arc[x radius=0.5, y radius=0.15, start angle=180, end angle=0];
	\fill[2facetBack] (\innradx,0) -- (\innradx,1)
		arc[x radius=\innradx, y radius=\innrady, start angle=0, end angle=180]
		-- (-\innradx, 0)
		arc[x radius=\innradx, y radius=\innrady, start angle=180, end angle=0];
	\draw[seam] (-\innradx,0)
		arc[x radius=\innradx, y radius=\innrady, start angle=180, end angle=0];
	\fill[2facetInner] (\innradx,0) -- (\innradx,1)
		arc[x radius=\innradx,y radius=\innrady,start angle=360,end angle=180]
		-- (-\innradx,0)
		arc[x radius=\innradx,y radius=\innrady,start angle=180,end angle=360];
	\draw[2facetLine]
		( \innradx, 1) -- ( \innradx, 0)
		(-\innradx, 1) -- (-\innradx, 0);
	\draw[seam] (\innradx, 0)
		arc[x radius=\innradx,y radius=\innrady,start angle=360,end angle=180];
	\ifx V#2\relax\ifx #1>%
		\draw[2facetLine,line width=1pt,->]
			(0,0.75) ++(220:0.5 and 0.15) -- (0,0.75);
	\fi\fi
	\fill[2facetFront] (0.5,0)
		-- ( 0.5,1) arc[x radius=0.5, y radius=0.15, start angle=360, end angle=180]
		-- (-0.5,0) arc[x radius=0.5, y radius=0.15, start angle=180, end angle=360];
	\draw[2facetLine] (-0.5,1) -- (-0.5,0) (0.5,1) -- (0.5,0)
		(0,1) ellipse[x radius=\innradx, y radius=\innrady]
		(0,1) ellipse[x radius=0.5, y radius=0.15];
	\draw[seam]
		(-0.5,0) arc[x radius=0.5, y radius=0.15, start angle=180, end angle=360];
	\ifx V#2\relax
		\ifx #1<
			\draw[2facetLine,line width=1pt,->]
				(0,0.75) ++(220:0.5 and 0.15) -- ++(220:0.5 and 0.15);
		\fi
	\else\ifx\relax#2\relax\else
		\foamarrow#1[2facetLine]0({$(0,0.85) + (-4pt,1pt)$});
		\foamarrow#1[2facetLine]180({$(0, 1.15) + (4pt,-1pt)$});
	\fi\fi
	\foamarrow#1[seam]7({$(0,-0.15) + (4pt, 1pt)$});
}

\deffoampict{2cup > 1plane}(-1.8,-0.4)--(1.8,1.25) M=0.4 [#1]=[.]{%
	\getlength\innradx({$(0.5,0) + (-1.7pt,0)$});
	\getlength\innrady({$(0,0.15) + (0,-1pt)$});
	\draw[1facetFront] (-1.8,-0.4) -- (-0.6,0.4) -- (1.7,0.4) -- (0.5,-0.4) -- cycle;
	\fill[2facetInner](0.5,1) .. controls (0.5,0) and (-0.5,0) .. (-0.5,1)
		arc[x radius=0.5, y radius=0.15, start angle=180, end angle=0];
	\fill[2facetBack] (\innradx,1)
		.. controls (\innradx,1.7pt) and (-\innradx,1.7pt) .. (-\innradx,1)
		arc[x radius=\innradx,y radius=\innrady,start angle=180,end angle=-180];
	\fill[2facetInner] (\innradx,1)
		.. controls (\innradx,1.7pt) and (-\innradx,1.7pt) .. (-\innradx,1)
		arc[x radius=\innradx,y radius=\innrady,start angle=180,end angle=360];
	\draw[2facetLine] (\innradx,1)
		.. controls (\innradx,1.7pt) and (-\innradx,1.7pt) .. (-\innradx,1)
		(0,1) ellipse[x radius=\innradx,y radius=\innrady];
	\ifx A#1\relax
		\draw[2facetLine,line width=1pt,->] (0,0.25) -- (0,0.6);
	\fi
	\fill[2facetFront] (0.5,1) .. controls (0.5,0) and (-0.5,0) .. (-0.5,1)
		arc[x radius=0.5, y radius=0.15, start angle=180, end angle=360];
	\draw[2facetLine]
		(0.5,1) .. controls (0.5,0) and (-0.5,0) .. (-0.5,1)
		(0,1) ellipse[x radius=0.5, y radius=0.15];
	\ifx V#1\relax
		\draw[2facetLine,line width=1pt,->] (0,0.25) -- (0,-0.1);
	\else
		\foamarrow#1[2facetLine]({$(0,0.85) + (-4pt,1pt)$});
		\foamarrow#1[2facetLine]180({$(0, 1.15) + (4pt,-1pt)$});
	\fi
}

\deffoampict{2planes at 1plane}(-1.8,-0.4)--(1.8,1.1) M=0.35 [#1]=[.]{%
	\draw[1facetFront] (-1.7,-0.4) -- (-0.5,0.4) -- (1.7,0.4) -- (0.5,-0.4) -- cycle;
	\draw[2facetInner] (-0.875, 0.15) rectangle ++(2.2,0.85);
	\draw[seam] (-0.875,0.15) -- ++(2.2,0);
	\draw[2facetBack]  (-0.875, 0.15) ++(-1.4pt,-0.96pt) rectangle ++(2.2,0.85);
	\draw[seam] (-0.875,0.15) ++(-1.4pt,-0.96pt) -- ++(2.2,0);
	\draw[2facetInner] (-1.325,-0.15) ++( 1.4pt, 0.96pt) rectangle ++(2.2,0.85);
	\draw[seam] (-1.325,-0.15) ++(1.4pt,0.96pt) -- ++(2.2,0);
	\foamarrow#1[seam]180({ 0.225, 0.15) ++(-0.7pt,-0.47pt});
	\draw[2facetFront] (-1.325,-0.15) rectangle ++(2.2,0.85);
	\draw[seam] (-1.325,-0.15) -- ++(2.2,0);
	\foamarrow#1[seam]({-0.225,-0.15) ++(0.7pt, 0.47pt});
}

\deffoampict{2saddle at 1plane}(-1.8,-0.4)--(1.8,1.1) M=0.35 [#1]=[.]{%
	\getlength\xradbig({$(0.95,0) + (-3.1pt,0)$});
	\getlength\xradsmall({$(0.5,0) + (-0.3pt,0)$});
	\getlength\yrad({$(0,0.15) + (0,-0.96pt)$});
	\getlength\radmid({$(0,0.375) + (0, 1.7pt)$});
	\draw[1facetFront] (-1.7,-0.4) -- (-0.5,0.4) -- (1.7,0.4) -- (0.5,-0.4) -- cycle;
	\fill[2facetInner] (-0.875, 0.15)
		arc[x radius=0.5, y radius=0.15, start angle=90, end angle=0]
		arc[radius=0.375, start angle=180, end angle=0]
		arc[x radius=0.95, y radius=0.15, start angle=180, end angle=90]
		-- ++(0,0.85) -- ++(-2.2,0) -- cycle;
	\draw[2facetLine] (-0.875,0.15) |- ++(2.2,0.85) -- ++(0,-0.85);
	\draw[seam]
		(-0.875, 0.15) arc[x radius=0.5, y radius=0.15, start angle=90, end angle=0]
		++(0.75, 0) arc[x radius=0.95, y radius=0.15, start angle=180, end angle=90];
	\fill[2facetBack] (-0.875, 0.15) ++(-1.4pt,-0.96pt)
		arc[x radius=\xradsmall, y radius=\yrad, start angle=90, end angle=0]
		arc[radius=\radmid, start angle=180, end angle=0]
		arc[x radius=\xradbig, y radius=\yrad, start angle=180, end angle=90]
		-- ++(0,0.85) -- ++(-2.2,0) -- cycle;
	\draw[2facetLine] (-0.875, 0.15) ++(-1.4pt,-0.96pt) |- ++(2.2,0.85) -- ++(0,-0.85);
	\draw[seam] (-0.875, 0.15) ++(-1.4pt,-0.96pt)
		arc[x radius=\xradsmall, y radius=\yrad, start angle=90, end angle=0]
		++(\radmid, 0) ++(\radmid, 0)
		arc[x radius=\xradbig, y radius=\yrad, start angle=180, end angle=90];
	\draw[2facetInner] (-1.325,-0.15) ++(1.4pt, 0.96pt)
		arc[x radius=\xradbig, y radius=\yrad, start angle=270, end angle=360]
		arc[radius=\radmid, start angle=180, end angle=0]
		arc[x radius=\xradsmall, y radius=\yrad, start angle=180, end angle=270]
		-- ++(0,0.85) -- ++(-2.2,0) -- cycle;
	\draw[seam] (-1.325,-0.15) ++(1.4pt, 0.96pt)
		arc[x radius=\xradbig, y radius=\yrad, start angle=270, end angle=360]
		++(\radmid, 0) ++(\radmid, 0)
		arc[x radius=\xradsmall, y radius=\yrad, start angle=180, end angle=270];
	\draw[2facetFront] (-1.325,-0.15)
		arc[x radius=0.95, y radius=0.15, start angle=270, end angle=360]
		arc[radius=0.375, start angle=180, end angle=0]
		arc[x radius=0.5, y radius=0.15, start angle=180, end angle=270]
		-- ++(0,0.85) -- ++(-2.2,0) -- cycle;
	\draw[seam]
		(-1.325,-0.15) arc[x radius=0.95, y radius=0.15, start angle=270, end angle=360]
		++(0.75,0) arc[x radius=0.5, y radius=0.15, start angle=180, end angle=270];
	\foamarrow#1[seam]8({-0.823,-0.128) ++(0.7pt, 0.47pt});
	\foamarrow#1[seam]-10({0.588,-0.123) ++(0.7pt, 0.47pt});
}

\deffoampict{2bubble at 1plane}(-1.4,-0.3)--(1.4,0.8) M=0 [#1]=[.]{%
	\getlength\xrad(0.5,0);
	\getlength\yrad(0,0.15);
	\getlength\yyrad(0,0.7);
	\draw[1facetFront] (-0.6,0.3) -- (-1.4,-0.3) -- (0.6,-0.3) -- (1.4,0.3) -- cycle;
	\fill[2facetInner] (-0.5,0) ++(-0.85pt,0)
		arc[start angle=180, end angle=0, x radius=\xrad+0.85pt, y radius=\yrad+0.7pt]
		arc[start angle=0, end angle=180, x radius=\xrad+0.85pt, y radius=\yyrad+0.85pt];
	\draw[seam] (-0.5,0) ++(-0.85pt,0)
		arc[start angle=180, end angle=0, x radius=\xrad+0.85pt, y radius=\yrad+0.7pt];
	\fill[2facetBack] (-0.5,0) ++(0.85pt,0)
		arc[start angle=180, end angle=0, x radius=\xrad-0.85pt, y radius=\yrad-0.7pt]
		arc[start angle=0, end angle=180, x radius=\xrad-0.85pt, y radius=\yyrad-0.85pt];
	\draw[seam] (-0.5,0) ++(0.85pt,0)
		arc[start angle=180, end angle=0, x radius=\xrad-0.85pt, y radius=\yrad-0.7pt];
	\fill[2facetInner] (-0.5,0) ++(0.85pt,0)
		arc[start angle=180, end angle=360, x radius=\xrad-0.85pt, y radius=\yrad-0.7pt]
		arc[start angle=0, end angle=180, x radius=\xrad-0.85pt, y radius=\yyrad-0.85pt];
	\draw[seam] (-0.5,0) ++(0.85pt,0)
		arc[start angle=180, end angle=360, x radius=\xrad-0.85pt, y radius=\yrad-0.7pt];
	\draw[2facetLine] (-0.5,0) ++(0.85pt,0)
		arc[start angle=180, end angle=0, x radius=\xrad-0.85pt, y radius=\yyrad-0.85pt];
	\fill[2facetFront] (-0.5,0) ++(-0.85pt,0)
		arc[start angle=180, end angle=360, x radius=\xrad+0.85pt, y radius=\yrad+0.7pt]
		arc[start angle=0, end angle=180, x radius=\xrad+0.85pt, y radius=\yyrad+0.85pt];
	\draw[seam] (-0.5,0) ++(-0.85pt,0)
		arc[start angle=180, end angle=360, x radius=\xrad+0.85pt, y radius=\yrad+0.7pt];
	\draw[2facetLine] (-0.5,0) ++(-0.85pt,0)
		arc[start angle=180, end angle=0, x radius=\xrad+0.85pt, y radius=\yyrad+0.85pt];
	\foamarrow#1[seam](0,-0.15);
}

\deffoampict{no bubble at 1plane}(-1.4,-0.3)--(1.4,0.3) M=0 {%
	\draw[1facetFront] (-0.6,0.3) -- (-1.4,-0.3) -- (0.6,-0.3) -- (1.4,0.3) -- cycle;
}

\deffoampict{2tunel at 1plane}(-1.4,-0.3)--(1.4,0.8) M=0.2 [#1]=[.]{%
	\draw[1facetFront] (-0.6,0.3) -- (-1.4,-0.3) -- (0.6,-0.3) -- (1.4,0.3) -- cycle;
	\getlength\radx(0.2,0);
	\getlength\rady(0,0.6);
	\fill[2facetInner] (1.0,0.75) ++(0,0.85pt)
		arc[x radius=\radx+0.68pt, y radius=\rady+0.34pt,
		    start angle=90, end angle=0] -- ++(-2,0)
		arc[x radius=\radx+0.68pt, y radius=\rady+0.34pt,
		    start angle=0, end angle=90];
	\draw[seam] (1.2,0.15) ++(0.68pt,0.51pt) -- ++(-2,0);
	\draw[2facetLine] (-0.8,0.15) ++(0.68pt,0.51pt)
		arc[x radius=\radx+0.68pt, y radius=\rady+0.34pt,
		    start angle=0, end angle=90]
		++(2,0) arc[x radius=\radx+0.68pt, y radius=\rady+0.34pt,
		    start angle=90, end angle=0];
	\fill[2facetBack] (1.0,0.75) ++(0,-0.85pt)
		arc[x radius=\radx-0.68pt, y radius=\rady-0.34pt,
		    start angle=90, end angle=0] -- ++(-2,0)
		arc[x radius=\radx-0.68pt, y radius=\rady-0.34pt,
		    start angle=0, end angle=90];
	\draw[seam] (1.2,0.15) ++(-0.68pt,-0.51pt) -- ++(-2,0);
	\draw[2facetLine] (-0.8,0.15) ++(-0.68pt,-0.51pt)
		arc[x radius=\radx-0.68pt, y radius=\rady-0.34pt,
		    start angle=0, end angle=90]
		++(2,0) arc[x radius=\radx-0.68pt, y radius=\rady-0.34pt,
		    start angle=90, end angle=0];
	\getlength\rady(0,0.9);
	\fill[2facetInner] (1.0,0.75) ++(0,-0.85pt)
		arc[x radius=\radx-0.68pt, y radius=\rady-1.36pt,
		    start angle=90, end angle=180] -- ++(-2,0)
		arc[x radius=\radx-0.68pt, y radius=\rady-1.36pt,
		    start angle=180, end angle=90] -- cycle;
	\draw[seam] (0.8,-0.15) ++(0.68pt,0.51pt) -- ++(-2,0);
	\draw[2facetLine] (1.0,0.75) ++(0,-0.85pt)
		arc[x radius=\radx-0.68pt, y radius=\rady-1.36pt,
		    start angle=90, end angle=180]
		++(-2,0) arc[x radius=\radx-0.68pt, y radius=\rady-1.36pt,
		    start angle=180, end angle=90] -- ++(2,0);
	\foamarrow#1[seam]{180}(0.4,0.15);
	\fill[2facetFront] (1.0,0.75) ++(0, 0.85pt)
		arc[x radius=\radx+0.68pt, y radius=\rady+1.36pt,
		    start angle=90, end angle=180] -- ++(-2,0)
		arc[x radius=\radx+0.68pt, y radius=\rady+1.36pt,
		    start angle=180, end angle=90] -- cycle;
	\draw[seam] (0.8,-0.15) ++(-0.68pt,-0.51pt) -- ++(-2,0);
	\draw[2facetLine] (1.0,0.75) ++(0, 0.85pt)
		arc[x radius=\radx+0.68pt, y radius=\rady+1.36pt,
		    start angle=90, end angle=180] ++(-2,0)
		arc[x radius=\radx+0.68pt, y radius=\rady+1.36pt,
		    start angle=180, end angle=90] -- ++(2,0);
	\foamarrow#1[seam](-0.4,-0.15);
}

\deffoampict{2shells at 1plane}(-1.4,-0.3)--(1.4,0.8) M=0.2 [#1]=[.]{%
	\getlength\xunit(0.1,0);
	\getlength\yunit(0,0.1);
	\draw[1facetFront] (-0.6,0.3) -- (-1.4,-0.3) -- (0.6,-0.3) -- (1.4,0.3) -- cycle;
	\fill[2facetInner] (-1, 0.75) ++(0, 0.85pt)
		arc[x radius=\xunit*8.0+0.85pt, start angle=90, 
		    y radius=\yunit*7.5+0.85pt,   end angle=0]
		arc[x radius=\xunit*6.0+0.17pt, start angle=0, 
		    y radius=\yunit*1.5+0.51pt,   end angle=90]
		arc[x radius=\xunit*2.0+0.68pt, start angle=0,
		    y radius=\yunit*6.0+0.34pt,   end angle=90];
	\draw[2facetLine] (-1, 0.75) ++(0, 0.85pt)
		arc[x radius=\xunit*2.0+0.68pt, start angle=90,
		    y radius=\yunit*6.0+0.34pt,   end angle=0];
	\draw[seam] (-0.2,0) ++(0.85pt,0)
		arc[x radius=\xunit*6.0+0.17pt, start angle=0,
		    y radius=\yunit*1.5+0.51pt,   end angle=90];
	\fill[2facetBack] (-1, 0.75) ++(0,-0.85pt)
		arc[x radius=\xunit*8.0-0.85pt, start angle=90, 
		    y radius=\yunit*7.5-0.85pt,   end angle=0]
		arc[x radius=\xunit*6.0-0.17pt, start angle=0, 
		    y radius=\yunit*1.5-0.51pt,   end angle=90]
		arc[x radius=\xunit*2.0-0.68pt, start angle=0,
		    y radius=\yunit*6.0-0.34pt,   end angle=90];
	\draw[2facetLine] (-1, 0.75) ++(0,-0.85pt)
		arc[x radius=\xunit*2.0-0.68pt, start angle=90,
		    y radius=\yunit*6.0-0.34pt,   end angle=0];
	\draw[seam] (-0.2,0) ++(-0.85pt,0)
		arc[x radius=\xunit*6.0-0.17pt, start angle=0, 
		    y radius=\yunit*1.5-0.51pt,   end angle=90];
	\fill[2facetInner] (-1, 0.75) ++(0,-0.85pt)
		arc[x radius=\xunit*8.0-0.85pt, start angle=90, 
		    y radius=\yunit*7.5-0.85pt,   end angle=0]
		arc[x radius=\xunit*10 -1.53pt, start angle=360,
		    y radius=\yunit*1.5-0.51pt,   end angle=270]
		arc[x radius=\xunit*2.0-0.68pt, start angle=180,
		    y radius=\yunit*9.0-1.36pt,   end angle=90];
	\draw[2facetLine] (-0.2,0) ++(-0.85pt,0)
		arc[x radius=\xunit*8.0-0.85pt, start angle=0, 
		    y radius=\yunit*7.5-0.85pt,   end angle=90]
		arc[x radius=\xunit*2.0-0.68pt, start angle=90,
		    y radius=\yunit*9.0-1.36pt,   end angle=180];
	\draw[seam] (-0.2,0) ++(-0.85pt,0)
		arc[x radius=\xunit*10 -1.53pt, start angle=360,
		    y radius=\yunit*1.5-0.51pt,   end angle=270];
	\fill[2facetFront] (-1, 0.75) ++(0, 0.85pt)
		arc[x radius=\xunit*8.0+0.85pt, start angle=90, 
		    y radius=\yunit*7.5+0.85pt,   end angle=0]
		arc[x radius=\xunit*10 +1.53pt, start angle=360,
		    y radius=\yunit*1.5+0.51pt,   end angle=270]
		arc[x radius=\xunit*2.0+0.68pt, start angle=180,
		    y radius=\yunit*9.0+1.36pt,   end angle=90];
	\draw[2facetLine] (-0.2,0) ++(0.85pt,0)
		arc[x radius=\xunit*8.0+0.85pt, start angle=0, 
		    y radius=\yunit*7.5+0.85pt,   end angle=90]
		arc[x radius=\xunit*2.0+0.68pt, start angle=90,
		    y radius=\yunit*9.0+1.36pt,   end angle=180];
	\draw[seam] (-0.2,0) ++(0.85pt,0)
		arc[x radius=\xunit*10 +1.53pt, start angle=360,
		    y radius=\yunit*1.5+0.51pt,   end angle=270];
	\fill[2facetInner] (1,0.75) ++(0,0.85pt)
		arc[x radius=\xunit*2.0+0.68pt, start angle=90,  
		    y radius=\yunit*6.0+0.34pt,   end angle=0]
		arc[x radius=\xunit*10 +1.53pt, start angle=90,  
		    y radius=\yunit*1.5+0.51pt,   end angle=180]
		arc[x radius=\xunit*8.0+0.85pt, start angle=180,
		    y radius=\yunit*7.5+0.85pt,   end angle=90];
	\draw[2facetLine] (1,0.75) ++(0,0.85pt)
		arc[x radius=\xunit*2.0+0.68pt, start angle=90,  
		    y radius=\yunit*6.0+0.34pt,   end angle=0];
	\draw[seam] (0.2,0) ++(-0.85pt,0)
		arc[x radius=\xunit*10 +1.53pt, start angle=180,  
		    y radius=\yunit*1.5+0.51pt,   end angle=90];
	\fill[2facetBack] (1,0.75) ++(0,-0.85pt)
		arc[x radius=\xunit*2.0-0.68pt, start angle=90,  
		    y radius=\yunit*6.0-0.34pt,   end angle=0]
		arc[x radius=\xunit*10 -1.53pt, start angle=90,  
		    y radius=\yunit*1.5-0.51pt,   end angle=180]
		arc[x radius=\xunit*8.0-0.85pt, start angle=180,
		    y radius=\yunit*7.5-0.85pt,   end angle=90];
	\draw[2facetLine] (1,0.75) ++(0,-0.85pt)
		arc[x radius=\xunit*2.0-0.68pt, start angle=90,  
		    y radius=\yunit*6.0-0.34pt,   end angle=0];
	\draw[seam] (0.2,0) ++(0.85pt,0)
		arc[x radius=\xunit*10 -1.53pt, start angle=180,  
		    y radius=\yunit*1.5-0.51pt,   end angle=90];
	\fill[2facetInner] (1,0.75) ++(0,-0.85pt)
		arc[x radius=\xunit*8.0-0.85pt, start angle=90,  
		    y radius=\yunit*7.5-0.85pt,   end angle=180]
		arc[x radius=\xunit*6.0-0.17pt, start angle=180,  
		    y radius=\yunit*1.5-0.51pt,   end angle=270]
		arc[x radius=\xunit*2.0-0.68pt, start angle=180,  
		    y radius=\yunit*9.0-1.36pt,   end angle=90];
	\draw[2facetLine] (1,0.75) ++(0,-0.85pt)
		arc[x radius=\xunit*8.0-0.85pt, start angle=90,  
		    y radius=\yunit*7.5-0.85pt,   end angle=180]
		(1,0.75) ++(0,-0.85pt)
		arc[x radius=\xunit*2.0-0.68pt, start angle=90,  
		    y radius=\yunit*9.0-1.36pt,   end angle=180];
	\draw[seam] (0.2,0) ++(0.85pt,0)
		arc[x radius=\xunit*6.0-0.17pt, start angle=180,  
		    y radius=\yunit*1.5-0.51pt,   end angle=270];
	\fill[2facetFront] (1,0.75) ++(0, 0.85pt)
		arc[x radius=\xunit*8.0+0.85pt, start angle=90,  
		    y radius=\yunit*7.5+0.85pt,   end angle=180]
		arc[x radius=\xunit*6.0+0.17pt, start angle=180,  
		    y radius=\yunit*1.5+0.51pt,   end angle=270]
		arc[x radius=\xunit*2.0+0.68pt, start angle=180,  
		    y radius=\yunit*9.0+1.36pt,   end angle=90];
	\draw[2facetLine] (1,0.75) ++(0, 0.85pt)
		arc[x radius=\xunit*8.0+0.85pt, start angle=90,  
		    y radius=\yunit*7.5+0.85pt,   end angle=180]
		(1,0.75) ++(0, 0.85pt)
		arc[x radius=\xunit*2.0+0.68pt, start angle=90,  
		    y radius=\yunit*9.0+1.36pt,   end angle=180];
	\draw[seam] (0.2,0) ++(-0.85pt,0)
		arc[x radius=\xunit*6.0+0.17pt, start angle=180,  
		    y radius=\yunit*1.5+0.51pt,   end angle=270];
	\foamarrow#1[seam]-10(0.414,-0.115);
	\foamarrow#1[seam]8(-0.7,-0.13);
}

\deffoampict{1cup > 2plane}(-1.4,-0.8)--(1.4,1.15) M=0.2 [#1]=[.]{%
	\draw[2facetInner] (-0.6,0.3)
		++(0,-0.85pt) -- ++(-0.8, -0.6) -- ++(2, 0) -- ++(0.8, 0.6) -- cycle;
	\draw[2facetFront] (-0.6,0.3)
		++(0, 0.85pt) -- ++(-0.8, -0.6) -- ++(2, 0) -- ++(0.8, 0.6) -- cycle;
	\foamarrow#1[2facetLine](-0.6,-0.3);
	\foamarrow#1[2facetLine]180(0.6,0.3);
	\fill[1facetBack] (-0.5,1)
			arc[x radius=0.5, y radius=0.15, start angle=180, end angle=0]
			.. controls (0.5,0) and (-0.5,0) .. (-0.5,1);
	\fill[1facetFront] (-0.5,1)
			arc[x radius=0.5, y radius=0.15, start angle=180, end angle=360]
			.. controls (0.5,0) and (-0.5,0) .. (-0.5,1);
	\draw[1facetLine]
		(0,1) ellipse[x radius=0.5, y radius=0.15]
		(0.5,1) .. controls (0.5,0) and (-0.5,0) .. (-0.5,1);
}

\deffoampict{1cup < 2plane}(-1.4,-0.8)--(1.4,1.15) M=0.2 [#1#2]=[>]{%
	\fill[2facetInner] (-0.5,-0.85pt)
			arc[x radius=0.5, y radius=0.15, start angle=180, end angle=0] --
			++(0,0.3) -- ++(-1,0) -- cycle;
	\draw[2facetLine] (-0.5,-0.85pt) ++(0,0.3) -- ++(1,0);
	\fill[2facetFront] (-0.5, 0.85pt)
			arc[x radius=0.5, y radius=0.15, start angle=180, end angle=0] --
			++(0,0.3) -- ++(-1,0) -- cycle;
	\draw[2facetLine] (-0.5,0.85pt) ++(0,0.3) -- ++(1,0);
	\fill[1facetBack]  (-0.5, 1) -- (-0.5,0)
			.. controls (-0.5,-1) and (0.5,-1) .. (0.5,0) -- (0.5,1)
			arc[x radius=0.5, y radius=0.15, start angle=0, end angle=180];
	\draw[seam]
		(-0.5, 0.85pt) arc[x radius=0.5, y radius=0.15, start angle=180, end angle=0]
		(-0.5,-0.85pt) arc[x radius=0.5, y radius=0.15, start angle=180, end angle=0];
	\fill[1facetFront] (-0.5, 1) -- (-0.5,0)
			.. controls (-0.5,-1) and (0.5,-1) .. (0.5,0) -- (0.5,1)
			arc[x radius=0.5, y radius=0.15, start angle=360, end angle=180];
	\draw[1facetLine]
			(0,1) ellipse[x radius=0.5, y radius=0.15]
			(-0.5,0.85pt) .. controls (-0.5,-1) and (0.5,-1) .. (0.5,0.85pt);
	\fill[2facetInner] (0.5,-0.85pt)
			arc[x radius=0.5, y radius=0.15, start angle=360, end angle=180] |-
			++(-0.1,0.3) -- ++(-0.8,-0.6) -- ++(2,0) -- ++(0.8,0.6) -| cycle;
	\draw[seam] (0.5,-0.85pt)
			arc[x radius=0.5, y radius=0.15, start angle=360, end angle=180];
	\draw[2facetLine] (-0.5,0.3) ++(0,-0.85pt) --
			++(-0.1,0) -- ++(-0.8,-0.6) -- ++(2,0) -- ++(0.8,0.6) -- ++(-0.9,0);
	\fill[2facetFront] (0.5, 0.85pt)
			arc[x radius=0.5, y radius=0.15, start angle=360, end angle=180] |-
			++(-0.1,0.3) -- ++(-0.8,-0.6) -- ++(2,0) -- ++(0.8,0.6) -| cycle;
	\draw[seam] (0.5, 0.85pt)
			arc[x radius=0.5, y radius=0.15, start angle=360, end angle=180];
	\draw[2facetLine] (-0.5,0.3) ++(0, 0.85pt) --
			++(-0.1,0) -- ++(-0.8,-0.6) -- ++(2,0) -- ++(0.8,0.6) -- ++(-0.9,0);
	\draw[1facetLine] (-0.5,0.85pt) -- (-0.5,1) (0.5,0.85pt) -- (0.5,1);
	\ifx\relax#2\relax\else
		\foamarrow#1[2facetLine](-0.7,-0.3);
		\foamarrow#1[2facetLine]180(0.7,0.3);
	\fi
	\foamarrow#1[seam](0,-0.15);
}

\deffoampict{1cap > 2plane}(-1.4,-1.15)--(1.4,0.8) M=-0.2 [#1#2]=[>]{%
	\fill[2facetInner] (0.5,-0.85pt)
		arc[x radius=0.5, y radius=0.15, start angle=0, end angle=180] --
		++(-0.9,-0.3) -- ++(0.8,0.6) -- ++(2,0) -- ++(-0.8,-0.6) -- cycle;
	\draw[2facetLine] (0.5,-0.85pt)
		arc[x radius=0.5, y radius=0.15, start angle=0, end angle=180]
		++(-0.9,-0.3) -- ++(0.8,0.6) -- ++(2,0);
	\fill[2facetFront] (0.5, 0.85pt)
		arc[x radius=0.5, y radius=0.15, start angle=0, end angle=180] --
		++(-0.9,-0.3) -- ++(0.8,0.6) -- ++(2,0) -- ++(-0.8,-0.6) -- cycle;
	\draw[2facetLine] (0.5, 0.85pt)
		arc[x radius=0.5, y radius=0.15, start angle=0, end angle=180]
		++(-0.9,-0.3) -- ++(0.8,0.6) -- ++(2,0);
	\fill[1facetBack] (-0.5,-1) -- (-0.5,0)
		.. controls (-0.5,1) and (0.5,1) .. (0.5,0) -- (0.5,-1)
		arc[x radius=0.5, y radius=0.15, start angle=0, end angle=180];
	\draw[1facetLine] (0.5,-1)
		arc[x radius=0.5, y radius=0.15, start angle=0, end angle=180];
	\draw[1facetFront] (-0.5,-1) -- (-0.5,0)
		.. controls (-0.5,1) and (0.5,1) .. (0.5,0) -- (0.5,-1)
		arc[x radius=0.5, y radius=0.15, start angle=360, end angle=180];
	\fill[2facetInner] (-0.5, -0.85pt)
		arc[x radius=0.5, y radius=0.15, start angle=180, end angle=360] --
		++(0.1,-0.3) -- ++(-2,0) -- cycle;
	\draw[2facetLine] (-0.5, -0.85pt)
		arc[x radius=0.5, y radius=0.15, start angle=180, end angle=360]
		++(0.9, 0.3) -- ++(-0.8,-0.6) -- ++(-2,0);
	\fill[2facetFront] (-0.5, 0.85pt)
		arc[x radius=0.5, y radius=0.15, start angle=180, end angle=360] --
		++(0.1,-0.3) -- ++(-2,0) -- cycle;
	\draw[2facetLine] (-0.5, 0.85pt)
		arc[x radius=0.5, y radius=0.15, start angle=180, end angle=360]
		++(0.9, 0.3) -- ++(-0.8,-0.6) -- ++(-2,0);
	\ifx\relax#2\relax\else
		\foamarrow#1[2facetLine](-0.7,-0.3);
		\foamarrow#1[2facetLine]180(0.7,0.3);
	\fi
	\foamarrow#1[seam](0,-0.15);
}

\deffoampict{1cap < 2plane}(-1.4,-1.15)--(1.4,0.8) M=-0.2 [#1]=[.]{%
	\fill[1facetBack] (-0.5,-1)
			arc[x radius=0.5, y radius=0.15, start angle=180, end angle=0]
			.. controls (0.5,0) and (-0.5,0) .. (-0.5,-1);
	\draw[1facetLine]
			(-0.5,-1) arc[x radius=0.5, y radius=0.15, start angle=180, end angle=0];
	\draw[1facetFront] (-0.5,-1)
			arc[x radius=0.5, y radius=0.15, start angle=180, end angle=360]
			.. controls (0.5,0) and (-0.5,0) .. (-0.5,-1);
	\draw[2facetInner] (-0.6, 0.3) ++(0,-0.85pt) -- ++(-0.8,-0.6) --
		++(2,0) -- ++(0.8,0.6) -- ++(-2,0);
	\draw[2facetFront] (-0.6, 0.3) ++(0, 0.85pt) -- ++(-0.8,-0.6) --
		++(2,0) -- ++(0.8,0.6) -- ++(-2,0);
	\foamarrow#1[2facetLine](-0.6,-0.3);
	\foamarrow#1[2facetLine]180(0.6,0.3);
}

\deffoampict{Vsaddle > 2plane}(-1.5,-0.8)--(1.5,1.1) M=0.1 [#1]=[.]{%
	\draw[2facetInner] (-0.7,0.2) ++(0,-0.85pt) --
		++(0.4,0.4) -- ++(1.8,0) -- ++(-0.4,-0.4);
	\draw[2facetFront] (-0.7,0.2) ++(0, 0.85pt) --
		++(0.4,0.4) -- ++(1.8,0) -- ++(-0.4,-0.4);
	\fill[1facetBack] (-0.7,1.1)
			arc[x radius=0.4, y radius=0.2, start angle=90, end angle=0]
			arc[x radius=0.3, y radius=0.5, start angle=180, end angle=360]
			arc[x radius=0.8, y radius=0.2, start angle=180, end angle=90]
			-- (1.1,-0.4) -- (-0.7,-0.4) -- cycle;
	\draw[1facetLine] (1.1, 1.1) |- (-0.7,-0.4) -- (-0.7,1.1);
	\draw[seam]
		(-0.7,0.2) ++(0,-0.85pt) -- ++(1.8,0)
		(-0.7,0.2) ++(0, 0.85pt) -- ++(1.8,0);
	\foamarrow#1[seam]{180}(0.3,0.2);
	\fill[1facetFront] (-1.1,0.7)
			arc[x radius=0.8, y radius=0.2, start angle=270, end angle=360]
			arc[x radius=0.3, y radius=0.5, start angle=180, end angle=360]
			arc[x radius=0.4, y radius=0.2, start angle=180, end angle=270]
			-- (0.7,-0.8) -- (-1.1,-0.8) -- cycle;
	\draw[1facetLine]
			(0.3, 0.9) arc[x radius=0.3, y radius=0.5, start angle=360, end angle=180]
			(-0.7,1.1) arc[x radius=0.4, y radius=0.2, start angle= 90, end angle=0]
			           arc[x radius=0.8, y radius=0.2, start angle=360, end angle=270]
			-- (-1.1,-0.8) -- (0.7,-0.8) -- (0.7,0.7)
			           arc[x radius=0.4, y radius=0.2, start angle=270, end angle=180]
			           arc[x radius=0.8, y radius=0.2, start angle=180, end angle=90];
	\draw[2facetInner] (-1.1,-0.2) ++(0,-0.85pt) --
		++(-0.4,-0.4) -- ++(1.8,0) -- ++(0.4, 0.4);
	\draw[seam] (-1.1,-0.2) ++(0,-0.85pt) -- ++(1.8,0);
	\draw[2facetFront] (-1.1,-0.2) ++(0, 0.85pt) --
		++(-0.4,-0.4) -- ++(1.8,0) -- ++(0.4, 0.4);
	\draw[seam] (-1.1,-0.2) ++(0, 0.85pt) -- ++(1.8,0);
	\foamarrow#1[seam](-0.3,-0.2);
}

\deffoampict{Vsaddle < 2plane}(-1.5,-1.2)--(1.5,0.7) M=-0.3 [#1]=[.]{%
	\fill[2facetInner] (0.3,-0.85pt)
			arc[x radius=0.8, y radius=0.2, start angle=180, end angle=90]
			-- ++(0.4,0.4) -- ++(-1.8,0) -- ++(-0.4,-0.4)
			arc[x radius=0.4, y radius=0.2, start angle=90, end angle=0];
	\draw[2facetLine] (1.1,0.2) ++(0,-0.85pt)
		-- ++(0.4,0.4) -- ++(-1.8,0) -- ++(-0.4,-0.4);
	\fill[2facetFront] (0.3, 0.85pt)
			arc[x radius=0.8, y radius=0.2, start angle=180, end angle=90]
			-- ++(0.4,0.4) -- ++(-1.8,0) -- ++(-0.4,-0.4)
			arc[x radius=0.4, y radius=0.2, start angle=90, end angle=0];
	\draw[2facetLine] (1.1,0.2) ++(0, 0.85pt)
		-- ++(0.4,0.4) -- ++(-1.8,0) -- ++(-0.4,-0.4);
	\fill[1facetBack] (-0.7,0.7)
			arc[x radius=0.4, y radius=0.2, start angle=90, end angle=0] -- (-0.3,0)
			arc[x radius=0.3, y radius=0.4, start angle=180, end angle=360] -- (0.3,0.5)
			arc[x radius=0.8, y radius=0.2, start angle=180, end angle=90]
			-- (1.1,-0.8) -- (-0.7,-0.8) -- cycle;
	\draw[1facetLine] (1.1,0.7) |- (-0.7,-0.8) -- (-0.7,0.7);
	\draw[seam]
		( 0.3,-0.85pt) arc[x radius=0.8, y radius=0.2, start angle=180, end angle=90]
		( 0.3, 0.85pt) arc[x radius=0.8, y radius=0.2, start angle=180, end angle=90]
		(-0.3,-0.85pt) arc[x radius=0.4, y radius=0.2, start angle=0, end angle=90]
		(-0.3, 0.85pt) arc[x radius=0.4, y radius=0.2, start angle=0, end angle=90];
	\fill[1facetFront] (-1.1, 0.3)
			arc[x radius=0.8, y radius=0.2, start angle=270, end angle=360] -- (-0.3,0)
			arc[x radius=0.3, y radius=0.4, start angle=180, end angle=360] -- (0.3,0.5)
			arc[x radius=0.4, y radius=0.2, start angle=180, end angle=270]
			-- (0.7,-1.2) -- (-1.1,-1.2) -- cycle;
	\draw[1facetLine] (-0.3,0.5) -- (-0.3,0)
			arc[x radius=0.3, y radius=0.4, start angle=180, end angle=360] -- (0.3,0.5)
			(-0.7,0.7)
			arc[x radius=0.4, y radius=0.2, start angle=90, end angle=0]
			arc[x radius=0.8, y radius=0.2, start angle=360, end angle=270]
			|- (0.7,-1.2) -- (0.7,0.3)
			arc[x radius=0.4, y radius=0.2, start angle=270, end angle=180]
			arc[x radius=0.8, y radius=0.2, start angle=180, end angle=90];
	\fill[2facetInner] (0.3,-0.85pt)
			arc[x radius=0.4, y radius=0.2, start angle=180, end angle=270]
			-- ++(-0.4,-0.4) -- ++(-1.8,0) -- ++(0.4,0.4)
			arc[x radius=0.8, y radius=0.2, start angle=270, end angle=360];
	\draw[2facetLine] (0.7,-0.2) ++(0,-0.85pt) -- ++(-0.4,-0.4) -- ++(-1.8,0) -- ++(0.4,0.4);
	\draw[seam]
		( 0.3,-0.85pt) arc[x radius=0.4, y radius=0.2, start angle=180, end angle=270]
		(-0.3,-0.85pt) arc[x radius=0.8, y radius=0.2, start angle=360, end angle=270];
	\fill[2facetFront] (0.3, 0.85pt)
			arc[x radius=0.4, y radius=0.2, start angle=180, end angle=270]
			-- ++(-0.4,-0.4) -- ++(-1.8,0) -- ++(0.4,0.4)
			arc[x radius=0.8, y radius=0.2, start angle=270, end angle=360];
	\draw[2facetLine] (0.7,-0.2) ++(0, 0.85pt) -- ++(-0.4,-0.4) -- ++(-1.8,0) -- ++(0.4,0.4);
	\draw[seam]
		( 0.3, 0.85pt) arc[x radius=0.4, y radius=0.2, start angle=180, end angle=270]
		(-0.3, 0.85pt) arc[x radius=0.8, y radius=0.2, start angle=360, end angle=270];
	\foamarrow#1[seam]-19(0.468,-0.163);
	\foamarrow#1[seam]11(-0.782,-0.184);
}

\deffoampict{Asaddle > 2plane}(-1.5,-0.7)--(1.5,1.2) M=0.1 [#1]=[.]{%
	\fill[2facetInner] (0.3,-0.85pt)
			arc[x radius=0.8, y radius=0.2, start angle=180, end angle=90]
			-- ++(0.4,0.4) -- ++(-1.8,0) -- ++(-0.4,-0.4)
			arc[x radius=0.4, y radius=0.2, start angle=90, end angle=0];
	\draw[2facetLine] (1.1,0.2) ++(0,-0.85pt)
		-- ++(0.4,0.4) -- ++(-1.8,0) -- ++(-0.4,-0.4);
	\fill[2facetFront] (0.3, 0.85pt)
			arc[x radius=0.8, y radius=0.2, start angle=180, end angle=90]
			-- ++(0.4,0.4) -- ++(-1.8,0) -- ++(-0.4,-0.4)
			arc[x radius=0.4, y radius=0.2, start angle=90, end angle=0];
	\draw[2facetLine] (1.1,0.2) ++(0, 0.85pt)
		-- ++(0.4,0.4) -- ++(-1.8,0) -- ++(-0.4,-0.4);
	\fill[1facetBack] (-0.7,-0.3)
			arc[x radius=0.4, y radius=0.2, start angle=90, end angle=0] -- (-0.3,0)
			arc[x radius=0.3, y radius=0.4, start angle=180, end angle=0] -- (0.3,-0.5)
			arc[x radius=0.8, y radius=0.2, start angle=180, end angle=90]
			-- (1.1,1.2) -- (-0.7,1.2) -- cycle;
	\draw[1facetLine] (-0.3,-0.5)
			arc[x radius=0.4, y radius=0.2, start angle=0, end angle=90]
			-- (-0.7,1.2) -| (1.1,-0.3)
			arc[x radius=0.8, y radius=0.2, start angle=90, end angle=180];
	\draw[seam]
		( 0.3,-0.85pt) arc[x radius=0.8, y radius=0.2, start angle=180, end angle=90]
		( 0.3, 0.85pt) arc[x radius=0.8, y radius=0.2, start angle=180, end angle=90]
		(-0.3,-0.85pt) arc[x radius=0.4, y radius=0.2, start angle=0, end angle=90]
		(-0.3, 0.85pt) arc[x radius=0.4, y radius=0.2, start angle=0, end angle=90];
	\draw[1facetFront] (-1.1,-0.7)
			arc[x radius=0.8, y radius=0.2, start angle=-90, end angle=0] -- (-0.3,0)
			arc[x radius=0.3, y radius=0.4, start angle=180, end angle=0] -- (0.3,-0.5)
			arc[x radius=0.4, y radius=0.2, start angle=180, end angle=270]
			|- (-1.1,0.8) -- cycle;
	\fill[2facetInner] (0.3,-0.85pt)
			arc[x radius=0.4, y radius=0.2, start angle=180, end angle=270]
			-- ++(-0.4,-0.4) -- ++(-1.8,0) -- ++(0.4,0.4)
			arc[x radius=0.8, y radius=0.2, start angle=270, end angle=360];
	\draw[2facetLine] (0.7,-0.2) ++(0,-0.85pt) -- ++(-0.4,-0.4) -- ++(-1.8,0) -- ++(0.4,0.4);
	\draw[seam]
		( 0.3,-0.85pt) arc[x radius=0.4, y radius=0.2, start angle=180, end angle=270]
		(-0.3,-0.85pt) arc[x radius=0.8, y radius=0.2, start angle=360, end angle=270];
	\fill[2facetFront] (0.3, 0.85pt)
			arc[x radius=0.4, y radius=0.2, start angle=180, end angle=270]
			-- ++(-0.4,-0.4) -- ++(-1.8,0) -- ++(0.4,0.4)
			arc[x radius=0.8, y radius=0.2, start angle=270, end angle=360];
	\draw[2facetLine] (0.7,-0.2) ++(0, 0.85pt) -- ++(-0.4,-0.4) -- ++(-1.8,0) -- ++(0.4,0.4);
	\draw[seam]
		( 0.3, 0.85pt) arc[x radius=0.4, y radius=0.2, start angle=180, end angle=270]
		(-0.3, 0.85pt) arc[x radius=0.8, y radius=0.2, start angle=360, end angle=270];
	\foamarrow#1[seam]-19(0.468,-0.163);
	\foamarrow#1[seam]11(-0.782,-0.184);
}

\deffoampict{Asaddle < 2plane}(-1.5,-1.1)--(1.5,0.8) M=-0.3 [#1]=[.]{%
	\draw[2facetInner] (-0.7,0.2) ++(0,-0.85pt)
		-- ++(0.4,0.4) -- ++(1.8,0) -- ++(-0.4,-0.4);
	\draw[2facetFront] (-0.7,0.2) ++(0, 0.85pt)
		-- ++(0.4,0.4) -- ++(1.8,0) -- ++(-0.4,-0.4);
	\fill[1facetBack] (-0.7,-0.7)
			arc[x radius=0.4, y radius=0.2, start angle=90, end angle=0]
			arc[x radius=0.3, y radius=0.5, start angle=180, end angle=0]
			arc[x radius=0.8, y radius=0.2, start angle=180, end angle=90]
			-- (1.1,0.8) -- (-0.7,0.8) -- cycle;
	\draw[1facetLine] (-0.3,-0.9)
			arc[x radius=0.4, y radius=0.2, start angle=0, end angle=90]
			|- (1.1,0.8) -- (1.1,-0.7)
			arc[x radius=0.8, y radius=0.2, start angle=90, end angle=180];
	\draw[seam]
		(-0.7,0.2) ++(0,-0.85pt) -- ++(1.8,0)
		(-0.7,0.2) ++(0, 0.85pt) -- ++(1.8,0);
	\foamarrow#1[seam]{180}(0.3,0.2);
	\draw[1facetFront] (-1.1,-1.1)
			arc[x radius=0.8, y radius=0.2, start angle=-90, end angle=0]
			arc[x radius=0.3, y radius=0.5, start angle=180, end angle=0]
			arc[x radius=0.4, y radius=0.2, start angle=180, end angle=270]
			|- (-1.1,0.4) -- cycle;
	\draw[2facetInner] (-1.1,-0.2) ++(0,-0.85pt)
		-- ++(-0.4,-0.4)-- ++(1.8,0) -- ++(0.4,0.4);
	\draw[seam] (-1.1,-0.2) ++(0,-0.85pt) -- ++(1.8,0);
	\draw[2facetFront] (-1.1,-0.2) ++(0, 0.85pt)
		-- ++(-0.4,-0.4)-- ++(1.8,0) -- ++(0.4,0.4);
	\draw[seam] (-1.1,-0.2) ++(0, 0.85pt) -- ++(1.8,0);
	\foamarrow#1[seam](-0.3,-0.2);
}

\deffoampict{cup}(-0.7, 0.2)--(0.7,-1.4) [#1]=[1]{%
	\fill[#1facetBack]
		(-0.5,0) .. controls ++(0,-1.0) and ++(0,-1.0) .. (0.5,0)
		arc[x radius=0.5, y radius=0.15, start angle=0, end angle=180];
	\fill[#1facetFront]
		(-0.5,0) .. controls ++(0,-1.0) and ++(0,-1.0) .. (0.5,0)
		arc[x radius=0.5, y radius=0.15, start angle=360, end angle=180];
	\draw[#1facetLine]
		(-0.5,0) .. controls ++(0,-1.0) and ++(0,-1.0) .. (0.5,0)
		(0,0) ellipse[x radius=0.5, y radius=0.15];
}

\deffoampict{cap}(-0.7,-0.2)--(0.7,1.4) [#1]=[1]{%
	\fill[#1facetBack]
		(-0.5,0) .. controls ++(0,1.0) and ++(0,1.0) .. (0.5,0)
		arc[x radius=0.5, y radius=0.15, start angle=0, end angle=180];
	\draw[#1facetLine]
		(0.5,0) arc[x radius=0.5, y radius=0.15, start angle=0, end angle=180];
	\draw[#1facetFront]
		(-0.5,0) .. controls ++(0,1.0) and ++(0,1.0) .. (0.5,0)
		arc[x radius=0.5, y radius=0.15, start angle=360, end angle=180];
}

\deffoampict{saddle}(-1,-1.1)--(1,0.5) [#1]=[1]{%
	\fill[#1facetBack] (-0.6,0.45)
			arc[x radius=0.3, y radius=0.15, start angle=90, end angle=0]
			arc[x radius=0.3, y radius=0.60, start angle=180, end angle=360]
			arc[x radius=0.6, y radius=0.15, start angle=180, end angle=90]
			-- (0.9,-0.7) -- (-0.6,-0.7) -- cycle;
	\draw[#1facetLine] (0.9,0.45) |- (-0.6,-0.7) -- (-0.6,0.45);
	\fill[#1facetFront] (-0.9, 0.15)
			arc[x radius=0.6, y radius=0.15, start angle=270, end angle=360]
			arc[x radius=0.3, y radius=0.60, start angle=180, end angle=360]
			arc[x radius=0.3, y radius=0.15, start angle=180, end angle=270]
			-- (0.6,-1) -- (-0.9,-1) -- cycle;
	\draw[#1facetLine] (-0.3,0.3)
			arc[x radius=0.3, y radius=0.6, start angle=180, end angle=360]
			(-0.6,0.45)
			arc[x radius=0.3, y radius=0.15, start angle=90, end angle=0]
			arc[x radius=0.6, y radius=0.15, start angle=360, end angle=270]
			|- (0.6,-1) -- (0.6,0.15)
			arc[x radius=0.3, y radius=0.15, start angle=270, end angle=180]
			arc[x radius=0.6, y radius=0.15, start angle=180, end angle=90];
}

\deffoampict{dot}(-0.1,-0.1)--(1.3,1.6) [#1]=[1]{%
	\draw[#1facetFront]
		(0,0) .. controls ++(0.7,0.1) and ++(-0.5,-0.3) .. (1.2,0.4) --
		(1.2,1.5) .. controls ++(-0.5,-0.3) and ++(0.7,0.1) .. (0,1.1) -- cycle;
	\foamdot (0.65,0.7);
}

\deffoampict{merge}(-1.6,-0.2)--(1.6,2.2) M=1 [#1]=[1]{%
	\fill[#1facetBack] (-1.5,0)
		arc[x radius=0.5, y radius=0.15, start angle=180, end angle=0]
		.. controls ++(0,0.8) and ++(0,0.8) .. (0.5,0)
		arc[x radius=0.5, y radius=0.15, start angle=180, end angle=0]
		.. controls ++(0,1.2) and ++(0,-1) .. (0.5, 2)
		arc[x radius=0.5, y radius=0.15, start angle=0, end angle=180]
		.. controls ++(0,-1) and ++(0,1.2) .. (-1.5, 0);
	\draw[#1facetLine]
		(-1.5,0) arc[x radius=0.5, y radius=0.15, start angle=180, end angle=0]
		( 0.5,0) arc[x radius=0.5, y radius=0.15, start angle=180, end angle=0];
	\fill[#1facetFront] (-0.5,2)
		.. controls ++(0,-1) and ++(0,1.2) .. (-1.5,0)
		arc[x radius=0.5, y radius=0.15, start angle=180, end angle=360]
		.. controls ++(0,0.8) and ++(0,0.8) .. (0.5,0)
		arc[x radius=0.5, y radius=0.15, start angle=180, end angle=360]
		.. controls ++(0,1.2) and ++(0,-1) .. (0.5, 2)
		arc[x radius=0.5, y radius=0.15, start angle=360, end angle=180];
	\draw[#1facetLine] (-0.5,2)
		.. controls ++(0,-1) and ++(0,1.2) .. (-1.5,0)
		arc[x radius=0.5, y radius=0.15, start angle=180, end angle=360]
		.. controls ++(0,0.8) and ++(0,0.8) .. (0.5,0)
		arc[x radius=0.5, y radius=0.15, start angle=180, end angle=360]
		.. controls ++(0,1.2) and ++(0,-1) .. (0.5, 2)
		(0,2) ellipse[x radius=0.5, y radius=0.15];
}

\deffoampict{split}(-1.6, 0.2)--(1.6,-2.2) M=-1 [#1]=[1]{%
	\fill[#1facetBack] (-1.5,0)
		arc[x radius=0.5, y radius=0.15, start angle=180, end angle=0]
		.. controls ++(0,-0.8) and ++(0,-0.8) .. (0.5,0)
		arc[x radius=0.5, y radius=0.15, start angle=180, end angle=0]
		.. controls ++(0,-1.2) and ++(0,1) .. (0.5,-2)
		arc[x radius=0.5, y radius=0.15, start angle=0, end angle=180]
		.. controls ++(0,1) and ++(0,-1.2) .. (-1.5, 0);
	\draw[#1facetLine]
		(-0.5,-2) arc[x radius=0.5, y radius=0.15, start angle=180, end angle=0];
	\fill[#1facetFront] (-1.5,0)
		arc[x radius=0.5, y radius=0.15, start angle=180, end angle=360]
		.. controls ++(0,-0.8) and ++(0,-0.8) .. (0.5,0)
		arc[x radius=0.5, y radius=0.15, start angle=180, end angle=360]
		.. controls ++(0,-1.2) and ++(0,1) .. (0.5,-2)
		arc[x radius=0.5, y radius=0.15, start angle=360, end angle=180]
		.. controls ++(0,1) and ++(0,-1.2) .. (-1.5, 0);
	\draw[#1facetLine]
		(-0.5,0) .. controls ++(0,-0.8) and ++(0,-0.8) .. (0.5,0)
		( 1.5,0) .. controls ++(0,-1.2) and ++(0,1) .. (0.5,-2)
		arc[x radius=0.5, y radius=0.15, start angle=360, end angle=180]
		.. controls ++(0,1) and ++(0,-1.2) .. (-1.5, 0)
		(-1,0) ellipse[x radius=0.5, y radius=0.15]
		( 1,0) ellipse[x radius=0.5, y radius=0.15];
}

\deffoampict{pocket above}(-0.5,-0.95)--(0.85,0.95) {%
	\fill[2facetBack] (0.15,0.7)
		arc[start angle= 90, end angle=180, x radius=0.60, y radius=0.22]
		arc[start angle=180, end angle=270, x radius=0.45, y radius=0.48]
		arc[start angle=270, end angle=360, x radius=0.15, y radius=0.70];
	\fill[2facetFront] (-0.15,0.3)
		arc[start angle=270, end angle=180, x radius=0.30, y radius=0.18]
		arc[start angle=180, end angle=270, x radius=0.45, y radius=0.48]
		arc[start angle=270, end angle=180, x radius=0.15, y radius=0.30];
	\draw[2facetLine] (0,0)
		arc[start angle=270, end angle=180, x radius=0.45, y radius=0.48]
		(0.15,0.7)
		arc[start angle= 90, end angle=180, x radius=0.60, y radius=0.22]
		arc[start angle=180, end angle=270, x radius=0.30, y radius=0.18];
	\draw[1facetFront]
		(-0.3, 0.1) -- (0.3, 0.9) -- (0.3, -0.1) -- (-0.3, -0.9) -- cycle;
	\fill[2facetBack] (0.15, 0.7) -- (0.8, 0.7) -- (0.8, -0.3)
		arc[start angle=90, end angle=180, x radius=0.5, y radius=0.22]
		arc[start angle=0, end angle=90, x radius=0.3, y radius=0.52]
		arc[start angle=270, end angle=360, x radius=0.15, y radius=0.7];
	\draw[seam] (0,0) arc[start angle=270, end angle=360, x radius=0.15, y radius=0.7];
	\draw[2facetLine] (0.15, 0.7) -| (0.8, -0.3)
		arc[start angle=90, end angle=180, x radius=0.5, y radius=0.22];
	\fill[2facetFront] (-0.15, 0.3) -- (0.5, 0.3) -- (0.5, -0.7)
		arc[start angle=270, end angle=180, x radius=0.2, y radius=0.18]
		arc[start angle=0, end angle=90, x radius=0.3, y radius=0.52]
		arc[start angle=270, end angle=180, x radius=0.15, y radius=0.3];
	\draw[seam] (0,0) arc[start angle=270, end angle=180, x radius=0.15, y radius=0.3];
	\draw[2facetLine] (-0.15, 0.3) -- (0.5, 0.3) -- (0.5, -0.7)
		arc[start angle=270, end angle=180, x radius=0.2, y radius=0.18]
		arc[start angle=0, end angle=90, x radius=0.3, y radius=0.52];
}

\deffoampict{pocket below}(-0.5,-0.95)--(0.9,0.95) {%
	\fill[2facetBack] (0.15,-0.3)
		arc[start angle= 90, end angle=180, x radius=0.60, y radius=0.22]
		arc[start angle=180, end angle= 90, x radius=0.45, y radius=0.52]
		arc[start angle= 90, end angle=  0, x radius=0.15, y radius=0.30];
	\draw[2facetLine] (0.15,-0.3)
		arc[start angle= 90, end angle=180, x radius=0.60, y radius=0.22];
	\fill[2facetFront] (-0.15,-0.7)
		arc[start angle=270, end angle=180, x radius=0.30, y radius=0.18]
		arc[start angle=180, end angle= 90, x radius=0.45, y radius=0.52]
		arc[start angle= 90, end angle=180, x radius=0.15, y radius=0.7];
	\draw[2facetLine] (-0.15,-0.7)
		arc[start angle=270, end angle=180, x radius=0.30, y radius=0.18]
		arc[start angle=180, end angle= 90, x radius=0.45, y radius=0.52];
	\draw[1facetFront]
		(-0.3, 0.1) -- (0.3, 0.9) -- (0.3, -0.1) -- (-0.3, -0.9) -- cycle;
	\fill[2facetBack] (0.15,-0.3) -| (0.85, 0.7)
		arc[start angle=90, end angle=180, x radius=0.5, y radius=0.22]
		arc[start angle=0, end angle=-90, x radius=0.35, y radius=0.48]
		arc[start angle=90, end angle=0, x radius=0.15, y radius=0.3];
	\draw[seam] (0, 0) arc[start angle=90, end angle=0, x radius=0.15, y radius=0.3];
	\draw[2facetLine] (0.15,-0.3) -| (0.85,0.7)
		arc[start angle=90, end angle=180, x radius=0.5, y radius=0.22];
	\fill[2facetFront] (0.35,0.48)
		arc[start angle=0, end angle=-90, x radius=0.35, y radius=0.48]
		arc[start angle=90, end angle=180, x radius=0.15, y radius=0.7] -| (0.55,0.3)
		arc[start angle=270, end angle=180, x radius=0.2, y radius=0.18];
	\draw[seam] (0,0) arc[start angle=90, end angle=180, x radius=0.15, y radius=0.7];
	\draw[2facetLine] (-0.15,-0.7) -| (0.55,0.3)
		arc[start angle=270, end angle=180, x radius=0.2, y radius=0.18]
		(0,0) arc[start angle=270, end angle=360, x radius=0.35, y radius=0.48];
}

\deffoampict{red cup > blue plane}(-1.9,-0.5)--(1.8,1.3) M=0.2 {%
	\draw[1facetFront] (-1.8,-0.4) -- (-0.6,0.4) -- (1.7,0.4) -- (0.5,-0.4) -- cycle;
	\fill[2facetBack] (-0.5,1) .. controls (-0.5,0) and (0.5,0) .. (0.5,1)
		arc[x radius=0.5, y radius=0.15, start angle=0, end angle=180];
	\fill[2facetFront] (-0.5,1) .. controls (-0.5,0) and (0.5,0) .. (0.5,1)
		arc[x radius=0.5, y radius=0.15, start angle=360, end angle=180];
	\draw[2facetLine]
		(-0.5,1) .. controls (-0.5,0) and (0.5,0) .. (0.5,1)
	   (0,1) ellipse[x radius=0.5, y radius=0.15];
}

\deffoampict{red cup < blue plane}(-1.9,-0.8)--(1.8,1.3) M=0.2 {%
	\fill[2facetBack] (-0.5,0) .. controls (-0.5,-1) and (0.5,-1) .. (0.5,0)
		arc[x radius=0.5, y radius=0.15, start angle=0, end angle=180];
	\fill[2facetFront] (-0.5,0) .. controls (-0.5,-1) and (0.5,-1) .. (0.5,0)
		arc[x radius=0.5, y radius=0.15, start angle=360, end angle=180];
	\draw[2facetLine] (-0.5,0) .. controls (-0.5,-1) and (0.5,-1) .. (0.5,0);
	\draw[1facetFront] (-1.8,-0.4) -- (-0.6,0.4) -- (1.7,0.4) -- (0.5,-0.4) -- cycle;
	\fill[2facetBack] (0.5,0)
		-- ( 0.5,1) arc[x radius=0.5, y radius=0.15, start angle=0, end angle=180]
		-- (-0.5,0) arc[x radius=0.5, y radius=0.15, start angle=180, end angle=0];
	\draw[seam]
		(0.5,0) arc[x radius=0.5, y radius=0.15, start angle=0, end angle=180];
	\fill[2facetFront] (0.5,0)
		-- ( 0.5,1) arc[x radius=0.5, y radius=0.15, start angle=360, end angle=180]
		-- (-0.5,0) arc[x radius=0.5, y radius=0.15, start angle=180, end angle=360];
	\draw[2facetLine]
		(0,1) ellipse[x radius=0.5, y radius=0.15]
		( 0.5,1) -- ( 0.5,0)
		(-0.5,1) -- (-0.5,0);
	\draw[seam] (0.5,0)
		arc[x radius=0.5, y radius=0.15, start angle=360, end angle=180];
}

\deffoampict{red saddle > blue plane}(-1.5,-0.7)--(1.5,1.2) M=0.1 {%
	\fill[2facetBack]
		(0.3,-0.5) arc[x radius=0.8, y radius=0.2, start angle=180, end angle=90] --
		(1.1, 0.2) arc[x radius=0.8, y radius=0.2, start angle=90, end angle=180];
	\fill[2facetBack]
		(-0.3,-0.5) arc[x radius=0.4, y radius=0.2, start angle=0, end angle=90] --
		(-0.7, 0.2) arc[x radius=0.4, y radius=0.2, start angle=90, end angle=0];
	\draw[2facetLine]
		( 0.3,-0.5) arc[x radius=0.8, y radius=0.2, start angle=180, end angle=90]
		-- ( 1.1,0.2)
		(-0.3,-0.5) arc[x radius=0.4, y radius=0.2, start angle=0, end angle=90]
		-- (-0.7,0.2);
	\fill[2facetFront]
		( 0.3,-0.5) arc[x radius=0.4, y radius=0.2, start angle=180, end angle=270] --
		( 0.7,-0.2) arc[x radius=0.4, y radius=0.2, start angle=270, end angle=180];
	\fill[2facetFront]
		(-0.3,-0.5) arc[x radius=0.8, y radius=0.2, start angle=0, end angle=-90] --
		(-1.1,-0.2) arc[x radius=0.8, y radius=0.2, start angle=-90, end angle=0];
	\draw[2facetLine] (0.3,0) -- ( 0.3,-0.5)
		arc[x radius=0.4, y radius=0.2, start angle=180, end angle=270] -- (0.7,-0.2)
		(-0.3,0) -- (-0.3,-0.5)
		arc[x radius=0.8, y radius=0.2, start angle=0, end angle=-90] -- (-1.1,-0.2);
	\draw[1facetFront] (1.4,0.5) -- (0.4,-0.5) -- (-1.4,-0.5) -- (-0.4,0.5) -- cycle;
	\fill[2facetBack] (-0.7,0.2)
		arc[x radius=0.4, y radius=0.2, start angle=90, end angle=0]
		arc[x radius=0.3, y radius=0.4, start angle=180, end angle=0]
		arc[x radius=0.8, y radius=0.2, start angle=180, end angle=90] |- (-0.7,1.2);
	\draw[seam] ( 0.3,0) arc[x radius=0.8, y radius=0.2, start angle=180, end angle=90]
	            (-0.3,0) arc[x radius=0.4, y radius=0.2, start angle=  0, end angle=90];
	\draw[2facetLine] (1.1,0.2) |- (-0.7,1.2) -- (-0.7,0.2);
	\fill[2facetFront] (-1.1,0.8) -- (-1.1,-0.2)
		arc[x radius=0.8, y radius=0.2, start angle=270, end angle=360]
		arc[x radius=0.3, y radius=0.4, start angle=180, end angle=0  ]
		arc[x radius=0.4, y radius=0.2, start angle=180, end angle=270] |- cycle;
	\draw[seam] ( 0.3,0) arc[x radius=0.4, y radius=0.2, start angle=180, end angle=270]
	            (-0.3,0) arc[x radius=0.8, y radius=0.2, start angle=360, end angle=270];
	\draw[2facetLine] (-1.1,-0.2) -- (-1.1,0.8) -| (0.7,-0.2);
}

\deffoampict{red saddle < blue plane}(-1.5,-1.1)--(1.5,0.8) M=-0.3 {%
	\fill[2facetBack] (-0.7,-0.7)
		arc[x radius=0.4, y radius=0.2, start angle=90, end angle=0]
		arc[x radius=0.3, y radius=0.5, start angle=180, end angle=0]
		arc[x radius=0.8, y radius=0.2, start angle=180, end angle=90] |- (-0.7,0.2);
	\draw[2facetLine] (0.3,-0.9)
		arc[x radius=0.8, y radius=0.2, start angle=180, end angle=90] -- (1.1,0.2)
		(-0.3,-0.9)
		arc[x radius=0.4, y radius=0.2, start angle=0, end angle=90] -- (-0.7,0.2);
	\draw[2facetFront] (-1.1,-0.2) -- (-1.1,-1.1)
		arc[x radius=0.8, y radius=0.2, start angle=270, end angle=360]
		arc[x radius=0.3, y radius=0.5, start angle=180, end angle=0  ]
		arc[x radius=0.4, y radius=0.2, start angle=180, end angle=270] -- (0.7,-0.2);
	\draw[1facetFront] (1.4,0.5) -- (0.4,-0.5) -- (-1.4,-0.5) -- (-0.4,0.5) -- cycle;
	\draw[2facetBack]  (-0.7, 0.2) -- (-0.7, 0.8) -- (1.1, 0.8) -- (1.1, 0.2);
	\draw[seam] (-0.7, 0.2) -- (1.1, 0.2);
	\draw[2facetFront] (-1.1,-0.2) -- (-1.1, 0.4) -- (0.7, 0.4) -- (0.7,-0.2);
	\draw[seam] (-1.1,-0.2) -- (0.7,-0.2);
}

\deffoampict{inner cylinder}(-0.9,-0.5)--(1.2,2.5) M=1 {%
	\draw[2facetBack] (-0.8,0)
		arc[x radius=0.8, y radius=0.4, start angle=180, end angle=90]
		-- (1.1, 0.4) |- (0,2.4)
		arc[x radius=0.8, y radius=0.4, start angle= 90, end angle=180];
	\fill[1facetBack] (-0.4,0)
		arc[x radius=0.4, y radius=0.15, start angle=180, end angle=0] -- ++(0,2)
		arc[x radius=0.4, y radius=0.15, start angle=0, end angle=180];
	\draw[1facetLine]
		(-0.4,0) arc[x radius=0.4, y radius=0.15, start angle=180, end angle=0];
	\fill[1facetFront] (-0.4,0)
		arc[x radius=0.4, y radius=0.15, start angle=180, end angle=360] -- ++(0,2)
		arc[x radius=0.4, y radius=0.15, start angle=360, end angle=180];
	\draw[1facetLine]
		(0,2) ellipse[x radius=0.4, y radius=0.15]
		(-0.4,2) -- (-0.4,0)
		arc[x radius=0.4, y radius=0.15, start angle=180, end angle=360] -- ++(0,2);
	\draw[2facetFront] (-0.8,2) -- (-0.8,0)
		arc[x radius=0.8, y radius=0.4, start angle=180, end angle=270]
		-| (0.8, 1.6) -- (0,1.6)
		arc[x radius=0.8, y radius=0.4, start angle=270, end angle=180];
}

\deffoampict{inner cylinder cut}(-0.9,-0.5)--(1.2,2.5) M=1 {%
	\begin{scope}
		\clip (-0.6, 1) rectangle (0.6, 2);
		\fill[1facetBack, rotate around={-10:(0,1.5)}]
			(-0.4,1.5) arc[x radius=0.4, y radius=0.12, start angle=180, end angle=0]
			-- ++(-80:0.5) -- ++(190:0.8) -- cycle;
		\fill[1facetFront, rotate around={-10:(0,1.5)}]
			(-0.4,1.5) arc[x radius=0.4, y radius=0.12, start angle=180, end angle=360]
			-- ++(-80:0.5) -- ++(190:0.8);
	\end{scope}
	\draw[1facetLine]
		(-0.4,1.569) -- (-0.4,1)
		( 0.4,1.431) -- ( 0.4,1);
	\draw[2facetBack] (-0.8,0)
		.. controls (-0.8,1.0) and ( 0.6,0.5) .. ( 0.6,1)
		.. controls ( 0.6,1.5) and (-0.8,1.0) .. (-0.8,2)
		arc[x radius=0.8, y radius=0.4, start angle=180, end angle=90]
		-| (1.1, 0.4) -- (0,0.4)
		arc[x radius=0.8, y radius=0.4, start angle= 90, end angle=180];
	\begin{scope}
		\clip (-0.4,2) arc[x radius=0.4, y radius=0.15, start angle=180, end angle=0]
			-| ( 0.6,1) -| (-0.6,2) -- cycle;
		\fill[1facetBack, rotate around={-10:(0,1.5)}]
			(0.4,1.5) arc[x radius=0.4, y radius=0.12, start angle=0, end angle=180]
			-- ++(100:0.7) -- ++(10:0.8) -- cycle;
	\end{scope}
	\begin{scope}
		\clip (-0.4,2) arc[x radius=0.4, y radius=0.15, start angle=180, end angle=360]
			-| ( 0.6,1) -| (-0.6,2) -- cycle;
		\fill[1facetFront, rotate around={-10:(0,1.5)}]
			(0.4,1.5) arc[x radius=0.4, y radius=0.12, start angle=360, end angle=180]
			-- ++(100:0.7) -- ++(10:0.8) -- cycle;
	\end{scope}
	\draw[seam] (0,1.5) ellipse[x radius=0.4, y radius=0.12, rotate=-10];
	\draw[1facetLine]
		(0,2) ellipse[x radius=0.4, y radius=0.15]
		(-0.4,1.569) -- (-0.4,2)
		( 0.4,1.431) -- ( 0.4,2);
	\begin{scope}
		\clip (-0.4,0) arc[x radius=0.4, y radius=0.15, start angle=180, end angle=0]
			-| ( 0.6,1) -| (-0.6,2) -- cycle;
		\fill[1facetBack, rotate around={10:(0,0.5)}]
			(-0.4,0.5) arc[x radius=0.4, y radius=0.12, start angle=180, end angle=0]
			-- ++(260:0.7) -- ++(170:0.8) -- cycle;
		\draw[1facetLine]
			(-0.4,0) arc[x radius=0.4, y radius=0.15, start angle=180, end angle=0];
	\end{scope}
	\begin{scope}
		\clip (-0.4,0) arc[x radius=0.4, y radius=0.15, start angle=180, end angle=360]
			-| ( 0.6,1) -| (-0.6,2) -- cycle;
		\fill[1facetFront, rotate around={10:(0,0.5)}]
			(-0.4,0.5) arc[x radius=0.4, y radius=0.12, start angle=180, end angle=360]
			-- ++(260:0.8) -- ++(170:0.8) -- cycle;
	\end{scope}
	\draw[1facetLine]
		(-0.4,0.431) -- (-0.4,0)
		arc[x radius=0.4, y radius=0.15, start angle=180, end angle=360] -- (0.4,0.569);
	\draw[2facetFront] (-0.8,2)
		arc[x radius=0.8, y radius=0.4, start angle=180, end angle=270]
		-| (0.8,-0.4) -- (0,-0.4)
		arc[x radius=0.8, y radius=0.4, start angle=270, end angle=180]
		.. controls (-0.8,1.0) and ( 0.6,0.5) .. ( 0.6,1)
		.. controls ( 0.6,1.5) and (-0.8,1.0) .. (-0.8,2);
	\draw[seam,shift={(0,0.5)}] [rotate=10] (170:0.4 and 0.12)
		arc[x radius=0.4, y radius=0.12, start angle=170, end angle=-10];
	\begin{scope}
		\clip (-0.6, 1) rectangle (0.6, 0);
		\fill[1facetBack, rotate around={10:(0,0.5)}]
			(-0.4,0.5) arc[x radius=0.4, y radius=0.12, start angle=180, end angle=0]
			-- ++(80:0.5) -- ++(170:0.8) -- cycle;
		\fill[1facetFront, rotate around={10:(0,0.5)}]
			(-0.4,0.5) arc[x radius=0.4, y radius=0.12, start angle=180, end angle=360]
			-- ++(80:0.5) -- ++(170:0.8) -- cycle;
	\draw[seam,shift={(0,0.5)}] [rotate=10] (170:0.4 and 0.12)
		arc[x radius=0.4, y radius=0.12, start angle=170, end angle=370];
	\end{scope}
	\draw[1facetLine]
		( 0.4,0.569) -- ( 0.4,1)
		(-0.4,0.431) -- (-0.4,1);
}

\deffoampict{cutting plane}(-0.9,-0.9)--(1.2,2.9) M=1 {%
	\draw[1facetFront] (0.3,2.4) -- (0.45,2.8) -- (0.45,0.8) -- (0.3,0.4);
	\draw[2facetBack] (-0.8,0)
		arc[x radius=0.8, y radius=0.4, start angle=180, end angle=90]
		-- (1.1, 0.4) |- (0,2.4)
		arc[x radius=0.8, y radius=0.4, start angle= 90, end angle=180];
	\fill[1facetFront] (0, 1.6) -- (0.3,2.4) -- (0.3,0.4) -- (0,-0.4);
	\draw[seam] (0.3,2.4) -- (0.3,0.4);
	\draw[1facetLine] (0.3,0.4) -- (0,-0.4);
	\draw[2facetFront] (-0.8,2) -- (-0.8,0)
		arc[x radius=0.8, y radius=0.4, start angle=180, end angle=270]
		-| (0.8, 1.6) -- (0,1.6)
		arc[x radius=0.8, y radius=0.4, start angle=270, end angle=180];
	\fill[1facetFront] (-0.15,-0.8) -- (0,-0.4) -- (0,1.6) -- (-0.15,1.2);
	\draw[seam] (0,-0.4) -- (0,1.6);
	\draw[1facetLine] (0.45,2.8) -- (-0.15,1.2) -- (-0.15,-0.8) -- (0,-0.4);
}

\deffoampict{cutting plane cut}(-0.9,-0.9)--(1.2,2.9) M=1 {%
	\begin{scope}
		\clip (0.3, 2.4)
			arc[x radius=0.15, y radius=1.12, start angle=360, end angle=270] --
			(0.15,0.72) arc[x radius=0.15, y radius=0.32, start angle= 90, end angle=0]
			|- (-0.8,0) |- cycle;
		\fill[2facetBack] (-0.8,0)
			.. controls (-0.8,1.0) and ( 0.6,0.5) .. ( 0.6,1)
			.. controls ( 0.6,1.5) and (-0.8,1.0) .. (-0.8,2)
			arc[x radius=0.8, y radius=0.4, start angle=180, end angle=90]
			-| (1.1, 0.4) -- (0,0.4)
			arc[x radius=0.8, y radius=0.4, start angle= 90, end angle=180];
	\end{scope}
	\draw[2facetLine]
		(-0.8,0) arc[x radius=0.8, y radius=0.4, start angle=180, end angle=90]
		-- (0.3,0.4);
	\begin{scope}
		\clip (0,1.6)
			arc[x radius=0.15, y radius=0.32, start angle=180, end angle=270] --
			(0.15,0.72) arc[x radius=0.15, y radius=1.12, start angle=90, end angle=180]
			-| (-0.9,2.4) -| cycle;
		\fill[2facetFront] (-0.8,2)
			arc[x radius=0.8, y radius=0.4, start angle=180, end angle=270]
			-| (0.8,-0.4) -- (0,-0.4)
			arc[x radius=0.8, y radius=0.4, start angle=270, end angle=180]
			.. controls (-0.8,1.0) and ( 0.6,0.5) .. ( 0.6,1)
			.. controls ( 0.6,1.5) and (-0.8,1.0) .. (-0.8,2);
		\draw[2facetLine](0,-0.4)
			arc[x radius=0.8, y radius=0.4, start angle=270, end angle=180]
			.. controls (-0.8,1.0) and ( 0.6,0.5) .. ( 0.6,1)
			.. controls ( 0.6,1.5) and (-0.8,1.0) .. (-0.8,2);
	\end{scope}
	\draw[1facetFront]
		(-0.15,-0.8) -- (0.45, 0.8) -- (0.45, 2.8) -- (-0.15,1.2) -- cycle;
	\begin{scope}
		\clip (0.3, 2.4)
			arc[x radius=0.15, y radius=1.12, start angle=360, end angle=270] --
			(0.15,0.72) arc[x radius=0.15, y radius=0.32, start angle= 90, end angle=0]
			-| (1.1,2.4);
		\fill[2facetBack] (-0.8,0)
			.. controls (-0.8,1.0) and ( 0.6,0.5) .. ( 0.6,1)
			.. controls ( 0.6,1.5) and (-0.8,1.0) .. (-0.8,2)
			-- (-0.8,2.4) -| (1.1, 0.4) -| cycle;
	\end{scope}
	\draw[seam]
		(0.15,1.28) arc[x radius=0.15, y radius=1.12, start angle=270, end angle=360]
		(0.3 ,0.4 ) arc[x radius=0.15, y radius=0.32, start angle=0, end angle=90];
	\draw[2facetLine] (0.3,2.4) -| (1.1,0.4) -- (0.3,0.4);
	\begin{scope}
		\clip (0,1.6)
			arc[x radius=0.15, y radius=0.32, start angle=180, end angle=270] --
			(0.15,0.72) arc[x radius=0.15, y radius=1.12, start angle=90, end angle=180]
			-| (0.8,1.6) -- cycle;
		\fill[2facetFront] (-0.8,2)
			arc[x radius=0.8, y radius=0.4, start angle=180, end angle=270]
			-| (0.8,-0.4) -- (0,-0.4)
			arc[x radius=0.8, y radius=0.4, start angle=270, end angle=180]
			.. controls (-0.8,1.0) and ( 0.6,0.5) .. ( 0.6,1)
			.. controls ( 0.6,1.5) and (-0.8,1.0) .. (-0.8,2);
		\draw[2facetLine] (-0.8,0)
			.. controls (-0.8,1.0) and ( 0.6,0.5) .. ( 0.6,1)
			.. controls ( 0.6,1.5) and (-0.8,1.0) .. (-0.8,2);
	\end{scope}
	\draw[seam]
		(0, 1.6) arc[x radius=0.15, y radius=0.32, start angle=180, end angle=270]
		(0,-0.4) arc[x radius=0.15, y radius=1.12, start angle=180, end angle= 90];
	\draw[2facetLine]
		(0,-0.4) -| (0.8,1.6) -- (0,1.6)
		arc[x radius=0.8, y radius=0.4, start angle=270, end angle=90]
		-- (0.3,2.4);
}

\deffoampict{edge zip}(-0.8,-0.5)--(0.8,0.5) {%
	\coordinate (lend) at (-0.3,0.35);
	\coordinate (rend) at ( 0.3,0.35);
	\draw[2facetInner,draw=seamColor] ([yshift=0.6pt,xshift=-0.2pt]lend)
		.. controls ++([yshift=-0.9pt]0,-0.5) and ++([yshift=-0.9pt]0,-0.5)
		.. ([yshift=0.6pt]rend);
	\draw[2facetLine] ([yshift=0.6pt,xshift=-0.2pt]lend) -- ([yshift=0.6pt]rend);
	\fill[1facetBack] (-0.8, 0.5)
		arc[x radius=0.5, y radius=0.15, start angle=90,  end angle=0]
		.. controls ++(0,-0.5) and ++(0,-0.5) .. (0.3,0.35)
		arc[x radius=0.3, y radius=0.15, start angle=180, end angle=90]
		-- (0.6,-0.2) .. controls (0,-0.35) and (-0.2,-0.35) .. (-0.8,-0.2)
		-- cycle;
	\draw[1facetLine] (0.3,0.35)
		arc[x radius=0.3, y radius=0.15, start angle=180, end angle=90]
		-- (0.6,-0.2)
		.. controls (0,-0.35) and (-0.2,-0.35) .. (-0.8,-0.2) -- (-0.8,0.5)
		arc[x radius=0.5, y radius=0.15, start angle=90,  end angle=0];
	\fill[1facetFront] (0.8, 0.2)
		arc[x radius=0.5, y radius=0.15, start angle=270,  end angle=180]
		.. controls ++(0,-0.5) and ++(0,-0.5) .. (-0.3,0.35)
		arc[x radius=0.3, y radius=0.15, start angle=360, end angle=270]
		-- (-0.6,-0.5) .. controls (0,-0.35) and (0.2,-0.35) .. (0.8,-0.5)
		-- cycle;
	\draw[1facetLine] (-0.3,0.35)
		.. controls ++(0,-0.5) and ++(0,-0.5) .. (0.3,0.35);
	\draw[2facetFront,draw=seamColor] ([yshift=-0.6pt]lend)
		.. controls ++([yshift=-0.9pt]0,-0.5) and ++([yshift=-0.9pt]0,-0.5)
		.. ([yshift=-0.6pt,xshift=0.2pt]rend);
	\draw[2facetLine] ([yshift=-0.6pt]lend) -- ([yshift=-0.6pt,xshift=0.2pt]rend);
	\draw[1facetLine] (-0.3,0.35)
		arc[x radius=0.3, y radius=0.15, start angle=360, end angle=270]
		-- (-0.6,-0.5)  .. controls (0,-0.35) and (0.2,-0.35) .. (0.8,-0.5) -- (0.8,0.2)
		arc[x radius=0.5, y radius=0.15, start angle=270,  end angle=180];
}

\deffoampict{edge unzip}(-0.8,-0.5)--(0.8,0.5) {%
	\coordinate (lend) at (-0.3,-0.35);
	\coordinate (rend) at ( 0.3,-0.35);
	\fill[1facetBack] (-0.8,-0.2)
		arc[x radius=0.5, y radius=0.15, start angle=90,  end angle=0]
		.. controls ++(0, 0.5) and ++(0, 0.5) .. (0.3,-0.35)
		arc[x radius=0.3, y radius=0.15, start angle=180, end angle=90]
		-- (0.6, 0.5) .. controls (0, 0.35) and (-0.2, 0.35) .. (-0.8, 0.5)
		-- cycle;
	\draw[1facetLine] (0.3,-0.35)
		arc[x radius=0.3, y radius=0.15, start angle=180, end angle=90]
		-- (0.6, 0.5) .. controls (0, 0.35) and (-0.2, 0.35) .. (-0.8, 0.5)
		-- (-0.8,-0.2)
		arc[x radius=0.5, y radius=0.15, start angle=90,  end angle=0];
	\draw[2facetInner,draw=seamColor] ([yshift=0.6pt,xshift=-0.2pt]lend)
		.. controls ++([yshift=0.9pt]0,0.5) and ++([yshift=0.9pt]0,0.5)
		.. ([yshift=0.6pt]rend);
	\draw[2facetLine] ([yshift=0.6pt,xshift=-0.2pt]lend) -- ([yshift=0.6pt]rend);
	\draw[1facetFront] (0.8,-0.5)
		arc[x radius=0.5, y radius=0.15, start angle=270,  end angle=180]
		.. controls ++(0, 0.5) and ++(0, 0.5) .. (-0.3,-0.35)
		arc[x radius=0.3, y radius=0.15, start angle=360, end angle=270]
		-- (-0.6, 0.2) .. controls (0, 0.35) and (0.2, 0.35) .. (0.8, 0.2)
		-- cycle;
	\draw[2facetFront,draw=seamColor] ([yshift=-0.6pt]lend)
		.. controls ++([yshift=0.9pt]0,0.5) and ++([yshift=0.9pt]0,0.5)
		.. ([yshift=-0.6pt,xshift=0.2pt]rend);
	\draw[2facetLine] ([yshift=-0.6pt,xshift=0.2pt]rend) -- ([yshift=-0.6pt]lend);
}

\tracepicturesfalse
\drawpicturestrue

\def\authordata#1#2#3{#1\\\parbox{0.4\textwidth}{\centering\normalsize#2\\#3}}

\begin{document}

\title{%
	On the~functoriality of $\LieSL$ tangle homology
}

\author{%
	\authordata{Anna Beliakova}
		{Universit\"at Z\"urich}{Z\"urich, Switzerland}
\and
	\authordata{Matthew Hogancamp}
		{University of Southern California}{Los Angeles, CA}
\and
	\authordata{Krzysztof K.\ Putyra}
		{Universit\"at Z\"urich}{Z\"urich, Switzerland}
\and
	\authordata{Stephan M.\ Wehrli}
		{Syracuse University}{Syracuse, NY}
}

\maketitle

\begin{abstract}
	We construct an~explicit equivalence between the~(bi)category
	of $\LieGL$ webs and foams and the~Bar-Natan (bi)category of
	Temperley--Lieb diagrams and cobordisms. With this equivalence
	we can fix functoriality of every link homology theory that
	factors through the~Bar-Natan category. To achieve this, we
	define web versions of arc algebras and their quasi-hereditary
	covers, which provide strictly functorial tangle homologies.
	Furthermore, we construct explicit isomorphisms between these
	algebras and the~original ones based on Temperley--Lieb cup
	diagrams. The~immediate application
	is a~strictly functorial version of the~Beliakova--Putyra--Wehrli
	quantization of the annular link homology. 
\end{abstract}

\setcounter{tocdepth}{2}
\tableofcontents

\section{Introduction}

In 1999 Khovanov \cite{KhHom} defined for any link in the 3-sphere
a~chain complex, whose homotopy type---hence, homology---is a~link
invariant and whose Euler characteristic is the Jones polynomial.
It was later extended to tangles between even collections of points
\cite{KhArcAlgebras} and then to all tangles \cite{ChenKhov, BrunStroII}.
The~main advantage of the~Khovanov homology with respect to the~Jones
polynomial is that link cobordisms induce chain maps between Khovanov's
complexes \cite{KhFunct,Jacobsson,DrorCob}. Even though the~original
construction is not strictly functorial---the~sign of the~chain map 
associated with a~link cobordism depends on the~decomposition of
the~cobordism into elementary pieces \cite{Jacobsson}---it was used
by Rasmussen to provide a~lower bound for the~slice genus of a~knot
and a~combinatorial proof of the~Milnor conjecture \cite{SliceGenusBound}.

In the~last 15 years there were many attempts to fix the~functoriality
of Khovanov homology. In \cite{ClarkMorrisonWalker, CaprauFoams, Vogel}
this was done by modifying the Bar-Natan category \cite{DrorCob} and
enlarging the~ground ring. In 2014 Blanchet \cite{Blanchet} proposed
a~more elegant solution, which does not change the~ring of scalars, but
replaces circles and surfaces in the~Bar-Natan category with \emph{webs}
and \emph{foams}: certain planar trivalent graphs and singular cobordisms
between them respectively.
This construction, commonly referred to as \emph{$\LieGL$ homology},
has been widely accepted as the~most natural way to fix functoriality
of Khovanov homology.
A~priori potentially different, $\LieGL$ homology coincides
with Khovanov homology in case of links \cite{Blanchet}, but the~case
of tangles has been analyzed only partially by Ehrig, Stroppel and
Tubbenhauer in \cite{WebAlgebras}.

The~Hochschild homology of the~Chen--Khovanov invariant of an~$(n,n)$-tangle
$T$ has been identified in \cite{QntHom} with the~annular Khovanov homology
of the~annular closure $\widehat T$ of the~tangle. In the~same paper
the~annular invariant has been quantized by deforming the~Hochschild
homology. Our original goal was to make this quantized annular homology
functorial, in order to construct its colored version following
\cite{KhColored} and \cite{CooperKrushkal}. These quantized colored
homologies are treated in the~follow up paper \cite{QntColored,ViennaTalk},
where we also show that both complexes coincide when the~deformation
parameter is generic.
In order to obtain a~strictly functorial quantized annular homology,
we wanted first to understand the~Ehrig--Stroppel--Tubbenhauer isomorphism
between Khovanov's arc algebras and their web algebras, and then reconstruct
the~Chen--Khovanov functor in the~framework of webs and foams. However,
after a~chain of simplifications of their arguments, especially replacing
the~foam basis used in \cite{WebAlgebras} with another one, more natural
from the~topological perspective, we understood the~real reason why all
the~isomorphisms popped out: \emph{foams and cobordisms constitute equivalent
bicategories}. By using a~particularly nice basis of foams, we construct
such a~equivalence explicitly and use it to obtain a~web versions of
the~TQFT functors from \cite{KhArcAlgebras, ChenKhov, BrunStroII}.

In the~following sections we discuss the~above in more details.

%% ===================================================================

\subsection{The~equivalence of foams and Bar-Natan cobordisms}
In order to compute Khovanov homology of a~link $L$, one first picks its
diagram $D$ and constructs the~\emph{cube of resolutions} of $D$:
a~commutative diagram in the~shape of the~$c$-dimension cube, where
$c$ counts crossings in $D$, with vertices decorated by Kauffman resolutions
of $D$ and edges by saddle cobordisms between them \cite{KhHom}.
Applying a~2-dimensional TQFT to this cube, changing signs of some
maps, decorating edges, and collapsing the~cube along diagonals results
in an~actual chain complex, which---depending on the~choice of the~TQFT
functor---computes the~Khovanov homology of $L$ or its deformation.

It was observed by Bar-Natan that most of the~construction can be performed
formally \emph{before} applying a~TQFT functor to get an~invariant of a~tangle
$T$ in the~form of a~formal complex $\KhBracket{T}$ called the~\emph{Khovanov
bracket} of $T$ \cite{DrorCob}. This complex is constructed in the~\emph
{Bar-Natan} bicategory $\cat{BN}$, the~locally additive graded bicategory
with objects collections of points on a~line, 1-morphisms generated by flat
tangles, and 2-morphisms generated by surfaces with dots modulo the~following
local relations:
\begin{itemize}
	\item \textit{sphere evaluations:}
	\begin{equation}\label{rel:sphere-eval0}
	\tikzset{x=8mm,y=8mm}%
		\foampict{1sphere}    = 0\hskip 0.1\textwidth
		\foampict{1sphere}[1] = 1
	\end{equation}
	
	\item \textit{neck cutting relation:}
	\begin{equation}\label{rel:neck-cutting0}
	\tikzset{x=7mm,y=7mm}%
		\mathclap{\foampict{1neck}\, = \,\foampict{1cap 1cup}[+1]
		                          \, + \,\foampict{1cap 1cup}[-1]
			                       \, - h \foampict{1cap 1cup}}
	\end{equation}
	
	\item \textit{dot reduction:}
	\begin{equation}\label{rel:dots0}
	\tikzset{x=1cm,y=1cm}%
		\foampict{1plane}[2] = h\!\foampict{1plane}[1] + t\!\foampict{1plane}
	\end{equation}
\end{itemize}
Here $h$ and $t$ are fixed elements of the~ring of scalars $\scalars$.
When $h=0$, then the~neck cutting relation evaluates a~handle attached to
a~plane as a~dot scaled by 2. Because of that it is common to think of a~dot
as ,,half'' of a~handle, even when 2 is not an~invertible scalar. However,
this interpretation is not correct if $h\neq 0$, in particular in the~universal
case $\scalars = \Z[h,t]$.

The~formal bracket is projectively functorial \cite{DrorCob}. Indeed,
there is a~way to associate a~formal chain map with each Reidemeister move
as well as any cobordism with a~unique critical point. One constructs
a~formal chain map for any tangle cobordism by decomposing the~cobordism
into a~sequence of the~above elementary pieces and composing the~associated maps;
choosing a~different decomposition may at most change the~global sign of the~map.

In Blanchet's construction \cite{Blanchet} the~role of flat tangles is played
by \emph{webs}, trivalent graphs with each edge colored blue or red,%
\footnote{
	When compared to \cite{Blanchet}, blue edges are those with label 1
	and red edges are those with label 2.
}
and dotted surfaces are replaced with \emph{foams}, which are singular
cobordisms with each facet also colored blue or red. They
constitute a~bicategory $\cat{Foam}$, where certain local relations
between foams, including \eqref{rel:sphere-eval0}--\eqref{rel:dots0},
are imposed (see Definition~\ref{def:foams}). Following \cite{DrorCob}
we can construct a~formal complex $\wKhBracket{T}$ in $\cat{Foam}$,
which we refer to as the~\emph{Blanchet--Khovanov bracket}.

The~collection of blue edges of a~web $\web$ is a~flat tangle $\web_b$,
which we call the~\emph{underlying tangle} of $\web$. Likewise, there is
an~\emph{underlying surface} $\foam_b$ associated with any foam $\foam$.
It is tempting to consider a~2-functor $\cat{Foam} \to
\cat{BN}$ that forgets red edges in webs and red facets in foams.
However, this operation is not compatible with relations between foams,
and it is not clear at first how to solve this problem. For instance,
it was observed in \cite{KhViaHowe} that if such a~functor exists, then
it cannot be identity on all foams with no red facets.

We resolved the~above problem by taking into account the~orientation of
blue edges and facets. Shortly speaking, we fix an~orientation for each
flat tangle and surface in a~canonical way, reinterpreting them as webs
and foams respectively (recall that tangles and surfaces from $\cat{BN}$,
though orientable, come with no particular orientation). This results in
a~2-functor, which however does not reach every object of $\cat{Foam}$.
In order to fix this we replace $\cat{BN}$ with the~product
$\cat{wBN} := \cat{BN} \times \Z$, where $\Z$ is seen as a~discrete
bicategory. We use the~extra integer to determine how many
red points, edges, or facets has to be added to the~right of the~oriented
blue points, tangle, or surface respectively.%
\footnote{
	Compare this with the~relation between the~weight lattices
	of $\LieSL$ and $\LieGL$---the~latter is isomorphic to the~product
	of the~former with $\Z$.
}
This way we end up with a~2-functor
$\EqFunc\colon \cat{wBN} \to \cat{Foam}$, such that every object
of $\cat{Foam}$ is equivalent to one from the~image of $\EqFunc$.

\pretheorem{Theorem}{thm:equiv-of-bicats}{%
	The 2-functor $\EqFunc\colon \cat{wBN} \to \cat{Foam}$ is
	an~equivalence of bicategories.
}

From the~point of view of representation theory, % the~2-functor
$\EqFunc$ and its inverse can be understood as the~categorification
of the~induction--restriction pair between representations of
$\LieSL$ and $\LieGL$.

There is also a~local version of Theorem~\ref{thm:equiv-of-bicats}.
Having fixed a~collection $\bdry$ of oriented blue and red points on
$\partial\Disk$, write $\cat{Foam}(\bdry)$ for the~category of webs in $\Disk$
bounded by $\bdry$ and foams in $\DxI$ between such webs. Likewise we
consider the~category $\cat{BN}(\bdry_b)$ of flat tangles bounded by
$\bdry_b$ and dotted surfaces between them, where $\bdry_b$ is
the~collection of blue points from $\bdry$.
We construct a~functor $\EqFunc_\bdry\colon \cat{BN}(\bdry_b) \to
\cat{Foam}(\bdry)$ in Section~\ref{sec:local BN-->Foam} 
by extending coherently all flat tangles to webs bounded by $\bdry$
and surfaces to foams.

\pretheorem{Theorem}{thm:equiv-of-cats}{%
	The~functor $\EqFunc_\bdry \colon \cat{BN}(\bdry_b)
		\to \cat{Foam}(\bdry)$
	is an~equivalence of categories.
}

%\begin{itheorem}\label{thm:cat-equiv}
%	The~functor $\EqFunc_\bdry \colon \cat{BN}(\bdry_b)
%		\to \cat{Foam}(\bdry)$
%	is an~equivalence of categories.
%\end{itheorem}

We construct the~functor $\EqFunc_\bdry$ explicitly as well as
its inverse $\EqFunc_\bdry^\vee$. The~latter not only forgets red
facets of foams, but also scales them by a~sign when necessary;
we provide an~explicit way to compute these signs in terms of
the~Blanchet evaluation of foams.
When combined with a~homological argument presented in \cite{OddKh, ChCob},
Theorem~\ref{thm:equiv-of-cats} implies that for every tangle $T$
the~image of the~Khovanov bracket $\KhBracket{T}$ under $\EqFunc_\bdry$
is isomorphic to the~Blanchet--Khovanov bracket $\wKhBracket{T}$.
Hence, any TQFT functor on $\cat{BN}(\bdry_b)$ that leads to
an~invariant tangle or link homology can be precomposed with
$\EqFunc_\bdry^\vee$ to obtain a~functor on $\cat{Foam}(\bdry)$
that computes the~same homology groups, but which is strictly
functorial with respect to tangle cobordisms.

\subsubsection{Main tools: shadings and bicolored isotopies}

The~key step in the~proofs of Theorems~\ref{thm:equiv-of-bicats} and
\ref{thm:equiv-of-cats} is to understand how foams with the~same underlying
surface are related. We achieve this by constructing foams from
\emph{shadings}. A~shading is a~union of two possibly intersecting surfaces:
a~non-oriented blue and an~oriented red one, that are in general position
in $\R^3$, together with a~checkerboard black and white coloring of
the~connected components of their complement, called regions.
Forgetting those red facets of a~shading, the~orientations of which
disagree with the~one induced from the~white regions, results in a~foam,
and all foams can be constructed this way. The~same applies to webs.

A~particularly nice feature of representing foams by shadings is
the~flexibility of this construction, which we call the~\emph{bicolored
isotopy argument}: deforming any of the~two surfaces by an~isotopy results
in a~foam that differs from the~original one only up to a~sign or replacing
some dots with their duals (see Proposition~\ref{prop:foams-red} in
Section~\ref{sec:foams} for a~precise statement). This has a~number
of important consequences:
\begin{itemize}
	\item closed foams can be evaluated (Theorem~\ref{thm:foam-evaluation})
	using the~bicolored isotopy argument by moving the~blue and red facets
	away from each other,
	\item more generally, foams with the~same boundary and underlying surfaces
	coincide up to a~sign and types of dots (Proposition~\ref{prop:foams-red}),
	\item a~foam, the~underlying surface of which is a~product $\web\times[0,1]$,
	is invertible.
\end{itemize}
We then use the~above to construct a~basis of the~space of foams bounded by
a~closed web $\web$. It is given in terms of shadings of a~plane that extends
$\web$, the~blue loops of which may carry dots. The~foam associating with
such a~picture $\web^+$ is given by attaching blue and red cups to the~loops
of $\web^+$---red cups above all blue ones---and placing a~dot at the~minimum
of every blue cup attached to a~loop that is marked by a~dot.
This leads to an~explicit description of the~tautological TQFT
functor on $\cat{Foam}(\emptyset)$ that associates the~space
$\Hom_{\cat{Foam}(\emptyset)}(\emptyset, \web)$ with a~closed web $\web$,
presented in Section~\ref{sec:TQFT}. When compared with \cite{WebAlgebras},
our basis is not only easier to visualize, but also the~formula for
the~action of foams involves less signs.

%% ===================================================================
\subsection{Functorial tangle homology}
Khovanov extended his construction first to tangles with an~even number of
boundary points at each side \cite{KhArcAlgebras}. For this he constructed
a~2-functor $\FKh^\circ\colon \cat{BN}^\circ \to \cat{Bimod}$,
where $\cat{BN}^\circ$ is the~subbicategory of $\cat{BN}$ with only even
collections of points as objects.
The~2-functor $\FKh^\circ$ associates with a~collection of $2n$ points
the~\emph{arc algebra}
\begin{equation}\label{eq:arc-alg ala Lev}
	\arcalg n := \bigoplus_{a,b}
		\Hom_{\cat{BN}}(a,b),
\end{equation}
where $a$ and $b$ run through the~set of Temperley--Lieb cup diagrams in
$\R\times(-\infty,0]$ with $2n$ boundary points at the~top boundary line.%
\footnote{
	This presentation of $\arcalg n$ comes from \cite{StableKh}.
}
This algebra $\arcalg n$ is known to categorify the~invariant subspace
$\mathrm{Inv}(V^{\otimes n})$ of $V^{\otimes n}$, where $V$ is the~fundamental
representation of $\Uqsl$.
Cup diagrams parametrize indecomposable projective $\arcalg n$-modules,
which in turn correspond to elements of the~canonical basis of $V^{\otimes n}$.
Let $\KhCom(T)$ be the~chain complex associated with an~$(2n,2n')$-tangle
$T$, i.e.\ the~result of applying $\FKh^\circ$ to $\KhBracket{T}$.
The~functors $\KhCom(T)\otimes(\blank)$ lift the~action of tangles
on $\mathrm{Inv}(V^{\otimes n})$ to the~derived categories of
the~arc algebras \cite{KhArcAlgebras}.

In order to categorify the~whole tensor power $V^{\otimes n}$, Chen and
Khovanov considered a~family of algebras $\CKalg{k,n-k}\rics$, where
$0\leq k\leq n$, each constructed as a~subquotient of $\arcalg n$.
These algebras were discovered independently by Stroppel \cite{ParabolicO},
who proved with Brundan than they are quasi-hereditary covers of arc algebras
and Koszul \cite{BrunStroI, BrunStroII}. Furthermore, projective modules
over $\CKalg{k,n-k}$ categorify the~weight space $V^{\otimes n}(\lambda)$
with $\lambda = n-2k$ \cite{ChenKhov, BrunStroI}. As in the~case
of arc algebras, there is a~family of 2-functors
$\FKh^\lambda\colon \cat{BN} \to \cat{Bimod}$,
such that $\FKh^\lambda$ assigns to a~collection of $n$ points the~algebra
$\CKalg{k,n-k}$ with $\lambda = n-2k$ \cite{ChenKhov, BrunStroII}.
Write $\KhCom(T;\lambda)$ for the~result of applying $\FKh^\lambda$ to
$\KhBracket{T}$. Then the~functor $\KhCom(T;\lambda)\otimes(\blank)$
lifts the~action of $T$ on the~weight space $V^{\otimes n}(\lambda)$.

Using Theorem~\ref{thm:equiv-of-bicats} we can contruct a~strictly functorial
version of both Khovanov and Chen--Khovanov homologies by precomposing
$\FKh^\circ$ and $\FKh^\lambda$ with $\EqFunc^\vee$. We provide
a~direct construction of both invariants.

Following \cite{WebAlgebras} we call the~web version of $\arcalg n$
the~\emph{Blanchet--Khovanov algebra}. It is defined for any collection
of oriented red and blue points $\bdry$ that is \emph{balanced}, i.e.
bounds a~web, as the~direct sum
\[
	\webalg{\mathcal B} :=
		\bigoplus_{a,b\in\mathcal B}
			\Hom_{\cat{Foam}(\bdry)}(a,b),
\]
where $\mathcal B$ is a~\emph{cup basis} of webs bounded by $\bdry$;
its elements play the~role of cup diagrams for $\arcalg n$. Although
$\webalg{\mathcal B}$ depends a~priori on $\mathcal B$, we show that
different choices of basis lead to isomorphic algebras. Moreover, there
is a~special basis of webs---the~\emph{red-over-blue} basis---such that
forgetting red facets in cup foams is compatible with multiplication.
In particular, $\webalg{\mathcal B}$ admits a~\emph{positive basis}.
This results immediately in an~algebra isomorphism $\webalg{\mathcal B}
\cong \arcalg n\rics$, where $n$ is half of the~blue points in $\bdry$.
We further extend this construction to a~2-functor
$\Fweb^\circ \colon \cat{Foam}^\circ \to \cat{Bimod}$
following the~construction of $\FKh^\circ$.

Suppose that $T$ is an~oriented tangle, the~input and output of which
are balanced. Then all resolutions of $T$ are in $\cat{Foam}^\circ$
and $\Fweb^\circ$ can  be applied to $\wKhBracket{T}$ to produce a~chain
complex of bimodules $\WebCom(T)$. We call it
the~\emph{Blanchet--Khovanov complex}.

\pretheorem{Theorem}{thm:FKh-vs-Fweb}{%
	The~2-functor $\Fweb^\circ$ is equivalent to
	$\FKh^\circ\circ\EqFunc^\vee\rics$. In particular, the~complexes
	$\WebCom(T)$ and $\KhCom(T)$ are isomorphic for any
	tangle\/ $T\ric$ with balanced input and output.
}

The~construction of a~web version of Chen--Khovanov algebras
is more challenging. We first describe two extensions of a~sequence
$\bdry$ to a~balanced one $\bdry^\circ$ by inserting extra blue points
to the~left and to the~right of $\bdry$.
Then we pick a~basis $\mathcal B$ of webs bounded by $\bdry^\circ$
and the~corresponding Blanchet--Khovanov algebra $\webalg{\mathcal B}$.
The~\emph{extended Blanchet--Khovanov algebra} $\qtwebalg{\bdry,\lambda}$,
where $\lambda\in \Z$ has the~same parity as the~number of blue points
in $\bdry$, is a~certain subquotient of $\webalg{\mathcal B}$.
Following the~same procedure we associate a~bimodule with a~web
and a~bimodule map with a~foam for every $\lambda\in \Z$, obtaining
a~family of 2-functors $\Fweb^\lambda\colon \cat{Foam}
\to \cat{Bimod}$, each defined on the~entire foam bicategory.
As in the~previous construction, $\Fweb^\lambda$ is compatible with
relations between foams, so that applying it to $\wKhBracket{T}$
results in an~invariant chain complex of bimodules $\WebCom(T; \lambda)$.
We call it the~\emph{extended Blanchet--Khovanov complex} of $T$.

We construct an~explicit isomorphism
$\qtwebalg{\mathcal B,\lambda} \cong \CKalg{k,n-k}\rics$, where $n$ counts
blue points in $\bdry$ and $\lambda = n-2k$. Contrary to the~previous case,
it is not enough to forget red facets in cup foams to get the~isomorphism,
because the~basic webs from $\mathcal B$ may have too many blue arcs.
This issue is resolved by \emph{stabilization}---adding beneath webs
and foams extra blue arcs and disks respectively. We then extend
this isomorphism to bimodules and prove the~following fact.

\pretheorem{Theorem}{thm:FCKh-vs-Fweb}{%
	The~2-functor $\Fweb^\lambda$ is equivalent to
	$\FKh^\lambda\circ\EqFunc^\vee\rics$. In particular, the~complexes
	$\WebCom(T; \lambda)$ and $\KhCom(T; \lambda)$ are isomorphic
	for any tangle\/ $T\ric$.
}

All the~isomorphisms are constructed explicitly and---in case nice
bases are used---given by very simple formulas. Furthermore,
by the~discussion following Theorem~\ref{thm:equiv-of-cats},
the~tangle homology computed with $\Fweb^\circ$ and $\Fweb^\lambda$
are isomorphic to the~Khovanov and Chen--Khovanov invariants
respectively.

\subsection{Functoriality of quantized annular Khovanov homology}

The~above results allow us to construct a~strictly functorial version
of the~quantized annular Khovanov homology, which was the~motivation
for this paper.
Combining Theorem~\ref{thm:FCKh-vs-Fweb} with \cite[Proposition 6.6]
{QntHom} we get

\begin{icor}\label{cor:qHH-of-extBlKh}
	Suppose $\scalars$ is flat over\/ $\Z[q^{\pm 1}]$. Then the~quantum
	Hochschild homology groups $\qHoHom_{\!i}(\qtwebalg{\mathcal B,\lambda})$
	with coefficients in $\scalars$ vanish for $i>0$, whereas the~Chern
	character map
	\[
		h\colon K_0\left( \qtwebalg{\mathcal B,\lambda}\right)
			\otimes_{\Z[q^{\pm 1}]}{\scalars}
		\to
			\qHoHom_{\!0}\left( \qtwebalg{\mathcal B,\lambda}\right)
	\]
	is an~isomorphism.
\end{icor}

Choose now an~oriented tangle $T$ that is bounded at both top and bottom
by the~same collection of oriented points $\bdry$. We define for its annular
closure $\widehat T$ the~\emph{quantum annular $\LieGL$ complex} as
\[
	\qAKh(\widehat T) := \bigoplus_\lambda
		\qHoHom_{\!\bullet}(
			\qtwebalg{\mathcal B,\lambda},
			\qtwebmod{T; \lambda}
		)
\]
where $\mathcal B$ is a~cup basis of webs bounded by $\bdry$ and 
$\qtwebmod{T; \lambda}$---the~chain complex of bimodules obtained
by applying $\Fweb^\lambda$ to $\wKhBracket{T}$.
Corollary~\ref{cor:qHH-of-extBlKh} together with
\cite[Theorem B]{QntHom} imply the following.

\begin{icor}
	The~quantum annular $\LieGL$ homology $\qAKh(L)$ is a~triply graded
	invariant of annular links that is strictly functorial with respect
	to annular link cobordisms. Moreover, it admits an~action of\/
	$\mathcal U_q(\LieGL)$ that commutes with the differential and
	the~maps induced by annular link cobordisms.
\end{icor}

It follows now from Theorem~\ref{thm:FCKh-vs-Fweb} and the~following
discussion that $\qAKh(L)$ is isomorphic with the~quantized annular
complex as constructed in \cite{QntHom}.

\subsection{Further generalizations}

The~Khovanov homology has been extended by Asaeda, Przytycki, and Sikora
to links in thickened surfaces \cite{APS},
but the~functoriality has not been addressed until the~resent paper of
Quefellec and Wedrich \cite{KhSurfaces}.
There they have defined $\LieGL$ foams in thickened \emph{oriented} surfaces,
and the~natural question is whether the~results of this paper can be
extended to show equivalence of the~two constructions.
This is addressed in a~follow up paper, where we also discuss foams
in arbitrary 3-manifolds, including non-orientable ones.

Another natural question is about $\LieGL[N]$ foams for $N>2$. Again there
are two (bi)categories involved: of \emph{enhanced} and \emph{not enhanced}
foams, the~latter allowing only facets of labels up to $N-1$. We expect
that a~proper generalization of this paper would prove equivalence of
both (bi)categories, hence, also of the~associated link homologies.
Notice that functoriality of $\LieGL[N]$ homology has been shown in
\cite{KhRozFunctorial} using enhanced foams.

%% ===================================================================

\subsection{Organization of the~paper}
Section~\ref{sec:preliminaries} provides a~brief exposition of webs
and foams. All the~results presented there are well-known, except
perhaps the~choice of defining relations. Section~\ref{sec:shadings}
discusses shadings, their connection to webs and foams, and bicolored
isotopies. It ends with a~construction of a~basis of the~space of foams
bounded by a~given web. The~equivalence
of bicategories $\cat{BN}$ and $\cat{Foam}$ together with the~local
versions are constructed in Section~\ref{sec:equivalences}, in which
we also compare the~two versions of the~Khovanov bracket.
Finally, Sections~\ref{sec:TQFT}--\ref{sec:subquotients}
provide detailed constructions of TQFT functors: a~description of
the~tautological functor on $\cat{Foam}(\emptyset)$ in terms of planar
pictures, the~constructions of the~Blanchet--Khovanov algebras, their
subquotients, and the~2-functors $\Fweb^\circ$ and
$\Fweb^\lambda\rics$.

%% ===================================================================

\subsection{Conventions and notation}
Throughout the~paper we fix a~commutative unital ring $\scalars$
and linearity means $\scalars$--linearity.
We denote by $\{d\}$ the~upward degree shift by $d$,
i.e. $M\{d\}_i = M_{i-d}$ for a~graded module $M$.
Hence, a~homogeneous $m\in M$ has degree $\deg(m) + d$
when seen as an~element of $M\{d\}$.
We write $\HCom(\cat{C})$ for the~homotopy category of a~linear
category $\cat{C}$, the~objects of which are formal complexes
in $\cat{C}$ and morphisms---homotopy classes of chain maps.

Manifolds are assumed to be smooth (or at least piecewise smooth
when necessary) and submanifolds are neat---that is $N\subset M$
is transverse to $\partial M$ and $\partial N = M\cap \partial M$
\cite{DiffTop}. Orientation of a~surface
$S \subset \R^3$ is often identified with the~\emph{canonical normal
vector field} $\nu$, defined by the~property that for each $p\in S$
the~triple $(e_1,e_2,\nu_p)$, where $(e_1,e_2)$ is an~oriented basis
of $T_pS$, is an~oriented basis of $T_p\R^3$. Such a~vector field
is unique up to an~isotopy and can be found by the~right hand rule.

%% ===================================================================
\subsection{Acknowledgments}
The authors are grateful to the~organizers of the~program
,,Homology Theories in Low Dimensional Topology'' in spring 2017
at the~Isaac Newton Institute for Mathematical Sciences in Cambridge,
where they have started to work on this project.
A.B. and K.P. are supported by the NCCR SwissMAP
founded by Swiss National Science Foundation.

\section{Main players}
\label{sec:preliminaries}

This section provides basic definitions and facts about webs and foams.
Most of the~material is well-known \cite{Blanchet, KhViaHowe},
except perhaps the~choice of defining relations, and the~main purpose
of this part is to fix notation and introduce terms used throughout
the~paper.

%% ===================================================================
\subsection{Webs}
\label{sec:webs}

%% Definition of webs
A~\emph{web} is an~oriented trivalent graph with edges colored blue
or red%
\footnote{
	Red edges are drawn as double thick lines to make the~difference
	visible when the~paper is printed black and white.
}
in such a~way, that at each vertex either two blue edges \emph{merge}
to a~red one, or a~red edge \emph{splits} into two blue edges:
\begin{equation}\label{eq:flow-cond}
	\tikzset{x=1.5em,y=1.5em}%
	\mathclap{\webpict{merge up}}
	\hskip 0.3\textwidth
	\mathclap{\webpict{split up}}
\end{equation}
In this paper webs will be always embedded in a~disk $\Disk$ or a~sphere
$\S^2$ with a~fixed basepoint $\ast$ that lies on $\partial\Disk$ in
the~case of a~disk.
%for a~technical reason explained later.
Edges of a~web in a~disk can be attached
transversely to the~boundary circle away from $\ast$; each boundary point
inherits then both the~color and orientation from the~attached edge: outwards
(resp.\ inwards) oriented edges terminate with positive (resp.\ negative) points.
A~web is \emph{closed} if its boundary is empty.

\begin{rmk}
	By moving the~basepoint $\ast$ to the~infinity, we can consider
	webs in $\Disk$ or $\S^2$ as embedded in a~half plane $\R\times(-\infty,0]$
	or a~full plane $\R^2$ respectively.
\end{rmk}

%% Module of webs
\begin{defn}
	We write $\set{Web}$ for the~module generated by isotopy%
	\footnote{
		Isotopies are assumed to fix points on the~boundary circle.
	}
	classes of webs in a~disk, modulo the~local%
	\footnote{
		The~word \emph{local} means that two webs are identified
		if there is a~disk outside of which the~webs coincide and
		inside they look like in the~pictures.
	}
	relations
	\tikzset{x=2em,y=2em}%
	\begin{align}
	\label{rel:circle-evaluation}
		\webpict{circle}[V1] &= q+q^{-1}
	&
		\webpict{circle}[V2] &= 1
	\\[2ex]
	\label{rel:web-detach-and-saddle}
		\webpict{T-shape} &= \webpict{line and cup}
	&
		\webpict{hor arcs} &= \webpict{vert arcs}
	\end{align}
	where the~webs above can carry any coherent orientation unless indicated.
	For each collection of oriented red and blue points $\bdry \subset \partial\Disk$
	there is a~submodule $\set{Web}(\bdry)$ generated by webs bounded by $\bdry$
	and $\set{Web}$ is the~direct sum of all of them.
\end{defn}

%\begin{rmk}
%	Webs bounded by $\partial \subset \partial\Disk$ generated a~submodules
%	$\set{Web}(\bdry)$. Clearly, the~direct sum of all these submodules
%	recovers $\set{Web}$.
%\end{rmk}

\begin{exer}
	Show that webs satisfy the~following local relations:
	\tikzset{x=2em,y=2em}%
	\begin{align}
	\label{rel:web-bigons}
		\webpict{bigon>} &= \webpict{vline}[V1]
	&
		\webpict{bubble} &= (q+q^{-1}) \webpict{vline}[V2]
	\end{align}
	\begin{hint}
		Start with the~left relation in \eqref{rel:web-detach-and-saddle}.%
	\end{hint}
\end{exer}

%% Case of planar webs
Blue edges of a~web $\web$ form a~crossingless tangle $\web_b$, which we call
the~\emph{underlying tangle of $\web$}. In particular, it is a~collection of
disjoint circles when $\web$ is closed. Write $\ell(\web)$ for the~number of
blue loops in $\web_b$. Let $\mathrm r(\web)$ be a~web, the~underlying tangle
of which is $\web_b$ with closed loops removed. We call it a~\emph{reduction
of $\web$}. We construct it later using the~bicolored isotopy argument and
show the~following fact, which implies in particular that $\mathrm r(\web)$
does not depend on the~placement of red edges.

\begin{prop}\label{prop:webs-red}
	Webs with same boundary and isotopic underlying tangles coincide
	in $\set{Web}$. In particular, $\web = (q+q^{-1})^{\ell(\web)}
	\mathrm r(\web)$ for any web $\web$.
\end{prop}

Let $-\web$ be the~result of reversing orientation of all edges in a~web
$\web$. This operation preserves the~relations \eqref{rel:circle-evaluation}
and \eqref{rel:web-detach-and-saddle}, hence it induces an~involution
on $\set{Web}$. It does not preserve the~submodules $\set{Web}(\bdry)$,
but there is a~pairing
\begin{equation}\label{eq:web-pairing}
	(\web,\web') := (q+q^{-1})^{\ell(-\web \cup \web')},
\end{equation}
which can be visualized by placing $-\web$ and $\web'$ on the~lower and
upper hemisphere of a~sphere and applying Proposition~\ref{prop:webs-red}
to the~resulting web (entirely red webs evaluate to $1$).

\begin{lem}\label{lem:web-pairing-is-nondeg}
	The~pairing \eqref{eq:web-pairing} is non-degenerate.
\end{lem}
\begin{proof}
	Choose a~nonzero $w\in\set{Web}(\bdry)$ and write it as a~linear
	combination $c_1\web_1 + \dots + c_r\web_r$ of pair-wise non-isotopic
	webs $\web_1,\dots,\web_r$, the~underlying tangles of which contain
	no loops. We may further assume that the~polynomial $c_1$ contains
	a~term $q^d$ with the~maximal value of $|d|$ among all $c_i$.
	Because $\ell(-\web_1\cup \web_i) < \ell(-\web_1\cup \web_1)$ for any $i\neq 1$,
	the~term $q^d(\web_1, \web_1)$ is not canceled in the~expansion of $(w,\web_1)$.
	Hence, $(w,\web_1) \neq 0$.
\end{proof}

%% ===================================================================
\subsection{Foams}
\label{sec:foams}

A~\emph{foam} is a~collection of \emph{facets}, oriented blue and red%
\footnote{
	As in the~case of webs, red facets of a~foam are doubled in pictures.
}
surfaces, 
embedded in a~3-ball $\Ball$ with boundary components attached transversely to
$\partial\Ball$ or glued together along singular curves called \emph{bindings}
in a~way, such that locally two blue facets merge into a~red one in
an~orientation preserving way as shown in Figure~\ref{fig:local-model-for-foams}.
Furthermore, blue facets may carry \emph{dots}, but not the~red ones, and bindings
inherit orientation from blue facets. We say that a~foam is \emph{closed} if its
boundary is empty. Otherwise it is bounded by a~web in $\partial\Ball$.
Notice that blue facets alone form a~surface $\foam_b$ with dots,
the~\emph{underlying surface of $\foam$}. As in the~case of webs,
we fix a~basepoint $\ast \in \partial\Ball$ away from $\partial\foam$.
By moving it to infinity we can reinterpret foams as embedded in a~half
3-space $\R^2\times(-\infty,0]$.

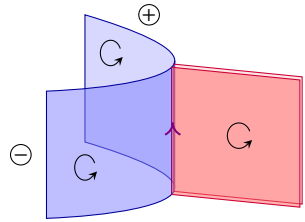
\begin{figure}[ht]%
	\centering
	% Requires TikZ libraries:
%	calc
%	intersections
%
\begin{tikzpicture}[x=4em,y=4em]
	% back facet
	\draw[1facetBack]
		(0,0.44) .. controls (0, 0.55) and (-0.3, 0.68) .. (-0.7,0.8) --
		++(0,-1) .. controls (-0.3,-0.32) and (0,-0.45) .. (0,-0.56);
	% front facet
	\draw[1facetFront]
		(0,0.44) .. controls (0, 0.3) and (-0.5, 0.22) .. (-1, 0.2) --
		++(0,-1) .. controls (-0.5,-0.78) and (0,-0.7) .. (0,-0.56) -- ++(0,1);
	% red facet
	\path[name path=arc]
		(0,0.44) .. controls (0, 0.3) and (-0.5, 0.22) .. (-1, 0.2);
	\path[name path=bline] (-0.5,0.4) ++(0, 0.6pt) -- ++(1,0);
	\path[name path=fline] (-0.5,0.4) ++(0,-0.6pt) -- ++(1,0);
	\draw[2facetInner, name intersections={of=arc and bline}]
		(intersection-1) coordinate[name=back]
		-- ++(1,-0.1) -- ++(0,-1) -- ++(-1,0.1);
	\draw[seam] (back)  -- ++(0,-1);
	\draw[2facetFront, name intersections={of=arc and fline}]
		(intersection-1) coordinate[name=front]
		-- ++(1,-0.1) -- ++(0,-1) -- ++(-1,0.1);
	\draw[seam] (front) -- ++(0,-1);
	% signs of blue facets
	\draw (-0.2, 0.8) node[shape=circle,inner sep=0.2pt,draw] {$\scriptstyle+$}
	      (-1.2,-0.3) node[shape=circle,inner sep=0.2pt,draw] {$\scriptstyle-\vphantom+$};
	% seam
	\coordinate (midSeam) at ($0.5*(front) + 0.5*(back) + (0,-0.5)$);
	\foamarrow>[seam]90(midSeam);
	% orientations of facets
	\foamorient (0.5,-0.15) Rx=0.09 Ry=0.1 S=20 E=310;
	\foamorient (-0.5, 0.5) Rx=0.08 Ry=0.1 S=20 E=310;
	\foamorient (-0.7,-0.4) Rx=0.08 Ry=0.1 S=20 E=310;
\end{tikzpicture}
	\caption{%
		The~local model for a~foam. The~orientation of the~binding 
		is coherent with the~orientation of the~blue facets, but opposite
		to the~one induced from the~red facet. The~cyclic order is counter-%
		clockwise, when seen from above, so that the~front blue facet
		is the~negative one.}%
%		and the~cyclic ordering of facets follows the~right-hand rule.
%		This convention is used for all pictures in this paper.}%
	\label{fig:local-model-for-foams}%
\end{figure}

There is a~canonical cyclic order of facets attached to a~binding
that follows the~right hand rule: point the~thumb of your right
hand along the~binding curve and slightly bend the~other fingers%
---they indicate the~orientation of a~small circle around the~binding,
hence, a~cyclic order of facets. We call a~blue facet \emph{positive}
or \emph{negative} depending on whether it succeeds or precedes the~red
facet respectively. For non-embedded foams this cyclic order is usually
provided explicitly by drawing small arrows around the~binding,
see \cite{Blanchet}.

%% Module of foams
\begin{defn}\label{def:foams}
%	Given a~web $\web$ in $\partial\Ball$ we define $\set{Foam}(\web)$
%	as the~module generated by isotopy classes of foams in $\Ball$ bounded
%	by $\web$ with the~following local relations imposed:
	We write $\set{Foam}$ for the~module generated by isotopy classes of
	foams in $\Ball$ with the~following local relations imposed:
	\begin{itemize}
	\item \textit{sphere evaluations:}
	\begin{equation}\label{rel:sphere-eval}
	\tikzset{x=8mm,y=8mm}%
		\foampict{1sphere}    = 0\hskip 0.1\textwidth
		\foampict{1sphere}[1] = 1\hskip 0.1\textwidth
		\foampict{2sphere}    =-1
	\end{equation}
	
	\item \textit{neck cutting relations:}
	\begin{equation}\label{rel:neck-cutting}
	\tikzset{x=7mm,y=7mm}%
		\mathclap{\foampict{1neck}\, = \,\foampict{1cap 1cup}[+1]
		                          \, + \,\foampict{1cap 1cup}[-1]
			                       \, - h \foampict{1cap 1cup}}
	\hskip 0.4\textwidth
		\mathclap{\foampict{2neck}\, = -\foampict{2cap 2cup}}
	\end{equation}
	
	\item \textit{dot reduction and dot moving relations:}
	\begin{equation}\label{rel:dots}
	\tikzset{x=1cm,y=1cm}%
		\foampict{1plane}[2] = h\!\foampict{1plane}[1] + t\!\foampict{1plane}
	\hskip 2em
		\foampict{3facets}[+] = h\foampict{3facets}\, - \!\foampict{3facets}[-]
	\end{equation}
	
	\item \textit{red facet detachments:} %%%%%%%%%%%%%%%%%%%%%%%%%%%%%%
	\tikzset{x=7mm,y=7mm}%
	\begin{gather}\label{rel:detach-cylinder}
		\foampict{2cyl at 1plane}[<]
			\mskip 8mu\raisebox{-1ex}{$\,=$}\mskip-6mu
		\foampict{2cup > 1plane}
	\hskip 4em
		\foampict{2cyl at 1plane}[>]
			\mskip 8mu\raisebox{-1ex}{$\,=-$}\mskip-6mu
		\foampict{2cup > 1plane}
	\\[2ex]
		\label{rel:detach-saddle}
		\foampict{2planes at 1plane}[<]
			\mskip 8mu = \mskip-6mu
		\foampict{2saddle at 1plane}[<]
	\hskip 4em
		\foampict{2planes at 1plane}[>]
			\mskip 8mu =- \mskip-6mu
		\foampict{2saddle at 1plane}[>]
	\end{gather}
	\end{itemize}
	Foams bounded by a~web $\web\subset\S^2$ (with $\ast \notin \web$)
	generate a~submodule $\set{Foam}(\web)$.  As in the~case of webs,
	$\set{Foam}$ is the~direct sum of all these submodules.
\end{defn}

\begin{rmk}\label{rmk:sign-from-normal-vector}
	The~sign in \eqref{rel:detach-cylinder} and \eqref{rel:detach-saddle} can
	be read easily from the~direction of the~canonical normal vector at
	the~critical point on the~red surface: it is positive exactly
	when the~normal vector is directed towards the~blue plane. For example,
	\eqref{rel:detach-cylinder} can be written as
	\[
	\tikzset{x=8mm,y=8mm}%
		\foampict{2cyl at 1plane}[<V]
			\mskip 8mu\raisebox{-1ex}{$\,=$}\mskip-6mu
		\foampict{2cup > 1plane}[V]
	\hskip 4em
		\foampict{2cyl at 1plane}[>V]
			\mskip 8mu\raisebox{-1ex}{$\,=-$}\mskip-6mu
		\foampict{2cup > 1plane}[A]
	\]
\end{rmk}

When $\Bbbk$ is graded with $h$ and $t$ homogeneous in degree 2 and 4
respectively, then $\set{Foam}$ is a~graded module with a~foam $\foam$
being a~homogeneous element in degree
\begin{equation}
	\deg(\foam) := -\chi(\foam_b) + 2\mathit{dots}(\foam).
\end{equation}
Here $\chi(\foam_b)$ stands for the~Euler characteristic of the~underlying
surface and $\mathit{dots}(\foam)$ counts dots carried by the~foam.

%% Dual dot
The~dot moving relation (the~right one in \eqref{rel:dots}) takes
a~particularly simple form for $h=0$: it allows to move a~dot on
the~underlying surface at a~cost of a~sign. To have a~similar
interpretation in the~general case, we introduce the~\emph{dual dot}
as the~difference
\begin{equation}
	\foampict{1plane}[4] := h\foampict{1plane} - \foampict{1plane}[1].
\end{equation}
The~following exercise lists several relations satisfied by dual dots.

\begin{exer}\label{rel:dual-dot}
	Show the~following equalities between foams:
	\begin{gather*}
%		\tikzset{x=8mm, y=8mm}
		\mathclap{\foampict{1sphere}[4] = -1}
		\hskip 0.4\linewidth
		\mathclap{\foampict{1sphere}[5] = 0}
	\\[2ex]
%		\tikzset{x=1cm, y=1cm}
		\mathclap{\foampict{1plane}[6] = h\foampict{1plane}[4] + t\foampict{1plane}}
		\hskip 0.5\linewidth
		\mathclap{\foampict{3facets}[+] \,=\, \foampict{3facets}[-2]}
	\\[2ex]
		\tikzset{x=7mm, y=7mm}
		\foampict{1cap 1cup}[+1]\, - \,\foampict{1cap 1cup}[-2]
			\quad=\quad
		\foampict{1neck}
			\quad=\quad
		\foampict{1cap 1cup}[-1]\, - \,\foampict{1cap 1cup}[+2]
	\end{gather*}
\end{exer}

%% Other red-blue interaction relations
The~detaching relations \eqref{rel:detach-cylinder} and
\eqref{rel:detach-saddle} can take many other forms.
For instance, redrawing them to make red facets horizontal
results in
\begin{tikzset}{x=2em,y=2em}%
\begin{align}
	\label{rel:cup-vs-plane}
		\foampict{1cup < 2plane}[>] &=  \foampict{1cup > 2plane} &
		\foampict{1cup < 2plane}[<] &= -\foampict{1cup > 2plane}
	\\[1ex]
	\label{rel:saddle-vs-plane}
		\foampict{Asaddle > 2plane}[>] &=  \foampict{Asaddle < 2plane}[>] &
		\foampict{Vsaddle < 2plane}[>] &= -\foampict{Vsaddle > 2plane}[>]
\intertext{%
	Likewise, \eqref{rel:detach-cylinder} together with
	\eqref{rel:sphere-eval} allow us to remove a~red membrane
	attached to a~blue cup
}
	\label{rel:cup-vs-disk}
		\foampict{1cup and disk}[>] &=  \foampict[baseline=-0.4em]{1cup} &
		\foampict{1cup and disk}[<] &= -\foampict[baseline=-0.4em]{1cup}
\intertext{%
	and other well-known relations arise by redrawing	
	\eqref{rel:detach-saddle} and \eqref{rel:cup-vs-disk}
	in a~way, such that blue facets form a~horizontal plane
	and the~boundary of red facets is vertical:
}
	\label{rel:red-cap-vs-plane}
	\foampict{2bubble at 1plane}[>] &=  \foampict{no bubble at 1plane} &
	\foampict{2bubble at 1plane}[<] &= -\foampict{no bubble at 1plane}
\\[2ex]
	\label{rel:stripe-vs-plane}
	\foampict{2tunel at 1plane}[>] &=  \foampict{2shells at 1plane}[>] &
	\foampict{2tunel at 1plane}[<] &= -\foampict{2shells at 1plane}[<]
\end{align}
\end{tikzset}%
Notice that in each case the~sign can be read from the~direction of the~normal
vector as explained in Remark~\ref{rmk:sign-from-normal-vector}.

We interpret the~above relations later as isotopies between two surfaces,
a~blue and a~red one. This will be a~key ingredient in the~proofs
of the~two facts listed below.
In what follows we write $\foam \foamequiv \foam'$ if foams $\foam$ and $\foam'$
differ only by a~sign and dualizing dots. For instance, $\foam\foamequiv\foam'$
when $\foam'$ is the~result of moving a~dot on the~underlying surface of $\foam$.

\begin{prop}\label{prop:foams-red}
	Let $\foam$ and $\foam'$ be foams with isotopic underlying surfaces
	and same boundary. Then $\foam \foamequiv \foam'$ in $\set{Foam}$.
\end{prop}

We prove the~above proposition in the~following section. An~important
consequence of it is the~uniqueness (up to a~sign) of a~foam $\cupfoam(\web)$,
the~underlying surface of which is a~collection of disjoint disks
bounded by $\web_b$. We call it the~\emph{cup foam} associated to
$\web$. Then for any family $X$ of blue loops in $\web_b$ we denote
by $\cupfoam(\web, X)$ the~cup foam with a~dot placed on every blue
disk that bounds a~curve from $X$. These foams constitute a~linear basis
of $\set{Foam}(\web)$ as shown in Section~\ref{sec:cup basis}.

\begin{thm}\label{thm:basis-for-Foam(w)}
	Choose a~closed web $\web$.
	The~set $\{ \cupfoam(\web,X)\ |\ X\subset BL(\omega) \}$
	is a~linear basis of $\set{Foam}(\web)$. In particular,
	$\set{Foam}(\web)$ is a~free graded module of rank
	$(q+q^{-1})^{\ell(\web)}\rics$.
\end{thm}

%% ===================================================================
\subsection{Decategorification}
% Foam with corners: in DxI = B^3
% Input: -w0, output: w1, side: vertical lines
% Redefine grading

%Fix a~balanced collection of colored oriented points $\bdry \subset \Disk$.
Fix a~collection of red and blue oriented points $\bdry \subset \Disk$.
A~\emph{foam with corners in $\bdry$} is a~foam $\foam$ in $\DxI$ with
$\foam \cap \partial\Disk = \bdry\times [0,1]$. We gather them into a~category
$\cat{Foam}(\bdry)$, in which
\begin{itemize}
	\item objects are webs bounded by $\bdry$ with no relation imposed,
	\item morphisms from $\web_0$ to $\web_1$ are generated by foams with
	      corners in $\bdry$, with $\web_1$ at the~top and $-\web_0$ at
			the~bottom disk of $\DxI$, modulo the~relations
			\eqref{rel:sphere-eval}--\eqref{rel:detach-saddle}, and
	\item the~composition is given by stacking foams, one on top of the~other.
\end{itemize}
We further enhance it to a~graded additive category by introducing formal
direct sums and formal degree shifts, so that objects are of the~form
$\web_1\{d_1\} \oplus \dots \oplus \web_r\{d_r\}$, and redefining the~degree
of a~foam $\foam\colon \web_0\{a\} \to \web_1\{b\}$ as
\begin{equation}
	\deg(\foam) := (b-a)
			- \chi(\foam_b) + 2\mathit{dots}(\foam)
			+ \frac{\#\bdry_b}{2},
\end{equation}
where, as before, $\chi(\foam_b)$ is the~Euler characteristic of the~underlying
surface of $\foam$ and $\mathit{dots}(\foam)$ counts dots on $S$, whereas
$\#\bdry_b$ is the~number of blue points in $\bdry$.
The~reason for the~last term is to make the~identity foam a~morphism of degree
zero; it also makes the~degree additive under the~composition of foams.
Furthermore, reinterpreting foams with corners as foams in $\Ball^3$ leads
to an~isomorphism of graded $\scalars$--modules
\begin{equation}\label{eq:Foam(a,b) as Foam(ab)}
	\Hom_{\cat{Foam}(\bdry)}(\web,\web')
		\cong
	\set{Foam}(-\web \cup \web')\big\{\tfrac{\#\bdry_b}{2}\big\}
	% x in V{1} has degree deg(x)+1, right?
\end{equation}
for any webs $\web$ and $\web'$ bounded by $\bdry$.

The~orientation reversing diffeomorphism of the~thickened
disk $(p,t) \mapsto (p,1-t)$ induces a~contravariant involutive
functor
\begin{equation}\label{eq:foam-reversion}
	\Hom_{\cat{Foam}(\bdry)}(\web,\web') \ni \foam
		\longmapsto
	\foam^! \in \Hom_{\cat{Foam}(\bdry)}(\web',\web)
\end{equation}
that flips a~foam vertically and reverses orientation of its facets.
We check directly that all the~defining relations
\eqref{rel:sphere-eval}--\eqref{rel:detach-saddle} are preserved.

Foams with corners categorify webs. Indeed, web relations are lifted to
isomorphisms:
\begin{align}
	\label{isom:circle-removal}
	&\begin{tikzpicture}[x=6em,baseline=(X.base)]
		\node (X) at (0,0) {$\vcenter{\hbox{\phantom X}}$};
		\node[anchor=east] (1circle) at (0,0) {$\webpict[x=2em,y=2em]{circle}$};
		\node[anchor=west] (qq) at (1,0) {$\emptyset\{-1\}\oplus\emptyset\{+1\}$};
		\draw[transform canvas={yshift=-2pt},<-] (1circle) -- (qq);
		\draw[transform canvas={yshift= 2pt},->] (1circle) -- (qq);
		\node[anchor=south] at (0.5, 2pt) {$\scriptstyle\left[\substack{
			\fnt{foam-1cap} \\[1ex]
			\fnt{foam-1cap*}\raisebox{-0.5ex}{$\,-h$}\fnt{foam-1cap}
		}\right]$};
		\node[anchor=north] at (0.5,-2pt) {$\scriptstyle\left[
			\fnt{foam-1cup*} \quad \fnt{foam-1cup}
		\right]$};
	\end{tikzpicture}
&&\hskip 0.4em\begin{tikzpicture}[x=5em,baseline=(X.base)]
		\node (X) at (0,0) {$\vcenter{\hbox{\phantom X}}$};
		\node[anchor=east] (2circle) at (0,0) {$\webpict[x=2em,y=2em]{circle}[V2]$};
		\node[anchor=west] (q) at (1,0) {$\emptyset$};
		\draw[transform canvas={yshift=-2pt},<-] (2circle) -- (q);
		\draw[transform canvas={yshift= 2pt},->] (2circle) -- (q);
		\node[anchor=south] at (0.5, 2pt)
			{$\mathllap-\scriptstyle\fnt{foam-2cap}$};
		\node[anchor=north] at (0.5,-2pt)
			{$\scriptstyle\fnt{foam-2cup}$};
	\end{tikzpicture}
\\
	\label{isom:detach-and-saddle}
	&\begin{tikzpicture}[x=6em,baseline=(X.base)]
		\node (X) at (0,0) {$\vcenter{\hbox{\phantom X}}$};
		\node[anchor=east] (Tshape) at (0,0) {$\webpict[x=2em,y=2em]{T-shape}$};
		\node[anchor=west] (lines)  at (1,0) {$\webpict[x=2em,y=2em]{line and cup}$};
		\draw[transform canvas={yshift=-2pt},<-] (Tshape) -- (lines);
		\draw[transform canvas={yshift= 2pt},->] (Tshape) -- (lines);
		\node[anchor=south] at (0.5, 2pt)
			{$\mathllap{\raisebox{-1ex}{$\pm$}}\scriptstyle\fnt{foam-detach-cup}$};
		\node[anchor=north] at (0.5,-2pt)
			{$\scriptstyle\fnt{foam-attach-cup}$};
	\end{tikzpicture}
&&\begin{tikzpicture}[x=5em,baseline=(X.base)]
		\node (X) at (0,0) {$\vcenter{\hbox{\phantom X}}$};
		\node[anchor=east] (H) at (0,0) {$\webpict[x=2em,y=2em]{hor arcs}$};
		\node[anchor=west] (V) at (1,0) {$\webpict[x=2em,y=2em]{vert arcs}$};
		\draw[transform canvas={yshift=-2pt},<-] (H) -- (V);
		\draw[transform canvas={yshift= 2pt},->] (H) -- (V);
		\node[anchor=south] at (0.5, 2pt)
			{$\mathllap{\raisebox{-0.75ex}{$-$}}\scriptstyle\fnt{foam-saddle-front}$};
		\node[anchor=north] at (0.5,-2pt)
			{$\scriptstyle\fnt{foam-saddle-side}$};
	\end{tikzpicture}
\end{align}
where the~sign in the~bottom left corner depends on the~orientation of
the~edges. Therefore, there is a~well-defined epimorphism
$\gamma\colon \set{Web}(\bdry) \to K_0(\cat{Foam}(\bdry)) \otimes_{\Zq} \scalars$
that takes a~web $\web$ to its class $[\web]$ in the~Grothendieck group.

\begin{thm}\label{thm:K0(foam)=Web}
	The~linear map
	$\gamma\colon \set{Web}(\bdry) \to K_0(\cat{Foam}(\bdry)) \otimes_{\Zq} \scalars$
	is an~isomorphism.
\end{thm}
\begin{proof}
	We have to show that $\gamma$ is injective. Consider
	a~bilinear form $\langle\blank,\blank\rangle$ on
	$K_0(\cat{Foam}(\bdry))$ defined for webs $\web$ and $\web'$ as
	\(
%		(\web, \web') = (q+q^{-1})^{\ell(-\web \cup \web')}
%	\quad\text{and}\quad
		\langle [\web], [\web'] \rangle
		:= \rk_q\Hom_{\cat{Foam}(\bdry)}(\web, \web')
	\).
	It is well-defined, because the~rank of the~morphism space
	depends only on the~images of webs in the~Grothendieck group.
	Theorem~\ref{thm:basis-for-Foam(w)} and the~isomorphism
	\eqref{eq:Foam(a,b) as Foam(ab)} imply together that
	$\langle [\web], [\web']\rangle = (\web, \web')$,
	where the~latter is the~nondegenerate pairing from
	\eqref{eq:web-pairing}. Hence, $\gamma(w) = 0$ forces
	$(w,\blank)=0$, so that $w$ must be zero.
\end{proof}

%% ===================================================================
\subsection{Higher structures}

It is common to consider webs embedded in a~horizontal stripe $\RxI$
instead of a~disk. This is equivalent to picking two basepoints on $\partial
\Disk$, $\ast$ and $\ast'$, and placing them at the~left and right infinities
respectively. Such webs are morphisms of a~linear category $\cat{Web}$,
the~objects of which are finite collections of oriented red and blue points
on a~line, whereas the~composition is defined by stacking stripes vertically:
% Picture: composition of webs

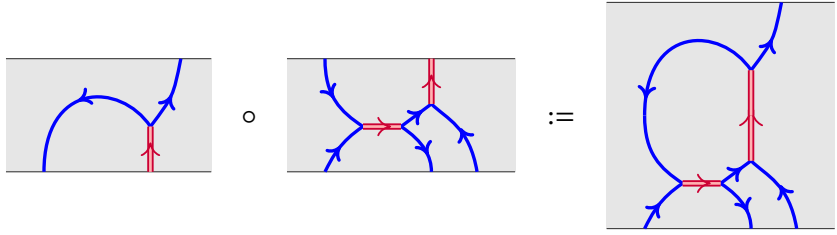
\begin{equation*}
	\def\webrect(#1,#2) -- (#3,#4);{%
		\fill[web/bg] (#1,#2) rectangle (#3,#4);
		\draw[web/bdry] (#1,#2) -- (#3,#2) (#3,#4) -- (#1,#4);
	}%
	\begin{tikzpicture}[baseline=(X.base)]
		\useasboundingbox (-0.5,-0.2) rectangle (2.5,1.7);
		\node (X) at (0, 0.75) {\phantom0};
		\webrect (-0.5,0) -- (2.2,1.5);
		\draw[V1,->-] (1.4,0.6) .. controls (1.7,1) .. (1.8,1.5);
		\draw[V1,->-] (1.4,0.6) .. controls (1,1.2) and (0,1.2) .. (0,0);
		\draw[V2,->-=0.7] (1.4,0) -- (1.4,0.6);
	\end{tikzpicture}
	\circ
	\begin{tikzpicture}[baseline=(X.base)]
		\useasboundingbox (-0.8,-0.2) rectangle (2.8,1.7);
		\node (X) at (0, 0.75) {\phantom0};
		\webrect (-0.5,0) -- (2.5,1.5);
		\draw[V1,->-] (0,0) .. controls (0.1,0.2) and (0.2,0.4) .. (0.5,0.6);
		\draw[V1,->-] (0,1.5) .. controls (0,1) and (0.2,0.8) .. (0.5,0.6);
		\draw[V1,->-=0.7] (1,0.6) .. controls (1.3,0.4) and (1.4,0.2) .. (1.4,0);
		\draw[V1,->-=0.7] (1,0.6) -- (1.4,0.9);
		\draw[V1,-<-=0.6] (1.4,0.9) .. controls (1.6,0.7) and (1.9,0.6) .. (2,0);
		\draw[V2,->-=0.8] (0.5,0.6) -- (1,0.6);
		\draw[V2,->-=0.7] (1.4,0.9) -- (1.4,1.5);
	\end{tikzpicture}
	:=
	\begin{tikzpicture}[baseline=(X.base)]
		\useasboundingbox (-0.8,0) rectangle (2.5,3);
		\node (X) at (0, 1.5) {\phantom0};
		\webrect (-0.5,0) -- (2.5,3);
		\begin{scope}
			\draw[V1,->-] (0,0) .. controls (0.1,0.2) and (0.2,0.4) .. (0.5,0.6);
			\draw[V1] (0,1.5) .. controls (0,1) and (0.2,0.8) .. (0.5,0.6);
			\draw[V1,->-=0.7] (1,0.6) .. controls (1.3,0.4) and (1.4,0.2) .. (1.4,0);
			\draw[V1,->-=0.7] (1,0.6) -- (1.4,0.9);
			\draw[V1,-<-=0.6] (1.4,0.9) .. controls (1.6,0.7) and (1.9,0.6) .. (2,0);
			\draw[V2,->-=0.8] (0.5,0.6) -- (1,0.6);
			\draw[V2,->] (1.4,0.9) -- (1.4,1.6);
		\end{scope}
		\begin{scope}[shift={(0,1.5)}]
			\draw[V1,->-] (1.4,0.6) .. controls (1.7,1) .. (1.8,1.5);
			\draw[V1,->-=0.9] (1.4,0.6) .. controls (1,1.2) and (0,1.2) .. (0,0);
			\draw[V2] (1.4,0) -- (1.4,0.6);
		\end{scope}
	\end{tikzpicture}
\end{equation*}
Formally, $\Hom_{\cat{Web}}(\bdry,\bdry') = \set{Web}(-\bdry \cup \bdry')$.
This category is closely related to representations of $\mathcal U_q(\LieGL)$
\cite{WebsHowe}: there is a~monoidal functor
$\mathcal V\colon \cat{Web} \to \cat{Rep}(\mathcal U_q(\LieGL))$ such that
\begin{itemize}
	\item a~blue positive (resp.\ negative) point is assigned the~fundamental
	representation $V$ (resp.\ its dual $V^\ast$) and a~red positive
	(resp.\ negative) point---the~determinant representation
	$\extpower^2V$ (resp.\ $\extpower^2V^\ast)$, whereas a~sequence
	of such points is assigned the~tensor product of the~corresponding
	representations,
	
	\item the~merge and split webs \eqref{eq:flow-cond} are assigned
	the~canonical inclusion and quotient maps between representations, and
	
	\item cups and caps represent coevaluation and evaluation maps.
\end{itemize}
The~relations between webs make the~above functor faithful.

Define the~\emph{weight} of a~point from $\bdry$ according to the~table below.
\begin{center}%
\def\point#1#2{\,\begin{centertikz}[V#2dot]
	\filldraw (0,0) circle[radius=\the\dimexpr#2pt/2+1.5pt\relax]
		node[right,text=black]{$\scriptstyle#1$}
		node[left]{$\scriptstyle\phantom#1$};
\end{centertikz}}%
\def\arraystretch{1.2}%
	\begin{tabular}{l|cccc}
	\hline
	point
		& $\point+1$
		& $\point-1$
		& $\point+2$
		& $\point-2$
	\\
	\hline
	weight
		& $+1$
		& $-1$
		& $+2$
		& $-2$
	\\
	\hline
\end{tabular}\end{center}
The~\emph{total weight} $w(\bdry)$ of $\bdry$ is the~sum of weights of its points.
A~quick analysis of the~local model for webs \eqref{eq:flow-cond}
reveals that webs exist only between objects of the~same weight. Hence,
the~category of webs decomposes into \emph{weight blocks} $\cat{Web}^k\ric$,
each spanned by objects of weight $k\in \Z$.
%\[
%	\cat{Web} = \bigoplus_{k\in\Z} \cat{Web}^k.
%\]
In particular, $\Hom_{\cat{Web}}(\emptyset, \bdry) \neq 0$ only when
$w(\bdry) = 0$; such collections are called \emph{balanced}.

In a~similar matter one collects the~foam categories $\cat{Foam}(\bdry)$
into a~bicategory $\cat{Foam}$, which also decomposes into blocks
$\cat{Foam}^k$ parametrized with $k\in\Z$.
Theorem~\ref{thm:K0(foam)=Web} can be then rephrased to say that
$\cat{Foam}^k$ categorifies $\cat{Web}^k$, i.e.\ 
the~category of webs is obtained by replacing morphism categories of
$\cat{Foam}$ with their Grothendieck groups.

%% ===================================================================
\subsection{Blanchet evaluation formula}

We end this section recalling the~evaluation formula for closed
foams in a~3-ball $\Ball$ following \cite{Blanchet}.
It requires two 2-dimensional TQFTs, one for blue and
one for red facets.
Each is uniquely determined by the~(associative)
commutative Frobenius algebra assigned to a~circle. 
We choose the~algebras
\[
%	A_b:=\Bbbk[X]/X^2
	A_b := \Bbbk[X]/(X^2 - hX - t)
	\quad\textrm{and}\quad
	A_r:=\Bbbk
\]
for blue and red circles respectively,
where $h,t \in \Bbbk$ are fixed parameters (the~standard choice is $h=t=0$).
The~comultiplications and counits are defined by the~formulas
\begin{align*}
	\Delta_b(1) &= 1\otimes X+X \otimes 1 - h 1\otimes 1, &
	\Delta_b(X) &= X\otimes X + t 1\otimes 1, &
	\Delta_r(1) &= -1\otimes 1,
\\
	\epsilon_b(1) &=  0, &
	\epsilon_b(X) &=  1, &
	\epsilon_r(1) &= -1.
\end{align*}
A~dot on a~blue surface is interpreted as the~multiplication with $X$.
Notice that $h-X$, which represents a~dual dot, satisfies
the~polynomial relation defining $A_b$, so that $\overline X := h-X$ extends
to a~conjugation compatible with multiplication. One checks directly that
$\overline{\Delta_b(a)} = -\Delta_b(\overline a)$ and $\epsilon_b(\overline a)
= -\epsilon_b(a)$ for any $a\in A_b$.

When $\Bbbk$ is graded with $h$ and $t$ homogeneous in degree 2 and 4
respectively, then we make $A_b$ a~graded algebra by setting $\deg(X)=2$;
comultiplication and counit increase and decrease the~degree by $2$
respectively.
Assigning now $A_b\{-1\}$ to a~blue circle produces a~graded TQFT:
$\deg(1) = -1$ and $\deg(X) = +1$, in which case both multiplication and
comultiplication are homogeneous in degree 1, matching the~degree of
a~saddle. Likewise for the~unit and counit.
The~other TQFT is upgraded by inheriting the~grading on $A_r$ from $\Bbbk$.

Assume that a~closed foam $\foam$ is obtained from a~blue surface $\foam_b$
and a~red one $\foam_r$ by identifying boundary circles $C^+_i, C^-_i
\subset \partial\foam_b$ with $C^0_i\subset \partial \foam_r$ for $1\leq i\leq m$,
such that $C^+_i$ and $C^-_i$ come from the~positive and negative facet
respectively. Let
\[
	\Blanchet_b(\foam_b)\in (A_b \otimes A_b)^{\otimes m}
\quad\textrm{and}\quad
	\Blanchet_r(\foam_r)\in (A_r )^{\otimes m}
\]
be the~elements assigned by the~two TQFTs to the~blue and red surface,
where the~first factor in $A_b\otimes A_b$ corresponds to $C_i^+$ and
the~second to $C_i^-$. The~evaluation assigns to $\foam$ the~value
\begin{equation}\label{eq:Blclfoam}
	\Blanchet(\foam) = \tr^{\otimes m}\left(  
		\pi^{\otimes m}(\Blanchet_b(\foam_b))
	\otimes
		\eta^{\otimes m}(\Blanchet_r(\foam_r))
	\right) \in \Bbbk,
\end{equation}
where $\pi\colon A_b\otimes A_b\to A_b$ sends $x \otimes y$ to
$x\overline y$ and $\eta\colon A_r\to A_b$
is the~inclusion of algebras;
the~trace map $\tr\colon A_b\otimes A_b\to \Bbbk$
is the~composition of the~multiplication with the~counit of $A_b$.

\begin{exam}\label{ex:dotted-sphere}
	Let $S$ be a~blue sphere with a~red disk inside and one dot, as shown below.
	It decomposes into three cups, two blue and a~red one, where one of the~blue
	cups carries a~dot:
	\begin{equation*}
		\foampict{sphere and membrane}[<1]
			\quad=\quad
		\foampict{1cup} \cup \foampict{2cup} \cup \foampict{1cup}[1]
	\end{equation*}
	The~orientation of the~binding determines that the~dotted cup is attached
	to the~negative boundary. Hence,
	\begin{equation*}
		\Blanchet_b(S_b) = 1 \otimes X,
			\qquad
		\Blanchet_r(S_r) = 1,
	\end{equation*}
	resulting in $\Blanchet(S) = \tr(1\otimes\overline X) = \epsilon_b(h-X) = -1$.
\end{exam}

The~relations \eqref{rel:cup-vs-disk} and \eqref{rel:sphere-eval}
evaluate the~foam $S$ from the~example above to $-1$ as well.
This is not a~coincidence: the~defining relations were looked up
in the~kernel of $\Blanchet$. In fact, Proposition~\ref{prop:foams-red}
implies a~stronger statement. It was first proven in \cite{Blanchet}.

\begin{thm}[cp.\ \cite{Blanchet}]\label{thm:foam-evaluation}
	The~evaluation \eqref{eq:Blclfoam} descends to an~isomorphism
	$Z\colon \set{Foam}(\emptyset) \to \scalars$.
\end{thm}
\begin{proof}
	We first check that $Z$ is well-defined, i.e.\ it preserves the~relations
	\eqref{rel:sphere-eval}--\eqref{rel:detach-saddle}.
	Those involving facets of one color can be checked directly, whereas
	moving a~dot through an $i$-th binding corresponds to taking it from
	a~facet attached to $C_i^+$ (multiplication by $X$) and placing it on
	the~facet attached to $C_i^-$ (multiplication by $\overline X = h-X$).
	Hence, \eqref{rel:dots} is satisfied. We follow now
	Example~\ref{ex:dotted-sphere} to compute
	\begin{equation}\label{eq:sphere with a membrane}
		Z\left(\foampict{sphere and membrane}[>]\right) = 0,\qquad
		Z\left(\foampict{sphere and membrane}[>1]\right) = 1,\qquad
		Z\left(\foampict{sphere and membrane}[<1]\right) = -1,
	\end{equation}
	which immediately implies \eqref{rel:detach-cylinder}:
	using \eqref{rel:neck-cutting} cut both the~red cylinder and the~plane
	around the~binding to obtain a~sum of three foams,
	each consisting of a~red cup, a~blue plane, and a~blue sphere with a~red
	membrane inside. Two of these foams have an~additional dot, one on
	the~plane and the~other the~sphere; only the~latter term survives
	and the~sign comes from \eqref{eq:sphere with a membrane}.
	We leave \eqref{rel:detach-saddle} as an~exercise.

	Assume now that $Z(S) = 0$. By Proposition~\ref{prop:foams-red},
	$\foam$ coincides up to a~sign with an~entirely blue foam $\foam'$,
	which is the~blue surface $\foam_b$, perhaps with some dots replaced
	with dual dots.
	However, applying the~blue neck cutting relation \eqref{rel:neck-cutting}
	to any component of positive genus reduces $\foam'$ further
	to a~sum of collections of dotted spheres. These in turn can be
	completely evaluated with \eqref{rel:dots} and \eqref{rel:sphere-eval}.
	Hence, $\foam' = Z(\foam') = 0$, which shows that $Z$ is invertible.
\end{proof}

\section{Shadings and a~basis of foams}
\label{sec:shadings}

This part is the~backbone of the~paper. We introduce here shadings
of manifolds, use them to construct webs and foams, and prove
the~bicolored isotopy lemma: isotopic shadings encode equal
webs and foams (the~latter up to a~sign and type of dots). Using
this language we introduce then a~basis of foams that is especially
easy to visualize.

%% ===================================================================
\subsection{Shadings and trivalent manifolds}

A~\emph{shading} of a~manifold $\mfld$ consists of two codimension
1 submanifolds, an~oriented $U_r$ and a~non-oriented $U_b$, that are
transverse to each other and to $\partial\mfld$, together with a~\emph{%
checkerboard coloring} of $\mfld$: a~choice of color, white or black,
for each connected component of the~complement of $U_r\cup U_b$.
We refer to $U_r$ and $U_b$ as \emph{red} and \emph{blue} respectively.
The~components of the~intersectoin $U_r \cap U_b$ are called \emph{bindings};
they decompose both $U_r$ and $U_b$ into \emph{facets}. Finally, we refer
to the~components of the~complement of $U_r\cup U_b$ in $\mfld$ as
\emph{regions}.

\begin{lem}\label{lem:shadibility}
	Assume $\mfld$ is simply connected and fix a~point $\ast\in\mfld\rics$.
	Then a~pair of codimension 1 submanifolds of $\mfld\rics$, that are
	transverse to each other and away from $\ast$, determines a~unique
	shading of\/ $\mfld\rics$ with the~region containing $\ast$ painted white.
\end{lem}
\begin{proof}
	Given a~pair of transverse codimension 1 submanifolds $(U_r, U_b)$
	we construct a~desired shading as follows.
	Given $p \in \mfld \setminus (U_r\cup U_b)$ choose a~path $\gamma$
	from $\ast$ to $p$, transverse to both $U_r$ and $U_b$, and let 
	$d(\gamma) := \#(\gamma \cap U_r) + \#(\gamma \cap U_b)$
	count the~intersection points of $\gamma$ with both submanifolds.
	Color $p$ white or black depending on whether $d(\gamma)$ is even
	or odd. Because $\mfld$ is simply connected, the~parity of $d(\gamma)$ does
	not depend on the~choice of $\gamma$ and the~color of $p$ is well-defined.
\end{proof}

\begin{rmk}\label{rmk:standard-orientation}
	It follows from Lemma~\ref{lem:shadibility} that every codimension 1
	submanifold $U$ of a~simple connected manifold $\mfld$ admits a~\emph
	{standard orientation}: the~one induced from white regions, when $U$
	is considered as a~shading with $U_r = \emptyset$. When $\mfld$ is
	a~line and $U$ a~collection of blue points, then the~standard orientation
	on $U$ is the~alternating one. Likewise for the~case $\mfld = \mathbb S^1\rics$,
	assuming the~cardinality of $U$ is even (otherwise it does not extend
	to a~shading).
\end{rmk}

A~\emph{trivalent manifold} embedded in $\mfld$ is a~generalization of
webs and foams. It is a~collection of \emph{facets}, oriented codimension
1 submanifolds colored blue or red, with boundary components
attached transversely to $\partial\mfld$ or glued together along \emph{bindings}
in a~way, such that locally two blue facets merge into a~red one. In other
words, each point of a~trivalent manifold has a~neighborhood diffeomorphic to
either $\R^{n-1}$ or $Y\times\R^{n-2}$, where $Y$ is an~oriented \emph{merge}
or a~\emph{split} from \eqref{eq:flow-cond}.

Given a~shading $(U_r,U_b)$ of $\mfld$ we construct a~trivalent manifold
$\Gamma(U_r, U_b)$ by examining the~orientation on facets induced
from white regions:
\begin{itemize}
	\item blue facets inherit the~orientation,
	\item red facets are preserved (,,amplified'') if the~induced orientation
	      agrees with the~given one or annihilated otherwise.
\end{itemize}
An~example is presented in Figure~\ref{fig:shading->foam}.
\begin{figure}[ht]%
	\centering
	% Shading of a disk
\begin{tikzpicture}[x=2em,y=2em]
% Circles containing the arcs:
%		left blue:  (-3,   1.8, r=2.8)
%		right blue: ( 4,  -1.2, r=3.5)
%		top red:    ( 2.5,   5, r=5.0)
%		bottom red: ( 0.7,-2.3, r=1.5)
%
\begin{scope}[shift={(-4,0)}]
	% shaded triangles
	\begin{scope}
		\clip (0,0) circle[radius=2];
		\clip (2.5, 0)
			arc[radius=5, start angle=270, end angle=180]
			-- (2.6, 1.0) -- (2.6,-2.3) -- (-0.8,-2.3)
			arc[radius=1.5, start angle=180, end angle=0];
		\fill[BCshade]
			(-3, 1.8) circle[radius=2.8]
			( 4,-1.2) circle[radius=3.5];
	\end{scope}
	% shaded quadrilateral
	\begin{scope}
		\clip (-0.2,1.8)
			arc[radius=2.8, start angle=0, end angle=-90]
			-- (-2,-2) -- (0.4,-2) -- (0.5,-1.2)
			arc[radius=3.5, start angle=180, end angle=135];
		\clip (-0.8,-2.3)
			arc[radius=1.5, start angle=180, end angle=60]
			-- (2.5, 0)
			arc[radius=5.0, start angle=270, end angle=180]
			-- (-2,-2);
		\fill[BCshade] (0,0) circle[radius=2];
	\end{scope}
	% boundary circle
	\draw[thin] ( 0,0) circle[radius=2];
	\clip (0,0) circle[radius=2];
	% cutting lines
	\draw[BCedge]
		(-3,1.8) circle[radius=2.8]
		(4,-1.2) circle[radius=3.5];
	\draw[BCedge+,->-]
		(2.5,5) ++(222.3:5) arc[radius=5, start angle=222.3, end angle=264.5];
	\draw[BCedge+,-<-=0.6]
		(0.7,-2.3) ++(50:1.5) arc[radius=1.5, start angle=50, end angle=164];	
	% orientation arrows at white regions
	\foreach \x/\y in {-1.4/0.4, 0/-1.6, 1.3/-0.4, 0.5/1.3}
		\draw[line width=0.5pt, {-Stealth[length=4pt, width=5pt]}] (\x,\y) ++(-150:0.85ex)
			arc[radius=0.85ex, start angle=-150, end angle=150]
			-- ++(220:1pt);
\end{scope}

\draw[->] (-1,0) -- node[midway,above]{$\Gamma$} (1,0);
		
% The web
\begin{scope}[shift={(4,0)}]
	% boundary circle
	\draw[thin] (0,0) circle[radius=2];
	\clip (0,0) circle[radius=2];
	% deleted red arcs
	\draw[thin,gray,dashed]
		(2.5, 5.0) circle[radius=5]
		(0.7,-2.3) circle[radius=1.5];
	% remaining red arcs
	\draw[V2,->-=0.6]
		(2.5, 5.0) ++(235.25:5) arc[radius=5,start angle=235.25,end angle=248.5];
	\draw[V2,->-=0.6]
	   (0.7,-2.3) ++( 97:1.5)  arc[radius=1.5, start angle=97, end angle=50];
	% blue arcs
	\draw[V1,->-=0.5, -<-=0.75]
	   (-3,-1) arc[radius=2.8,start angle=270,end angle=380];
	\draw[V1,-<-=0.55,->-=0.75,-<-=0.92]
		(4, 2.7) arc[radius=3.5, start angle=90, end angle=200];
\end{scope}
\end{tikzpicture}%
	\caption{%
		The~construction of a~web from a~shading of a~disk. The~annihilated
		red edges are drawn as dashed lines on the~right diagram.
	}%
	\label{fig:shading->foam}%
\end{figure}
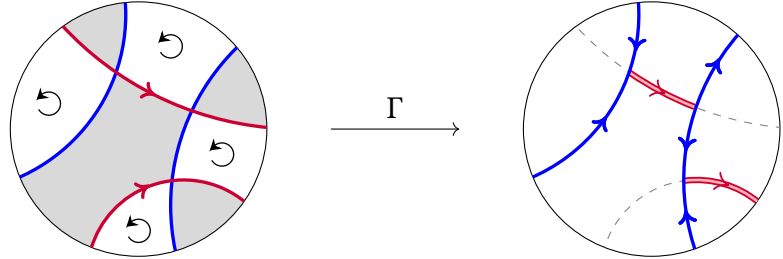%
In particular, $\Gamma(\emptyset, U_b)$ is $U_b$ with its standard
orientation as defined in Remark~\ref{rmk:standard-orientation}.
It appears that every trivalent manifold arises this way,
the~proof of which is presented below and visualized in
Figure~\ref{fig:foam->shading}.
Hence, shadings can be considered as \emph{completions} of trivalent
manifolds, because of which we shall refer to shadings of $\Disk$ and
$\Ball$ as \emph{completed webs} and \emph{completed foams} respectively.

%% Surjectivity
\begin{lem}\label{lem:trivalent->shading}
	Choose a~trivalent manifold $V\subset \mfld\rics$, such that
	$\partial V = \Gamma(\tilde U_r, \tilde U_b)$ for some
	shading $(\tilde U_r, \tilde U_b)$ of\/ $\partial\mfld\rics$.
	Then there exists a~shading $(U_r,U_b)$ of\/ $\mfld\ric$ that restricts
	to $(\tilde U_r,\tilde U_b)$ on $\partial\mfld\rics$ and
	satisfies\/ $\Gamma(U_r,U_b) = V\ric$.
\end{lem}
\begin{proof}
	Consider the~orientation of $\partial\mfld$ induced from $\mfld$
	and reverse it at all points painted black in the~given shading.
	Then the~boundary of any region $R\subset\mfld$ is a~union of facets
	of $V$ and regions in $\partial\mfld\rics$, such that oppositely oriented
	components meet only in two situations: when they are both contained in
	the~boundary (so that they meet at a~facet of $\partial V$)
	or both are blue facets of $V$ adjacent to a~red facet outside of $R$.
	Consider the~union of those components of $\partial R$, the~orientation
	of which does not match the~one induced from $R$.
	They constitute certain oriented $(n-1)$-dimensional submanifolds $U_1,\dots,U_k$.
	Taking a~red colored copy $U'_i$ of each $U_i$ push its interior inside $R$
	and paint the~newly created region black.
	Repeating this for each region produces a~desired shading.
\end{proof}

\begin{figure}[ht]%
	\centering
	% From web to a shading of a disk
% Circles supporting the arcs:
%		left blue:  (-3,   1.8, r=2.8)
%		right blue: ( 4,  -1.2, r=3.5)
%		top red:    ( 2.5,   5, r=5.0)
%		bottom red: ( 0.7,-2.3, r=1.5)
\begin{tikzpicture}[x=2em,y=2em]
% a web with shaded boundary
\begin{scope}[shift={(-7,0)}]
	% boundary circle
	\draw[thin] (0,0) circle[radius=2];
	\fill (150:2) circle[radius=2pt]
	      (243:2) circle[radius=2pt]
	      (  5:2) circle[radius=2pt];
	\foreach \x/\y in {22/b,68/w,124/b,176/w,222/b,265/w,306/b,345/w}
		\node at(\x:2.2) {\scriptsize\y};
	\clip (0,0) circle[radius=2];
	% red edges
	\draw[V2,->-=0.6]
		(2.5, 5.0) ++(235.25:5) arc[radius=5,start angle=235.25,end angle=248.5];
	\draw[V2,->-=0.6]
	   (0.7,-2.3) ++( 97:1.5)  arc[radius=1.5, start angle=97, end angle=50];
	% blue edges
	\draw[V1,->-=0.5, -<-=0.75]
	   (-3,-1) arc[radius=2.8,start angle=270,end angle=380];
	\draw[V1,-<-=0.55,->-=0.75,-<-=0.92]
		(4, 2.7) arc[radius=3.5, start angle=90, end angle=200];
\end{scope}
% a web with oriented boundary
\begin{scope}
	% bounary circle
	\draw[thin] (0,0) circle[radius=2];
	\fill (150:2) circle[radius=2pt]
	      (243:2) circle[radius=2pt]
	      (  5:2) circle[radius=2pt];
	\foreach \angleA/\angleB in {
		 23/ 22,  72/ 73, 124/123, 176/177,
		223/222, 264/265, 308/307, 345/346
	} \draw[->] (\angleA:2)
		arc[radius=2, start angle=\angleA, end angle=\angleB];
	\clip (0,0) circle[radius=2];
	% red edges
	\draw[V2,->-=0.6]
		(2.5, 5.0) ++(235.25:5) arc[radius=5,start angle=235.25,end angle=248.5];
	\draw[V2,->-=0.6]
		(0.7,-2.3) ++( 97:1.5)  arc[radius=1.5, start angle=97, end angle=50];
	% blue edges
	\draw[V1,->-=0.5, -<-=0.75]
	   (-3,-1) arc[radius=2.8,start angle=270,end angle=380];
	\draw[V1,-<-=0.55,->-=0.75,-<-=0.92]
		(4, 2.7) arc[radius=3.5, start angle=90, end angle=200];
\end{scope}
% a shading
\begin{scope}[shift={(7,0)}]
	% shade: top left
	\begin{scope}
		\clip ( 0,0  ) circle[radius=2];
		\clip (-3,1.8) circle[radius=2.8];
		\fill[BCshade]
			(150:2) .. controls (150:1.5) and (-0.75, 1.7)
			.. (-0.60,1.6) .. controls (-0.45, 1.5) and (-0.71,1.17)
			.. (-0.35, 0.89) -- (0,0.89) -- (0,2) -- (-2,2) -- cycle;
	\end{scope}
	% shade: top & bottom right
	\begin{scope}
		\clip (0,0) circle[radius=2];
		\clip (4,-1.2) circle[radius=3.5];
		\fill[BCshade]
			(0.83,0.29) .. controls (1.38, 0.13) and (1.3,0.8)
			.. (1.5, 0.75) .. controls (1.7,0.7) and (5:1.6)
			.. (5:2) -- (2,2) -- (0.83,2) -- cycle;
		\fill[BCshade] (0.7,-2.3) circle[radius=1.5];
	\end{scope}
	% shade: center
	\begin{scope}
		\clip (-0.2,1.8)
			arc[radius=2.8, start angle=0, end angle=-90]
			-- (-2,-2) -- (0.4,-2) -- (0.5,-1.2)
			arc[radius=3.5, start angle=180, end angle=135];
		\clip (-0.8,-2.3)
			arc[radius=1.5, start angle=180, end angle=60]
			-- (2.5,0)
			arc[radius=5.0, start angle=270, end angle=180]
			-- (-2,-2);
		\fill[BCshade]
			(243:2) .. controls (243:1.2) and (-1.8,-1.2)
			.. (-1,-0.8) .. controls (-0.2,-0.4) and (-0.3, 0.4)
			.. (0.1,0.2) .. controls (0.5,0) and (0,-0.87)
			.. (0.52,-0.81) -- (2,0) arc[radius=2,start angle=0,end angle=243];
	\end{scope}
	% boundary circle
	\draw[thin] (0,0) circle[radius=2];
	\clip (0,0) circle[radius=2];
	% thin red arc: top left
	\draw[color=V2lineColor,->-=0.2,->-=0.88]
		(150:2) .. controls (150:1.5) and (-0.75, 1.7) ..
		(-0.60,1.6) .. controls (-0.45, 1.5) and (-0.71,1.17) .. (-0.35, 0.89)
		% top right
		(0.83,0.29) .. controls (1.38, 0.13) and (1.3,0.8) ..
		(1.5, 0.75) .. controls (1.7,0.7) and (5:1.6) .. (5:2);
	% thin red arc: bottom right
	\draw[color=V2lineColor,rounded corners=6pt,fill=white,->-=0.15]
		(318:1.85) -- (293:1.85) -- (305:1.15) -- cycle;
	% thin red arc: center
	\draw[color=V2lineColor,->-]
		(243:2) .. controls (243:1.2) and (-1.8,-1.2) ..
		(-1,-0.8) .. controls (-0.2,-0.4) and (-0.3, 0.4) ..
		(0.1,0.2) .. controls (0.5,0) and (0,-0.87) .. (0.52,-0.81);
	% red arcs
	\draw[BCedge+,->-=0.6]
		(2.5, 5.0) ++(235.25:5) arc[radius=5,start angle=235.25,end angle=250.5];
	\draw[BCedge+,->-=0.6]
		(0.7,-2.3) ++( 97:1.5)  arc[radius=1.5, start angle=97, end angle=50];
	% blue arcs
	\draw[BCedge]
		(-3,1.8) circle[radius=2.8]
		(4,-1.2) circle[radius=3.5];
\end{scope}
% arrows
\draw[->] (-4.5,0) -- (-2.5,0)
	node[midway, anchor=south] {\small orient}
	node[midway, anchor=north] {\small boundary};
\draw[->] ( 2.5,0) -- ( 4.5,0)
	node[midway, anchor=south] {\small push}
	node[midway, anchor=north] {\small \&\ shade};
\end{tikzpicture}
	\caption{%
		The~construction of a~shading from a~planar web that extends a~given shading
		of its boundary.
		The~boundary of the~disk is oriented in the~middle picture, whereas
		the~curves $U'_i$ are identified and pushed inwards the~corresponding regions
		in the~third picture.}%
	\label{fig:foam->shading}%
\end{figure}
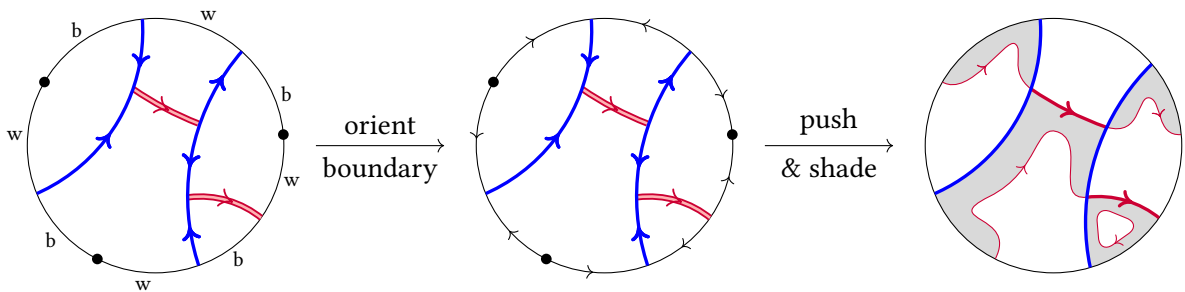

A~useful consequence of Lemma~\ref{lem:trivalent->shading} is that tangles and
surfaces can be extended to webs and foams with given boundary. Recall that
a~collection $\bdry \subset \partial\Disk$ of oriented red and blue points
is \emph{balanced} if it bounds a~web, which is equivalent to being of weight
zero.

\begin{prop}\label{prop:blue->binded}\ 
	\begin{enumerate}
		\item	
		Let\/ $\bdry \subset \partial\Disk$ be a~balanced collection of oriented
		red and blue points and $\tau$ a~tangle bounded by $\bdry_b$.
		Then there exists a~web $\web$ bounded by\/ $\bdry$ with $\web_b = \tau\ric$.
		
		\item
		Let $\web \subset \partial\Ball$ be a~web and\/ $W\ric$ a~surface
		bounded by $\web_b$. Then there is a~foam $\foam$ bounded by $\web$
		with $\foam_b = W\ric$.
	\end{enumerate}
\end{prop}
\begin{proof}
	Extend $\bdry$ to a~shading $\widetilde\bdry = (\bdry_r \cup \bdry'_r, \bdry_b)$.
	Then $\widetilde\bdry$ has an~even number of points and the~orientation of
	points from $\bdry$ matches
	the~one induced from white regions. Let $b$, $r$, and $r'$ be the~sums of
	orientations of blue points in $\bdry$, red points in $\bdry$, and red
	points added to $\widetilde\bdry$ respectively. Then $b + 2r = 0$, because
	$\bdry$ is balanced, and $b + r + r' = 0$, because the~orientation of
	points in $\widetilde\bdry$ alternate. Subtracting the~two equalities
	reveals that $r - r' = 0$. It follows that there is an~oriented collection
	of disjoint intervals $\tau_r \subset \Disk$ bounded by $\widetilde\bdry_r$,
	the~orientation of which agree with the~points from $\bdry_r$ and disagree
	with those from $\bdry'_r$. Hence, $\web := \Gamma(\tau_r, \tau)$ is
	the~desired web.
	
	The~second statement is even easier to show. Extend the~web $\web$ to
	a~shading $\dbltan$. Then $\dbltan_r$ is a~collection of disjoint loops
	and each such collection bounds a~family $W_r$ of disjoint disks in $\Ball$.
	Therefore, $\foam := \Gamma(W_r, W)$ is the~desired foam.
\end{proof}

%% ===================================================================
\subsection{Bicolored isotopies}

Choose an~isotopy $\Phi$ of $\mfld$ and a~subset $A$. The~set
\(
	\Tr_\Phi(A) = \{ (\Phi_t(a), t)\ |\ a\in A, t\in [0,1]\}
\)
is called the~\emph{trace of $A\subset\mfld$ under} $\Phi$ \cite{DiffTop}.
We say that a~pair of isotopies $(\Phi, \Psi)$ of $\mfld$ is an~\emph{isotopy
of a~shading} $(U_r, U_b)$ if $(\Tr_\Phi(U_r), \Tr_\Phi(U_b))$ is a~shading of
$\mfld\times[0,1]$ that coincides with $(U_r, U_b)$ at the~level $t=0$.
When $\mfld$ is a~disk, then a~generic pair of isotopies can be encoded by
a~sequence of \emph{bigon moves}
\begin{equation}\label{rel:isot-webs}
	\begin{centertikz}[x=5em,y=5em]
		\draw[BCedge ] (0.3,0.2) .. controls (0.7,0.4) and (0.7,0.6) .. (0.3,0.8);
		\draw[BCedge+] (0.7,0.2) .. controls (0.3,0.4) and (0.3,0.6) .. (0.7,0.8);
	\end{centertikz}
	\quad\leftrightarrow\quad
	\begin{centertikz}[x=5em,y=5em]
		\draw[BCedge ] (0.3,0.2) .. controls (0.5,0.5) .. (0.3,0.8);
		\draw[BCedge+] (0.7,0.2) .. controls (0.5,0.5) .. (0.7,0.8);
	\end{centertikz}
\end{equation}
whereas in case of a~3-ball two moves are necessary:
\begin{equation}\label{rel:isot-foams}
	\tikzset{x=1.5em,y=1.5em}%
	\foampict{red cup < blue plane}
		\,\leftrightarrow\,
	\foampict{red cup > blue plane}
\qquad\text{and}\qquad
	\foampict{red saddle < blue plane}
		\,\leftrightarrow\,
	\foampict{red saddle > blue plane}
\end{equation}
In each move a~shading of one side determines a~shading of the~other.
Hence, we obtain the~following characterization of isotopies of shadings
in these cases.%
\footnote{
	This can be extended to all manifolds by a~detailed analysis of
	singular levels of a~pair of isotopies.
}
\begin{lem}
	Two shadings of\/ $\Disk$ or $\Ball$ are isotopic if and only if their
	codimension 1 components are isotopic, possibly by different isotopies.
%	Let $\mfld$ be a~disk or a~3-ball. Then shadings $(U_r,U_b)$ and
%	$(U'_r, U'_b)$ are isotopic if and only if $U_r$ is isotopic to $U'_r$
%	and $U_b$ is isotopic to $U'_b$.
\end{lem}

When a~basepoint $\ast\in\mfld$ is present, then one must be careful how it
behaves under the~isotopy. There is no problem when $\Psi$ and $\Phi$ coincide
at $\ast$ (and in this paper we always assume that both $\Psi$ and $\Phi$ fix
$\ast$). Otherwise, the~basepoint should stay at the~same region if possible.
However, when the~region disappears, then the~basepoint has to reappear in
a~white region. For instance, when $\ast$ lies in the~small bigon on the~left
hand side of \eqref{rel:isot-webs}, then it reappears between the~two strands
on the~right hand side. The~same can be done for both moves in
\eqref{rel:isot-foams}.

Recall from Section~\ref{sec:foams} that we write $\foam \foamequiv \foam'$
for foams $\foam$ and $\foam'$ if they agree up to a~sign and replacing some
dots with their duals.

\begin{lem}[Bicolored Isotopy]\label{lem:rb-isotopy}\ 
	\begin{enumerate}
		\item $\Gamma(\dbltan_r,\dbltan_b) = \Gamma(\dbltan'_r,\dbltan'_b)$
		in $\set{Web}$ if $(\dbltan_r,\dbltan_b)$ and $(\dbltan'_r,\dbltan'_b)$
		are isotopic shadings of\/ $\Disk\rics$.
		
		\item $\Gamma(\dblcob_r,\dblcob_b) \foamequiv \Gamma(\dblcob'_r,\dblcob'_b)$
		in $\set{Foam}$ if $(\dblcob_r,\dblcob_b)$ and $(\dblcob'_r,\dblcob_b')$
		are isotopic shadings of\/ $\Ball\rics$.
	\end{enumerate}
\end{lem}
\begin{proof}
	It is enough to consider the~case of elementary isotopies.
	When applied to each side of the~bigon move \eqref{rel:isot-webs},
	$\Gamma$ removes red edges in both pictures from the~same side of
	the~blue line. Hence, $\Gamma(\dbltan_r,\dbltan_b)$ and $\Gamma
	(\dbltan'_r,\dbltan'_b)$ are related by the~left relation in either
	\eqref{rel:web-detach-and-saddle} or \eqref{rel:web-bigons}.
	Likewise, the~moves \eqref{rel:isot-foams} correspond to
	the~detaching relations \eqref{rel:detach-cylinder} and
	\eqref{rel:detach-saddle}.
\end{proof}

The~above result has far reaching consequences when paired with
Lemma~\ref{lem:trivalent->shading}. The~statements about comparing webs
and foams with isotopic blue pieces follows, which in turn were used
in the~proof of Theorem~\ref{thm:foam-evaluation} to show bijectivity
of the~Blanchet evaluation map $Z$.

\begin{proof}[Proof of Proposition~\ref{prop:webs-red}]
	Let $\web$ and $\web'$ have isotopic underlying tangles
	and take the~trace of $\web_b$ under this isotopy as $\foam_b$;
	it is the~underlying surface of a~foam $\foam\colon\web \to \web'$
	due to Proposition~\ref{prop:blue->binded}.
	Extend the~foam to a~shading $(\widetilde\foam_r, \foam_b)$ of $\DxI$.
	When in generic position, it can be represented by a~finite
	sequence of level sets, such that in between any two consecutive levels
	$\widetilde\foam_r$ has either no critical points
	(so that the~level sets are related by the~Bicolored Isotopy Lemma) or
	a~unique Morse type critical point---a~cap, a~cup, or a~saddle---in which
	case the~corresponding webs coincide (if the~affected red edges are erased)
	or are identified by the~right relations in \eqref{rel:circle-evaluation}
	and \eqref{rel:web-detach-and-saddle} (if the~red edges survive).
	Notice that $\foam_b$ has no critical points.
	
	For the~second part, extend $\web$ to a~shading $(\widetilde\web_r, \web_b)$
	of $\Disk$ and isotope closed blue loops, so that they do not intersect
	$\web_r$. Applying $\Gamma$ results in a~new web $\web'$ that coincides
	with $\web$ as shown above. Removing blue circles from $\web'$ results in
	$\mathrm r(\web)$ and the~desired equality follows from
	\eqref{rel:circle-evaluation}.
\end{proof}

\begin{proof}[Proof of Proposition~\ref{prop:foams-red}]
	Let foams $\foam_1$ and $\foam_2$ have isotopic blue parts.
	Extend them to shadings $\dblcob_1$ and $\dblcob_2$ respectively
	and pick a~ball $\mathcal O$ in the~interior of $\Ball$, outside of which
	the~red facets of the~shadings coincide. Using Lemma~\ref{lem:rb-isotopy}
	isotope blue facets away from $\mathcal O$ (this may dualize dots), reducing
	the~problem to showing equality for foams with only red facets.
	In such case, use the~neck cutting relation \eqref{rel:neck-cutting}
	to reduce each foam to a~collection of disjoint disks and spheres;
	this may change the~sign of the~foam. The~thesis follows, because
	each red sphere evaluates to $-1$ and the~disks are uniquely
	determined up to an~isotopy by the boundary circles.
\end{proof}

It follows immediately from Proposition~\ref{prop:foams-red} that the~foam
used in the~proof of Proposition~\ref{prop:webs-red} is invertible. That
would be enough to prove the~latter if we knew that $\cat{Foam}$ categorifies
$\set{Web}$. However, the~proof of the~categorification result is based on
Theorem~\ref{thm:basis-for-Foam(w)}, which is proven only in the~next section.

% =============================================================================
\subsection{Cup foams}
\label{sec:cup basis}

We will now apply the~above results to show that cup foams, as defined in
Section~\ref{sec:foams}, constitute a~free basis of spaces of foams.
In particular, the~category of foams is non-degenerate.

Let $\web$ be a~closed web, so that $\web_b$ is a~collection of blue loops.
Orient them in a~standard way (see Remark~\ref{rmk:standard-orientation})
and pick a~foam $I_\web \in \Hom_{\cat{Foam}(\emptyset)}(\web_b, \web)$ with
$\web_b \times [0,1]$ as its underlying surface; the~existence of such a~foam
follows from Proposition~\ref{prop:blue->binded}.
According to Proposition~\ref{prop:foams-red}, there is a~sign
$\sgn(\web) = \pm1$ satisfying
\[
	\flip I_\web \, I_\web = \sgn(\web) (\web_b\times[0,1]),
\]
where $\flip I_\web \in \Hom_{\cat{Foam}}(\web, \web_b)$ is the~vertical flip
of $I_\web$ as defined in \eqref{eq:foam-reversion}. The~sign
can be also computed directly as
\[
	\sgn(\web)
		= \sgn(\web)(\dotted C_\web \, C_\web)
		= Z(\dotted C_\web \, \flip I_\web \, I_\web \, C_\web),
\]
where $C_\web \in \Hom_{\cat{Foam}}(\emptyset, \web_b)$ is a~collection of disks
bounded by $\web_b$ and $\dotted C_\web \in \Hom_{\cat{Foam}}(\web_b, \emptyset)$
the~same collection, except that each disk is decorated by a~dot. Hence,
$\sgn(\web)$ is a~well-defined integer, which we call the~\emph{sign of
the~web $\web$}.

\begin{lem}\label{lem:sign-of-web}
	The~sign $\sgn(\web)$ does not depend on the~choice of $I_\web$.
\end{lem}
\begin{proof}
	Let $\foam \in \Hom_{\cat{Foam}}(\web_b, \web)$ be another foam with
	$\foam_b = \web_b \times [0,1]$. Then $\foam = \pm I_\web$
	by Proposition~\ref{prop:foams-red} and $\foam^!\,\foam =
	\flip I_\web\, I_\web$, because the~same sign relates $\foam^!$
	with $\flip I_\web$.
%	Hence, the~thesis follows.
\end{proof}

Let $BL(\web)$ be the~collection of blue loops in $\web$.
For each subset $X \subset BL(\web)$ we construct the~cup foam
$\cupfoam(\web, X)$ by attaching blue disks to the~input of $I_\web$
and placing a~dot on each disk bounded by a~loop from $X$. Notice that red
facets of $\cupfoam(\web, X)$ are above all dots and minima of blue facets.
Therefore, we say that $\cupfoam(\web, X)$ is a~\emph{red-over-blue cup foam
decorated by}  $X\ric$.
We construct likewise a~\emph{cap foam} $\capfoam(\web, X)
\in \Hom_{\cat{Foam}(\emptyset)}(\web, \emptyset)$ by reflecting
$\cupfoam(\web, X)$ vertically and replacing each dot with the~dual
one scaled by $-1$. For instance, we have the~following
correspondence between cup and cap foams bounded by two blue loops:
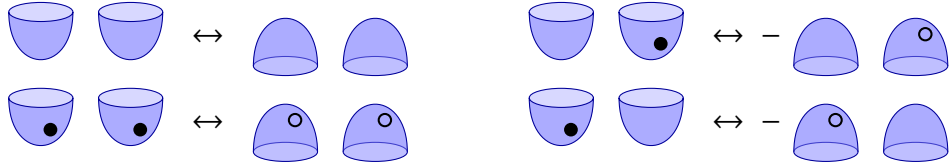
\[
	\tikzset{x=2em,y=2em}%
	\setlength\arraycolsep{3pt}%
	\begin{array}{rclcrcl}
		\foampict{1cup}\foampict{1cup} & \leftrightarrow &
		\foampict{1cap}\foampict{1cap}
		& \mbox{\hskip 2em} &
		\foampict{1cup}\foampict{1cup}[1] & \leftrightarrow &
		-\foampict{1cap}\foampict{1cap}[2]
		\\
		\foampict{1cup}[1]\foampict{1cup}[1] & \leftrightarrow &
		\foampict{1cap}[2]\foampict{1cap}[2]
		&&
		\foampict{1cup}[1]\foampict{1cup} & \leftrightarrow &
		-\foampict{1cap}[2]\foampict{1cap}
	\end{array}
\]
Let us now represent a~foam $\foam \in \Hom_{\cat{Foam}(\emptyset)}
(\web, \web')$ by a~vertical cylinder labeled $\foam\ric$, with $\web$ and $\web'$
at the~bottom and top disk respectively.
When no label is present, it is understood that $\web = \web'$ and the~cylinder
represents the~identity foam $\web\times[0,1]$.
We emphasize the~cases $\web=\emptyset$ and $\web'=\emptyset$ by drawing a~cup
or a~cap instead and, to simplify notation, we decorate it directly with
$X \subset BL(\web)$ when $S$ is a~cup or a~cap foam:
\[
	\cupfoam(\web, X) =
	\begin{tikzpicture}[x=1.5em,y=1.5em,baseline=-1.5em]
		\draw (0,0) ellipse[x radius=1, y radius=0.35]
	    	  (1,0) arc[x radius=1, y radius=1.5, start angle=360, end angle=180];
		\node (X) at (0,-1) {$X$};
		\node at (0, 0) {$\web$};
	\end{tikzpicture}
\qquad\text{and}\qquad
	\capfoam(\web, X) =
	\begin{tikzpicture}[x=1.5em,y=1.5em,baseline=0.5em]
		\draw (1,0)
	    	  	arc[x radius=1, y radius=1.5, start angle=  0, end angle=180]
			arc[x radius=1, y radius=0.35,start angle=180, end angle=360];
		\draw[dashed] (1,0)
	    	  	arc[x radius=1, y radius=0.35, start angle=0, end angle=180];
		\node (X) at (0, 1) {$X$};
		\node at (0, 0) {$\web$};
	\end{tikzpicture}
\]
Moreover, $X^c := BL(\omega) \setminus X$ stands for the~complement of a~subset
$X\subset BL(\omega)$.

\begin{lem}\label{lem:cup-cap-foam-relations}
	Foams satisfy the~following relations:
	\begin{align}
	\label{rel:bubble-evaluation}
		\begin{centertikz}[x=1.5em, y=1.5em]
			\draw (0,0) node {$\web$} ellipse[x radius=1, y radius=1.5]
				(1,0) arc[x radius=1, y radius=0.35, start angle=360, end angle=180];
			\draw[dashed] (1,0) arc[x radius=1, y radius=0.35, start angle=0, end angle=180];
			\node at (0,-1) {$X$};
			\node at (0, 1) {$Y$};
		\end{centertikz}
		\quad&=\ 
		\begin{cases}
			\sgn(\web),& \text{if\/ } Y = X^c,\\[1ex]
			0,&\text{otherwise},
		\end{cases}
	\\[2ex]
	\label{rel:cylinder-cutting}
		\begin{centertikz}[x=1.5em, y=1.5em]
			\draw (0,0) node {$\web$} (0,4) node {$\web$} ellipse[x radius=1, y radius=0.35]
				  (-1,4) -- (-1,0) (1,4) -- (1,0)
				  arc[x radius=1, y radius=0.35, start angle=360, end angle=180]
				  (0,2);
			\draw[dashed] (1,0) arc[x radius=1, y radius=0.35, start angle=0, end angle=180];
		\end{centertikz}
		\quad&= \sgn(\web) \sum_{X\subset BL(\web)}
		\begin{centertikz}[x=1.5em, y=1.5em]
			\draw (0,0) node {$\web$} (1,0)
				arc[x radius=1, y radius=0.35, start angle=360, end angle=180]
				arc[x radius=1, y radius=1.5, start angle=180, end angle=0];
			\draw[dashed] (1,0) arc[x radius=1, y radius=0.35, start angle=0, end angle=180];
			\node at (0,1) {$X^c$};
			\draw (0,4) node {$\web$} ellipse[x radius=1, y radius=0.35] (1,4)
				arc[x radius=1, y radius=1.5, start angle=360, end angle=180];
			\node at (0,3) {$X$};
		\end{centertikz}
	\end{align}
%	\begin{align}
%		\capfoam(\web, X) \cupfoam(\web, Y) &=
%			\begin{cases}
%				\sgn(\web),& \text{if\/ } Y = X^c,\\
%				0,&\text{otherwise},
%			\end{cases}
%		\\
%		\web\times[0,1] &= \sgn(\web) \sum_{X \subset BL(\web)}
%			\cupfoam(\web,X) \capfoam(\web,X^c).
%	\end{align}
\end{lem}
\begin{proof}
	From the~construction of cup and cap foams
	\[
		\tikzset{x=1.5em, y=1.5em}
		\begin{centertikz}
			\draw (0,0) node {$\web$} ellipse[x radius=1, y radius=1.5]
				(1,0) arc[x radius=1, y radius=0.35, start angle=360, end angle=180];
			\draw[dashed] (1,0) arc[x radius=1, y radius=0.35, start angle=0, end angle=180];
			\node at (0,-1) {$X$};
			\node at (0, 1) {$Y$};
		\end{centertikz}
		\quad=\quad
		\begin{centertikz}
			\draw (1,0) arc[x radius=1, y radius=0.35, start angle=360, end angle=180]
			      (1,2) arc[x radius=1, y radius=0.35, start angle=360, end angle=180]
			      (1,4) arc[x radius=1, y radius=0.35, start angle=360, end angle=180];
			\draw[dashed]
			      (1,0) arc[x radius=1, y radius=0.35, start angle=0, end angle=180]
			      (1,2) arc[x radius=1, y radius=0.35, start angle=0, end angle=180]
			      (1,4) arc[x radius=1, y radius=0.35, start angle=0, end angle=180];

			\node at (0,0) {$\web_b$};
			\node at (0,2) {$\web$};
			\node at (0,4) {$\web_b$};

			\draw (1,4) -- ( 1, 0) arc[x radius=1, y radius=1.5, start angle=360, end angle=180]
			            -- (-1, 4) arc[x radius=1, y radius=1.5, start angle=180, end angle=0]
			            -- cycle;
			            
			\node at (0,-1) {$X$};
			\node at (0, 5) {$Y$};
			\node at (0, 1) {$I_\web$};
			\node at (0, 3) {$\flip I_\web$};
		\end{centertikz}
		\quad=\sgn(\web)\,
		\begin{centertikz}
			\draw (0,0) node {$\web_b$} ellipse[x radius=1, y radius=1.5]
				(1,0) arc[x radius=1, y radius=0.35, start angle=360, end angle=180];
			\draw[dashed] (1,0) arc[x radius=1, y radius=0.35, start angle=0, end angle=180];
			\node at (0,-1) {$X$};
			\node at (0, 1) {$Y$};
		\end{centertikz}
	\]
	and the~right hand side is a~collection of spheres, each carrying at most one
	regular and one dual dot, scaled by $(-1)^{|Y|}$. Such a~sphere evaluates to
	1 or $-1$ when it carries either one regular or one dual dot respectively and
	vanishes otherwise (see Exercise~\ref{rel:dual-dot}). Hence,
	\eqref{rel:bubble-evaluation} follows.
	
	The~second relation follows from the~equality
	$\web\times[0,1] = \sgn(\web)\, I_\web\, \flip I_\web$
	and the~neck cutting relation from Exercise~\ref{rel:dual-dot}.
\end{proof}

We are ready to prove that cup foams form a~linear basis of foams.

\begin{proof}[Proof of Theorem~\ref{thm:basis-for-Foam(w)}]
	The~first relation of Lemma~\ref{lem:cup-cap-foam-relations}
	implies that cup foams are linearly independent. To show that
	they generate $\set{Foam}(\web) \cong \Hom_{\cat{Foam}(\emptyset)}
	(\emptyset, \web)$, use the~second relation to
	write a~foam $S$ bounded by $\web$ as a~sum
	\[
		\tikzset{x=1.5em, y=1.5em}
		\begin{centertikz}
			\draw (0,0) node{$\web$} ellipse[x radius=1, y radius=0.35]
				  (1,0) -- (1,-2)
			      arc[x radius=1, y radius=2.5, start angle=360, end angle=180] -- (-1,0)
			      (0,-2) node {$S$};
		\end{centertikz}
		\quad=
		\sgn(\web) \sum_{X\subset BL(\web)}
		\begin{centertikz}[x=1.5em, y=1.5em]
			\draw (0,0) node {$\web$}
				ellipse[x radius=1, y radius=1.5]
				 (1, 0) arc[x radius=1, y radius=0.35, start angle=360, end angle=180]
				 (0,-1) node {$S$};
			\draw[dashed] (1,0) arc[x radius=1, y radius=0.35, start angle=0, end angle=180];
			\node at (0,1) {$X^c$};
			\draw (0,4) node {$\web$} ellipse[x radius=1, y radius=0.35] (1,4)
				arc[x radius=1, y radius=1.5, start angle=360, end angle=180];
			\node at (0,3) {$X$};
		\end{centertikz}
	\]
	which is a~linear combination of cup foams, because closed foams evaluate
	to scalars. Finally,
	\[
		\deg(\cupfoam(\web,X)) = 2 |X| - \ell,
	\]
	as the~underlying surface of the~cup foam consists of $\ell$ disks decorated
	by $|X|$ dots, so that
	\[
		\rk_q \set{Foam}(\web) = \sum_X q^{2|X| - \ell}
			= \sum_{s=0}^\ell {\ell \choose s} q^{2s - \ell}
			= (q + q^{-1})^\ell
	\]
	as desired.
\end{proof}

\section{Equivalences of foam and cobordism categories}
\label{sec:equivalences}

% =============================================================================
In this section we prove Theorems~\ref{thm:equiv-of-bicats} and
\ref{thm:equiv-of-cats}, which state that foams and Bar-Natan cobordisms
constitute equivalent (bi)categories. We then relate the~formal complexes
$\KhBracket{T}$ and $\wKhBracket{T}$ associated with a~tangle $T\ric$.

% =============================================================================
\subsection{Embedding cobordisms into foams}
\label{sec:local BN-->Foam}

% as before: there is an `annular' web extending identity
% note: there is a unique collection of red arcs that do not
% intersect \ast \times [0,1]

Fix a~balanced collection $\bdry \subset \partial\Disk$ away from a~fixed
basepoint $\ast\in\partial\Disk$ and write $\bdry_b$ for the~subset consisting
of all blue points from $\bdry$. Consider first the~case when $\bdry = \bdry_b$
and the~points are oriented in a~standard way as explained in
Remark~\ref{rmk:standard-orientation}. This means that, when following
the~orientation of the~boundary circle, the~first point after the~basepoint
is negative and then the~orientation alternates.
Theorem~\ref{thm:equiv-of-cats} is in this case a~direct consequence of
Proposition~\ref{prop:webs-red} and Theorem~\ref{thm:basis-for-Foam(w)}:
each web is isomorphic to an~entirely blue one (and each such web is
a~flat tangle equipped with the~standard orientation) and for such webs
$\web$ and $\web'$ the~cup basis of $\Hom_{\cat{Foam}(\bdry)}(\web, \web')$
consists of foams with no red facets. Hence, the~naive map
$\Hom_{\cat{BN}(\bdry)}(\web, \web') \to \Hom_{\cat{Foam}(\bdry)}(\web, \web')$
that orients a~cobordism in a~standard way does the~job.
A~little more work has to be done to cover the~general case.

\begin{lem}
	There is a~web $\EqSymb_\bdry \subset \S\times[0,1]$ bounded by $\bdry$
	at $\S\times\{1\}$ and standardly oriented $\bdry_b$ at $\S\times\{0\}$,
	which is disjoint from $\{\ast\}\times[0,1]$ and with $\bdry_b \times [0,1]$
	as the~underlying tangle.
\end{lem}
\begin{proof}
	Let $\tau$ be a~collection of radial blue intervals connecting
	blue points at $\S\times\{0\}$ with those at $\S\times\{1\}$.	
	Cut the~annulus to a~disk along $\{\ast\}\times[0,1]$ and apply
	Proposition~\ref{prop:blue->binded} to get a~desired web.
\end{proof}

\begin{rmk}\label{rmk:canonical-annular-web}
	The~extension of a~tangle to a~web is constructed in
	Lemma~\ref{lem:trivalent->shading}
	from a~shading of the~disk, which is by no means unique.
	In case of an~annulus, however, the~situation is different:
	there is a~unique up to an~isotopy family of counter-clockwise
	oriented arcs that bounds a~given collection of oriented points
	at the~outer boundary circle. Some of the~arc may intersect
	the~interval $\{\ast\}\times[0,1]$; moving them through the~hole
	results in a~preferred shading and a~preferred web $\EqSymb_\bdry$.
\end{rmk}

\begin{figure}[ht]%
	\centering
	\begin{tikzpicture}[x=1.4em,y=1.4em]
	
	\begin{scope}[shift={(-10,0)}]
		\draw (0,0) circle[radius=2];
		\foreach \angle/\sign in {0/-,40/+,120/+,180/+,240/+,300/+}
			\filldraw[V1] (\angle:2) circle[radius=2pt]
				++(\angle:1.75ex) node {$\scriptstyle\sign$};
		\foreach \angle/\sign in {90/-,210/-,270/-,330/+}
			\draw[V2dot] (\angle:2) circle[radius=3pt]
				++(\angle:1.75ex) node {$\scriptstyle\sign$};
		\fill[white] (65:2) circle[radius=4pt] node[text=black]{$\ast$};
	\end{scope}

	\draw[->] (-6.5,0) -- (-4.5,0);

	\begin{scope}
		\clip (0,0) circle[radius=3];
	
		\draw[V2,-<-=0.4] (120:2)
			arc[x radius=2, y radius=1.5, start angle=-60, end angle=30];
		\draw[V0] (120:2)
			arc[x radius=2, y radius=1.5, start angle=-60, end angle=-120];
		
		\draw[V2,-<-=0.35] (300:2.25)
			arc[x radius=1.5, y radius=1, start angle=120, end angle=240];
		\draw[V0] (300:2.25)
			arc[x radius=1.5, y radius=1, start angle=120, end angle=0];
			
		\draw[V2,->-=0.23,-<-=0.75]
			(210:3) .. controls (210:1.5) and ++(150:0.5) .. (240:1.75)
			(340:3) .. controls (340:1.5) and ++(30:0.75) .. (300:1.75);
		\draw[V0] (240:1.75) arc[radius=1.75, start angle=240, end angle=300];
	\end{scope}
	
	\draw[fill=white] (0,0) circle[radius=1.25];
	\draw (0,0) circle[radius=3];

	\draw[V2dot]
		(0,  3) circle[radius=3pt] ++(0,1.75ex) node{$\scriptstyle-$}
		(210:3) circle[radius=3pt] ++(210:1.75ex) node{$\scriptstyle-$}
		(277:3) circle[radius=3pt] ++(277:1.75ex) node{$\scriptstyle-$}
		(340:3) circle[radius=3pt] ++(340:1.75ex) node{$\scriptstyle+$};
	
	\foreach \angle/\in/\out in {0/-/-,40/+/+,120/-/+,180/+/+,240/-/+,300/+/+}
		\filldraw[V1]
			(\angle:1.25) ++(\angle:-1.5ex) node{$\scriptstyle\in$}
			++(\angle:1.5ex) circle[radius=2pt] -- (\angle:3) circle[radius=2pt]
			++(\angle:1.75ex) node{$\scriptstyle\out$};
	
	\draw[dotted] (70:1.25) -- (70:3);
	\fill[white]
		(70:1.25) circle[radius=4pt] node[text=black]{$\ast$}
		(70:3) circle[radius=4pt] node[text=black]{$\ast$};
	
\end{tikzpicture}
	\caption{%
		A~collection of points $\bdry$ with four red and six blue points and
		an~annular web $\EqSymb_\bdry$ as in Remark~\ref{rmk:canonical-annular-web}.
		Dashed red lines are not part of the~web, but they represent the~additional
		red edges in the~associated shading.
	}%
	\label{fig:annular-diagram}%
\end{figure}
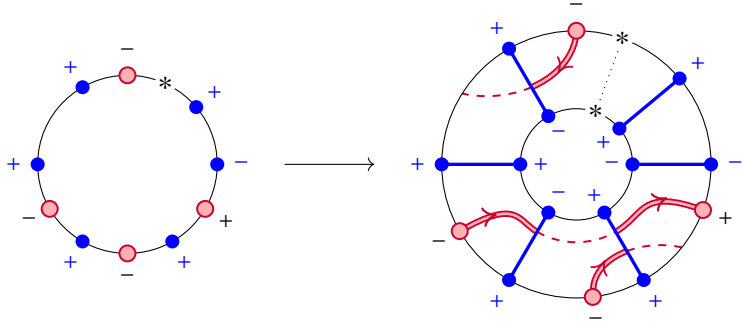

Inserting a~tangle inside the~web $\EqSymb_\bdry$ and a~surface inside the~foam
$\EqSymb_\bdry \times [0,1]$ results in a~functor
$\EqFunc_\bdry \colon \cat{BN}(\bdry_b) \to \cat{Foam}(\bdry)$,
as it preserves units and composition.

\insertpretheorem{thm:equiv-of-cats}
\begin{proof}
	It follows from Propositions \ref{prop:webs-red} and \ref{prop:foams-red}
	that $\EqFunc_\bdry$ is essentially surjective and full. Faithfulness
	follows from Theorem~\ref{thm:basis-for-Foam(w)}: both
	$\Hom_{\cat{BN}(\bdry_b)}(\web_b,\web_b')$ and
	$\Hom_{\cat{Foam}(\bdry)}(\web,\web')$ are free graded modules of graded
	rank $(q+q^{-1})^\ell\rics$, where $\ell$ counts blue loops in
	$-\web \cup \web'\rics$.
\end{proof}

% =============================================================================
\subsection{A~coherent way to forget red facets}
\label{sec:forget-red}

The~inverse functor to $\EqFunc_\bdry$ forgets red facets of foams,
but it may also change the~sign. To construct it explicitly,
fix for each web $\web$ an~invertible foam
$I_\web \in \Hom_{\cat{Foam}(\bdry)}(\EqFunc_\bdry(\web_b), \web)$.
Given a~non-vanishing foam $S\colon \web \to \web'$ consider the~square
\begin{equation}\label{diag:sgn of foam}
\begin{centertikz}[x=10em, y=10ex]
	\node (web0) at (0,1) {$\web$};
	\node (web1) at (1,1) {$\web'$};
	\node (tan0) at (0,0) {$\EqFunc_\bdry(\web_b)$};
	\node (tan1) at (1,0) {$\EqFunc_\bdry(\web'_b)$};
	\draw[->] (web0) -- (web1) node[midway, above] {$\scriptstyle S$};
	\draw[->] (tan0) -- (tan1) node[midway, above]
		{$\scriptstyle \sgn(S) \EqFunc_\bdry(S_b)$};
	\draw[->] (tan0) -- (web0) node[midway, left] {$\scriptstyle I_{\web}$};
	\draw[->] (tan1) -- (web1) node[midway, right] {$\scriptstyle I_{\web'}$};
\end{centertikz}
\end{equation}
where $\sgn(S) = \pm1$ is the~unique sign for which the~square commutes, i.e.
\[
	\foam\,I_\web = \sgn(S) \left( I_{\web'}\EqFunc_\bdry(S_b) \right).
\]
It exists by Proposition~\ref{prop:foams-red}.
Both sides are foams bounded by the~web $-\EqFunc_\bdry(\web_b) \cup \web'$
and they evaluate to a~nonzero scalar after glued with some cup foam
$C$.%
\footnote{
	Explicitly, $C = \cupfoam(-\EqFunc_\bdry(\web_b) \cup \web', X)$
	where $X$ contains exactly one boundary circle of each genus 0 component
	of $S_b$ that does not carry a~dot.
}
Therefore,
\[
	\sgn(S) = \frac
		{Z(C \cup \foam\,I_\web)}
		{Z(C \cup I_{\web'}\EqFunc_\bdry(S_b))},
\]
where $Z$ is the~Blanchet evaluation map.

\begin{prop}
	The~assignment
	\[
		\web \mapsto \web_b,
		\qquad
		\foam \mapsto \sgn(\foam)\foam_b
	\]
	defines a~functor
	$\EqFunc^\vee_\bdry\colon\cat{Foam}(\bdry) \to \cat{BN}(\bdry_b)$
	inverse to $\EqFunc_\bdry$.
\end{prop}
\begin{proof}
	We have to check that the~sign $\sgn(\foam)$ is multiplicative
	with respect to composition of foams. For that pick foams
	$\foam'\in\Hom_{\cat{Foam}(\bdry)}(\web, \web')$ and
	$\foam''\in\Hom_{\cat{Foam}(\bdry)}(\web',\web'')$,
	such that the~composition $\foam''\foam'$ does not vanish. Then
	\begin{align*}
		\foam''\foam'I_\web
		&=
			\sgn(\foam'')\sgn(\foam')\left(
				I_{\web''} \EqFunc_\bdry(\foam''_b) \EqFunc_\bdry(\foam'_b)
			\right) \\
		&= \sgn(\foam'')\sgn(\foam')\left(
				I_{\web''} \EqFunc_\bdry(\foam''_b\foam'_b)
			\right)
	%	\left(\sgn(\foam'') I^{-1}_{\web''} \EqFunc_\bdry(\foam''_b) I_{\web'} \right)&
	%	\left(\sgn(\foam')  I^{-1}_{\web'}  \EqFunc_\bdry(\foam'_b)  I_{\web}  \right)
	%	\\
	%	&=
	%	\sgn(\foam'')\sgn(\foam') 
	%	\left(
	%		I^{-1}_{\web''} \EqFunc_\bdry(\foam''_b) \EqFunc_\bdry(\foam'_b) I_\web
	%	\right)
	%	\\
	%	&=
	%	\sgn(\foam'') \sgn(\foam') 
	%	\left(
	%		I^{-1}_{\web''} \EqFunc_\bdry(\foam''_b\foam'_b) I_\web
	%	\right),
	\end{align*}
	which forces $\sgn(\foam''\foam') = \sgn(\foam'')\sgn(\foam')$.
	To end the~proof, we check directly that
	$\EqFunc^\vee_\bdry \circ \EqFunc_\bdry$ is the~identity functor on
	$\cat{BN}(\bdry_b)$, whereas the~collection of the~invertible foams
	$I_\web$ constitute a~natural isomorphism between
	$\EqFunc_\bdry \circ \EqFunc^\vee_\bdry$ and the~identity
	functor on $\cat{Foam}(\bdry)$.
\end{proof}

\begin{exam}\label{ex:forgetting red facets}
	\tikzset{x=2em,y=2em}%
	Let $\web$ be a~blue circle oriented clockwise. This is the~orientation
	induced from the~unbounded region, hence standard, so that $\web = \web_b$
	and $\EqFunc^\vee_\emptyset$ simply forgets orientation:
	\[
		\EqFunc^\vee_\emptyset\left(\foampict{1cup}[0<]\right) = \foampict{1cup}
	\qquad\text{and}\qquad
		\EqFunc^\vee_\emptyset\left(\foampict{1cup}[1<]\right) = \foampict{1cup}[1].
	\]
	However, when $\web$
	is oriented counter-clockwise, then the~invertible foam $I_\web$ is a~cylinder
	with a~red membrane and removing the~membrane may cost a~sign:
	\[
		\EqFunc^\vee_\emptyset\left(\foampict{1cup}[0>]\right) = \foampict{1cup}
	\qquad\text{and}\qquad
		\EqFunc^\vee_\emptyset\left(\foampict{1cup}[1>]\right) = -\foampict{1cup}[1].
	\]
\end{exam}

\begin{rmk}
	Although the~construction of $\EqFunc^\vee_\bdry$ depends on the~choice of
	foams $I_\web$, the~functor is unique up to a~unique natural isomorphism.
	To see this directly,
	suppose that $\tilde{\EqFunc}^\vee_\bdry$ is constructed using a~different
	family of foams $\tilde I_\web$. Then $\tilde I_\web = s(\web) I_\web$
	for a~well-defined sign $s(\web) = \pm1$ and it follows from a~direct
	computation that the~collection of morphisms
	$\iota_\web := s(\web)\cdot \web_b\times[0,1]$ is a~natural isomorphism
	from $\EqFunc^\vee_\bdry$ to $\tilde{\EqFunc}^\vee_\bdry$.
\end{rmk}

% =============================================================================
\subsection{An~equivalence of bicategories}

Recall that a~1-morphism $f\colon x\to y$ in a~bicategory $\cat{C}$ is
an~\emph{equivalence} if there exists $g\colon y\to x$ such that the~compositions
$f\circ g$ and $g\circ f$ are isomorphic to identity 1-morphisms.
A~2-functor $\mathcal F\colon \cat{C} \to \cat{D}$ is
an~\emph{equivalence of bicategories} when
\begin{itemize}
	\item it is a~\emph{local equivalence}, that is the~functor
	$\mathcal F_{x,y}\colon \cat{C}(x,y) \to \cat{D}(\mathcal F(x), \mathcal F(y))$
	is an~equivalence of categories for all objects $x,y$ of $\cat{C}$, and
	\item it is \emph{essentially surjective}: each object of $\cat{D}$ is
	equivalent to an~object of the~form $\mathcal F(x)$.
\end{itemize}
Indeed, the~above conditions imply the~existence of an~inverse of $\mathcal F$
\cite{BasicBicats}.

There is a~2-functor
\begin{equation}\label{func:BN-->Foam}
	\EqFunc^0 \colon \cat{BN} \to \cat{Foam}
\end{equation}
that equips points, tangles, and cobordisms with the~standard orientation.%
\footnote{
	Recall the~convention that the~basepoint $\ast$ is placed at the~left
	infinity, so that the~left unbounded region is painted white. This
	implies in particular that the~left most point of an~object of $\cat{BN}$
	receives the~positive orientation and the~left most vertical strand of
	a~1-morphism is oriented upwards.
}
It is a~local equivalence due to Theorem~\ref{thm:equiv-of-cats}, but
not essentially surjective: objects from the~image of 
$\EqFunc^0$ have weight 0 or 1, so that the~whole image
is contained in $\cat{Foam}^0 \sqcup \cat{Foam}^1$. We fix this by
enlarging the~source bicategory to $\cat{wBN} := \cat{BN}\times\Z$,
the~product of $\cat{BN}$ with $\Z$ seen as a~discrete bicategory.
In other words, objects of $\cat{wBN}$ are pairs $(\bdry, k)$
consisting of an~object $\bdry$ from $\cat{BN}$ and a~number $k\in\Z$,
whereas morphism categories are copied directly from $\cat{BN}$:
\begin{equation}
	\cat{wBN} ( (\bdry, k), (\bdry', k')) :=
		\cat{BN}(\bdry, \bdry').
\end{equation}
We then extend \eqref{func:BN-->Foam} to a~2-functor
\begin{equation}\label{func:BNxZ-->Foam}
	\EqFunc \colon \cat{wBN} \to \cat{Foam}
\end{equation}
in a~way, such that $(\bdry, k)$ is taken to the~collection $\EqFunc^0(\bdry)$
with $|k|$ red points added to the~right, all positive when $k>0$ and negative
otherwise. Likewise for 1- and 2-morphisms: $\EqFunc$ takes a~tangle $\tau$
(resp.\ a~cobordism $W$), orients it in a~standard way, and adds to the~right
$|k|$ vertical red lines (resp.\ vertical red squares) with the~appropriate
orientation.

\insertpretheorem{thm:equiv-of-bicats}
\begin{proof}
	By Theorem~\ref{thm:equiv-of-cats}, $\EqFunc$ is a~local equivalence. Hence,
	it is enough to show that it is essentially surjective. For that choose
	an~object $\bdry$ from $\cat{Foam}$ and let $k = \floor{w(\bdry)/2}$, where
	$w(\bdry)$ is the~weight of $\bdry$. Then $\bdry^0 := \EqFunc(\bdry_b, k)$
	has the~same weight. Considering $\RxI$ as a~disk with two boundary points
	removed, we can apply Lemma~\ref{lem:trivalent->shading} to the~collection of
	points $-\bdry^0 \cup \bdry$ to obtain a~web $\EqSymb_\bdry \colon \bdry^0
	\to \bdry$ with vertical lines as the~underlying tangle. Another application
	of Lemma~\ref{lem:trivalent->shading} combined with Proposition~\ref{prop:foams-red}
	shows that it is an~equivalence, with its mirror image the~inverse 1-morphism.
\end{proof}

We write $\EqFunc^\vee$ for the~2-functor inverse to $\EqFunc$. It can be
constructed explicitly like the~functors $\EqFunc^\vee_\bdry$, except that
the~computation of signs requires not only a~choice of isomorphisms between
webs, but also a~choice of equivalences between collections of points.
For the~latter one can use the~following webs
%\begin{rmk}
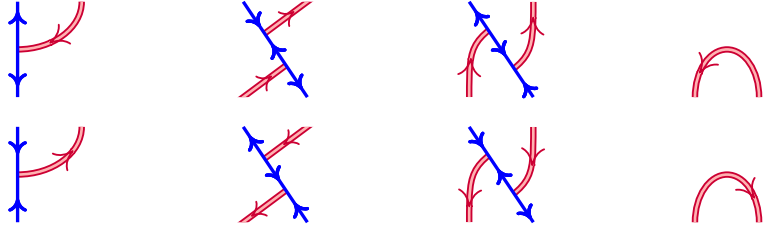
\begin{center}
	\tikzset{x=4em,y=3em}%
	\rule{0pt}{0pt}\hfill\hfill
		\begin{tikzpicture}
			\useasboundingbox (-0.2,-1) -- (0.7,0);
			\draw[V2,->]
				(0,-0.5) arc[radius=0.5, start angle=270, end angle=360]
				(-60:0.5) -- ++(31:-0.1pt);
			\draw[V1,-<-=0.25, ->-=0.9] (0,-1) -- (0,0);
		\end{tikzpicture}
	\hfill
		\begin{tikzpicture}
			\useasboundingbox (-0.2,0) -- (0.7,1);
			\begin{scope}
				\clip (-0.5,0) -- (0,1) -- (0.5,0) -- cycle;
				\draw[V2,->-=0.65] (-0.25,-0.25) -- (0.5,0.5);
			\end{scope}
			\begin{scope}
				\clip (1,1) -- (0,1) -- (0.5,0) -- cycle;
				\draw[V2,->-=0.55] (0,0.5) -- (0.75,1.25);
			\end{scope}
			\draw[V1,-<-=0.2,->-=0.55,-<-=0.85] (0.5,0) -- (0,1);
		\end{tikzpicture}
	\hfill
		\begin{tikzpicture}
			\useasboundingbox (-0.2,0) -- (0.7,1);
			\begin{scope}
				\clip (-0.5,0) -- (0,1) -- (0.5,0) -- cycle;
				\draw[V2,->] (0,0) .. controls (0,0.6) .. (0.5,1)
					(0.0135,0.405) -- ++(0.02pt,0.2pt);
			\end{scope}
			\begin{scope}
				\clip (1,1) -- (0,1) -- (0.5,0) -- cycle;
				\draw[V2,->] (0,0) .. controls (0.5,0.4) .. (0.5,1)
					(0.4995,0.837) -- ++(0.01pt,0.2pt);
			\end{scope}
			\draw[V1,->-=0.15,-<-=0.55,->-=0.9] (0.5,0) -- (0,1);
		\end{tikzpicture}
	\hfill
		\begin{tikzpicture}
			\useasboundingbox (-0.45,0) -- (0.45,1);
			\draw[V2,->]
				(-0.25,0) arc[x radius=0.25, y radius=0.5, start angle=180, end angle=0]
				(150:0.25 and 0.5) -- ++(218:0.01pt and 0.02pt);
		\end{tikzpicture}
	\hfill\hfill\rule{0pt}{0pt}\\[2ex]
	\rule{0pt}{0pt}\hfill\hfill
		\begin{tikzpicture}
			\useasboundingbox (-0.2,-1) -- (0.7,0);
			\draw[V2,->]
				(0,-0.5) arc[radius=0.5, start angle=270, end angle=360]
				(-30:0.5) -- ++(41:0.1pt);
			\draw[V1,->-=0.25, -<-=0.85] (0,-1) -- (0,0);
		\end{tikzpicture}
	\hfill
		\begin{tikzpicture}
			\useasboundingbox (-0.2,0) -- (0.7,1);
			\begin{scope}
				\clip (-0.5,0) -- (0,1) -- (0.5,0) -- cycle;
				\draw[V2,-<-=0.55] (-0.25,-0.25) -- (0.5,0.5);
			\end{scope}
			\begin{scope}
				\clip (1,1) -- (0,1) -- (0.5,0) -- cycle;
				\draw[V2,-<-=0.55] (0,0.5) -- (0.75,1.25);
			\end{scope}
			\draw[V1,->-=0.2,-<-=0.55,->-=0.9] (0.5,0) -- (0,1);
		\end{tikzpicture}
	\hfill
		\begin{tikzpicture}
			\useasboundingbox (-0.2,0) -- (0.7,1);
			\begin{scope}
				\clip (-0.5,0) -- (0,1) -- (0.5,0) -- cycle;
				\draw[V2,->] (0,0) .. controls (0,0.6) .. (0.5,1)
					(0.0005,0.163) -- ++(-0.01pt,-0.2pt);
			\end{scope}
			\begin{scope}
				\clip (1,1) -- (0,1) -- (0.5,0) -- cycle;
				\draw[V2,->] (0,0) .. controls (0.5,0.4) .. (0.5,1)
					(0.4915,0.595) -- ++(-0.01pt,-0.2pt);
			\end{scope}
			\draw[V1,-<-=0.15,->-=0.55,-<-=0.9] (0.5,0) -- (0,1);
		\end{tikzpicture}
	\hfill
		\begin{tikzpicture}
			\useasboundingbox (-0.45,0) -- (0.45,1);
			\draw[V2,->]
				(-0.25,0) arc[x radius=0.25, y radius=0.5, start angle=180, end angle=0]
				(30:0.25 and 0.5) -- ++(-38:0.01pt and 0.02pt);
		\end{tikzpicture}
	\hfill\hfill\rule{0pt}{0pt}
\end{center}
which are equivalences by Proposition~\ref{prop:foams-red}. They can be
used to construct an~explicit equivalence from a~collection $\bdry$ to
$\bdry^0 = \EqFunc(\bdry_b, \floor{w(\bdry)/2})$ by examining the~points
of $\bdry$ from left to right. The~details are left to the~reader.
%\end{rmk}

% =============================================================================
\subsection{Comparison of Khovanov brackets}
\label{sec:Kh-brackets}

We finish this section by comparing two invariant complexes for a~tangle $T$:
the~Khovanov bracket $\KhBracket{T}$ from \cite{DrorCob}, which is a~formal
complex of objects from $\cat{BN}(\partial T)$, and the~Blanchet--Khovanov
bracket $\wKhBracket{T}$ constructed using wbes and foams instead. In what
follows we recall the~construction of the~latter---forgetting red edges in
webs and red facets in foams recovers the~former.

Let $c$ be the~number of crossings in $T$, out of which $c_+$
are positive and $c_-$ are negative. The~first step to construct
$\wKhBracket{T}$ is to compute the~\emph{cube of resolutions} of
$\wKhCube{T}$: a~commutative diagram with resolutions of $T$ at
vertices of the~$c$-dimensional cube $[0,1]^c$.%
\begin{figure}[ht]%
	\centering
	\begin{tikzpicture}
			\node (Xing) at ( 0,2) {$\FNT{PosCr}$};
			\node (Hres) at ( 2,0) {\webpict{zip res}};
			\node (Vres) at (-2,0) {\webpict{arc res}};
			\draw[->] (Xing) -- (Vres) node[midway,above] {$0$};
			\draw[->] (Xing) -- (Hres) node[midway,above] {$1$};
			\draw[->] (Vres) -- (Hres);
			\node[below] at (0,-1ex) {$\foampict{edge zip}$};
	\end{tikzpicture}
	\hskip 4em
	\begin{tikzpicture}
			\node (Xing) at ( 0,2) {$\FNT{NegCr}$};
			\node (Hres) at (-2,0) {$\webpict{zip res}$};
			\node (Vres) at ( 2,0) {$\webpict{arc res}$};
			\draw[->] (Xing) -- (Hres) node[midway,above] {$0$};
			\draw[->] (Xing) -- (Vres) node[midway,above] {$1$};
			\draw[->] (Hres) -- (Vres);
			\node[below] at (0,-1ex) {$\foampict{edge unzip}$};
	\end{tikzpicture}
	\caption{%
		Web resolutions of a~positive (to the~left) and negative (to the~right)
		crossing, together with the~minimal foams between them.
	}%
	\label{fig:resolutions}%
\end{figure}
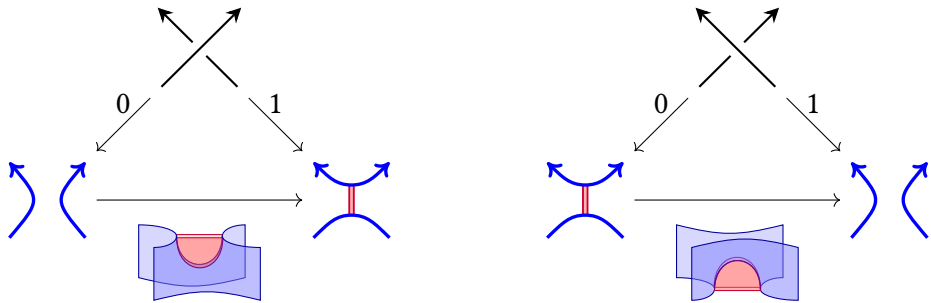
Namely, a~vertex $\xi = (\xi_1, \dots, \xi_c) \in \{0,1\}^c$ is decorated with
the~web $T_\xi$ obtained from $T$ by replacing each $i$-th crossing of the~tangle
with its resolution of type $\xi_i$, as shown in Figure~\ref{fig:resolutions}.
Let $\xi'$ be another vertex, obtained from $\xi$ by changing one coordinate
from $0$ to $1$. The~directed edge $\zeta\colon \xi \to \xi'$ is decorated
with the~minimal foam $T_\zeta\colon T_\xi \to T_{\xi'}$, which is a~collection
of vertical facets except over the~region where the~two resolutions do not match;
here $T_\zeta$ is a~zip or an~unzip as shown in Figure~\ref{fig:resolutions}.
It is evident that $\wKhCube{T}$ commutes: directed paths between same
vertices represent isotopic foams.

Pick a~\emph{sign assignment} $\epsilon$, that is a~collection of signs
$\epsilon(\zeta) = \pm1$, one sign per edge in the~cube, such that the~product
of signs around any square in the~cube is equal to $-1$. The~standard choice
is $\epsilon(\zeta) = (-1)^{s(\xi,\xi')}$, where $s(\xi,\xi')$ counts $1$'s
left to the~place at which $\xi$ and $\xi'$ disagree. Scaling each edge $\zeta$
by $\epsilon(\zeta)$ makes the~cube anticommute and it can be shown that
the~isomorphism type of the~cube is independent of the~sign assignment
(compare with \cite[Lemma 2.2]{OddKh} or \cite[Lemma 5.7]{ChCob}).
The~formal complex $\wKhBracket{T}$ is obtained by flattening the~cube
along diagonals and shifting degrees accordingly. Explicitly,
\[
	\wKhBracket{T}^i := \bigoplus_{|\xi| = i + c_-} T_\xi
		\{ c_- - c_+ - i \}
\]
where $|\xi| := \xi_1 + \dots + \xi_c$, with the~differential
\[
	d|_{T_\xi} = \sum_{\zeta\colon \xi \to \xi'}
		\epsilon(\zeta) T_\zeta.
\]
The~Khovanov bracket $\KhBracket{T}$ is constructed following the~same
steps, except that webs and foams are replaced with flat tangles and
cobordisms. In particular, one has to erase in Figure~\ref{fig:resolutions}
the~red edges in resolutions and red facets in foams.

% chain maps for foams

\begin{thm}\label{thm:strictification}
	The~homotopy type of\/ $\wKhBracket{T}$ is an~invariant of the~tangle $T$,
	strictly functorial with respect to tangle cobordism. Its image under
	$\EqFunc^\vee_{\partial T}$ is isomorphic to $\KhBracket{T}$.
\end{thm}
\begin{proof}
	Following \cite{DrorCob} one can show that $\KhBracket{T}$
	is functorial up to a~sign and strict functoriality is shown in
	\cite{Blanchet} in the~case of links, i.e.\ when $T$ has no
	endpoints. From these two facts strict functoriality follows,
	because every tangle can be closed to a~link.
	
	To compare $\wKhBracket{T}$ with $\KhBracket{T}$ consider the~cube
	of resolutions $\wKhCube{T}$ constructed in $\cat{Foam}(\partial T)$
	and let $\KhCube{T}'$ be its image in $\cat{BN}(\partial T)$ under
	the~equivalence of categories $\EqFunc^\vee_{\partial T}$. It differs
	from $\KhCube{T}$, the~cube of resolutions in $\cat{BN}(\partial T)$
	that computes $\KhBracket{T}$, only in signs at edges. Hence, the~two
	cubes are isomorphic and the~thesis follows.
\end{proof}

\begin{rmk}
	The~construction of $\wKhBracket{T}$ can be easily extended to
	an~invariant of knotted webs \cite{Hoel} and it is conjectured to
	be strictly functorial with respect to foams embedded in a~four
	dimensional space.
\end{rmk}

\section{A~diagrammatic TQFT on \texorpdfstring{$\cat{Foam}(\emptyset)$}{Foam(0)}.}
\label{sec:TQFT}

The~assignment of the~module $\set{Foam}(\web)$ to a~closed web $\web$
extends to a~functor
\[
	\Hom_{\cat{Foam}(\emptyset)}(\emptyset, \blank)
		\colon \cat{Foam}(\emptyset) \longrightarrow \cat{kMod}.
\]
In what follows we provide a~diagrammatic description of this functor
by representing red-over-blue cup foams from $\set{Foam}(\web)$ using
certain planar diagrams and examine how the~diagrams changes under
action of the~linear maps associated with foams.

% =============================================================================
\subsection{A~planar representation of cup foams}

Let $\web$ be a~bounded planar web and $\web^+$ its completion, which is
a~shading $(\web^+_r,\web^+_b)$ satisfying $\Gamma(\web^+_r, \web^+_b) = \web$.
It is assumed that the~basepoint $\ast$ marks the~unbounded region, so that
the~region is painted white. To simplify the~picture and make the~web $\web$
better visible, we do not color regions and we draw red edges as double
or dashed lines depending on whether they survive or disappear after $\Gamma$
is applied, see Figure~\ref{fig:completed webs}.
\begin{figure}[ht]%
	\centering
	\begin{tikzpicture}[x=1em,y=1em,baseline=0pt]
	\draw[V2]
		(-1, 1) -- (-2, 2)	( 1, 1) -- ( 2, 2)
		(-1,-1) -- (-2,-2)	( 1,-1) -- ( 2,-2);
	\draw[V2,->] ( 1.5, 1.5) -- ++( 3pt, 3pt);
	\draw[V2,->] (-1.5,-1.5) -- ++(-3pt,-3pt);
	\draw[V2,->] (-1.5, 1.5) -- ++( 3pt,-3pt);
	\draw[V2,->] ( 1.5,-1.5) -- ++(-3pt, 3pt);
	\draw[V1] (-1,1)
		.. controls (-0.6, 1.4) and ( 0.6, 1.4) .. ( 1, 1)
		.. controls ( 1.4, 0.6) and ( 1.4,-0.6) .. ( 1,-1)
		.. controls ( 0.6,-1.4) and (-0.6,-1.4) .. (-1,-1)
		.. controls (-1.4,-0.6) and (-1.4, 0.6) .. cycle;
	\draw[V1] (-2,2)
		.. controls (-1.3, 2.7) and ( 1.3, 2.7) .. ( 2, 2)
		.. controls ( 2.7, 1.3) and ( 2.7,-1.3) .. ( 2,-2)
		.. controls ( 1.3,-2.7) and (-1.3,-2.7) .. (-2,-2)
		.. controls (-2.7,-1.3) and (-2.7, 1.3) .. cycle;
	\draw[V0]
		(-2, 2) .. controls (-4, 4) and ( 4, 4) .. ( 2, 2)
		(-2,-2) .. controls (-4,-4) and ( 4,-4) .. ( 2,-2)
		(-1, 1) .. controls (-0.2, 0.2) and ( 0.2, 0.2) .. ( 1, 1)
		(-1,-1) .. controls (-0.2,-0.2) and ( 0.2,-0.2) .. ( 1,-1);
\end{tikzpicture}
\hskip 4em
\begin{tikzpicture}[x=1em,y=1em,baseline=0pt]
	\draw[V2]
		(-1, 1) -- (-2, 2)	( 1, 1) -- ( 2, 2)
		(-1,-1) -- (-2,-2)	( 1,-1) -- ( 2,-2);
	\draw[V2,->] ( 1.5, 1.5) -- ++( 3pt, 3pt);
	\draw[V2,->] (-1.5,-1.5) -- ++(-3pt,-3pt);
	\draw[V2,->] (-1.5, 1.5) -- ++( 3pt,-3pt);
	\draw[V2,->] ( 1.5,-1.5) -- ++(-3pt, 3pt);
	\draw[V1] (-1,1)
		.. controls (-0.6, 1.4) and ( 0.6, 1.4) .. ( 1, 1)
		.. controls ( 1.4, 0.6) and ( 1.4,-0.6) .. ( 1,-1)
		.. controls ( 0.6,-1.4) and (-0.6,-1.4) .. (-1,-1)
		.. controls (-1.4,-0.6) and (-1.4, 0.6) .. cycle;
	\draw[V1] (-2,2)
		.. controls (-1.3, 2.7) and ( 1.3, 2.7) .. ( 2, 2)
		.. controls ( 2.7, 1.3) and ( 2.7,-1.3) .. ( 2,-2)
		.. controls ( 1.3,-2.7) and (-1.3,-2.7) .. (-2,-2)
		.. controls (-2.7,-1.3) and (-2.7, 1.3) .. cycle;
	\draw[V0,->]
		( 0, 3.5) .. controls (-4, 3.5) and (-4.5, 3) .. (-4.5, 0)
		          .. controls (-4.5,-3) and (-4,-3.5) .. ( 0,-3.5)
					 .. controls ( 4,-3.5) and ( 4.5,-3) .. ( 4.5, 0)
					 .. controls ( 4.5, 3) and ( 4, 3.5) .. ( 0, 3.5);
	\draw[V0]
		( 2,-2) .. controls ( 4,-4) and ( 4, 4) .. ( 2, 2)
		(-2,-2) .. controls (-4,-4) and (-4, 4) .. (-2, 2)
		(-1, 1) .. controls (-0.2, 0.2) and ( 0.2, 0.2) .. ( 1, 1)
		(-1,-1) .. controls (-0.2,-0.2) and ( 0.2,-0.2) .. ( 1,-1);
\end{tikzpicture}%
	\caption{%
		Two completions of the~same web. The~surrounding dashed circle
		in the~right picture is required by the~condition that the~unbounded
		region is painted white.}%
	\label{fig:completed webs}%
\end{figure}
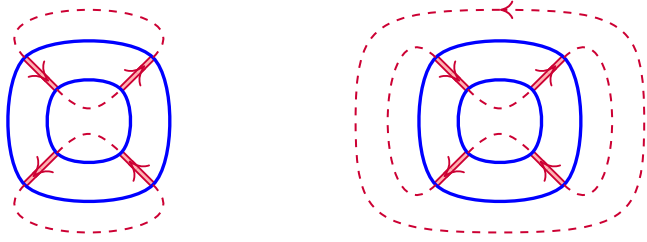
Furthermore, we allow to mark blue loops of $\web^+$ with (any number of) dots.
We assign to such a~planar diagram a~completed foam $\cupfoam(\web^+) =
\cupfoam_r(\web^+) \cup \cupfoam_b(\web^+)$ bounded by $\web^+\rics$
that satisfies the~following conditions:
\begin{enumerate}[label={(CF\arabic{*})},ref={CF\arabic{*}},leftmargin=4em]
	\item\label{mcf:red disks}
		$\cupfoam_r(\web^+) \subset \R^2\times[-1,0]$ and consists of disks
		that project injectively onto $\R^2\times\{0\}$,
	\item\label{mcf:blue disks}
		$\cupfoam_b(\web^+)$ is a~collection of disks such that
	   $\cupfoam_b(\web^+) \cap (\R^2 \times [-1,0]) = \web^+_b \times [-1,0]$,
%		and
	\item\label{mcf:dots}
		each blue disk is decorated with as many dots as its boundary loop
	   in $\web^+\rics$, all placed at heights smaller than $-1$ (hence,
		below all red facets).
\end{enumerate}
%The~intersection of red and blue disks in $\cupfoam(\web^+)$ is minimal
%among all completed cup foams bounded by $\web^+$ with red facets above
%minima of blue cups.
%---because of which
%we call $\cupfoam(\web^+)$ a~\emph{minimal completed cup foam}.
Painting the~unbounded region white extends to a~unique shading supported
by $\cupfoam(\web^+)$. The~resulting foam $\Gamma(\cupfoam(\web^+)) \in
\set{Foam}(\web)$ is a~red-over-blue cup foam. We call it the~\emph{cup
foam associated to $\web^+$}. The~following observation is an~immediate
consequence of Theorem~\ref{thm:basis-for-Foam(w)}.

\begin{lem}
	Choose a~completion $\web^+$ of $\web$ and consider the~family of all
	dotted completed webs obtained from $\web^+$ by placing at most one dot
	on each blue loop. Then the~corresponding cup foams form a~linear basis
	of $\set{Foam}(\web)$.
\end{lem}

Notice that dots in this pictures only mark loops. In particular,
moving a~dot along a~loop---even passing through a~crossing with a~red
strands---does not affect the~cup foam represented by the~diagram.

\begin{exam}
	Let $\web$ be a~blue circle. Then $\set{Foam}(\web)$ is generated by
	two blue cups: one with and the~other without a~dot. These are the~cup
	foams associated to $\web^+$ when $\web$ is oriented clockwise, because
	this orientation is oriented from the~unbounded region, hence $\web^+ = \web$.
	Otherwise, $\web^+$ is $\web$ surrounded by a~dashed red circle, which results
	in the~change of the~sign of the~cup with a~dot, 
	see Table~\ref{tab:basis for circles}. This is consistent with
	the~computation from Example~\ref{ex:forgetting red facets}.
\end{exam}

\begin{table}[ht]%
	\centering
	\def\arrayrowstretch{1.2}%
	\tikzset{x=2em,y=2em}%
	\begin{tabular}{ccl}
		\hline
		Web & Completion & Basis \\
		\hline
		$\begin{centertikz}
			\draw[V1,->]
				(0,0.4) arc[radius=0.4, start angle=90, end angle=-270] -- ++(3pt,0);
		\end{centertikz}$
		&
		$\begin{centertikz}
			\draw[V1] (0,0) circle[radius=0.4];
		\end{centertikz}$
		&
		$\begin{centertikz}
			\useasboundingbox (-0.8,-0.5) rectangle (0.8,0.5);
			\draw[V1] (0,0) circle[radius=0.4];
		\end{centertikz}
			\longleftrightarrow
		\foampict{1cup}[0]$,
		\\
		&&$\begin{centertikz}
			\useasboundingbox (-0.8,-0.5) rectangle (0.8,0.5);
			\draw[V1] (0,0) circle[radius=0.4];
			\webdrawdots 1(60:0.4);
		\end{centertikz}
			\longleftrightarrow
		\foampict{1cup}[1]$
		\\
		\hline
		$\begin{centertikz}
			\draw[V1,->]
				(0,0.4) arc[radius=0.4, start angle=90, end angle=450] -- ++(-3pt,0);
		\end{centertikz}$
		&
		$\begin{centertikz}
			\useasboundingbox (-0.8,-0.8) rectangle (0.8,0.8);
			\draw[V1] (0,0) circle[radius=0.4];
			\draw[V0,->] (0.7,0) arc[radius=0.7, start angle=0, end angle=360]
				-- ++(0,3pt);
		\end{centertikz}$
		&
		$\begin{centertikz}
			\useasboundingbox (-0.8,-0.8) rectangle (0.8,0.8);
			\draw[V1] (0,0) circle[radius=0.4];
			\draw[V0,->] (0.7,0) arc[radius=0.7, start angle=0, end angle=360]
				-- ++(0,3pt);
		\end{centertikz}
			\longleftrightarrow
		\foampict{1cup and disk}[>] = \phantom-\foampict[baseline=-0.4em]{1cup}$,
		\\
		&&$\begin{centertikz}
			\useasboundingbox (-0.8,-0.8) rectangle (0.8,0.8);
			\draw[V1] (0,0) circle[radius=0.4];
			\webdrawdots 1(60:0.4);
			\draw[V0,->] (0.7,0) arc[radius=0.7, start angle=0, end angle=360]
				-- ++(0,3pt);
		\end{centertikz}
			\longleftrightarrow
		\foampict{1cup and disk}[>1] = -\foampict[baseline=-0.4em]{1cup}[1]$
		\\
		\hline
	\end{tabular}
	\caption{%
		A~basis for foams bounded by a~circle, represented as planar diagrams
		and as foams.}
	\label{tab:basis for circles}
\end{table}

% =============================================================================
\subsection{Action of foams}

We now provide a~description of the~linear maps associated to foams
in terms of the~dotted completed webs. In fact, it is enough to analyze
the~following \emph{elementary completed foams}:
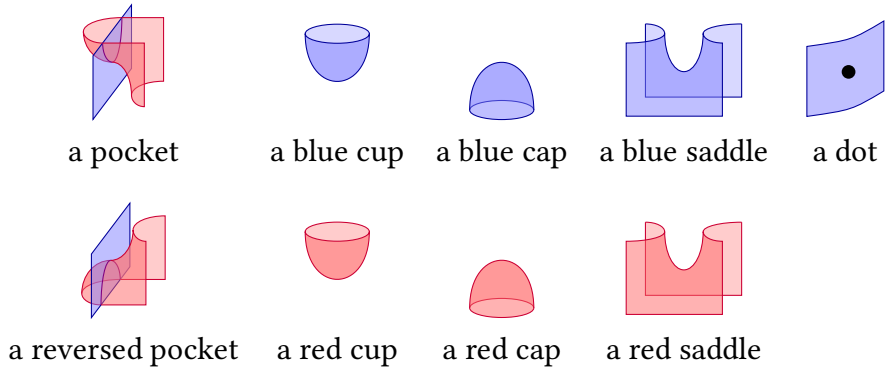
\begin{center}
	\tikzset{x=2em,y=2em}%
	\begin{tabular}{ccccc}
		\foampict{pocket above} &
		\foampict{cup} & \foampict{cap} & \foampict{saddle} & \foampict{dot} \\
		a~pocket & a~blue cup & a~blue cap & a~blue saddle & a~dot \\[2ex]
		\foampict{pocket below} &
		\foampict{cup}[2] & \foampict{cap}[2] & \foampict{saddle}[2] & \\
		a~reversed pocket & a~red cup & a~red cap & a~red saddle & 
	\end{tabular}
\end{center}
because every foam can be decomposed into these.

\subsubsection*{Pockets and bicolored isotopies}
A~bicolored isotopy is a~sequence of several bigon moves \eqref{rel:isot-webs},
which are realized by the~pocket foams. When applied to a~(completed)
cup foam, it results in a~collection of disks which may or may not be minimal,
see Figure~\ref{fig:bigon-moves-on-foams}.
\begin{figure}[ht]%
	\centering
	% Bigon moves applied to cup foams
%
\begin{tikzpicture}[x=0.08\linewidth,y=0.08\linewidth]

% a cup right to the plane
\begin{scope}[shift={(-3.5,0)}]
	\draw[1facetFront] (-0.4,-0.3) -- (0.4,0.3) -- (0.4,-0.5) -- (-0.4,-1.1) -- cycle;
	\fill[2facetBack] (0.2, 0)
		arc[x radius=0.7, y radius=0.15, start angle=180, end angle= 90]
		arc[x radius=0.2, y radius=0.75, start angle=360, end angle=270]
		arc[x radius=0.5, y radius=0.60, start angle=270, end angle=180];
	\fill[2facetFront] (0.2, 0)
		arc[x radius=0.3, y radius=0.15, start angle=180, end angle=270]
		arc[x radius=0.2, y radius=0.45, start angle=180, end angle=270]
		arc[x radius=0.5, y radius=0.60, start angle=270, end angle=180];
	\draw[2facetLine] (0.2, 0)
		arc[x radius=0.3, y radius=0.15, start angle=180, end angle=270]
		arc[x radius=0.2, y radius=0.45, start angle=180, end angle=270]
		arc[x radius=0.2, y radius=0.75, start angle=270, end angle=360]
		arc[x radius=0.7, y radius=0.15, start angle= 90, end angle=180]
		arc[x radius=0.5, y radius=0.60, start angle=180, end angle=270];
	
	\begin{scope}[shift={(0,-2)}]
		\clip (-0.5, 0.3) |- (0.3,-0.3) -- (1.1,0.3) -- cycle;
		\draw[BCedge+,fill=V2innerColor] (1,0.15) -- (0.9, 0.15)
			arc[x radius=0.65, y radius=0.15, start angle= 90, end angle=180]
			arc[x radius=0.25, y radius=0.15, start angle=180, end angle=270]
			-- ++(0.1,0);
		\draw[BCedge] (-0.6,-0.45) -- (0.6, 0.45);
	\end{scope}
\end{scope}

% a cup through the plane
\begin{scope}
	% red surface - left
	\begin{scope}
		\clip (0,0.005) arc[x radius=0.17, y radius=0.62, start angle=-90, end angle=0]
			|- (-0.6, 0.8) |- (0,-0.2);
		\fill[2facetBack]
			(0, 0) .. controls ++(-0.3,0.1) and ++(0,-0.4) .. (-0.5,0.5)
			arc[x radius=1.4, y radius=0.15, start angle=180, end angle= 90];
	\end{scope}
	\begin{scope}
		\clip (0,0.005)
			arc[x radius=0.152, y radius=0.38, start angle=270, end angle=180]
			-- (0.4,0.8) -| (-0.6, -0.2);
		\fill[2facetFront](0, 0) .. controls ++(-0.3,0.1) and ++(0,-0.4) .. (-0.5,0.5)
			arc[x radius=1.0, y radius=0.15, start angle=180, end angle=270];
		\draw[2facetLine] (0, 0) .. controls ++(-0.3,0.1) and ++(0,-0.4) .. (-0.5,0.5);
	\end{scope}
	% blue plane
	\draw[1facetFront] (-0.4, 0.2) -- (0.4,0.8) -- (0.4,-0.5) -- (-0.4,-1.1) -- cycle;
	% red surface - right
	\begin{scope}
		\clip (0,0.005) arc[x radius=0.17, y radius=0.62, start angle=-90, end angle=0]
			-- (1.0, 0.8) |- (0,-1);
		\fill[2facetBack](-0.5,0.5)
			arc[x radius=1.4, y radius=0.15, start angle=180, end angle= 90]
			arc[x radius=0.2, y radius=1.25, start angle=360, end angle=270]
			.. controls ++(-0.5,0) and ++(0.3,-0.1) .. (0, 0);
	\end{scope}
	\draw[seam]
		(0,0.005) arc[x radius=0.17, y radius=0.62, start angle=-90, end angle=0];
	\begin{scope}
		\clip (0,0.005)
			arc[x radius=0.152, y radius=0.38, start angle=270, end angle=180]
			-- (1.0, 0.8) |- (0,-1);
		\fill[2facetFront] (-0.5,0.5)
			arc[x radius=1.0, y radius=0.15, start angle=180, end angle=270]
			arc[x radius=0.2, y radius=0.95, start angle=180, end angle=270]
			.. controls ++(-0.5,0) and ++(0.3,-0.1) .. (0, 0) -- (-0.5,0);
	\end{scope}
	\draw[2facetLine] (0.7, -0.6)
		arc[x radius=0.2, y radius=0.95, start angle=270, end angle=180]
		arc[x radius=1.0, y radius=0.15, start angle=270, end angle=180]
		arc[x radius=1.4, y radius=0.15, start angle=180, end angle= 90]
		arc[x radius=0.2, y radius=1.25, start angle=360, end angle=270]
		.. controls ++(-0.5,0) and ++(0.3,-0.1) .. (0, 0) -- ++(-0.6pt, 0.2pt);
	\draw[seam]
		(0,0.005) arc[x radius=0.152, y radius=0.38, start angle=270, end angle=180];
	% to make top blue edge over red
	\draw[1facetLine] (-0.2,0.35) -- (0.2,0.65);
	
	\begin{scope}[shift={(0,-2)}]
		\clip (-0.6, 0.3) |- (0.3,-0.3) -- (1.1,0.3) -- cycle;
		\draw[BCedge+,fill=V2innerColor] (1,0.15) -- (0.9, 0.15)
			arc[x radius=1.4, y radius=0.15, start angle= 90, end angle=180]
			arc[x radius=1.0, y radius=0.15, start angle=180, end angle=270]
			-- ++(0.1,0);
		\draw[BCedge] (-0.6,-0.45) -- (0.6, 0.45);
	\end{scope}
\end{scope}

% not a cup foam
\begin{scope}[shift={(3.5,0)}]
	% red - left cup
	\fill[2facetBack] (0,-0.2)
		.. controls ++(-0.3,0.1) and ++(0,-0.3) .. (-0.5,0.2)
		.. controls ++(0,0.3) and ++(-0.3,-0.1) .. ( 0,0.6)
		(0,0.595) arc[x radius=0.15, y radius=0.395, start angle=90, end angle=-90];
	\fill[2facetFront] (0,-0.2)
		.. controls ++(-0.3,0.1) and ++(0,-0.3) .. (-0.5,0.2)
		.. controls ++(0,0.3) and ++(-0.3,-0.1) .. ( 0,0.6) --
		(0,0.595) arc[x radius=0.15, y radius=0.395, start angle=90, end angle=270];
	\draw[2facetLine] (0,-0.2)
		.. controls ++(-0.3,0.1) and ++(0,-0.3) .. (-0.5,0.2)
		.. controls ++(0,0.3) and ++(-0.3,-0.1) .. ( 0,0.6);
	% blue plane
	\draw[1facetFront] (-0.4, 0.7) -- (0.4,1.3) -- (0.4,-0.5) -- (-0.4,-1.1) -- cycle;
	% red - right piece
	\fill[2facetBack] (0,0.6)
		.. controls ++(0.1,0.033) and ++(0,-0.4) .. (0.2, 1)
		arc[x radius=0.7, y radius=0.15, start angle=180, end angle= 90]
		arc[x radius=0.2, y radius=1.75, start angle=360, end angle=270]
		.. controls ++(-0.5,0) and ++(0.3,-0.1) .. (0,-0.2) -- (0,-0.195)
		arc[x radius=0.15, y radius=0.395, start angle=-90, end angle=90];
	\draw[seam] (0,-0.195)
		arc[x radius=0.15, y radius=0.395, start angle=-90, end angle=90];
	\fill[2facetFront] (0,0.6)
		.. controls ++(0.1,0.033) and ++(0,-0.4) .. (0.2, 1)
		arc[x radius=0.3, y radius=0.15, start angle=180, end angle=270]
		arc[x radius=0.2, y radius=1.45, start angle=180, end angle=270]
		.. controls ++(-0.5,0) and ++(0.3,-0.1) .. (0,-0.2) -- (0,-0.195)
		arc[x radius=0.15, y radius=0.395, start angle=270, end angle=90];
	\draw[seam]
		(0,-0.195) arc[x radius=0.15, y radius=0.395, start angle=270, end angle=90];
	\draw[2facetLine] (0,0.6) ++(-0.6pt,-0.2pt) --
		(0,0.6) .. controls ++(0.1,0.033) and ++(0,-0.4) .. (0.2, 1)
		arc[x radius=0.7, y radius=0.15, start angle=180, end angle= 90]
		arc[x radius=0.2, y radius=1.75, start angle=360, end angle=270]
		.. controls ++(-0.5,0) and ++(0.3,-0.1) .. (0,-0.2) -- ++(-0.6pt,0.2pt)
		(0.2, 1)
		arc[x radius=0.3, y radius=0.15, start angle=180, end angle=270]
		arc[x radius=0.2, y radius=1.45, start angle=180, end angle=270];

	\begin{scope}[shift={(0,-2)}]
		\clip (-0.6, 0.3) |- (0.3,-0.3) -- (1.1,0.3) -- cycle;
		\fill[black!20] (0.9, 0.16)
			arc[x radius=1.4, y radius=0.16, start angle= 90, end angle=180]
			arc[x radius=1.0, y radius=0.16, start angle=180, end angle=270];
		\draw[BCedge+,fill=V2innerColor] (1,0.15) -- (0.9, 0.15)
			arc[x radius=0.65, y radius=0.15, start angle= 90, end angle=180]
			arc[x radius=0.25, y radius=0.15, start angle=180, end angle=270]
			-- ++(0.1,0);
		\draw[BCedge] (-0.6,-0.45) -- (0.6, 0.45);
	\end{scope}
\end{scope}

% arrows
\draw[->] (-2.2,-0.1) -- (-0.8,-0.1)
	node[midway,above] {\small push}
	node[midway,below] {\small leftwards};

\draw[->] (1.3,-0.1) -- (2.7,-0.1)
	node[midway,above] {\small pull}
	node[midway,below] {\small rightwards};

\end{tikzpicture}%
	\caption{%
		A~cup foam with its projection on the~horizontal plane (the~left column)
		and the~results of applying the~bigon move twice (the~middle and right
		columns). The~middle foam is again a~cup foam, but not the~right one:
		the~projection has double points---the~shaded region outside of the cup---%
		coming from the~disk to the~side of the~vertical plane.
	}%
	\label{fig:bigon-moves-on-foams}%
\end{figure}
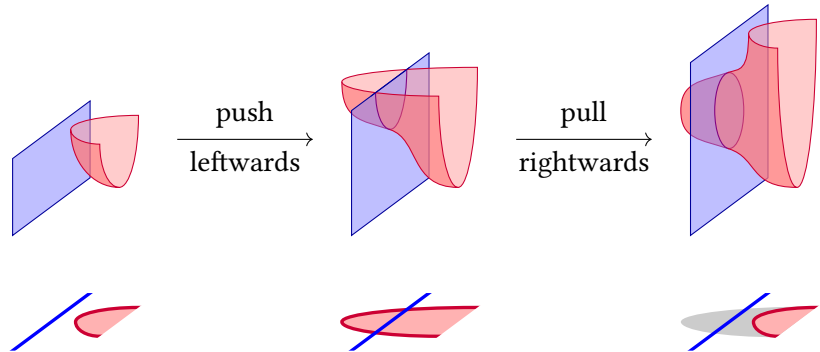
The~resulting cup foam is minimal whenever the~projection of a~red disk is pushed
through a~blue arc (this creates no double points in the~projection), so that
the~associated map takes a~dotted web to the~result of applying the~bigon move:
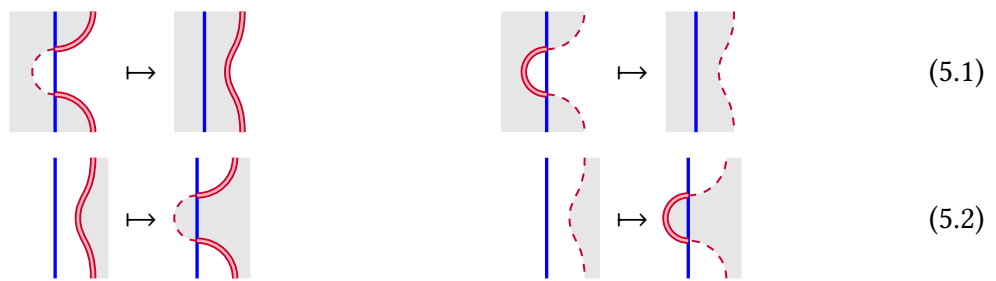
\begin{align}
\label{move:ext bigon->no bigon}
	\webpict{ext bigon}[V0,V2,] &\mapsto\webpict{ext cup}[V2,]   &
	\webpict{ext bigon}[V2,V0,] &\mapsto\webpict{ext cup}[V0,]   \\
\label{move:no bigon->inn bigon}
	\webpict{inn cup}[V2,]      &\mapsto\webpict{inn bigon}[V0,V2,] &
	\webpict{inn cup}[V0,]      &\mapsto\webpict{inn bigon}[V2,V0,]
\end{align}
The~shaded regions are the~projections of the~red disks in the~corresponding
completed cup foams.
However, pulling the~projection of a~red disk off a~blue arc creates double points,
like in the~right column of Figure~\ref{fig:bigon-moves-on-foams}. Indeed, the~new
red disk intersects the~blue surface in a~circle, so that either of
\eqref{rel:isot-foams} has to be applied. This may cost a~sign, depending on
the~orientation of the~edges:
\begin{align}
\label{move:inn bigon->no bigon,+}
	\webpict{inn bigon}[V0,V2,+] &\mapsto\phantom-\webpict{inn cup}[V2,+] &
	\webpict{inn bigon}[V2,V0,-] &\mapsto\phantom-\webpict{inn cup}[V0,-] \\
\label{move:inn bigon->no bigon,-}
	\webpict{inn bigon}[V0,V2,-] &\mapsto        -\webpict{inn cup}[V2,-] &
	\webpict{inn bigon}[V2,V0,+] &\mapsto        -\webpict{inn cup}[V0,+] \\
\label{move:no bigon->ext bigon,+}
	\webpict{ext cup}[V2,+]     &\mapsto\phantom-\webpict{ext bigon}[V0,V2,+] &
	\webpict{ext cup}[V0,-]     &\mapsto\phantom-\webpict{ext bigon}[V2,V0,-] \\
\label{move:no bigon->ext bigon,-}
	\webpict{ext cup}[V2,-]     &\mapsto        -\webpict{ext bigon}[V0,V2,-] &
	\webpict{ext cup}[V0,+]     &\mapsto        -\webpict{ext bigon}[V2,V0,+]
\end{align}
Indeed, the~left moves in \eqref{move:inn bigon->no bigon,+} and
\eqref{move:inn bigon->no bigon,-} are realized by detaching red
cylinders with \eqref{rel:detach-cylinder} whereas the~right ones---by
eliminating red caps with \eqref{rel:red-cap-vs-plane}. Likewise,
the~relations \eqref{rel:stripe-vs-plane} and \eqref{rel:detach-saddle}
give the~signs for the~left and right sides respectively of both
\eqref{move:no bigon->ext bigon,+} and \eqref{move:no bigon->ext bigon,+}.

\subsubsection*{Placing a~dot}
Placing a~dot on $\cupfoam(\web^+)$ near the~boundary violates \eqref{mcf:dots}.
To obtains a~minimal cup foam, the~dot has to be moved down.

Let $p$ be the~projection of the~dot onto the~horizontal plane and assume that
it does not lie on a~red loop. We define the~\emph{nestedness} $n(p)$ as the~number
of red loops encircling $p$. It counts red facets below $p$ in the~cup foam, hence,
the~number of times the~dot-moving relation \eqref{rel:dots} has to be applied to
move the~dot from top to the~bottom of a~blue disk. Therefore, placing a~dot on
a~blue loop results in the~following map:
\begin{equation}
\label{move:place a dot}
	\webpict{line}[V10] \mapsto \left\{\begin{array}{ll}
		\phantom h\webpict{line}[V11] & \text{ if }n(p)\text{ is even},\\
		h\webpict{line}[V10] - \webpict{line}[V11] & \text{ if }n(p)\text{ is odd}.
	\end{array}\right.
\end{equation}

\subsubsection*{Blue cups, caps, and saddles}
Suppose now that $\dblcob$ is a~completed foam with $\web^+$ at its bottom
and a~unique critical point that lies on the~blue surface.
In this case $\dblcob \cup \cupfoam(\web^+)$ is no longer a~cup foam associated
with the~output of $\dblcob$: to have one, the~critical point of $\dblcob$ has
to be slid downwards, below all red facets, and this may cost a~sign. Moreover,
a~cap creates a~sphere that has to be evaluated, whereas a~split creates a~neck
that has to be cut.

Let $p$ be the~projection of the~critical point onto the~horizontal plane and
assume that $p\notin \web^+$. We say that a~red loop $\gamma$ encircling $p$
is \emph{evenly distanced} if any generic path connecting $p$ to a~point $q$
from a~solid (resp.\ dashed) red arc of $\gamma$ intersects blue circles in
an~even (resp.\ odd) number of points. Otherwise, $\gamma$ is \emph{oddly distanced}.
Let $s(p)$ count \emph{oddly distanced} anti-clockwise and \emph{evenly distanced}
clockwise red loops surrounding $p$. This corresponds to two types of red facets
below $p$: downwards oriented ones that survive in the~cup foam $\Gamma(\web^+)$
and upwards oriented ones that are removed. These are exactly those situations,
in which there is a~sign in relations \eqref{rel:detach-cylinder} and
\eqref{rel:detach-saddle}.
Hence $s(p)$ determines the~result of isotoping the~blue critical point below
all red facets. Therefore, the~maps induced by critical blue points are the~usual
ones scaled by $(-1)^{s(p)}$:
\begin{align}
\label{move:blue cup}
	&\text{a~cup:} &
		\emptyset &\mapsto (-1)^{s(p)}\webpict{circle}[V1] \\
\label{move:blue cap}
	&\text{a~cap:} &
		\webpict{circle}[V1] &\mapsto (-1)^{s(p)} \\
\label{move:blue merge}
	&\text{a~merge:} &
		\webpict{circles}[V1] &\mapsto (-1)^{s(p)}\webpict{merged}[V1] \\
\label{move:blue split}
	&\text{a~split:} &
		\webpict{merged}[V1] &\mapsto (-1)^{s(p)}\left(
				\webpict{circles}[V110] + \webpict{circles}[V101]
		\right).
\end{align}

\subsubsection*{Red cups, caps, and saddles}
Placing a~red cup at the~top of $\cupfoam(\web^+)$ results in a~cup foam.
Hence, no sign appears. Conversely, capping off an~isolated red circle creates
a~red sphere, which can be removed by \eqref{rel:sphere-eval} at a cost of sign.
Hence, we obtain the~following maps:
\begin{align}
\label{move:red birth}
	&\text{a~cup:} & 
	\emptyset &\mapsto \webpict{circle}[V2] &
	\emptyset &\mapsto \webpict{circle}[V0] \\
\label{move:red death}
	&\text{a~cap:} &
	\webpict{circle}[V2] &\mapsto -1 &
	\webpict{circle}[V0] &\mapsto  1
\end{align}
The~behavior of merges and splits depends on whether the~two red circles (those
being merged or the~result of a~split) are nested or not. In the~latter case,
a~merge takes a~minimal cup foam to a~minimal cup foam, whereas splitting a~red
circle creates a~neck that has to be cut with \eqref{rel:neck-cutting}
if it survives in the~foam. Therefore, the~correspondings maps are
\begin{align}
\label{move:red merge}
	&\text{a~merge:} &
	\webpict{circles}[V2] &\mapsto \phantom-\webpict{merged} [V2] &
	\webpict{circles}[V0] &\mapsto  \webpict{merged} [V0] \\
\label{move:red split}
	&\text{a~split:} &
	\webpict{merged} [V2] &\mapsto -\webpict{circles}[V2] &
	\webpict{merged} [V0] &\mapsto \webpict{circles}[V0]
\end{align}
Merging nested red loops results in a~red disk, which does not project injectively
onto the~horizontal plane: to fix this, the~disk has to be isotoped into a~croissant:
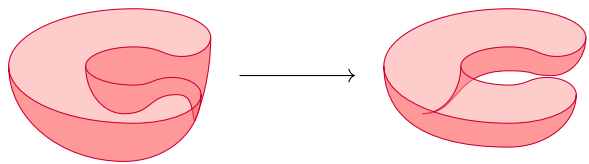
\begin{center}
	\begin{tikzpicture}[x=3em,y=3em]
\begin{scope}[shift={(1.7,0)}]
	\fill[2facetBack] (1.3,0.3)
		arc[x radius=0.35, y radius=0.4, start angle=360, end angle=270]
		.. controls (0.7, -0.1) and (0.8,-0.05) .. (0.5,-0.05)
		.. controls (-0.1,-0.05) and (0,-0.5) .. (-0.4, -0.5)
		arc[x radius=0.4, y radius=0.5, start angle=270, end angle=180]
		arc[x radius=1.3, y radius=0.6, start angle=180, end angle= 90]
		arc[x radius=0.8, y radius=0.3, start angle= 90, end angle=  0];
	\fill[2facetFront] (1.3, 0.3)
		arc[x radius=0.35, y radius=0.4, start angle=360, end angle=270]
		.. controls (0.7, -0.1) and (0.8,-0.05) .. (0.5,-0.05)
		.. controls (-0.1,-0.05) and (0,-0.5) .. (-0.4, -0.5)
		arc[x radius=0.4, y radius=0.5, start angle=270, end angle=360]
		arc[x radius=0.5, y radius=0.2, start angle=180, end angle= 90] % in tl
		.. controls (0.8,0.2) and (0.8,0.1) .. (1.0,0.1)                % in tr
		arc[x radius=0.3, y radius=0.2, start angle=270, end angle=360]; % out tr
	\draw[2facetLine] (1.3, 0.3)
		arc[x radius=0.35, y radius=0.4, start angle=360, end angle=270]
		.. controls (0.7, -0.1) and (0.8,-0.05) .. (0.5,-0.05)
		.. controls (-0.1,-0.05) and (0,-0.5) .. (-0.4, -0.5);
	\fill[2facetBack] (-0.8,0)
		.. controls (-0.8,-0.45) and (-0.5,-0.85) .. (0.5,-0.85)
		arc[x radius=0.7, y radius=0.55, start angle=270, end angle=360]
		arc[x radius=0.3, y radius=0.18, start angle=  0, end angle= 90] % out br
		.. controls (0.7,-0.12) and (0.7,-0.2) .. (0.5,-0.2)
		arc[x radius=0.5, y radius=0.2, start angle=270, end angle=180]
		arc[x radius=0.4, y radius=0.5, start angle=360, end angle=180];
	\draw[2facetLine] (0,0)
		arc[x radius=0.4, y radius=0.5, start angle=360, end angle=270];
	\fill[2facetFront] (1.2,-0.3)
		arc[x radius=0.7, y radius=0.55, start angle=360, end angle=270]
		.. controls (-0.5,-0.85) and (-0.8, -0.45) .. (-0.8, 0)
		arc[x radius=1.3, y radius=0.6, start angle=180, end angle=270] % out bl
		arc[x radius=0.7, y radius=0.3, start angle=270, end angle=360]; % out br
	\draw[2facetLine] (0,0)
		arc[x radius=0.5, y radius=0.2, start angle=180, end angle= 90] % in tl
		.. controls (0.8,0.2) and (0.8,0.1) .. (1.0,0.1)                % in tr
		arc[x radius=0.3, y radius=0.2, start angle=270, end angle=360] % out tr
		arc[x radius=0.8, y radius=0.3, start angle=  0, end angle= 90] % out tr
		arc[x radius=1.3, y radius=0.6, start angle= 90, end angle=180] % out tl
		arc[x radius=1.3, y radius=0.6, start angle=180, end angle=270] % out bl
		arc[x radius=0.7, y radius=0.3, start angle=270, end angle=360] % out br
		arc[x radius=0.3, y radius=0.18, start angle=  0, end angle= 90] % out br
		.. controls (0.7,-0.12) and (0.7,-0.2) .. (0.5,-0.2)
		arc[x radius=0.5, y radius=0.2, start angle=270, end angle=180] % in bl
		(1.2,-0.3)
		arc[x radius=0.7, y radius=0.55, start angle=360, end angle=270]
		.. controls (-0.5,-0.85) and (-0.8, -0.45) .. (-0.8, 0);
\end{scope}

\begin{scope}[shift={(-2.2,0)}]
	\fill[2facetBack] (1.3, 0.3)
		arc[x radius=0.8, y radius=0.3, start angle=  0, end angle= 90]
		arc[x radius=1.3, y radius=0.6, start angle= 90, end angle=180]
		arc[x radius=1.2, y radius=1.0, start angle=180, end angle=270]
		arc[x radius=0.8, y radius=0.7, start angle=270, end angle=360]
		-- (1.13,-0.58) .. controls (1.2,-0.42) and (1.3,-0.1) .. (1.3,0.3);
	\fill[2facetFront] (0, 0)
		arc[x radius=0.4, y radius=0.5, start angle=180, end angle=270]
		.. controls (0.7,-0.5) and (0.65,-0.3) .. (0.95, -0.3)
		.. controls (1.1,-0.3) and (1.1,-0.42) .. (1.13,-0.58)
		.. controls (1.2,-0.42) and (1.3,-0.1) .. (1.3,0.3)
		arc[x radius=0.3, y radius=0.2, start angle=360, end angle=270]
		.. controls (0.8,0.1) and (0.8,0.2) .. (0.5,0.2)
		arc[x radius=0.5, y radius=0.2, start angle=90, end angle=180];
	\draw[2facetLine]
		(1.3, 0.3) .. controls (1.3,-0.1) and (1.2,-0.42) .. (1.13,-0.58);
	\fill[2facetBack] (0,0)
		arc[x radius=0.4, y radius=0.5, start angle=180, end angle=270]
		.. controls (0.7,-0.5) and (0.65,-0.3) .. (0.95, -0.3)
		.. controls (1.1,-0.3) and (1.1,-0.42) .. (1.13,-0.58) -- (1.2,-0.3)
		arc[x radius=0.3, y radius=0.18, start angle=  0, end angle= 90]
		.. controls (0.7,-0.12) and (0.7,-0.2) .. (0.5,-0.2)
		arc[x radius=0.5, y radius=0.2, start angle=270, end angle=180];
	\draw[2facetLine] (0, 0)
		arc[x radius=0.4, y radius=0.5, start angle=180, end angle=270]
		.. controls (0.7,-0.5) and (0.65,-0.3) .. (0.95, -0.3)
		.. controls (1.1,-0.3) and (1.1,-0.42) .. (1.13,-0.58);
	\fill[2facetFront] (-0.8, 0)
		arc[x radius=1.2, y radius=1.0, start angle=180, end angle=270]
		arc[x radius=0.8, y radius=0.7, start angle=270, end angle=360]
		arc[x radius=0.7, y radius=0.3, start angle=360, end angle=270] % out br
		arc[x radius=1.3, y radius=0.6, start angle=270, end angle=180]; % out bl
	\draw[2facetLine] (0,0)
		arc[x radius=0.5, y radius=0.2, start angle=180, end angle= 90] % in tl
		.. controls (0.8,0.2) and (0.8,0.1) .. (1.0,0.1)                % in tr
		arc[x radius=0.3, y radius=0.2, start angle=270, end angle=360] % out tr
		arc[x radius=0.8, y radius=0.3, start angle=  0, end angle= 90] % out tr
		arc[x radius=1.3, y radius=0.6, start angle= 90, end angle=180] % out tl
		arc[x radius=1.3, y radius=0.6, start angle=180, end angle=270] % out bl
		arc[x radius=0.7, y radius=0.3, start angle=270, end angle=360] % out br
		arc[x radius=0.3, y radius=0.18, start angle=  0, end angle= 90] % out br
		.. controls (0.7,-0.12) and (0.7,-0.2) .. (0.5,-0.2)
		arc[x radius=0.5, y radius=0.2, start angle=270, end angle=180] % in bl
		(-0.8, 0)
		arc[x radius=1.2, y radius=1.0, start angle=180, end angle=270]
		arc[x radius=0.8, y radius=0.7, start angle=270, end angle=360];
\end{scope}

\draw[->] (-0.6,-0.1) -- (0.6,-0.1);
%	node[midway,above] {\small finger}
%	node[midway,below] {\small move};

\end{tikzpicture}
\end{center}
This isotopy can be described as a~\emph{finger move}: place your finger vertically
near the~saddle and move it inwards, pushing the~red disk. The~disk is then isotoped
through any blue facet attached to a~blue arc that cuts the~inner red circle or
a~cylinder attached to a~blue circle surrounded by the~inner red circle:
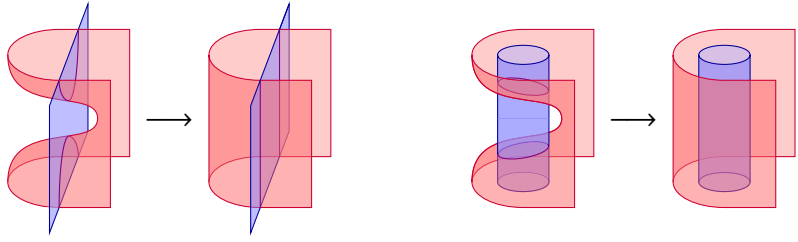
\[
	\tikzset{x=2em,y=2em}
	\foampict{cutting plane cut} \longrightarrow \foampict{cutting plane}
\hskip 4em
	\foampict{inner cylinder cut} \longrightarrow \foampict{inner cylinder}
\]
Depending on which red facets survive, each move represents two relations
between foams. We leave it to the~reader to check that the~foams involved
in the~right move are always equal, whereas the~left move costs a~sign only
in the~following two configurations:
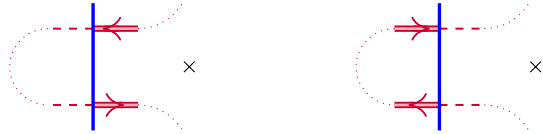
\begin{center}
	\tikzset{x=2em,y=2em}%
	\begin{tikzpicture}
		\draw[V0] (0, 0.6) -- (-0.7, 0.6)
		          (0,-0.6) -- (-0.7,-0.6);
		\draw[V2] (0, 0.6) -- ( 0.7, 0.6)
		          (0,-0.6) -- ( 0.7,-0.6);
		\draw[V2,->] (0.4,-0.6) -- ++( 2pt,0);
		\draw[V2,->] (0.4, 0.6) -- ++(-6pt,0);
		\draw[V0lineColor,dotted]
			(-0.7, 0.6) arc[radius=0.6, start angle=90, end angle=270]
			( 0.7, 0.6) arc[radius=0.8, start angle=270, end angle=330]
			( 0.7,-0.6) arc[radius=0.8, start angle= 90, end angle= 30];
		\draw[V1] (0,1) -- (0,-1);
		\draw[shift={(1.5,0)}] (-2pt,-2pt) -- (2pt,2pt) (-2pt,2pt) -- (2pt,-2pt);
	\end{tikzpicture}
	\hskip 2cm
	\begin{tikzpicture}
		\draw[V2] (0, 0.6) -- (-0.7, 0.6)
		          (0,-0.6) -- (-0.7,-0.6);
		\draw[V2,->] (-0.4,-0.6) -- ++(-2pt,0);
		\draw[V2,->] (-0.4, 0.6) -- ++( 6pt,0);
		\draw[V0] (0, 0.6) -- ( 0.7, 0.6)
		          (0,-0.6) -- ( 0.7,-0.6);
		\draw[V2lineColor,dotted]
			(-0.7, 0.6) arc[radius=0.6, start angle=90, end angle=270]
			( 0.7, 0.6) arc[radius=0.8, start angle=270, end angle=330]
			( 0.7,-0.6) arc[radius=0.8, start angle= 90, end angle= 30];
		\draw[V1] (0,1) -- (0,-1);
		\draw[shift={(1.5,0)}] (-2pt,-2pt) -- (2pt,2pt) (-2pt,2pt) -- (2pt,-2pt);
	\end{tikzpicture}
\end{center}
where the~position of a~saddle is marked with a~cross.
Let $c$ be the~number of such configurations.
Keeping in mind that a~neck has to be cut in
case of a~split, we end up with the~following formulas:
\begin{align}
\label{move:inner merge}
	\webpict{nested}[V2] &\mapsto \phantom-(-1)^c\webpict{croissant}[V2] &
	\webpict{nested}[V0] &\mapsto (-1)^c\webpict{croissant}[V0] \\
\label{move:inner split}
	\webpict{croissant}[V2] &\mapsto -(-1)^{c}\webpict{nested}[V2] &
	\webpict{croissant}[V0] &\mapsto (-1)^{c}\webpict{nested}[V0]
\end{align}

\subsubsection*{Other common foams}
There are other moves of interest, such as blue saddles with vertical
red facets (in particular, zips and unzips) or red cups and caps that
intersect vertical blue facets. All can be represented as compositions
of those described above. For instance, a~zip is isotopic to a~pocket
move followed by a~saddle:
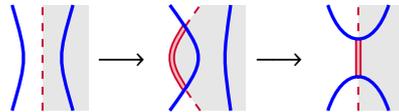
\[
	\webpict{blue-red-blue} \longrightarrow
	\webpict{blue-bigon}    \longrightarrow
	\webpict{zipped}
\]
As before, the~shaded regions represent a~projection of a~red disk,
and it is clear that the~first move takes a~basic cup foam to a~basic
cup foam, so that signs are governed by the~second move. Therefore,
the~map induced by a~zip is one of the~following two
\begin{align}
\label{move:merge zip}
	&\text{a~merging zip:} &
		\webpict{unjoined circles} &\mapsto (-1)^{s^+(p)}\webpict{joined merged} \\
\label{move:split zip}
	&\text{a~splitting zip:} &
		\webpict{unjoined merged} &\mapsto (-1)^{s^+(p)}\left(
				\webpict{joined circles}[10] + \webpict{joined circles}[01]
		\right),
\end{align}
where $s^+(p)$ is computed like $s(p)$, except that we take into account
the~loop passing through the~created red edge if it is oriented anticlockwise.

In the~unzip the~saddle precedes the~pocket and to ensure that the~latter
does not affect the~sign, we perform the~saddle to the~side of the~red
disk attached to the~red edge:
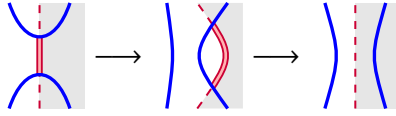
\[
	\webpict{zipped}     \longrightarrow
	\webpict{bigon-blue} \longrightarrow
	\webpict{blue-red-blue}
\]
Therefore, the~induced map is one of the~following two:
\begin{align}
\label{move:merge unzip}
	&\text{a~merging unzip:} &
		\webpict{joined circles} &\mapsto (-1)^{s^-(p)}\webpict{unjoined merged} \\
\label{move:split unzip}
	&\text{a~splitting unzip:} &
		\webpict{joined merged} &\mapsto (-1)^{s^-(p)}\left(
				\webpict{unjoined circles}[10] + \webpict{unjoined circles}[01]
		\right),
\end{align}
where again $s^-(p)$ is computed like $s(p)$ without counting
the~loop passing through the~removed red edge.

\section{The~Blanchet--Khovanov invariant of tangles with balanced boundaries}

Let $\cat{Foam}^\circ$ be the~subbicategory of $\cat{Foam}$ generated
by balanced sequences. In what follows we construct a~TQFT functor
$\Fweb^\circ\colon \cat{Foam}^\circ \to \cat{Bimod}$. If $T$ is
an~oriented tangle with balanced input and output collections of
points, then its resolutions are in $\cat{Foam}^\circ\rics$, so that
applying $\Fweb^\circ$ to $\wKhBracket{T}$ results in a~chain complex
of bimodules. We then show that this chain complex is isomorphic to
the~Khovanov's tangle invariant \cite{KhArcAlgebras}, but it admits
a~strictly functorial action of tangle cobordisms.

\subsection{A~linear basis of webs}

A~web $\web \subset \R\times(-\infty,0]$ is called a~\emph{cup web}
if its underlying tangle is a~cup diagram, i.e.\ a~collection of disjoint
arcs. All cup webs with the~same boundary any extending the~same cup
diagram coincide in $\set{Web}$ due to Proposition~\ref{prop:webs-red}
(and are isomorphic as objects of $\cat{Foam}$). Moreover, choosing a~cup
web for each cup diagram results in a~basis of the~space of webs with
given boundary, which we call a~\emph{cup basis}.

We describe now a~particular nice cup basis of webs bounded by a~balanced
$\bdry$ (see also Figure~\ref{fig:cup-web}).
Let $n$ be half of the~number of blue points in $\bdry$ (being balanced, $\bdry$
has an~even number of blue points). Proposition~\ref{prop:blue->binded} provides
an~invertible web $\EqSymb_\bdry\colon \bdry_b \to \bdry$ with $2n$ vertical
lines as blue edges. To obtain a~cup basis, attach cup diagrams to the~bottom
of $\EqSymb_\bdry$. In other words, the~basis is the~image of cup diagrams
under the~equivalence $\EqFunc_\bdry$ from Section~\ref{sec:equivalences}.
We call it the~\emph{red-over-blue basis of type} $\EqSymb_\bdry$,
because all red edges in the~webs appear above minima of blue cups.

\begin{figure}[ht]%
	\centering
	% A series of pictures explaining the construction of basic cup webs.
%
\begin{tikzpicture}[x=0.0625\columnwidth,y=0.0625\columnwidth,
	shadow/.style={color=black!10}]
	% BEGIN shortcuts
	\makeatletter
	\def\points{%
		\draw[thin] (-1,0) -- (5.5,0);
		\foreach \x/\sign in {0/-1, 1/+2, 2/+1, 3/-1, 4/+1, 5/-2} do
			\expandafter\webendpoint\sign(\x,0);
	}%
	\def\fullpoints{%
		\draw[thin] (-1,0) -- (5.5,0);
		\foreach \x/\sign in {-0.5/-2, 0/-1, 1/+2, 1.5/+2, 2/+1, 3/-1, 4/+1, 5/-2} do
			\expandafter\webendpoint\sign(\x,0);
	}%
	\def\caption{\@ifnextchar[\captionAt{\captionAt[-2.25]}}
	\def\captionAt[#1]{%
		\node[anchor=north, font=\small, align=center] at (2.5, #1)
	}
	\makeatother
	% END shortcuts
	
	% initial diagram
	\begin{scope}[shift={(-4,0)}]
		\points
		\caption[-0.5] {Step 0: the~initial collection of points.};
	\end{scope}		
	
	% Step 1
	\begin{scope}[shift={(4,0)}]
		\fullpoints
		\caption[-0.5] {Step 1: additional red points inserted.};
	\end{scope}		

	% Step 2
	\begin{scope}[shift={(-4,-2)}]
		\fill[BCshade]
			    (0,0) rectangle (2,-2)
				 (3,0) rectangle (4,-2)
		    (-0.5,0) arc[start angle=180, end angle=360, radius=0.75]
		    ( 1.5,0) arc[start angle=180, end angle=360, y radius=1.25, x radius=1.75];
		\begin{scope}
			\clip (0,0) rectangle (2,-2);
			\fill[white] (-0.5,0) arc[start angle=180, end angle=360, radius=0.75];
		\end{scope}
		\begin{scope}
			\clip (1,1pt) -- (1,-2) -- (2,-2) -- (2,0) -- (3,0) -- (3,-2) -- (4,-2) -- (4,1pt) -- cycle;
			\fill[white] ( 1.5,0) arc[start angle=180, end angle=360, y radius=1.25, x radius=1.75];
		\end{scope}
		%
		% add orienting arrows
		%
		\foreach \x in {0,2,3,4} do
			\draw[BCedge] (\x,0) -- (\x,-2);
		\draw[BCedge+,->-=0.6] (-0.5,0) arc[start angle=180, end angle=360, radius=0.75];
		\draw[BCedge+,-<-=0.6] ( 1.5,0) arc[start angle=180, end angle=360, y radius=1.25, x radius=1.75];
		\fullpoints
		\caption {Step 2: a~bicolored cut with a~shading.};
	\end{scope}		

	% Step 3
	\begin{scope}[shift={(4,-2)}]
		\draw[V0]
			(-0.5,0) arc[start angle=180, end angle=360, radius=0.75]
			( 1.5,0) arc[start angle=180, end angle=360, y radius=1.25, x radius=1.75];
		\begin{scope}
			\clip (0,1pt) -- (0,-2) -- (1.2,-2) -- (1.2,0) --
			      (2,0  ) -- (2,-2) -- (3  ,-2) -- (3  ,0) --
			      (4,0  ) -- (4,-2) -- (5.5,-2) -- (5.5,1pt) -- cycle;
			\draw[V2,->-=0.7] (-0.5,0)
				arc[start angle=180, end angle=360, radius=0.75];
			\draw[V2,-<-=0.35,-<-=0.85] ( 1.5,0)
				arc[start angle=180, end angle=360, y radius=1.25, x radius=1.75];
		\end{scope}
		\foreach \x in {0,2,3,4} do
			\draw[V1] (\x,0) -- (\x,-2);
		\points
		\caption {Step 3: applying $\Gamma$ produces the~web $\EqSymb_\bdry$.};
	\end{scope}		

	% Step 4
	\begin{scope}[shift={(-4,-6)}]
		\draw[V0]
			(-0.5,0) arc[start angle=180, end angle=360, y radius=1, x radius=0.75]
			( 1.5,0) arc[start angle=180, end angle=360, y radius=0.8, x radius=1.75];
		\begin{scope}
			\clip (0,0) arc[x radius=2, y radius=2, start angle=180, end angle=270] |- (0,0);
			\draw[V2,-<-=0.5] (1,0) arc[start angle=0, end angle=-90, y radius=1, x radius=0.75];
		\end{scope}
		\begin{scope}
			\clip (2,0) arc[x radius=0.5, y radius=1.25, start angle=180, end angle=360] -- cycle;
			\draw[V2,-<-=0.7]
				( 1.5,0) arc[start angle=180, end angle=270, y radius=0.8, x radius=1.75];
		\end{scope}
		\begin{scope}
			\clip (4,0) arc[x radius=2, y radius=2, start angle=0, end angle=-90] -| (5.5,0);
			\draw[V2,->-=0.5] ( 5,0) arc[start angle=0, end angle=-90, y radius=0.8, x radius=1.75];
		\end{scope}			
		\draw[V1]
			(0,0) arc[x radius=2, y radius=2, start angle=180, end angle=360]
			(2,0) arc[x radius=0.5, y radius=1.25, start angle=180, end angle=360];
		\points
	\end{scope}
	\begin{scope}[shift={(4,-6)}]
		\draw[V0]
			(-0.5,0) arc[start angle=180, end angle=360, y radius=1, x radius=0.75]
			( 1.5,0) arc[start angle=180, end angle=360, y radius=0.8, x radius=1.75];
		\begin{scope}
			\clip (0,0) arc[x radius=1, y radius=1.5, start angle=180, end angle=360] -- cycle;
			\draw[V2,-<-=0.5] (1,0) arc[start angle=0, end angle=-90, y radius=1, x radius=0.75];
		\end{scope}
		\begin{scope}
			\clip
				(2,0) arc[x radius=1, y radius=1.5, start angle=360, end angle=270] --
				(3.5,-1.5) arc[x radius=0.5, y radius=1.5, start angle=270, end angle=180]
				-- cycle;
			\draw[V2,-<-=0.7]
				( 1.5,0) arc[start angle=180, end angle=270, y radius=0.8, x radius=1.75];
		\end{scope}
		\begin{scope}
			\clip (4,0) arc[x radius=0.5, y radius=1.5, start angle=0, end angle=-90] -| (5.5,0);
			\draw[V2,->-=0.5] ( 5,0) arc[start angle=0, end angle=-90, y radius=0.8, x radius=1.75];
		\end{scope}			
		\draw[V1]
			(0,0) arc[x radius=1, y radius=1.5, start angle=180, end angle=360]
			(3,0) arc[x radius=0.5, y radius=1.5, start angle=180, end angle=360];
		\points
	\end{scope}
	\caption[-8.25] {Step 4: connect bottom endpoints in all possible ways
	to obtain a~web basis.};
\end{tikzpicture}
	\protected\gdef\mypoints{\begin{tikzpicture}[x=1em, baseline=-0.5ex]
		\foreach \sign/\x in {-1/0,+2/1,+1/2,-1/3,+1/4,-2/5}
			\expandafter\webendpoint\sign(\x,0);
	\end{tikzpicture}}%
	\caption{%
		The~construction of a~cup basis for
		$\bdry = \bdrypoints{-1,+2,+1,-1,+1,-2}$,
		together with completions of the~cup webs (the~edges erased
		in the~third step are drawn as dashed arcs). The~first three
		steps follow the~proof of Proposition~\ref{prop:blue->binded}.
	}
	\label{fig:cup-web}
\end{figure}

\subsection{Blanchet--Khovanov algebras}

Fix a~balanced collection of points $\bdry$ and let $\mathcal B$
be a~cup basis of $\set{Web}(\bdry)$.

\begin{defn}
	The~\emph{Blanchet--Khovanov algebra} $\webalg{\mathcal B}$
	associated with $\mathcal B$ is the~direct sum of spaces of
	foams with corners
	\begin{equation}\label{def-eq:webalg}
		\webalg{\mathcal B} :=
		\bigoplus_{a,b\in\mathcal B}
			\Hom_{\cat{Foam}(\bdry)}(a,b)
	\end{equation}
	with multiplication given by the~composition (and zero if foams cannot
	be composed).
\end{defn}

\begin{rmk}
	The~above algebra appeared first in \cite{WebAlgebras} for $\bdry$
	a~collection of positively oriented blue points followed by negatively
	oriented red points, the~latter drawn in \cite{WebAlgebras} at the~bottom.
\end{rmk}

Choose a~completion $a^+$ for any cup web $a$ and write
$\revmatching a$ (resp.\ $\revmatching{(a^+)}$) for the~result
of reflecting $a$ (resp.\ $a^+$) along the~horizontal line and
reversing orientation of edges. Using the~natural isomorphisms
$\Hom_{\cat{Foam}(\bdry)}(a,b) \cong \set{Foam}(\revmatching ba)\{n\}$
we can represent elements of the~algebra by dotted completed webs,
in which case the~multiplication is induced from the~family of
\emph{generalized saddles}
\begin{equation}\label{eq:mult-foam}
	\revmatching{(c^+)} b^+ \sqcup
	\revmatching{(b^+)} a^+ \xrightarrow{\ S_{c,b,a}\ }
	\revmatching{(c^+)} a^+
\end{equation}
each consisting of the~identity foams $\revmatching{(c^+)}\times[0,1]$ and
$a^+\times[0,1]$ glued to the~half-rotation of $b^+$ around the~boundary line.
These foams take a~particularly nice form when $\mathcal B$
is a~red-over-blue basis, as they involve then only three types of moves:
\begin{itemize}
	\item merging \eqref{move:blue merge} and splitting \eqref{move:blue split}
	blue loops at points outside of all red circles,
	\item merging unnested red loops \eqref{move:red merge}, and
	\item removing bigons \emph{external} to the~projection of
	      red disks \eqref{move:ext bigon->no bigon}.
\end{itemize}
Hence, the~product of two dotted diagrams is a~positive linear combination of
other diagrams.

\begin{cor}
	The~algebra $\webalg{\mathcal B}$ admits a~positive basis.
\end{cor}

When $\bdry$ is a~collection of blue points oriented in the~alternating
way and $\mathcal B$ consists of oriented cup diagrams (i.e.\ webs with
no red edges), then $\webalg{\mathcal B}$ coincides with the~arc algebra
$\arcalg n$ from \cite{KhArcAlgebras}. Indeed, $\webalg{\mathcal B}$
is the~image of $\arcalg n$ under the~embedding of bicategories
$\EqFunc \colon \cat{BN} \to \cat{Foam}$.
However, when $\mathcal B$ is not a~red-over-blue basis, then the~generalized
saddles \eqref{eq:mult-foam} may involve moves on red arcs that cost a~sign, 
such as splits \eqref{move:red split} or nested saddles \eqref{move:inner merge}
and \eqref{move:inner split}. Hence, cup foams do not constitute a~positive
basis of the~algebra in such case. Yet, it is still isomorphic to
the~arc algebra.

\begin{thm}\label{thm:webalg=arcalg}
	Let $\bdry$ be a~balanced collection of points with $2n$ blue points.
	Then there is an~algebra isomorphism $\webalg{\mathcal B} \cong \arcalg n$
	for any cup basis $\mathcal B$ of webs bounded by $\bdry$. When $\mathcal B$
	is a~red-over-blue basis, then the~isomorphism simply forgets
	red facets of basic cup foams.
\end{thm}
\begin{proof}
	Assume first that $\mathcal B$ is a~red-over-blue basis of type
	$\EqSymb_\bdry$. Then $\webalg{\mathcal B}$ is the~image of $\arcalg n$
	under the~equivalence of categories $\EqFunc_\bdry$, which equips
	a~collection of dotted cups with its standard orientation. The~inverse
	of $\EqSymb_\bdry$ simply forgets red edges in webs and red facets in foams.
	Hence, the~thesis follows.

	Let now $\mathcal B'$ be any cup basis and pick for each cup web
	$a'\in\mathcal B'$ the~isomorphic cup web $a\in\mathcal B$,
	an~invertible foam $I_a \in \cat{Foam}(a, a')$,
	and $s(a) = \pm 1$ such that $\flip I_a\,I_a = s(a) \,\web_a\times[0,1]$.
	Then the~collection of linear isomorphisms
	\[
		\varphi_{ba} \colon \Foam(a',b') \xrightarrow{\ \cong\ } \Foam(a,b),
	\qquad
		S \mapsto s(b)\, \flip I_b\,S\,I_a
	\]
	constitute an~isomorphism of algebras $\webalg{\mathcal B'}
	\cong \webalg{\mathcal B}\rics$, where the~latter
	is isomorphic to $\arcalg n\rics$.
\end{proof}

\begin{exam}
	\def\extcups#1{%
	\draw[V2]
		(1.5,0) arc[x radius=1.5, y radius=0.75, start angle=#1180, end angle=0];
	\ifx#1+%
		\draw[V2,<-] (3,0) ++(100:1.5 and 0.75) -- ++(200:0.15pt and 0.075pt);
	\else
		\draw[V2,<-] (3,0) ++(240:1.5 and 0.75) -- ++(340:0.15pt and 0.075pt);
	\fi
	\fill[white]
		(1,0) arc[x radius=0.5, y radius=1.25, start angle=#1180, end angle=0] --
		(3,0) arc[x radius=0.5, y radius=1.25, start angle=#1180, end angle=0];
	\begin{scope}
		\clip
			(1,0) arc[x radius=0.5, y radius=1.25, start angle=#1180, end angle=0] --
			(3,0) arc[x radius=0.5, y radius=1.25, start angle=#1180, end angle=0] --
			++(0,#11pt) -- (1,#11pt) -- cycle;
		\draw[V0]
			(1.5,0) arc[x radius=1.5, y radius=0.75, start angle=#1180, end angle=0];
	\end{scope}
	\draw[V1]
		(1,0) arc[x radius=0.5, y radius=1.25, start angle=#1180, end angle=0]
		(3,0) arc[x radius=0.5, y radius=1.25, start angle=#1180, end angle=0];
}%
\def\inncups#1{%
	\draw[V2]
		(1.5,0) arc[x radius=1.5, y radius=0.75, start angle=#1180, end angle=0];
	\ifx#1+%
		\draw[V2,<-] (3,0) ++(100:1.5 and 0.75) -- ++(200:0.15pt and 0.075pt);
	\else
		\draw[V2,<-] (3,0) ++(240:1.5 and 0.75) -- ++(340:0.15pt and 0.075pt);
	\fi
	\fill[white]
		(1,0) arc[x radius=1.5, y radius=2, start angle=#1180, end angle=0] --
		(3,0) arc[x radius=0.5, y radius=1.25, start angle=0, end angle=#1180];
	\begin{scope}
		\clip (1,0)
			(1,0) arc[x radius=1.5, y radius=2, start angle=#1180, end angle=0] --
			(3,0) arc[x radius=0.5, y radius=1.25, start angle=0, end angle=#1180];
		\draw[V0]
			(1.5,0) arc[x radius=1.5, y radius=0.75, start angle=#1180, end angle=0];
	\end{scope}
	\draw[V1]
		(1,0) arc[x radius=1.5, y radius=2, start angle=#1180, end angle=0]
		(2,0) arc[x radius=0.5, y radius=1.25, start angle=#1180, end angle=0];
}%
\def\algdiagram #1(#2|#3){\begin{tikzpicture}[x=1.25em,y=1.25em,baseline=(X.base)]
	\node (X) at (0.4,0) {$\phantom+$};
	\node (Y) at (5.1,0) {$\phantom+$};
	\csname #2cups\endcsname+
	\csname #3cups\endcsname-
	\draw[arc hline] (0.5,0) -- (5,0);
	\foreach \x in {#1} do
		\fill[V1] (\x,0) circle[radius=3pt];
\end{tikzpicture}}%
\def\cupdiagram#1{\begin{tikzpicture}[x=1.25em,y=1.25em,baseline=-1.5ex]
	\csname #1cups\endcsname-
	\draw[arc hline] (0.5,0) -- (5,0);
	\node at (1,1ex) {$\scriptstyle+$};
	\node at (2,1ex) {$\scriptstyle+$};
	\node at (3,1ex) {$\scriptstyle-\vphantom+$};
	\node at (4,1ex) {$\scriptstyle+$};
	\node at (4.75,1ex) {$\scriptstyle-\vphantom+$};
\end{tikzpicture}}%
% This is not an actual picture,
	                              % but a collection of macros that
											% are used only in this example.
	Let $\bdry = \bdrypoints{+1, +1, -1, +1, -2}$. Then the~cup basis
	$\mathcal B$ consists of two elements
	\begin{center}
		\cupdiagram{ext}
		\qquad\text{and}\qquad
		\cupdiagram{inn}
	\end{center}
	that form four pairs: two of them have two blue loops and each of the~other
	two has one blue loop. Hence, $\dim\webalg{\mathcal B} = 12$.
	The~multiplication looks a~lot like in $\arcalg 2\rics$. For instance,
	\begin{align*}
		\algdiagram(ext|inn) \cdot \algdiagram(inn|ext)
			&=
		\algdiagram 2(ext|ext) + \algdiagram 3(ext|ext)
		\\ \\
		\algdiagram(inn|ext) \cdot \algdiagram 3(ext|inn)
			&=
		\algdiagram 3,4(inn|inn)
	\end{align*}
	Erasing red edges recovers the~usual diagrammatic calculus of
	$\arcalg 2\rics$.
\end{exam}

\begin{exam}\label{ex:EST-algebra}
Recall the~Blanchet--Khovanov algebra from \cite{WebAlgebras}: it is defined
using webs that have only vertical red edges, $2n$ positive blue endpoints
at the~top and positive $n$ red endpoints at the~bottom. We call them here
\emph{EST webs}.
To fit this construction into our framework, we move the~bottom endpoints
rightwards and to the~top by appending a~collection of nested red cups
(see Figure~\ref{fig:EST web}).
\begin{figure}[ht]%
	\centering
	\begin{tikzpicture}[x=1.5em,y=1.5em]
\begin{scope}[shift={(-8,0)}]
	\draw[V2]
		(1.5,-0.5) -- (1.5,-2.5)
		(3.5,-0.5) -- (3.5,-2.5)
		(5.5,-0.5) -- (5.5,-2.5);
	\draw[V1,fill=white]
		(3,0) arc[radius=1.5, start angle=180, end angle=360];
	\draw[V2]
		(4.5,-0.5) -- (4.5,-1.5);
	\draw[V1]
		(1,0) arc[radius=0.5, start angle=180, end angle=360]
		(4,0) arc[radius=0.5, start angle=180, end angle=360];
	\draw[arc hline] (0.5,0) -- (6.5,0);
	\draw[V2,<-] (1.5,-1.6) -- ++(0,-0.1pt);
	\draw[V2,<-] (3.5,-1.6) -- ++(0,-0.1pt);
	\draw[V2,<-] (5.5,-1.6) -- ++(0,-0.1pt);
	\draw[V2,<-] (4.5,-0.8) -- ++(0,-0.1pt);
\end{scope}
\draw[|->] (-1,-1) -- (1,-1);
\begin{scope}[shift={(1.5,0)}]
	\draw[V2]
		(1.5,-0.5) -- (1.5,-1.5)
			arc[radius=3.5, start angle=180, end angle=360] -- (8.5,0)
		(3.5,-0.5) -- (3.5,-1.5)
			arc[radius=2.0, start angle=180, end angle=360] -- (7.5,0)
		(5.5,-0.5) -- (5.5,-1.5)
			arc[radius=0.5, start angle=180, end angle=360] -- (6.5,0);
	\fill[white]
		(3,0) arc[radius=1.5, start angle=180, end angle=360];
	\draw[V0]
		(1.5,0) -- (1.5,-0.5)
		(2.5,0) -- (2.5,-1.4)
			arc[radius=2.75, start angle=180, end angle=360] -- (8,0)
		(4.5,0) -- (4.5,-0.5) (4.5,-1.5)
			arc[radius=1.25, start angle=180, end angle=360] -- (7,0);
	\begin{scope}
		\clip (3,0) arc[radius=1.5, start angle=180, end angle=360] -- cycle;
		\draw[V0]
			(3.5,0) -- (3.5,-1.5)
			(5.5,0) -- (5.5,-1.5);
	\end{scope}
	\draw[V2]
		(4.5,-0.5) -- (4.5,-1.5);
	\draw[V1]
		(1,0) arc[radius=0.5, start angle=180, end angle=360]
		(3,0) arc[radius=1.5, start angle=180, end angle=360]
		(4,0) arc[radius=0.5, start angle=180, end angle=360];
	\draw[arc hline] (0.5,0) -- (9,0);
	\draw[V2,<-] (6.0,-1.5) ++(225:0.5) -- ++(-25:0.1pt);
	\draw[V2,<-] (5.5,-1.5) ++(240:2.0) -- ++(-25:0.1pt);
	\draw[V2,<-] (5.0,-1.5) ++(240:3.5) -- ++(-25:0.1pt);
	\draw[V2,<-] (4.5,-0.8) -- ++(0,-0.1pt);
\end{scope}
\end{tikzpicture}
	\caption{%
		Turning red edges rightwards and to the~top produces
		a~cup web from an~EST web. There is a~natural completion,
		visualized by dashed arcs, with minima of red cups below
		all blue ones.}%
	\label{fig:EST web}%
\end{figure}
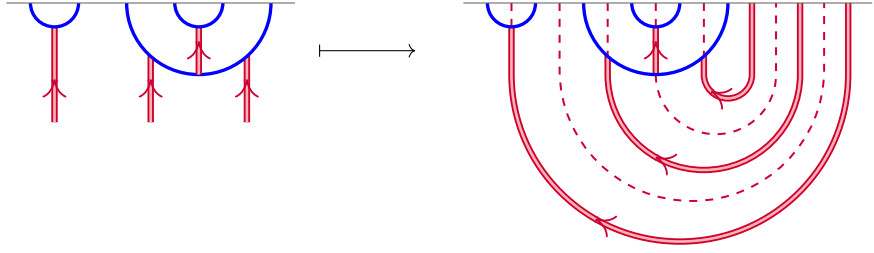
Contrary to the~case of red-over-blue bases, minima of red cups in EST webs appear
below all blue cups, because of which the~formula for multiplication involves
lots of signs.
Yet Theorem~\ref{thm:webalg=arcalg} provides a~direct isomorphism
between this algebra and the~arc algebra. Such an~isomorphism
was explicitly constructed in \cite{WebVsArcAlgebra} by providing a~sign
for each generator of the~algebra, then checking directly that these signs
result in a~homomorphism of algebras.
\end{exam}

\subsection{Blanchet--Khovanov bimodules}

% the space of foams with three parts of boundary
% action of the web algebra & tensor products
% independence of basis and red edges
% corollary: an isomorphism with arc bimodules

Pick now two balanced collections $\bdry$ and $\bdry'$ with cup bases
$\mathcal B$ and $\mathcal B'$ respectively. We assign to a~web
$\web\colon \bdry \to \bdry'$ its \emph{Blanchet--Khovanov bimodule}
$\Fweb^\circ(\web)$, which is
the~$(\webalg{\mathcal B}, \webalg{\mathcal B'})$-bimodule
\begin{equation}
	\Fweb^\circ(\web) := \bigoplus_{\substack{
		a \in \mathcal B\phantom'\\
		b \in \mathcal B'
	}} \cat{Foam}(a,\web,b),
\end{equation}
where $\cat{Foam}(a,\web,b)$ is the~space of foams bounded by
$\revmatching b\cup \web \cup a$ and seen as foams in a~cube with
$\web$ at the~top facet, whereas $a$ and $b$ lie on opposite vertical
facets (see Figure~\ref{fig:Foam(a,w,b)}).
\begin{figure}[ht]%
	\centering
	\begin{tikzpicture}[x=2em,y=2em]
	
		\coordinate (inp1) at (0.4, 3.2);
		\coordinate (inp2) at (0.8, 2.9);
		\coordinate (inp3) at (1.2, 2.6);
		\coordinate (inp4) at (1.6, 2.3);
		
		\coordinate (inp-cup1) at (1.1, 1.9);	% radius: 0.3/1 & 0.1/0.7
		\coordinate (inp-cup2) at (1.1, 1.3);	% radius: 0.7/1.9 & 0.5/1.0
		
		\coordinate (out1) at (5.0, 3.125);
		\coordinate (out2) at (6.0, 2.375);

		\coordinate (out-cup) at (5.6, 1.225);	% radius: 0.6/1.9 & 0.4/1.15
		
		\coordinate (top-cup) at (1.8, 3.03);		% radius: 1.4/0.17 & 1/0.13
		\coordinate (cusp) at (3.2, 1.9);
		
		\coordinate (red1-l) at (1.61641, 2.94688);
		\coordinate (red1-r) at (2.5552,  2.6546 );
		\coordinate (red1-b) at (2.0787,  2.11804);
		
		\coordinate (red2-l) at (4.2, 2.35);
		\coordinate (red2-r) at (5.5, 2.75);
		\coordinate (red2-b) at (5.896, 1.5746);
		
		% cube: back
		\draw[dotted]
			% cube: skew edges
			(0.0,1.5) -- ++(2,-1.5) 
			(4.5,1.5) -- ++(2,-1.5)
			(0,1.5) rectangle ++(4.5,2);

%		% front blue facet: inner
		\fill[1facetFront]
			(inp2) arc[x radius=1, y radius=0.13, start angle=270, end angle=360]
			.. controls ++(0,-0.8) and ++(-0.2,0) .. (cusp)
			.. controls ++(-0.7,0) and ++(0.5,0) .. (red1-b)
			.. controls ++(-0.5,0) and ++(0.7,0) .. (inp-cup1)
			arc[x radius=0.3, y radius=1, start angle=270, end angle=180];

		\draw[1facetLine]
			(top-cup) .. controls ++(0,-0.8) and ++(-0.2,0) .. (cusp);

		% left red facet
		\fill[2facetInner]
			([shift={(1pt,0.2pt)}]red1-l) .. controls ++(0.2,-0.1) and ++(-0.1,0.1) ..
			([shift={(1pt,0.2pt)}]red1-r) .. controls ++(-0.2,-0.8) and ++(0,-1) .. cycle;
		\draw[2facetLine]
			([shift={(1pt,0.2pt)}]red1-l) .. controls ++(0.2,-0.1) and ++(-0.1,0.1) ..
			([shift={(1pt,0.2pt)}]red1-r);
		\draw[seam,shift={(1pt,0.2pt)}]
			([shift={(1pt,0.2pt)}]red1-r) .. controls ++(-0.2,-0.8) and ++(0,-1) ..
			([shift={(1pt,0.2pt)}]red1-l);

		\fill[2facetFront]
			([shift={(-1pt,-0.4pt)}]red1-l) .. controls ++(0.2,-0.1) and ++(-0.1,0.1) ..
			([shift={(-1pt,-0.4pt)}]red1-r) .. controls ++(-0.2,-0.8) and ++(0,-1) ..
			cycle;
		\draw[2facetLine]
			([shift={(-1pt,-0.4pt)}]red1-l) .. controls ++(0.2,-0.1) and ++(-0.1,0.1) ..
			([shift={(-1pt,-0.4pt)}]red1-r);
		\draw[seam,shift={(-1pt,-0.4pt)}]
			([shift={(-1pt,-0.4pt)}]red1-r) .. controls ++(-0.2,-0.8) and ++(0,-1) ..
			([shift={(-1pt,-0.4pt)}]red1-l);

		% back blue facet: inner & outer
		\fill[1facetBack]
			(inp1) arc[x radius=0.7, y radius=1.9, start angle=180, end angle=270]
			.. controls ++(1,0) and ++(-1,0) .. (3.5,0.8)
			.. controls ++(1,0) and ++(-1,0) .. (out-cup)
			arc[x radius=0.6, y radius=1.9, start angle=270, end angle=180]
			.. controls ++(-2,0) and ++(3,0) .. (inp3)
			arc[x radius=0.1, y radius=0.7, start angle=360, end angle=270]
			arc[x radius=0.3, y radius=1.0, start angle=270, end angle=180]
			arc[x radius=1.0, y radius=0.13, start angle=270, end angle=360]
			arc[x radius=1.4, y radius=0.17, start angle=0, end angle=90];

		\draw[1facetLine] (out-cup)
			arc[x radius=0.6, y radius=1.9, start angle=270, end angle=180]
			.. controls ++(-2,0) and ++(3,0) .. (inp3)
			arc[x radius=0.1, y radius=0.7, start angle=360, end angle=270]
			arc[x radius=0.3, y radius=1.0, start angle=270, end angle=180]
			arc[x radius=1.0, y radius=0.13, start angle=270, end angle=360]
			arc[x radius=1.4, y radius=0.17, start angle=0, end angle=90]
			arc[x radius=0.7, y radius=1.9, start angle=180, end angle=270]
			(cusp) .. controls ++(-0.7,0) and ++(0.5,0) .. (red1-b)
			.. controls ++(-0.5,0) and ++(0.7,0) .. (inp-cup1);

		% right red facet
		\draw[2facetInner]
			($(red2-b) + (-0.3pt,-0.8pt)$) .. controls ++(-0.3,0) and ++(0.3,-0.5) ..
			($(red2-l) + (-1.2pt,0)$) .. controls ++(45:0.4) and ++(-0.3,0) ..
			($(red2-r) + (-1.0pt,0.75pt)$) .. controls ++(0,-0.5) and ++(-0.3,0.2) ..
			($(red2-b) + (-0.3pt,-0.8pt)$);
		\draw[2facetFront]
			($(red2-b) + (0.3pt,0.8pt)$) .. controls ++(-0.3,0) and ++(0.3,-0.5) ..
			($(red2-l) + (1.2pt,0)$) .. controls ++(45:0.4) and ++(-0.3,0) ..
			($(red2-r) + (1.0pt,-0.75pt)$) .. controls ++(0,-0.5) and ++(-0.3,0.2) ..
			($(red2-b) + (0.3pt,0.8pt)$);

		% front blue facet: outer
		\draw[1facetFront]
			(inp-cup2) arc[x radius=0.5, y radius=1, start angle=270, end angle=360]
			.. controls ++(2,0) and ++(-2,0) .. (out2)
			arc[x radius=0.4, y radius=1.15, start angle=360, end angle=270]
			.. controls ++(-1,0) and ++(1,0) .. (3.5,0.8)
			.. controls ++(-1,0) and ++(1,0) .. (inp-cup2);

		\begin{scope}	% handle
			\clip (4,1.5) arc[radius=0.8, start angle=300, end angle=240]
				-- ++(0,0.5) -| cycle;
			\draw[1facetLine, fill=white, opacity=1] (4,1.4)
				arc[radius=0.8, start angle=60, end angle=120];
		\end{scope}
		\draw[1facetLine] (4,1.5) arc[radius=0.8, start angle=300, end angle=240];
		
		% right seam
		\draw[seam]
			($(red2-b) + (-0.3pt,-0.8pt)$) .. controls ++(-0.3,0) and ++(0.3,-0.5)
				.. ($(red2-l) + (-1.2pt,0)$)
			($(red2-b) + (0.3pt,0.8pt)$) .. controls ++(-0.3,0) and ++(0.3,-0.5)
				.. ($(red2-l) + (1.2pt,0)$);
		
		% cube: front
		\draw[dotted]
			(0.0,3.5) -- ++(2,-1.5)
			(4.5,3.5) -- ++(2,-1.5)
			(2,0.0) rectangle ++(4.5,2);
	
\end{tikzpicture}
	\caption{%
		An~element of $\cat{Foam}(a,\web,b)$ is a~foam in a~cube, bounded
		by the~webs $\web$, $a$, and $b$ at the~top and opposite vertical
		facets of the~cube respectively.}%
	\label{fig:Foam(a,w,b)}%
\end{figure}
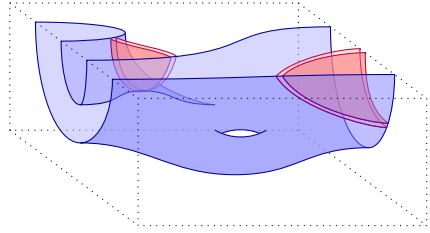
The~algebras $\webalg{\mathcal B}$ and $\webalg{\mathcal B'}$ act on
the~left and on the~right respectively, and there is a~diagrammatic
presentation of this bimodule as explained in Section~\ref{sec:TQFT}.
Moreover, placing a~foam $\foam \in \cat{Foam}(\web, \web')$ on top
results in a~bimodule map $\Fweb^\circ(\foam) \colon \Fweb^\circ(\web)
\to \Fweb^\circ(\web')$.
We check that $\Fweb^\circ(\foam) = \Fweb^\circ(\foam')$ if the~two foams
coincide in $\cat{Foam}$, so that the~map is well-defined. In particular,
$\Fweb^\circ(\web) \cong \Fweb^\circ(\web')$ if $\foam$ is invertible.
Finally, horizontal composition of foams induces a~canonical homomorphism
of graded bimodules
\begin{equation}\label{eq:tensor-isomorphism}
	\Fweb^\circ(\web') \webtimes{\bdry'} \Fweb^\circ(\web)
		\longrightarrow \Fweb^\circ(\web'\web)
\end{equation}
for any pair of composable webs $\web\colon\bdry\to\bdry'$ and
$\web'\colon\bdry'\to\bdry''$. The~proof of \cite[Theorem 1]{KhArcAlgebras}
can be adapted to our framework to show that
\eqref{eq:tensor-isomorphism} is an~isomorphism.

\begin{rmk}
	Contrary to the~case of Blanchet--Khovanov algebras, the~isomorphism
	\eqref{eq:tensor-isomorphism} may not take a~pair of cup foams into
	a~positive combination of cup foams. When using the~diagrammatics of
	completed webs, \eqref{eq:tensor-isomorphism} is induced by a~collection
	of generalized saddles, the~description of which---contrary to the~case
	of algebras---may involve moves on red loops that cost a~sign, such as
	\eqref{move:red split}--\eqref{move:inner split}.
	However, this is not the~case when both webs have only blue endpoints,
	oriented in an~alternating way---in this case all red loops lie inside
	the~webs and are not affected when webs are composed.
\end{rmk}

As in the~case of algebras, $\Fweb^\circ(\web)$ coincides with the~arc bimodule
$\FKh^\circ(\web)$ defined in \cite{KhArcAlgebras} when $\web$ is a~standardly
oriented flat tangle and both $\mathcal B$ and $\mathcal B'$ are collections
of cup diagrams.
Although in general $\Fweb^\circ(\web)$ is not a~priori a~bimodule over arc
algebras, it can be made such through the~algebra isomorphisms
$\arcalg n\cong \webalg{\mathcal B}$ and $\arcalg{n'}\cong \webalg{\mathcal B'}$
provided by Theorem~\ref{thm:webalg=arcalg}.
Hence, it makes sense to compare $\Fweb^\circ(\web)$ with $\FKh^\circ(\web_b)$.

\begin{thm}\label{thm:webmod=arcmod}
	Let $\web\colon\bdry \to \bdry'$ be a~web between balanced
	collections of points with $2n$ and $2n'$ blue points respectively.
	Then there is an~isomorphism of\/ $(\arcalg n,\arcalg{n'})$--bimodules
	$\Fweb^\circ(\web) \cong \FKh^\circ(\web_b)$. The~isomorphism simply forgets
	red facets of basic cup foams when $\mathcal B$ and $\mathcal B'$
	are red-over-blue web bases.
\end{thm}
\begin{proof}
	Assume first that $\mathcal B$ and $\mathcal B'$ are red-over-blue
	cup bases of types $\EqSymb_\bdry$ and $\EqSymb_{\bdry'}$ respectively.
	Fix a~foam $I_\web$ in a~cube with vertical rectangles as blue facets,
	bounded by $\web_b$ and $\web$ at the~bottom and top facets, and with
	$\EqSymb_\bdry$ and $\EqSymb_{\bdry'}$ at appropriate vertical facets.
	Placing it on top of an~element from $\cat{Foam}(a,\web_b,b)$ results
	in a~$\scalars$--linear isomorphism
	$\cat{Foam}(a,\web_b,b) \cong
		\cat{Foam}(\EqSymb_\bdry \cup a, \web, \EqSymb_{\bdry'}\cup b)$.
	It is straightforward to check that these isomorphisms are compatible
	with the~action of the~arc algebras, so that they constitute an~isomorphism
	of bimodules $\FKh^\circ(\web_b) \cong \Fweb^\circ(\web)$; it takes a~collection
	of dotted cups to a~basic cup foam. Forgetting red facets is the~inverse map.
	
	The~general case is reduced to the~above as in the~proof of Theorem~%
	\ref{thm:webalg=arcalg}: choose a~collection of invertible foams,
	one per $a\in\mathcal B$ and one per $b\in\mathcal B'$, and glue them
	to the~sides of foams generating $\FKh^\circ(\web)$.
\end{proof}

\begin{rmk}
	When $\mathcal B$ is a~red-over-blue basis of type $\EqSymb_\bdry$, then
	the~action of $\arcalg n$ can be understood pictorially as follows: a~dotted
	surface $\foam \in \arcalg n$ is standardly
	oriented and combined with $\EqSymb_\bdry \times [0,1]$ before acting on
	$\FKh^\circ(\web)$. The~same applies to $\arcalg{n'}$ if $\mathcal B'$
	is a~red-over-blue basis.
\end{rmk}

We say that a~linear basis $\{x_1,\dots,x_d\}$ of an~$(A,B)$-bimodule
is \emph{positive with respect to bases} $\{a_i\}$ of $A$ and $\{b_j\}$
of $B$, when each $a_ix_k$ and $x_kb_j$ has positive coefficients
in this basis. Because dotted cups constitute a~positive basis of
arc bimodules, Theorem~\ref{thm:webmod=arcmod} implies the~existence
of a~positive basis for Blanchet--Khovanov bimodules.

\begin{cor}
	Suppose that both $\mathcal B$ and $\mathcal B'\rics$ are red-over-blue
	cup bases of webs. Then basic cup foams constitute a~positive basis
	for $\Fweb^\circ(\web)$.
\end{cor}

Although the~formulas for the~actions of the~algebras on a~Blanchet--Khovanov
bimodule involve no signs when red-over-blue bases as used, this is not the~case
for action of foams: the~following squares commutes only up to a~sign
\begin{equation}\label{eq:FKh-vs-Fweb}
	\vcenter{\hbox{\begin{tikzpicture}[x=8em,y=10ex]
		\node[anchor=base] (web1) at (0,1) {$\Fweb^\circ(\web)$};
		\node[anchor=base] (web2) at (1,1) {$\Fweb^\circ(\web')$};
		\node[anchor=base] (tan1) at (0,0) {$\FKh^\circ(\web_b)$};
		\node[anchor=base] (tan2) at (1,0) {$\FKh^\circ(\web'_b)$};
		\draw[->] (web1) -- (web2) node[midway,above] {$\scriptstyle \Fweb^\circ(S)$};
		\draw[->] (tan1) -- (tan2) node[midway,above] {$\scriptstyle \FKh^\circ(S_b)$};
		\draw[<-] (web1) -- (tan1) node[midway,left ] {$\scriptstyle I_\web$}
		                           node[midway,right] {$\scriptstyle \cong$};
		\draw[<-] (web2) -- (tan2) node[midway,left ] {$\scriptstyle I_{\web'}$}
		                           node[midway,right] {$\scriptstyle \cong$};
	\end{tikzpicture}}}
\end{equation}
where we abuse the~notation and denote the~isomorphism from the~proof of
Theorem~\ref{thm:webmod=arcmod} by the~same symbol $I_\web$ as the~foam
used to construct it. However, the~sign does not depend on the~direct summand
of the~bimodule: it is determined by the~configuration of red loops (see
Section~\ref{sec:TQFT}) and the~configuration is the~same for all closures
$\revmatching b\web a$.

\subsection{A~functorial homology for tangles with balanced boundaries}

The~previous sections describe a~morphism fo bicategories
$\Fweb^\circ\colon\cat{Foam}^\circ \to \cat{Bimod}$, which we
extend naturally to $\HCom(\cat{Foam}^\circ)$. As mentioned
in the~introduction, we can apply it to the~formal bracket
$\wKhBracket{T}$ of a~tangle $T$ with balanced input and output,
producing a~chain complex $\WebCom(T)$. Invariance and functoriality
of the~bracket implies that the~homotopy type of $\WebCom(T)$ is
an~invariant of the~tangle $T$ that is functorial with respect to
tangle cobordisms.

\insertpretheorem{thm:FKh-vs-Fweb}
\begin{proof}
	The~two functors coincide on objects by Theorem~\ref{thm:webalg=arcalg}
	and on 1-morphisms by Theorem~\ref{thm:webmod=arcmod}. Furthermore,
	the~collection of isomorphisms $I_\web$ is natural in $\web$, because
	the~square \eqref{eq:FKh-vs-Fweb} commutes when $\foam_b$ is replaced
	with $\EqFunc^\vee(\foam)$. Indeed, the~sign relating $\Fweb^\circ(S)
	\circ I_\web$ with $I_{\web'} \circ \FKh^\circ(\foam_b)$ is exactly the~one
	provided by $\EqFunc^\vee$. The~last statement is a~direct consequence
	of Theorem~\ref{thm:strictification}.
\end{proof}

\section{Subquotient algebras and an~invariant for all tangles}
\label{sec:subquotients}

Inspired by \cite{ChenKhov} we use the~TQFT from the~previous section
to define a~family of 2-functors $\Fweb^\lambda\colon \cat{Foam}
\to \cat{Bimod}$ parametrized by $\lambda\in\Z$, which are defined on
the~whole bicategory of foams. As before, these 2-functors lead to invariant
chain complexes for tangles that are strictly functorial versions of
the~Chen--Khovanov tangle invariants. Contrary to the~previous sections,
we assume here that $h=t=0$. In particular, a~foam vanishes when it has
a~blue facet with two dots.

\subsection{Balancing}
\label{sec:balancing}

Suppose that $\bdry$ has $m$ red and $n$ blue points and choose
$0\leqslant k \leqslant n$. We say that a~sequence
$\bdry^\circ$ on a~line with platforms is a~\emph{balancing of\/ $\bdry$ of
type $n-2k$} if it is balanced and obtained from $\bdry$ by placing $\ell$
and $r$ blue points to the~left and right of $\bdry$ respectively, where
$r-\ell = n-2k$. We say that the~extra points lie on \emph{platforms},
which are drawn as dashed lines.
In what follows we describe two methods to balance a~given sequence,
see Figure~\ref{fig:balancing-a-sequence}.

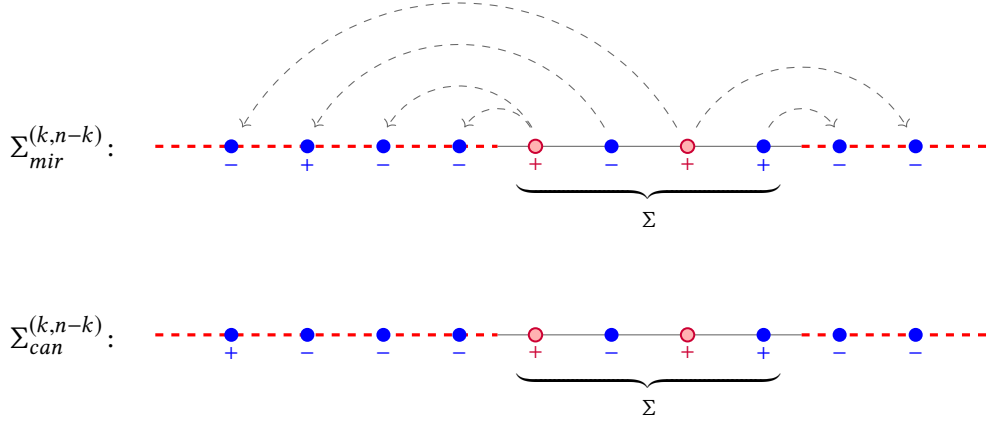
\begin{figure}[ht]%
	\centering
	\begin{tikzpicture}
		\useasboundingbox (-4, -3.5) rectangle (7, 2.25);
		\draw[arc hline] (0.5,0) -- (4.5,0);
		\webendpoint[below]+2(1,0);
		\webendpoint[below]-1(2,0);
		\webendpoint[below]+2(3,0);
		\webendpoint[below]+1(4,0);
		\draw[arc platform] (-4,0) -- (0.5,0);
		\webendpoint[below]-1(-3,0);
		\webendpoint[below]+1(-2,0);
		\webendpoint[below]-1(-1,0);
		\webendpoint[below]-1( 0,0);
		\draw[arc platform] ( 4.5,0) -- (7,0);
		\webendpoint[below]-1(5,0);
		\webendpoint[below]-1(6,0);
		\draw[help line,->] (0.5, 0.0 ) ++(30:0.5) arc[radius=0.5, start angle=30, end angle=150];
		\draw[help line,->] (0.0,-0.3 ) ++(30:1.1) arc[radius=1.1, start angle=30, end angle=150];
		\draw[help line,->] (0.0,-0.85) ++(30:2.2) arc[radius=2.2, start angle=30, end angle=150];
		\draw[help line,->] (0.0,-1.4 ) ++(30:3.3) arc[radius=3.3, start angle=30, end angle=150];
		\draw[help line,->] (4.5, 0.0 ) ++(150:0.5) arc[radius=0.5, start angle=150, end angle=30];
		\draw[help line,->] (4.5,-0.55) ++(150:1.6) arc[radius=1.6, start angle=150, end angle=30];
%
%		\draw[help line,->] (0.5, 0.0 ) ++(30:0.5) arc[radius=0.5, start angle=30, end angle=150];
%		\draw[help line,->] (0.5,-0.55) ++(30:1.6) arc[radius=1.6, start angle=30, end angle=150];
%		\draw[help line,->] (0.5,-1.1 ) ++(30:2.7) arc[radius=2.7, start angle=30, end angle=150];
%		\draw[help line,->] (4.5, 0.0 ) ++(150:0.5) arc[radius=0.5, start angle=150, end angle=30];
%		\draw[help line,->] (4.5,-0.55) ++(150:1.6) arc[radius=1.6, start angle=150, end angle=30];
%
		\node[below] at (2.5, -1.25ex) {$\underbrace{\hskip 3.5cm}_{\bdry}$};
		\node[left] at (-4.25,0) {$\bdry_{mir}^{(k,n-k)}\colon$};

	\begin{scope}[shift={(0,-2.5)}]
		\useasboundingbox (-4, -1) rectangle (7, 1.75);
		\draw[arc hline] (0.5,0) -- (4.5,0);
		\webendpoint[below]+2(1,0);
		\webendpoint[below]-1(2,0);
		\webendpoint[below]+2(3,0);
		\webendpoint[below]+1(4,0);
		\draw[arc platform] (-4,0) -- (0.5,0);
		\webendpoint[below]+1(-3,0);
		\webendpoint[below]-1(-2,0);
		\webendpoint[below]-1(-1,0);
		\webendpoint[below]-1( 0,0);
		\draw[arc platform] ( 4.5,0) -- (7,0);
		\webendpoint[below]-1(5,0);
		\webendpoint[below]-1(6,0);
		\node[below] at (2.5, -1.25ex) {$\underbrace{\hskip 3.5cm}_{\bdry}$};
		\node[left] at (-4.25,0) {$\bdry_{can}^{(k,n-k)}\colon$};
	\end{scope}
\end{tikzpicture}
	\caption{A~visualization of two ways to balance a~sequence for $k=2$.
	For the~mirror balancing (above) replace each red point with two blue
	points first, then copy the~left $k+m=4$ points to the~left and
	the~remaining ones to the~right platform. In the~canonical balancing
	(below) the~points on each platform are ordered with respect to
	their orientation.}%
	\label{fig:balancing-a-sequence}%
\end{figure}

The~\emph{mirror balancing} $\bdry^{k,n-k}_{mir}$ of $\bdry$ of type
$\lambda = n-2k$ is constructed as follows.
First, replace each red point by two blue points oriented the~same way
and call the~new sequence $\bdry'$. Then $\bdry^{k,n-k}_{mir}$ is obtained
from $\bdry$ by placing the~first $k+m$ points of $\bdry'$ on
the~left and the~remaining ones on the~right platform, both in the~reversed
order; we also reflect the~orientation of points (compare this with
Figure~\ref{fig:balancing-a-sequence}). It is a~balanced sequence,
which is an~alternating sequence of blue points if $\bdry$ is such.
However, it depends heavily on the~orientations of points of $\bdry$.
The~next construction does not share his drawback.

The~\emph{canonical balancing} $\bdry^{k,n-k}_{can}$ of type $\lambda = n-2k$
is constructed again by placing $k+m$ points on the~left and $n-k+m$
points on the~right platform, except that now we order the~points in a~way,
such that, when read from left to right, positive points on each platform
appear first. Moreover, we want the~minimal number of negative
(resp.\ positive) points on the~left (resp.\ right) platform. This leads
to one of the~following distributions, depending on the~total weight
$w = w(\bdry)$ of the~sequence $\bdry$.

\begingroup
\def\pointswhen#1:#2x#3:#4x#5:{%
	\par\medskip\noindent\emph{Case $#1$}\nopagebreak
	\par\noindent\ \hskip 1.5em\begin{tabular}{||ll}
		Left platform: & \ifx\relax#2\relax\else place $#2$ points of type $#3$,
			then \fi fill with $+$'s.\\
		Right platform:& \ifx\relax#4\relax\else place $#4$ points of type $#5$,
			then \fi fill with $-$'s.
	\end{tabular}
}%
\pointswhen
	|w| \geqslant |\lambda|:%
	{\frac 12(|w|-\lambda)}x{-\sgn(w)}:%
	{\frac 12(|w|+\lambda)}x{-\sgn(w)}:%
\pointswhen
	|w| < \lambda:%
	x:%
	{\frac 12(\lambda-w)}x+:%
\pointswhen
	|w| < -\lambda:%
	{\frac 12(w-\lambda)}x-:%
	x:%
\endgroup

\medskip\noindent
We check directly that in each case we obtain a~balanced sequence
with at most $k+m$ negative and at most $n-k+m$ positive points on
the~left and right platform respectively.

\begin{rmk}
	The~distribution of points on platforms in $\bdry^{k,n-k}_{can}$
	depends only on the~total weight $w$ of the~sequence and the~type
	$\lambda$ of the~balancing, but not directly on the~number of points
	nor their orientation. This is why we call it \emph{canonical}.
\end{rmk}

\subsection{Webs and foams with platforms}

We now allow foams to meet the~side vertical facets of the~ambient cube in
collections of horizontal blue lines. More precisely, fix a~web $\web \colon
\bdry_0 \to \bdry_1$ together with balanced collections $\bdry_0^\circ$ and
$\bdry_1^\circ$, such that the~first $\ell$ and last $r$ points of both
$\bdry_1^\circ$ and $\bdry_1^\circ$ are blue, oriented the~same way,
and removing them recovers $\bdry_1$ and $\bdry_2$ respectively.
Given cup webs $a$ and $b$ bounded by $\bdry_0^\circ$ and $\bdry_1^\circ$
respectively, we write $\cat{pFoam}^{(\ell,r)}(a,\web,b)$ for the~space of
foams embedded in a~cube with the~following boundary:
\begin{itemize}
	\item the~web $\web$ at the~top facet of the~cube,
	\item $\ell$ and $r$ horizontal blue lines at the~vertical facets
			to the~left and to the~right of $\web$ respectively, and
	\item the~cup webs $a$ and $b$ at the~vertical facets attached to the~input
			and output of $\web$.
\end{itemize}
Figure~\ref{fig:cob with platforms} provides examples of such foams.
\begin{figure}[ht]%
	\centering
	\begin{tikzpicture}[x=1.5em,y=1.5em]

		\draw[stealth-] ( 2.6,2.9) -- ++(-3pt,2pt)
			arc[x radius=0.5, y radius=1, start angle=245, end angle=180]
			node[above] {\small left platform};
		\draw[dotted]
			(0,0.7) -- ++(1.8,-1.2)
			(4,0.7) -- ++(1.8,-1.2)
			(0,0.7) rectangle ++(4,2.5);
		\fill[red!30,opacity=0.3]
			(0,0.7) rectangle ++(4,2.5);
		
		% back
		\fill[1facetBack]
			(0,1.6) .. controls ++(0.5,0.1) and ++(0,-0.5) .. (0.6,2.8)
			arc[x radius=1.7, y radius=0.2, start angle=90, end angle=0]
			.. controls ++(0,-1) and ++(0,-1) .. (3.5,2.6)
			arc[x radius=1.1, y radius=0.2, start angle=180, end angle=90]
			.. controls ++(0,-0.5) and ++( 0.5,0.1) .. (4,1.6) -- (0,1.6);
		\draw[1facetLine]
			(4.6,2.8) .. controls ++(0,-0.5) and ++( 0.5,0.1) .. (4,1.6) -- (0,1.6);
		
		% red facet - back
		\draw[2facetInner, draw=seamColor]
			(2.3,2.6) ++(-0.5pt,0.75pt)
			.. controls ++(0,-1) and ++(0,-1) .. ([yshift=0.75pt]3.5,2.6);
		\draw[2facetLine]
			(2.3,2.6) ++(-0.5pt,0.75pt) -- ([yshift=0.75pt]3.5,2.6);

		% front
		\fill[1facetFront]
			(1.2,2.4) .. controls ++(0,-0.5) and ++(-0.5,0.2) .. (1.8,0.4) --
			(5.8,0.4) .. controls ++(-0.5,0.2) and ++(0,-0.5) .. (5.2,2.4)
			arc[x radius=1.7, y radius=0.2, start angle=270, end angle=180]
			.. controls ++(0,-1) and ++(0,-1) .. (2.3,2.6)
			arc[x radius=1.1, y radius=0.2, start angle=0, end angle=-90];
		\draw[1facetLine]
			(0,1.6) .. controls ++(0.5,0.1) and ++(0,-0.5) .. (0.6,2.8)
			arc[x radius=1.7, y radius=0.2, start angle=90, end angle=0]
			arc[x radius=1.1, y radius=0.2, start angle=0, end angle=-90]
			.. controls ++(0,-0.5) and ++(-0.5,0.2) .. (1.8,0.4) -- (5.8,0.4)
			(5.8,0.4) .. controls ++(-0.5,0.2) and ++(0,-0.5) .. (5.2,2.4)
			arc[x radius=1.7, y radius=0.2, start angle=270, end angle=180]
			arc[x radius=1.1, y radius=0.2, start angle=180, end angle= 90]
			(2.3,2.6) .. controls ++(0,-1) and ++(0,-1) .. (3.5,2.6);
			
		% red facet - front
		\draw[2facetFront, draw=seamColor]
			(3.5,2.6) ++(0.5pt,-0.75pt)
			.. controls ++(0,-1) and ++(0,-1) .. ([yshift=-0.75pt]2.3,2.6);
		\draw[2facetLine]
			(3.5,2.6) ++(0.5pt,-0.75pt) -- ([yshift=-0.75pt]2.3,2.6);

		\draw[stealth-] ( 5.6,1.5)
			arc[x radius=1, y radius=0.5, start angle=270, end angle=330]
			node[anchor=210] {$b$};
		\draw[dotted]
			(0,3.2) -- ++(1.8,-1.2)
			(4,3.2) -- ++(1.8,-1.2)
			(1.8,-0.5) rectangle ++(4,2.5);
		\fill[red!30,opacity=0.3]
			(1.8,-0.5) rectangle ++(4,2.5);
		% description
		\draw[stealth-] ( 0.2,1.3)
			arc[x radius=1, y radius=0.5, start angle=330, end angle=270]
			node[left] {$a$};
		\draw[stealth-] ( 3.8,-0.1) -- ++(3pt,-2pt)
			arc[x radius=0.5, y radius=1, start angle=65, end angle=0]
			node[below] {\small right platform};
\end{tikzpicture}
\hskip \the\dimexpr 0.5\columnwidth - 12em\relax
\begin{tikzpicture}[x=1.5em,y=1.5em]

		\draw[stealth-] ( 2.6,2.9) -- ++(-3pt,2pt)
			arc[x radius=0.5, y radius=1, start angle=245, end angle=180]
			node[above] {\small left platform};
		\draw[dotted]
			(0,0.7) -- ++(1.8,-1.2)
			(4,0.7) -- ++(1.8,-1.2)
			(0,0.7) rectangle ++(4,2.5);
		\fill[red!30,opacity=0.3]
			(0,0.7) rectangle ++(4,2.5);
		
		% back
		\fill[1facetBack]
			(1.8,0) .. controls ++(-1.2,0.8) and ++(0,-1) .. (0.6,2.8)
			arc[x radius=1.7, y radius=0.2, start angle=90, end angle=0]
			.. controls ++(0,-1) and ++(0,-1) .. (3.5,2.6)
			arc[x radius=1.1, y radius=0.2, start angle=180, end angle=90]
			.. controls ++(0,-1) and ++(-1.2,0.8) .. (5.8,0) -- (1.8,0);
		\draw[1facetLine] (5.8,0) -- 
			(1.8,0) .. controls ++(-1.2,0.8) and ++(0,-1) .. (0.6,2.8)
			(4.6,2.8) .. controls ++(0,-1) and ++(-1.2,0.8) .. (5.8,0);
		
		% red facet - back
		\draw[2facetInner, draw=seamColor]
			(2.3,2.6) ++(-0.5pt,0.75pt)
			.. controls ++(0,-1) and ++(0,-1) .. ([yshift=0.75pt]3.5,2.6);
		\draw[2facetLine]
			(2.3,2.6) ++(-0.5pt,0.75pt) -- ([yshift=0.75pt]3.5,2.6);

		% front
		\fill[1facetFront]
			(1.8,1) .. controls ++(-0.6,0.4) and ++(0,-0.5) .. (1.2,2.4)
			arc[x radius=1.1, y radius=0.2, start angle=-90, end angle=0]
			.. controls ++(0,-1) and ++(0,-1) .. (3.5,2.6)
			arc[x radius=1.7, y radius=0.2, start angle=180, end angle=270]
			.. controls ++(0,-0.5) and ++(-0.6,0.4) .. (5.8,1) -- (1.8,1);
		\draw[1facetLine] (0.6,2.8)
			arc[x radius=1.7, y radius=0.2, start angle=90, end angle=0]
			arc[x radius=1.1, y radius=0.2, start angle=0, end angle=-90]
			.. controls ++(0,-0.5) and ++(-0.6,0.4) .. (1.8,1) -- (5.8,1)
			.. controls ++(-0.6,0.4) and ++(0,-0.5) .. (5.2,2.4)
			arc[x radius=1.7, y radius=0.2, start angle=270, end angle=180]
			arc[x radius=1.1, y radius=0.2, start angle=180, end angle=90]
			(2.3,2.6) .. controls ++(0,-1) and ++(0,-1) .. (3.5,2.6);
			
		% red facet - front
		\draw[2facetFront, draw=seamColor]
			(3.5,2.6) ++(0.5pt,-0.75pt)
			.. controls ++(0,-1) and ++(0,-1) .. ([yshift=-0.75pt]2.3,2.6);
		\draw[2facetLine]
			(3.5,2.6) ++(0.5pt,-0.75pt) -- ([yshift=-0.75pt]2.3,2.6);

		\draw[stealth-] ( 5.6,1.5)
			arc[x radius=1, y radius=0.5, start angle=270, end angle=330]
			node[anchor=210] {$b$};
		\draw[dotted]
			(0,3.2) -- ++(1.8,-1.2)
			(4,3.2) -- ++(1.8,-1.2)
			(1.8,-0.5) rectangle ++(4,2.5);
		\fill[red!30,opacity=0.3]
			(1.8,-0.5) rectangle ++(4,2.5);
		% description
		\draw[stealth-] ( 0.2,1.3)
			arc[x radius=1, y radius=0.5, start angle=330, end angle=270]
			node[left] {$a$};
		\draw[stealth-] ( 3.8,0.4) -- ++(3pt,-2pt)
			arc[x radius=0.5, y radius=1.4, start angle=65, end angle=0]
			node[below] {\small right platform};
\end{tikzpicture}
	\caption{Examples of foams with platforms. The~right surface is killed,
	because it is connected and intersects the~right platform (the~front facet
	in the~picture) in two lines. The~left surface is not killed unless it
	carries a~dot.}%
	\label{fig:cob with platforms}%
\end{figure}
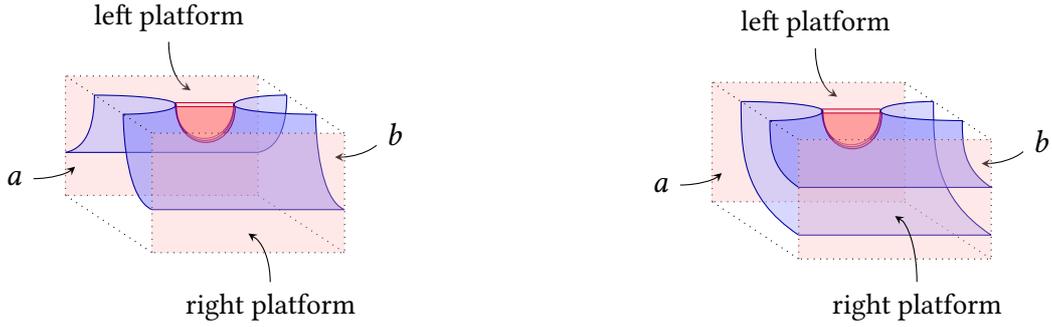
We say that such a~foam is \emph{violating} if it has a~connected
component that either
\begin{itemize}
	\item meets a~platform in more than one lines, or
	\item intersects a~platform and carries a~dot.
\end{itemize}
It is straightforward to check that the~property of being a~violating
foam is preserved by foam relations. Hence, violating foams generate
a~linear subspace of $\cat{pFoam}^{(\ell,r)}(a,\web,b)$.
We write $\cat{Foam}^{(\ell,r)}(a,\web,b)$ for the~quotient space,
or simply $\cat{Foam}^{(\ell,r)}(a,b)$ when $\web$ is the~identity web.

Gluing foams horizontally results in a~linear map
\[
	\cat{pFoam}^{(\ell,r)}(a,\web_0,b) \otimes
	\cat{pFoam}^{(\ell,r)}(b,\web_1,c) \longrightarrow
	\cat{pFoam}^{(\ell,r)}(a,\web_0\web_1,c)
\]
and it is straightforward to notice that a~foam $\foam'\foam$ is violating
when either $\foam$ or $\foam'$ is a~violating foam. Hence, there is an~induced
linear map
\[
	\cat{Foam}^{(\ell,r)}(a,\web_0,b) \otimes
	\cat{Foam}^{(\ell,r)}(b,\web_1,c) \longrightarrow
	\cat{Foam}^{(\ell,r)}(a,\web_0\web_1,c).
\]

Consider now \emph{webs with platforms} as discussed in
Section~\ref{sec:balancing}. Their blue arcs fall into three
families visualized in Figure~\ref{fig:blue-arc-types}:
\begin{itemize}%[label={(A\arabic{*})}, ref={A\arabic{*}}]
	\item\label{arc:inner} \emph{inner arcs}, with at least one
	endpoint not on a~platform,
	\item\label{arc:outer} \emph{outer arcs}, with each endpoint
	on a~different platform, and,
	\item\label{arc:bad} \emph{violating arcs}, with both endpoints
	on the~same platform.
\end{itemize}
Webs with no violating arcs and no red endpoints on platforms are \emph{admissible}. 
Outer arcs of an~admissible web are nested one in another and the~most nested
one of them encloses all inner arcs. Notice that $\cat{Foam}^{(\ell,r)}(a,\web,b)
= 0$ when either $a$ or $b$ has a~violating arc.
\begin{figure}[ht]%
	\centering
		\begin{tikzpicture}[x=1.25em,y=1.25em]
		\def\platforms[#1,#2,#3]{%
			\draw[arc hline] (#1,0) -- (#2,0);
			\draw[arc platform] (0,0) -- (#1,0) (#2,0) -- (#3,0);
		}
		\def\description{%
			\node[anchor=north, font=\small, align=center, text width=5cm]
			at (3.5,-2)
		}
		\begin{scope}[shift={(-9,0)}]
			\platforms[2,5.7,7]
			\foreach \x in {1,3,4,5} do
				\fill[V1lineColor] (\x,0) circle[radius=2pt];
			\draw[V1]
				(1,0) arc[start angle=180, end angle=360, x radius=2, y radius=1.5]
				(3,0) arc[start angle=180, end angle=360, radius=0.5];
			\description {Inner arcs};
		\end{scope}
		\begin{scope}
			\platforms[2.5,4.5,7]
			\fill[V1lineColor]
				(1,0) circle[radius=2pt]
				(6,0) circle[radius=2pt];
			\draw[V1]
				(1,0) arc[start angle=180, end angle=360, x radius=2.5, y radius=1.5];
			\description {An~outer arc};
		\end{scope}
		\begin{scope}[shift={(9,0)}]
			\platforms[4,6,7]
			\fill[V1lineColor]
				(1,0) circle[radius=2pt]
				(3,0) circle[radius=2pt];
			\draw[V1]
				(1,0) arc[start angle=180, end angle=360, radius=1];
			\description {A~violating arc};
		\end{scope}
	\end{tikzpicture}
	\caption{Three types of blue arcs in a~cup diagram with platforms.}%
	\label{fig:blue-arc-types}%
\end{figure}
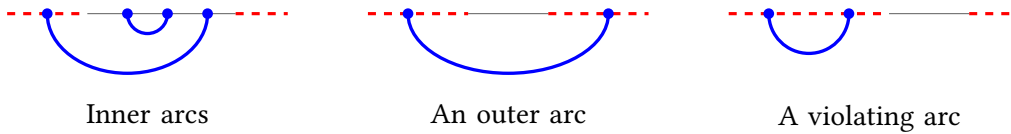

\begin{lem}\label{lem:count-outer-arcs}
	Let\/ $\bdry^\circ$ be a~balancing of\/ $\bdry$ with $\ell$ and $r$ points
	on the~left and on the~right platform, and let $n$ count the~blue points of\/
	$\bdry$. Then an~admissible cup web bounded by $\bdry^\circ$ has at least
	$(\ell+r-n)/2$ outer arcs.
\end{lem}
\begin{proof}
	An~admissible web bounded by $\bdry^\circ$ has at most $n$ inner
	arcs, so that at least $(\ell+r) - n$ points from the~platforms
	must be connected by outer arcs.
\end{proof}

Given a~cup basis $\mathcal B$ of webs bounded by $\bdry^\circ$, we shall
write $\mathcal B^{\ell,r}$ for the~subset of admissible webs with
$\ell$ and $r$ points on the~left and right platform respectively.

\subsection{Stabilization}

We say that a~foam $\widehat\foam$ is a~\emph{stabilization} of a~foam $\foam$,
when it is obtained by placing a~blue horizontal rectangle below $\foam$.
Likewise, \emph{stabilizing a~web} means adding an~additional outer arc.
It follows that $\widehat\foam$ is a~violating foam
if and only if the~foam $\foam$ is violating. Hence, there is
a~well-defined injection
\begin{equation}\label{map:stabilization}
	\cat{Foam}^{(\ell,r)}(a,\web,b) \xrightarrow{\ \widehat{(\blank)}\ }
	\cat{Foam}^{(\ell+1,r+1)}(\hat a, \web, \hat b)
\end{equation}
where $\hat a$ and $\hat b$ are appropriate stabilizations of the~webs.
It is also surjective: by applying the~neck
cutting relation \eqref{rel:neck-cutting} we can write every foam
$\foam \in \cat{Foam}^{(\ell,r)}(\hat a, \web, \hat b)$
as a~sum $\foam_0 + \foam_1$, such that the~lowest blue boundary curve bounds
a~blue disk in each $\foam_i$, see Figure~\ref{fig:destabilization}.
Furthermore, stabilization is natural with
respect to placing foams on top as well as to the~horizontal composition
of foams, i.e.
\[
	W \cup_\web \widehat\foam = \widehat{W \cup_\web \foam}
\qquad\text{and}\qquad
	\widehat\foam' \widehat\foam = \widehat{\foam'\foam}
\]
for any $S\in\cat{Foam}^{(\ell,r)}(a,\web,b)$, $S'\in\cat{Foam}^{(\ell,r)}
(b,\web',c)$ and $W\colon \web \to \web''$.

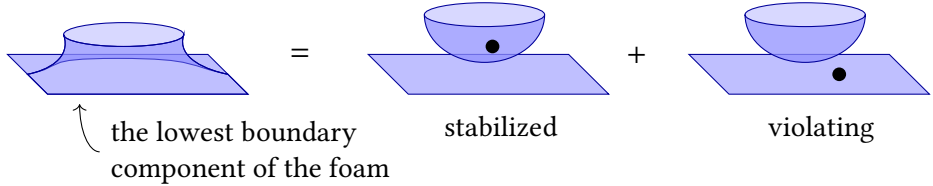
\begin{figure}[ht]%
	\centering
	\begin{tikzpicture}[x=1.25em,y=1.25em]
%\begin{scope}[shift={(-8,0)}]
%	\draw[1facetFront]
%		(-2,-0.5) -- (3,-0.5) -- (2,0.5) -- (-3,0.5) -- cycle;
%	\draw[1facetLine, fill=white,rotate=-1]
%		(0,0) ellipse[x radius=1.5, y radius=0.3];
%\end{scope}
\begin{scope}
	
	\fill[1facetFront] (-3,0.5) --
		(-2.5,0  ) .. controls ++(1,0.2) and ++(-2,0) ..
		(-0.25,0.4) .. controls ++(2,0) and ++(-1,0.2) ..
		(2.5,0) -- (2,0.5) -- cycle;
	\draw[1facetLine] (2.5,0) -- (2,0.5) -- (-3,0.5) -- (-2.5,0);
	
	\fill[1facetBack]
		(-1.6 , 1  ) .. controls ++( 0.2,-0.6) and ++( 0.5, 0.1) .. 
		(-2.5 , 0  ) .. controls ++( 1  , 0.2) and ++(-2  , 0  ) ..
		(-0.25, 0.4) .. controls ++( 2  , 0  ) and ++(-1  , 0.2) ..
		( 2.5 , 0  ) .. controls ++(-0.5, 0.1) and ++(-0.2,-0.6) ..
		( 1.4 , 1  ) arc[x radius=1.5, y radius=0.3, start angle=0, end angle=180];
	\draw[1facetLine]
		(-2.5 , 0  ) .. controls ++(1  , 0.2) and ++(-2  , 0  ) ..
		(-0.25, 0.4) .. controls ++(2  , 0  ) and ++(-1  , 0.2) ..
		( 2.5 , 0  );
	
	\draw[1facetFront]
		(-1.6 , 1  ) .. controls ++( 0.2,-0.6) and ++( 0.5, 0.1) .. 
		(-2.5 , 0  ) -- (-2,-0.5) -- (3,-0.5) --
		( 2.5 , 0  ) .. controls ++(-0.5, 0.1) and ++(-0.2,-0.6) ..
		( 1.4 , 1  ) arc[x radius=1.5, y radius=0.3, start angle=360, end angle=180];
	\draw[1facetLine]
		(-1.6, 1) .. controls ++( 0.2,-0.6) and ++( 0.5, 0.1) .. 
		(-2.5, 0) -- (-2,-0.5) -- (3,-0.5) --
		( 2.5, 0) .. controls ++(-0.5, 0.1) and ++(-0.2,-0.6) .. (1.4, 1)
		(-0.1,1) ellipse[x radius=1.5, y radius=0.3];
	
\end{scope}
\begin{scope}[shift={(9,0)}]
	\draw[1facetFront]
		(-2,-0.5) -- (3,-0.5) -- (2,0.5) -- (-3,0.5) -- cycle;
	\fill[1facetBack] (-1.6,1.5)
		arc[x radius=1.5, y radius=1.2, start angle=180, end angle=360]
		arc[x radius=1.5, y radius=0.3, start angle=0,   end angle=180];
	\fill[1facetFront] (-1.6,1.5)
		arc[x radius=1.5, y radius=1.2, start angle=180, end angle=360]
		arc[x radius=1.5, y radius=0.3, start angle=360, end angle=180];
	\draw[1facetLine]
		(-0.1,1.5) ellipse[x radius=1.5, y radius=0.3]
		(-1.6,1.5) arc[x radius=1.5, y radius=1.2, start angle=180, end angle=360];
	%\foamdot()
	\fill (0.1,0.7) circle[radius=0.5ex];
\end{scope}
\begin{scope}[shift={(17,0)}]
	\draw[1facetFront]
		(-2,-0.5) -- (3,-0.5) -- (2,0.5) -- (-3,0.5) -- cycle;
	\fill[1facetBack] (-1.6,1.5)
		arc[x radius=1.5, y radius=1.2, start angle=180, end angle=360]
		arc[x radius=1.5, y radius=0.3, start angle=0,   end angle=180];
	\fill[1facetFront] (-1.6,1.5)
		arc[x radius=1.5, y radius=1.2, start angle=180, end angle=360]
		arc[x radius=1.5, y radius=0.3, start angle=360, end angle=180];
	\draw[1facetLine]
		(-0.1,1.5) ellipse[x radius=1.5, y radius=0.3]
		(-1.6,1.5) arc[x radius=1.5, y radius=1.2, start angle=180, end angle=360];
	%\foamdot()
	\fill (0.75,0) circle[radius=0.5ex];
\end{scope}
%	\node at (-4.3,0.5) {$=$};
	\node at ( 4.3,0.5) {$=$};
	\node at (12.7,0.5) {$+$};
	\node[below,font=\small] at ( 9.3,-0.75) {stabilized};
	\node[below,font=\small] at (17.3,-0.75) {violating};
	\draw[<-] (-1.2,-0.75)
		arc[x radius=0.5, y radius=1.2, start angle=180, end angle=270]
		node[right,font=\small,align=left] {the~lowest boundary\\ component of the~foam};
\end{tikzpicture}%
	\caption{A~way to destabilize a~foam bounded by stabilized webs.}%
	\label{fig:destabilization}%
\end{figure}

Let $\mathcal B$ be a~cup basis of webs bounded by $\bdry^\circ$ and
$\mathcal B^{\ell,r}$ the~subset of admissible webs. We write
$\widehat{\mathcal B}^{\ell,r}$ for the~set of stabilized basic webs;
they are bounded by a~bigger collection $\widehat{\bdry^\circ}$.
It is in general only a~subset of a~basis of admissible webs bounded
by $\widehat{\bdry^\circ}$. However, it is a~basis when platforms
carry sufficiently many points.

\begin{lem}\label{lem:stabilizing webs}
	Let $\ell$ and $r$ count points of\/ $\bdry^\circ$ on the~left and
	on the~right platform, whereas $n$---the~number of the~remaining blue
	points. Then $\widehat{\mathcal B}^{\ell,r}$ is a~cup basis if\/
	$\ell+r \geqslant n$.
\end{lem}
\begin{proof}
	The~collection $\widehat{\bdry^\circ}$ has $\ell+1$ points on the~left
	and $r+1$ points on the~right platform. Hence, by
	Lemma~\ref{lem:count-outer-arcs}, every admissible web bounded
	by $\widehat{\bdry^\circ}$ has an~outer arc.
\end{proof}

\subsection{Subquotient algebras and bimodules}

We are now ready to construct a~foam version of the~subquotient algebras
and bimodules from \cite{ChenKhov}. Let $\bdry$ be a~sequence of $n$ blue
and $m$ red points, and pick $\lambda = n-2k$ with $0\leqslant k\leqslant n$.
Choose a~balancing $\bdry^\circ$ with $k+m$ and $n-k+m$ points on the~left
and right platform respectively and a~cup basis $\mathcal B$ for webs
bounded by it; the~admissible webs form a~subset $\mathcal B^{k+m,n-k+m}$.

\begin{defn}
	The~\emph{extended Blanchet--Khovanov algebra} $\qtwebalg{\mathcal B,\lambda}$
	is the~direct sum of spaces of foams with platforms
	\begin{equation}\label{eq:subalg-for-subqnt}
		\qtwebalg{\mathcal B,\lambda} :=
		\bigoplus_{a,b\in\mathcal B^{k+m,n-k+m}}
			\cat{Foam}^{(k+m,n-k+m)}(a, b)
	\end{equation}
	with multiplication given by the~composition (and zero if foams cannot
	be composed).
\end{defn}

It follows from the~definition that $\qtwebalg{\mathcal B,\lambda}$
is a~quotient of a~subalgebra of $\webalg{\mathcal B}$. In particular,
it inherits the~description in terms of dotted completed webs with
the~following modifications:
\begin{itemize}
	\item the~horizontal line, along which cup webs are glued, has platform
	      sections on its sides,
	\item we allow only completions of admissible cup foams, and
	\item such a~diagram vanishes when it contains a~blue loop
	      intersecting a~platform at least twice or a~blue loop
			intersecting a~platform and carrying a~dot at the~same time.
\end{itemize}
In particular, $\qtwebalg{\mathcal B,\lambda}$ is isomorphic to the~Chen--%
Khovanov algebra $\CKalg{k,n-k}$ when $\bdry$ consists of $n$ blue points
that are oriented in an~alternating way and $\bdry^\circ$ is the~mirror
balancing. With the~help of the~stabilization map \eqref{map:stabilization}
we can find such isomorphisms for all extended Blanchet--Khovanov algebras.

\begin{thm}\label{thm:qtwebalg=CKalg}
	Let $\bdry$ be a~collection of $m$ red and $n$ blue points
	with a~balancing $\bdry^\circ$ of type $\lambda = n-2k$, where
	$0\leqslant k\leqslant n$.
	Then there is an~algebra isomorphism
	$\qtwebalg{\mathcal B,\lambda} \cong \CKalg{k,n-k}$ for any
	cup basis $\mathcal B$ of webs bounded by $\bdry^\circ$.
	When $\mathcal B$ is a~red-over-blue basis, then the~isomorphism
	simply forgets red facets of basic cup foams and drops $m$ lowest
	blue rectangles.
\end{thm}
\begin{proof}
	We assume that $\mathcal B$ is a~red-over-blue basis---the~general
	case is proven the~same way as in Theorem~\ref{thm:webalg=arcalg}.
	Let $\mathcal B_0$ be the~collection of admissible cup diagrams with $2n$
	blue endpoints, $k$ of which are on the~left and $n-k$ on the~right platform.
	It follows from Lemma~\ref{lem:count-outer-arcs} that each cup web from
	$\mathcal B$ has $m$ outer arcs, so that it is constructed by placing
	an~invertible web $\EqSymb$ on top of cup diagrams from
	$\mathcal B_0$ stabilized $m$ times. Hence, as a~$\scalars$-module,
	$\qtwebalg{\mathcal B,\lambda}$ is isomorphic to
	$\mathrm{stab}^{(m)}(\CKalg{k,n-k})$,
	the~algebra $\CKalg{k,n-k}$ stabilized $m$ times. Because stabilization
	is compatible with horizontal composition of foams, we obtain an~algebra
	isomorphism
	\[
		\CKalg{k,n-k}
			\xrightarrow{\ \widehat{(\blank)}^{(m)}\ }
		\mathrm{stab}^{(m)}(\CKalg{k,n-k})
			\xrightarrow{\ (\EqSymb\times[0,1]) \cup (\blank)\ }
		\qtwebalg{\mathcal B,\lambda}
	\]
	which takes a~dotted surface with platforms, adds $m$ extra horizontal
	rectangles below, and then glues $\EqSymb\times[0,1]$ to it along
	the~top and platforms. The~inverse of this map simply forgets red facets
	and drops the~extra $m$ blue rectangles as desired.
\end{proof}

% Bimodules
We follow the~same ideas to construct a~collection of bimodules
for a~web $\web \colon \bdry_0 \to \bdry_1$. Let $n_i$ be the~number
of blue points in $\bdry_i$. Choose $0\leqslant k_i\leqslant n_i$
for $i=0,1$ such that $n_0-2k_0 = n_1-2k_1 =: \lambda$, and let $\bdry_0^\circ$
and $\bdry_1^\circ$ be the~\emph{canonical} balancings of type $\lambda$,
except that we stabilize one of them, so that both sequences have
the~same numbers of points on platforms: $\ell$ on the~left and $r$
on the~right one. Notice that $\bdry_0$ and $\bdry_1$ have same weight,
so that their balancings agree on platforms. Choose cup bases
$\mathcal B_0$ and $\mathcal B_1$ for $\bdry_0^\circ$ and $\bdry_1^\circ$
respectively. We assign to the~web $\web$ the~$(\qtwebalg{\mathcal B_0,\lambda},
\qtwebalg{\mathcal B_1,\lambda})$-bimodule
\[
	\Fweb^\lambda(\web) := \bigoplus_{\substack{
		a\in\mathcal B_0^{\ell,r} \\
		b\in\mathcal B_1^{\ell,r}
	}} \cat{Foam}^{(\ell,r)}(a,\web,b),
\]
which we call the~\emph{extended Blanchet--Khovanov bimodule of weight $\lambda$}.
The~algebras $\qtwebalg{\mathcal B_0,\lambda}$ and $\qtwebalg{\mathcal B_1,\lambda}$
act on the~left and right by composing foams horizontally, stabilized sufficiently
many times when necessary. It should be also clear that placing a~foam
$W\colon \web \to \web'$ on top induces a~bimodule map
$\Fweb^\lambda(W)\colon \Fweb^\lambda(\web) \to \Fweb^\lambda(\web')$,
and that taking tensor products over the~algebras corresponds to composing webs:
\[
	\Fweb^\lambda(\web)
		\qtwebtimes{\mathcal B,\lambda}
	\Fweb^\lambda(\web')
		\xrightarrow{\ \cong\ }
	\Fweb^\lambda(\web'\web).
\]
As before, Theorem~\ref{thm:qtwebalg=CKalg} allows us to think of
$\Fweb^\lambda(\web)$ as a~bimodule over the~algebras $\CKalg{k_i,n_i-k_i}$,
so that we can compare it with the~bimodule $\CKmod{\web_b; \lambda}$
assigned to the~flat tangle $\web_b$ in \cite{ChenKhov}.
We leave the~following as an~easy exercise.

\begin{thm}\label{thm:qtwebmod=CKmod}
	Let $\web\colon\bdry\to \bdry'$ be a~web between collections of points
	with $n$ and $n'$ blue points respectively and choose $0\leqslant k\leqslant n$
	and $0\leqslant k'\leqslant n'$, such that $n-2k = n'-2k' = \lambda$.
	Then there is an~isomorphism of $(\CKalg{k,n-k}, \CKalg{k',n'-k'})$-bimodules
	$\Fweb^\lambda(\web) \cong \CKmod{\web_b,\lambda}$, which---up to stabilization%
	---forgets red facets of cup foams when $\Fweb^\lambda(\web)$
	is constructed using red-over-blue web bases.
\end{thm}

\subsection{A~functorial homology for all tangles}

The~above sections describe a~family of 2-functors
$\Fweb^\lambda\colon \cat{Foam} \to \cat{Bimod}$ parametrized
with $\lambda\in\Z$, which---as before--we extend to the~bicategory
of formal complexes $\HCom(\cat{Foam})$. Applying $\Fweb^\lambda$ to
the~bracket $\wKhBracket{T}$ of a~tangle $T$ results in a~chain complexes
of bimodules, the~homotopy type of which is an~invariant of $T$ and
which is functorial with respect to tangle cobordisms.

\insertpretheorem{thm:FCKh-vs-Fweb}
\begin{proof}
	The~equivalence of $\Fweb^\lambda$ and $\FKh^\lambda\circ\EqFunc^\vee$
	follows from Theorems~\ref{thm:qtwebalg=CKalg} and \ref{thm:qtwebmod=CKmod}
	along the~same lines as in the~proof of Theorem~\ref{thm:FKh-vs-Fweb}.
	The~second statement is a~direct consequence of Theorem~\ref{thm:strictification}.
\end{proof}

\bibliography{references}

\begin{thebibliography}{10}

\bibitem{APS}
M.~Asaeda, J.~H. Przytycki, and A.~S. Sikora.
\newblock Categorification of the~{K}auffman bracket skein module of
  ${I}$-bundles over surfaces.
\newblock {\em Algebr. Geom. Topol.}, 4(2):1177--1210, 2004.

\bibitem{DrorCob}
D.~Bar-Natan.
\newblock Khovanov's homology for tangles and cobordisms.
\newblock {\em Algebr. Geom. Topol.}, 9(3):1443--1499, 2005.

\bibitem{QntColored}
A.~Beliakova, M.~Hogancamp, K.~K. Putyra, and S.~M. Wehrli.
\newblock Quantum link homology via trace functor {II}: colored homology.
\newblock Work in progress.

\bibitem{QntHom}
A.~Beliakova, K.~K. Putyra, and S.~M. Wehrli.
\newblock Quantum link homology via trace functor {I}.
\newblock {\em Inventiones mathematicae}, 215(2):383--492, 2019.

\bibitem{Blanchet}
C.~Blanchet.
\newblock An oriented model for {K}hovanov homology.
\newblock {\em Journal of Knot Theory and its Ramifications}, 19(02):291--312,
  2010.

\bibitem{BrunStroII}
J.~Brundan and C.~Stroppel.
\newblock Highest weight categories arising from {K}hovanov's diagram algebra
  {II}: {K}oszulity.
\newblock {\em Transformation Groups}, 15(1):1--45, 2010.

\bibitem{BrunStroI}
J.~Brundan and C.~Stroppel.
\newblock Highest weight categories arising from {K}hovanov's diagram algebra
  {I}: cellularity.
\newblock {\em Mosc. Math. J.}, 11:685--722, 2011.

\bibitem{CaprauFoams}
C.~Caprau.
\newblock An $\mathfrak{sl}(2)$ tangle homology and seamed cobordisms.
\newblock \\Preprint: \arXiv{0806.1532v3}.

\bibitem{WebsHowe}
S.~Cautis, J.~Kamnitzer, and S.~Morrison.
\newblock Webs and quantum skew {H}owe duality.
\newblock {\em Mathematische Annalen}, 360(1):351--390, 2014.

\bibitem{ChenKhov}
Y.~Chen and M.~Khovanov.
\newblock An invariant of tangle cobordisms via subquotients of arc rings.
\newblock {\em Fundamenta Mathematicae}, 225:23--44, 2014.

\bibitem{ClarkMorrisonWalker}
D.~Clark, S.~Morrison, and K.~Walker.
\newblock Fixing the functoriality of {K}hovanov homology.
\newblock {\em Geom. Topol.}, 13(3):1499--1582, 2009.

\bibitem{CooperKrushkal}
B.~Cooper and V.~Krushkal.
\newblock Categorification of the {Jones--Wenzl} projectors.
\newblock {\em Quantum Topol.}, 3(2):139--180, 2012.

\bibitem{WebVsArcAlgebra}
M.~Ehrig, C.~Stroppel, and D.~Tubbenhauer.
\newblock Generic $\mathfrak{gl}_2$-foams, web and arc algebras.
\newblock Preprint: \arXiv{1601.08010v2}.

\bibitem{WebAlgebras}
M.~Ehrig, C.~Stroppel, and D.~Tubbenhauer.
\newblock {B}lanchet--{K}hovanov algebras.
\newblock In {\em Categorification and Higher Representation Theory}, volume
  683 of {\em Contemp. Math.}, pages 183--226, Providence, RI, 2017. Amer.
  Math. Soc.

\bibitem{KhRozFunctorial}
M.~Ehrig, D.~Tubbenhauer, and P.~Wedrich.
\newblock Functoriality of colored link homologies.
\newblock {\em Proceedings of the London Mathematical Society},
  117(5):996--1040, 2018.

\bibitem{DiffTop}
M.~W. Hirsch.
\newblock {\em Differential {T}opology}, volume~33 of {\em Graduate Texts in
  Mathematics}.
\newblock Springer-Verlag New York, 1976.

\bibitem{Jacobsson}
M.~Jacobsson.
\newblock An invariant of link cobordisms from {K}hovanov's homology.
\newblock {\em Algebr. Geom. Topol.}, 4(2):1211--1251, 2004.

\bibitem{KhHom}
M.~Khovanov.
\newblock A categorification of the {J}ones polynomial.
\newblock {\em Duke Math. J.}, 101(3):359--426, 2000.

\bibitem{KhArcAlgebras}
M.~Khovanov.
\newblock A functor-valued invariant of tangles.
\newblock {\em Algebr. Geom. Topol.}, 2(2):665--741, 2002.

\bibitem{KhColored}
M.~Khovanov.
\newblock Categorifications of the colored {J}ones polynomial.
\newblock {\em Journal of Knot Theory and Its Ramifications}, 14(01):111--130,
  2005.

\bibitem{KhFunct}
M.~Khovanov.
\newblock An invariant of tangle cobordisms.
\newblock {\em Transactions of the American Mathematical Society},
  358(1):315--327, 2006.

\bibitem{KhViaHowe}
A.~D. Lauda, H.~Queffelec, and D.~E.~V. Rose.
\newblock Khovanov homology is a~skew {H}owe 2-representation of categorified
  quantum {$\mathfrak{sl}_m$}.
\newblock {\em Algebr. Geom. Topol.}, 15(5):2517--2608, 2015.

\bibitem{BasicBicats}
T.~Leinster.
\newblock Basic bicategories.
\newblock Preprint: \arXiv{9810.017}.

\bibitem{OddKh}
P.~Ozsv\'ath, J.~Rasmussen, and Z.~Szab\'o.
\newblock Odd {K}hovanov homology.
\newblock {\em Algebr. Geom. Topol.}, 13(3):1465--1488, 2013.

\bibitem{ViennaTalk}
K.~K. Putyra.
\newblock A quantum colored {$\mathfrak{sl}(2)$} knot homology: three
  approaches, same invariant.
\newblock \textit{Categorification in quantum topology and beyond}, ESI,
  Vienna, 2019.\\
  \textsc{url:~\small\url{http://www.categorification.net/esi19}}.

\bibitem{ChCob}
K.~K. Putyra.
\newblock A 2-category of chronological cobordisms and odd {K}hovanov homology.
\newblock {\em Banach Center Publications}, 103:291--355, 2014.

\bibitem{Hoel}
H.~Queffelec.
\newblock Skein modules from skew {H}owe duality and affine extensions.
\newblock {\em {SIGMA}}, 11(030):1--36, 2015.

\bibitem{KhSurfaces}
H.~Queffelec and P.~Wedrich.
\newblock Khovanov homology and categorification of skein modules.
\newblock \\Preprint: \arXiv{1806.03416}.

\bibitem{SliceGenusBound}
J.~Rasmussen.
\newblock Khovanov homology and the slice genus.
\newblock {\em Inventiones mathematicae}, 182(2):419--447, 2010.

\bibitem{StableKh}
L.~Rozansky.
\newblock A categorification of the stable $\mathfrak{su}(2)$
  {W}itten--{R}eshteikhin--{T}uraev invariant of links in
  $\mathbb{S}^2\times\mathbb{S}^1$.
\newblock Preprint: \arXiv{1011.1958}.

\bibitem{ParabolicO}
C.~Stroppel.
\newblock Parabolic category {$\mathcal O$}, perverse sheaves on
  {G}rassmanians, {S}pringer fibers and {K}hovanov homology.
\newblock {\em Composition {M}ath.}, 145(4):954--992, 2009.

\bibitem{Vogel}
P.~Vogel.
\newblock Functoriality of {K}hovanov homology.
\newblock Preprint: \arXiv{1505.04545v3}.

\end{thebibliography}
\bibliographystyle{plain}

\end{document}